\documentclass[12pt]{amsart}

\pdfoutput=1

\usepackage[text={420pt,660pt},centering]{geometry}

\usepackage{color}
\usepackage{esint,amssymb}
\usepackage{graphicx}
\usepackage{MnSymbol}
\usepackage{mathtools}
\usepackage[colorlinks=true, pdfstartview=FitV, linkcolor=blue, citecolor=blue, urlcolor=blue,pagebackref=false]{hyperref}
\usepackage{microtype}

\usepackage{bm}

\usepackage{scalerel} 
\usepackage{dsfont}
\usepackage[font={footnotesize}]{caption}

\newtheorem{proposition}{Proposition}
\newtheorem{theorem}[proposition]{Theorem}
\newtheorem{lemma}[proposition]{Lemma}
\newtheorem{corollary}[proposition]{Corollary}

\theoremstyle{definition}
\newtheorem{remark}[proposition]{Remark}

\newcommand{\cref}[1]{Corollary~\ref{c.#1}}

\numberwithin{equation}{section}
\numberwithin{proposition}{section}

\newcommand{\A}{\mathcal{A}}
\newcommand{\Ahom}{\overline{\A}}

\newcommand{\Z}{\mathbb{Z}}
\newcommand{\N}{\mathbb{N}}
\newcommand{\R}{\mathbb{R}}

\newcommand{\E}{\mathbb{E}}
\renewcommand{\P}{\mathbb{P}}
\newcommand{\F}{\mathcal{F}}
\newcommand{\Zd}{\mathbb{Z}^d}
\newcommand{\Rd}{{\mathbb{R}^d}}

\newcommand{\ep}{\varepsilon}

\renewcommand{\a}{\mathbf{a}}

\newcommand{\f}{\mathbf{f}}

\renewcommand{\subset}{\subseteq}

\renewcommand{\a}{\mathbf{a}}
\renewcommand{\b}{\mathbf{b}}

\newcommand{\ahom}{{\overbracket[1pt][-1pt]{\a}}}  

\renewcommand{\subset}{\subseteq}

\newcommand{\cu}{{\scaleobj{1.2	}{\square}}}

\renewcommand{\L}{\mathcal{L}}
\renewcommand{\fint}{\strokedint}

\DeclareMathOperator{\dist}{dist}

\DeclareMathOperator*{\osc}{osc}

\newcommand{\X}{\mathcal{X}}  
\newcommand{\Y}{\mathcal{Y}}

\renewcommand{\bar}{\overline}

\renewcommand{\tilde}{\widetilde}

\newcommand{\indc}{\mathds{1}}

\renewcommand{\O}{\mathcal{O}}

\DeclareMathOperator{\data}{data}

\renewcommand{\hat}{\widehat}

\newcommand{\uhom}{{u_{\mathrm{hom}}}}

\renewcommand{\data}{\mathrm{data}}

\begin{document}

\title[Higher-order linearization and regularity in homogenization]{Higher-order linearization and regularity in nonlinear homogenization}

\begin{abstract}
We prove large-scale $C^\infty$ regularity for solutions of nonlinear elliptic equations with random coefficients, thereby obtaining a version of the statement of Hilbert's 19th problem in the context of homogenization. The analysis proceeds by iteratively improving three statements together: (i) the regularity of the homogenized Lagrangian~$\overline{L}$, (ii) the commutation of higher-order linearization and homogenization, and (iii) large-scale $C^{0,1}$-type regularity for higher-order linearization errors. We consequently obtain a quantitative estimate on the scaling of linearization errors, a Liouville-type theorem describing the polynomially-growing solutions of the system of higher-order linearized equations, and an explicit (heterogenous analogue of the) Taylor series for an arbitrary solution of the nonlinear equations---with the remainder term optimally controlled. These results give a complete generalization to the nonlinear setting of the large-scale regularity theory in homogenization for linear elliptic equations.  
\end{abstract}

\author[S. Armstrong]{Scott Armstrong}
\address[S. Armstrong]{Courant Institute of Mathematical Sciences, New York University, 251 Mercer St., New York, NY 10012}
\email{scotta@cims.nyu.edu}

\author[S. J. Ferguson]{Samuel J. Ferguson}
\address[S. J. Ferguson]{Courant Institute of Mathematical Sciences, New York University, 251 Mercer St., New York, NY 10012}
\email{sjf370@nyu.edu}

\author[T. Kuusi]{Tuomo Kuusi}
\address[T. Kuusi]{Department of Mathematics and Statistics, P.O. Box 68 (Gustaf H\"allstr\"omin katu 2), FI-00014 University of Helsinki, Finland}
 \email{tuomo.kuusi@helsinki.fi}

\date{\today}

\maketitle
\setcounter{tocdepth}{1}

\tableofcontents%


\section{Introduction}

\subsection{Motivation: quantitative homogenization for nonlinear equations}

This article is concerned with nonlinear, divergence-form, uniformly elliptic equations of the form
\begin{equation} 
\label{e.pde}
- \nabla \cdot \left( D_p L(\nabla u(x),x ) \right) = 0 \quad \mbox{in} \ U \subseteq \Rd, \quad d\geq 2. 
\end{equation}
The Lagrangian~$L(p,x)$ is assumed to be uniformly convex and regular in~$p$ and possess some mild regularity in~$x$. Furthermore,~$L$ is a stochastic object: it is sampled by a probability measure~$\P$ which is statistically stationary and satisfies a unit range of dependence. This essentially means that $x\mapsto L(\cdot,x)$ is a random field,  valued in the space of uniformly convex functions, the law of which is independent of~$x$ (or, to be precise, periodic in~$x$; see  Subsection~\ref{ss.assumptions} for the assumptions). 
%

\smallskip

The objective is to describe the statistical behavior of the solutions of~\eqref{e.pde}, with respect to the probability measure~$\P$, on \emph{large length scales}. In other words, we want to understand what the solution~$u$ looks like in the case that the ``macroscopic'' domain $U$ is very large relative to ``microscopic''  scale, which is the correlation length scale of the coefficients (taken to the unit scale). 

\smallskip

At a qualitative level, a satisfactory characterization of the solutions, in the regime in which the ratio of these two length scales is large, is provided by the principle of homogenization. First proved in this context by Dal Maso and Modica~\cite{DM1,DM2}, it asserts roughly that a solution of~\eqref{e.pde} is, with probability approaching one, close in~$L^2$ (relative to its size in~$L^2$) to a solution of a deterministic equation of the form
\begin{equation}
\label{e.pde.homog}
-\nabla \cdot \left( D_p\overline{L} \left( \nabla \uhom \right) \right)
= 0 \quad \mbox{in} \ U, 
\end{equation}
for an \emph{effective Lagrangian}~$\overline{L}$ which is also uniformly convex and~$C^{1,1}$. 

\smallskip

This result is of great importance, from both the theoretical and computation points of view, since the complexity of the homogenized equation~\eqref{e.pde.homog} is significantly less than that of~\eqref{e.pde} as it is both deterministic and spatially homogeneous. It hints that the structure of~\eqref{e.pde} should, on large domains and with high probability, possess some of the structure of a constant coefficient equation and thus we may expect it to be more amenable to our analysis than the worst-possible heterogeneous equation of the form~\eqref{e.pde}. In other words, since the Lagrangian~$L$ is sampled by a probability measure~$\P$ with nice ergodic properties, rather than given to us by the devil, can we expect its solutions to have a nicer structure? In order to answer this kind of question, we need to build a \emph{quantitative} theory of homogenization. 

\smallskip

To be of practical use, the principle of homogenization needs be made quantitative. We need to have answers to questions such as these:
\begin{itemize}
\item How large does the ratio of scale separation need to be before we can be reasonably sure that solutions of~\eqref{e.pde} are close to those of~\eqref{e.pde.homog}? In other words, what is the size of a typical error in the homogenization approximation in terms of the size of $U$? 
\item Can we estimate the probability of the unlikely event that the error is  large?
\item What is~$D_p\overline{L}$ and how can we efficiently compute it? How regular can we expect it to be? Can we efficiently compute its derivatives?
\item Can we describe the fluctuations of the solutions?  
\end{itemize}
In this paper we show that~\eqref{e.pde} has a $C^\infty$ structure. In particular, we will essentially answer the third question posed above by demonstrating that the effective Lagrangian~$\overline{L}$ is as regular as~$L(\cdot,x)$ with estimates for its derivatives. We will identify the higher derivatives of $\overline{L}$ as the homogenized coefficients of certain \emph{linearized} equations and give quantitative homogenization estimates for these, implicitly indicating a computational method for approximating them and thus a Taylor approximation for~$\overline{L}$. Finally, we will prove large-scale~$C^{k,1}$ type estimates for solutions of~\eqref{e.pde}, for~$k\in\N$ as large as can be expected from the regularity assumptions on~$L$, a result analogous to Hilbert's 19th problem, famously given for spatially homogeneous Lagrangians by De Giorgi and Nash. Our analysis reveals the interplay between these three seemingly different kinds results:  (i) the regularity of $\overline{L}$, (ii) the homogenization of linearized equations, and (iii) the large-scale regularity of the solutions. 
In analogy to the way that the Schauder estimates are iterated in the resolution of Hilbert's 19th problem, these three statements must be proved together, iteratively in the parameter~$k\in\N$ which represents the degree of regularity of~$\overline{L}$, the order of the linearized equation, and the order of the $C^{k,1}$ estimate. 

\subsection{Background: large-scale regularity theory and its crucial role in quantitative homogenization}

In the last decade, beginning with the work of Gloria and Otto~\cite{GO1,GNO1}, 
a quantitative theory of homogenization has been developed to give precise answers to  questions like the ones stated in the previous subsection. Until now, most of the progress has come in the case of \emph{linear} equations 
\begin{equation} 
\label{e.pde.lin}
-\nabla \cdot \a(x) \nabla u = 0,
\end{equation}
which corresponds to the special case~$L(p,x) = \frac12p\cdot \a(x)p$ of~\eqref{e.pde}, where~$\a(x)$ is a symmetric matrix. By now there is an essentially complete quantitative theory for linear equations, and we refer to the monograph~\cite{AKMbook} and the references therein for a comprehensive presentation of this theory. Quantitative homogenization for the nonlinear equation~\eqref{e.pde} has a comparatively sparse literature; in fact, the only such results of which we are aware are those of~\cite{AS,AM} (see also~\cite[Chapter 11]{AKMbook}), our previous paper~\cite{AFK} and a new paper of Fischer and Neukamm~\cite{FN} which was posted to the arXiv as we were finishing the present article. 

\smallskip

Quantitative homogenization is inextractably linked to regularity estimates on the solutions of the heterogeneous equation. This is not surprising when we consider that the homogenized flux~$D_p\overline{L}(\nabla u_{\mathrm{hom}})$ should be related to the spatial average (say, on some mesoscopic scale) of the heterogeneous flux $D_pL(\nabla u(x),x)$. In order for spatial averages of the latter to converge nicely, we need to have bounds. It could be unfortunate and lead to a very slow rate of homogenization if, for instance, the flux was concentrated on sets of very small measure which percolate only on very large scales. To rule this out we need good great estimates: ideally, we would like to know that the size of the flux on small scales is the same as on large scales, which amounts to a~$W^{1,\infty}$ estimate on solutions. 

\smallskip

Unfortunately, solutions of equations with highly oscillating coefficients do not possess very strong regularity, in general. Indeed, the best deterministic elliptic regularity estimate for solutions of~\eqref{e.pde}, which does not degenerate as the size of the domain~$U$ becomes large, is~$C^{0,\delta}$ in terms of H\"older regularity (the De Giorgi-Nash estimate) and~$W^{1,2+\delta}$ in terms of Sobolev regularity (the Meyers estimate). The tiny exponent $\delta>0$ in each estimate becomes small as the ellipticity ratio becomes large (see~\cite[Example 3.1]{AKMbook}) and thus both estimates are far short of the desired regularity class~$W^{1,\infty} = C^{0,1}$. 

\smallskip

One of the main insights in the quantitative theory of homogenization 
is that, compared to a generic (``worst-case'')~$L$, solutions of the equation~\eqref{e.pde} have much better regularity if~$L$ is sampled by an ergodic probability measure~$\P$. This is an effect of homogenization itself: on large scales,~\eqref{e.pde} should be a ``small perturbation'' of~\eqref{e.pde.homog}, and therefore better regularity estimates for the former can be inherited from the latter. This is the same idea used to prove the classical Schauder estimates. In the context of homogenization, the result states that there exists a random variable~$\X$ (sometimes called the minimal scale) which is finite almost surely such that, for every $\X < r < \frac 12R$ and solution $u\in H^1(B_R)$ of~\eqref{e.pde} with $U=B_R$, we have the estimate
\begin{equation} 
\label{e.C01estintro}
\fint_{B_r} \left| \nabla u \right|^2 
\leq 
C \left( 1+ \fint_{B_R} \left| \nabla u \right|^2 \right).
\end{equation}
Here~$C$ depends only on dimension and ellipticity and~$\fint_U := \frac{1}{|U|}\int_U$ denotes the mean of a function in~$U$. If we could send $r\to 0$ in~\eqref{e.C01estintro}, it would imply that
\begin{equation*} \label{}
\left| \nabla u(0) \right|^2 
\leq 
C \left( 1+ \fint_{B_R} \left| \nabla u \right|^2 \right), 
\end{equation*}
which is a true Lipschitz estimate, same estimate in fact as holds for the homogenized equation~\eqref{e.pde.homog}. As~\eqref{e.C01estintro} is valid only for $r>\X$, it is sometimes called a ``large-scale $C^{0,1}$ estimate'' or a ``Lipschitz estimate down to the microscopic scale." This estimate, first demonstrated in~\cite{AS} in the stochastic setting, is a generalization of the celebrated result in the case of (non-random) periodic coefficients due to Avellaneda and Lin~\cite{AL1}.  Of course, it then becomes very important to quantify the size of~$\X$. The estimate proved in~\cite{AS}, which is essentially optimal, states that $\X$ is bounded up to ``almost volume-order large deviations'': for every $s<d$ and $r\geq1$,
\begin{equation} 
\label{e.SI.X}
\P \left[ \X > r \right] \leq C\exp\left( -cr^s \right).
\end{equation}
Here the constant~$C$ depends only on~$s$,~$d$, and the ellipticity. A proof of this large-scale regularity estimate together with~\eqref{e.SI.X} can be found in~\cite[Chapter 3]{AKMbook} in the linear case and in~\cite[Chapter 11]{AKMbook} for the nonlinear case. The right side of~\eqref{e.SI.X} represents the probability of the unlikely event that the~$L$ sampled by~$\P$ will be a ``worst-case''~$L$ in the ball of radius~$r$. A proof of the optimality of~\eqref{e.SI.X} can be found in~\cite[Section 3.6]{AKMbook}.

\smallskip

This large-scale regularity theory introduced in~\cite{AS} was further developed in the case of~\eqref{e.pde.lin} in~\cite{GNO2,AM,FO,AKM1,AKM}  and now plays an essential role in the  quantitative theory of stochastic homogenization. Whether one employs functional inequalities~\cite{GNO2,DGO} or renormalization arguments~\cite{AKM1,AKM,GO6}, it is a crucial ingredient in the proof of the optimal error estimates in homogenization for~\eqref{e.pde.lin}: see the monograph~\cite{AKMbook} and the references therein for a complete presentation of these developments. 

\smallskip

The large-scale~$C^{0,1}$ estimate is, from one point of view, the best regularity one can expect solutions of~\eqref{e.pde} or~\eqref{e.pde.lin} to satisfy: since the coefficients are rapidly oscillation, there is no hope for the gradient to exhibit continuity on the macroscopic scale. However, as previously shown in the periodic case in~\cite{AL1,AL2}, the solutions of the linear equation~\eqref{e.pde.lin} still have a $C^\infty$ structure. To state what we mean, let us first think of an (interior) $C^{k,1}$ estimate not as a pointwise bound on the $(k+1)$th derivatives of a function, but as an estimate on how well a function may be approximated on small balls by a $k$th order polynomial. By Taylor's theorem, these are of course equivalent in the following sense:
\begin{equation*} \label{}
\left| \nabla^{k+1} u(0) \right|
\simeq 
\limsup_{r\to 0} 
\frac1{r^{k+1}}
\inf_{p\in\mathcal{P}_k} 
\left\| u - p \right\|_{\underline{L}^2(B_r)},
\end{equation*}
where $\mathcal{P}_k$ denotes the set of polynomials of degree at most~$k$ and and we use the notation $\| w \|_{\underline{L}^2(U)}:= \left( \fint_U |w|^2 \right)^{\frac12}$ to denote the volume-normalized~$L^2(U)$ norm. Thus the interior $C^{k,1}$ estimate for a harmonic function can be stated in the form: for any harmonic function $u$ in $B_R$ and any $r\in \left(0,\tfrac12R\right]$, 
\begin{equation} 
\label{e.Ck1harm}
\inf_{p\in \mathcal{P}_k}
\left\| u - p\right\|_{\underline{L}^2(B_r)}
\leq 
C \left( \frac{r}{R} \right)^{k+1}
\inf_{p\in \mathcal{P}_k}
\left\| u - p\right\|_{\underline{L}^2(B_R)}.
\end{equation}
Moreover, the infimum on the left side may be replaced by the set of \emph{harmonic} polynomials of degree at most~$k$. 

\smallskip

As we cannot expect a solution of~\eqref{e.pde} to have regularity beyond~$C^{0,1}$ in a classical (pointwise) sense, in order to make sense of a~$C^{k,1}$ estimate for the heterogeneous equation~\eqref{e.pde.lin} we need to replace the set of polynomials by a heterogeneous analogue. The classical Liouville theorem says that the set of harmonic functions which grow like $o(|x|^{k+1})$ is just the set of harmonic polynomials of degree at most~$k$. This suggests that we should use the (random) vector space
\begin{equation*} \label{}
\mathcal{A}_k:= 
\left\{ 
u\in H^1_{\mathrm{loc}}(\Rd)\,:\, -\nabla \cdot \a\nabla u = 0, \ \limsup_{r\to \infty} r^{k+1} \left\| u \right\|_{\underline{L}^2(B_r)} = 0 
\right\}. 
\end{equation*}
We think of these as~``$\a(x)$-harmonic polynomials.'' 
It turns out that one can prove that, ~$\P$--almost surely, this set is finite dimensional and has the same dimension as the set of at most $k$th order harmonic polynomials. In fact, one can match any~$\a(x)$-harmonic polynomial to an $\ahom$-harmonic polynomial in the highest degree, and vice versa. In close analogy to~\eqref{e.Ck1harm}, the statement of large-scale $C^{k,1}$ regularity is then as follows: there exists a minimal scale~$\X$ satisfying~\eqref{e.SI.X} such that, for any $R > 2\X$ and any solution $u\in H^1(B_R)$ of $-\nabla \cdot \a\nabla u =0$, we can find $\phi \in \A_k$ such that, for every $r\in \left[\X,\tfrac12 R \right]$, 
\begin{equation} 
\label{e.Ck1introd.lin}
\left\| u - \phi\right\|_{\underline{L}^2(B_r)}
\leq 
C \left( \frac{r}{R} \right)^{k+1}
\inf_{\phi \in \mathcal{A}_k}
\left\| u - \phi \right\|_{\underline{L}^2(B_R)}.
\end{equation}
See~\cite[Theorem 3.8]{AKMbook} for the full statement, which was first proved in the periodic setting by Avellaneda and Lin~\cite{AL4}. Subsequent versions of this result, which are based on the ideas of~\cite{AL1,AL4} in their more quantitative formulation given in~\cite{AS}, were proved in various works~\cite{GNO2,FO,AKM1}, with the full statement here given in~\cite{AKM,BGO}. 

\smallskip

In all of its various forms, higher regularity in stochastic homogenzations is based on the simple idea that solutions of the heterogeneous equation should be close to those of the homogenized equation, which should have much better regularity. In the case of the linear equation~\eqref{e.pde.lin}, it does not a large leap, technically or philosophically, to go from~\eqref{e.C01estintro} to~\eqref{e.Ck1introd.lin}.  
Indeed, to gain control over higher derivatives, one just needs to differentiate the equation (not in the microscopic parameters, of course, but in macroscopic ones) and, luckily, since the equation is linear, this does not change the equation. Roughly speaking, the idea is analogous to bootstrapping the regularity of a constant-coefficient, linear equation by differentiating it. Therefore the estimate~\eqref{e.Ck1introd.lin} is perhaps not too surprising once one has the large-scale~$C^{0,1}$ estimate in hand. 

\subsection{Summary of the results proved in this paper}
\label{ss.naive}

The situation is very different in the nonlinear case. When one differentiates the equation (again, in a macroscopic parameter) we get a new equation, namely the first-order linearized equation. If we want to apply a large-scale regularity result to this equation, we must first (quantitatively) homogenize it! Achieving higher-order regularity estimates requires repeatedly differentiating the equation, which lead to a hierarchy of linearized equations requiring homogenization estimates. 

\smallskip

Let us be a bit more explicit. The gradient of the homogenized Lagrangian~$\overline{L}$ is given by the well-known formula
\begin{align} 
\label{e.DpLbar.form}
D_p\overline{L}(p) 
&
= \E \left[ \int_{[0,1]^d} D_pL\left( p+ \nabla \phi_p(x),x\right) \,dx \right]
\\ & \notag
= \lim_{r\to \infty} \fint_{B_r} D_pL\left( p+ \nabla \phi_p(x),x\right) \,dx,
\end{align}
where~$\phi_p$ is the first-order corrector with slope~$p\in\Rd$, that is, it satisfies 
\begin{equation*} \label{}
\left\{ 
\begin{aligned}
& -\nabla \cdot D_pL(p+\nabla \phi_p(x),x) = 0 \quad \mbox{in}  \ \Rd, \\
&
\nabla \phi_p \quad \mbox{is $\Zd$--stationary,} \quad \E \left[ \int_{[0,1]^d} \nabla \phi_p(x) \,dx \right]=0. 
\end{aligned} 
\right.
\end{equation*}
The limit in the second line of~\eqref{e.DpLbar.form} is to be understood in a~$\P$--almost sure sense, and it is a consequence of the ergodic theorem, which states that macroscopic averages of stationary fields must converge to their expectations. 
The formula~\eqref{e.DpLbar.form} says that~$D_p\overline{L}(p)$ is the flux per unit volume of the first-order corrector with slope~$p\in\Rd$. It naturally arises when we homogenize the nonlinear equation. We can try to show that $\overline{L}\in C^2$ by formally differentiating~\eqref{e.DpLbar.form}, which leads to the expression
\begin{equation*} \label{}
D_p \partial_{p_i} \overline{L}(p) 
= 
\E \left[ \int_{[0,1]^d} D_p^2L\left( p+ \nabla \phi_p(x),x\right) \left(e_i + \nabla \left(\partial_{p_i}\phi_p(x) \right) \right) \,dx \right].
\end{equation*}
If we define the linearized coefficients around~$\ell_p + \phi_p$ by 
\begin{equation*} \label{}
\a_p (x) :=D_p^2L\left( p+ \nabla \phi_p(x),x\right)
\end{equation*}
and put $\psi^{(1)}_{p,e_i}:= \partial_{p_i}\phi_p(x)$, then we see that $\psi^{(1)}_{p,e_i}$ is the first-order corrector with slope $e_i$ of the linear equation with coefficients $\a_p$:
\begin{equation*} \label{}
\left\{ 
\begin{aligned}
& -\nabla \cdot \a_p\left( e_i + \nabla \psi^{(1)}_{p,e_i} \right) = 0 \quad \mbox{in}  \ \Rd, \\
&
\nabla \psi^{(1)}_{p,e_i} \quad \mbox{is $\Zd$--stationary,} \quad \E \left[ \int_{[0,1]^d} \nabla \psi_{p,e_i}(x) \,dx \right]=0. 
\end{aligned} 
\right.
\end{equation*}
We call $\psi^{(1)}_{p,e}$ a \emph{first-order linearized corrector}. 
Moreover, we have the formula 
\begin{equation*} \label{}
D_p \partial_{p_i} \overline{L}(p) 
= 
\E \left[ \int_{[0,1]^d} \a_p(x) \left(e_i + \nabla \psi_{p,e_i} (x) \right) \,dx \right] = \ahom_p e_i.
\end{equation*}
That is, ``linearization and homogenization commute'': the Hessian of $\overline{L}$ at~$p$ should be equal to the homogenized coefficients~$\ahom_p$ corresponding to the linear equation with coefficient field~$\a_p=D^2_pL(p+\nabla \phi_p(\cdot),\cdot)$. This reasoning is only formal, but since the right side of the above formula for $D^2_p\overline{L}$ is well-defined (and only needs qualitative homogenization), we should expect that it should be rather easy to confirm it rigorously. Moreover, while quantitative homogenization of the original nonlinear equation gives us a $C^{0,1}$ estimate, we should expect that quantitative homogenization of the linearized equation should give us a $C^{0,1}$ estimate for \emph{differences of solutions} and a large-scale $C^{1,1}$ estimate for solutions. This is indeed the case and was proved in our previous paper~\cite{AFK}, where we were motivated by the goal of obtaining this regularity estimate for differences of solutions in anticipation of its  important role in the proof of optimal quantitative homogenization estimates. Indeed, 
in the very recent preprint~\cite{FN}, Fischer and Neukamm showed that this estimate can be combined with spectral gap-type assumptions on the probability measure to obtain quantitative bounds on the first-order correctors which are optimal in the scaling of the error.

\smallskip

We may attempt to differentiate the formula for the homogenized Lagrangian a second time, with the ambition of obtaining a $C^3$ estimate for $\overline{L}$, a $C^{2,1}$ estimate for solutions and a higher-order improvement of our $C^{0,1}$ estimate for differences (which will be a~$C^{0,1}$ estimate for \emph{linearization errors}): we get 
\begin{align*} \label{}
D_p \partial_{p_i} \partial_{p_j}  \overline{L}(p) 
& 
= 
\E \left[ \int_{[0,1]^d}  \a_p (x)\nabla \psi^{(2)}_{p,e_i,e_j} (x) \,dx \right]
\\ & \quad 
+ \E \left[ \int_{[0,1]^d} D^3L(p+\nabla \phi_p(x),x) \left( e_i + \nabla \psi^{(1)}_{p,e_i}  \right) \left( e_j + \nabla \psi^{(1)}_{p,e_j}  \right)
 \,dx \right].
\end{align*}
If we define the vector field
\begin{equation*} \label{}
\mathbf{F}_{2,p,e_i,e_j} (x) 
:=
D^3L(p+\nabla \phi_p(x),x) \left( e_i + \nabla \psi^{(1)}_{p,e_i}  \right) \left( e_j + \nabla \psi^{(1)}_{p,e_j}  \right), 
\end{equation*}
then we see that $\psi^{(2)}_{p,e_i,e_j}$ is the first-order corrector with slope zero of the linear equation
\begin{equation} 
\label{e.secondorderlin}
-\nabla \cdot \a_p \nabla \psi^{(2)}_{p,e_i,e_j} = \nabla \cdot \mathbf{F}_{2,p,e_i,e_j} \quad \mbox{in} \ \Rd, 
\end{equation}
and the formula for the tensor $D^3_p\overline{L}$ becomes 
\begin{equation} 
\label{e.thirdorderdervLbar}
D_p \partial_{p_i} \partial_{p_j}  \overline{L}(p) 
=
\E \left[ \int_{[0,1]^d}  \a_p (x)\nabla \psi^{(2)}_{p,e_i,e_j} (x) +\mathbf{F}_{2,p,e_i,e_j} (x) \,dx \right] = \overline{\mathbf{F}}_{2,p,e_i,e_j}, 
\end{equation}
the corresponding homogenized coefficient. Unlike the case for the Hessian of $\overline{L}$, we should \emph{not} expect this formula to be valid under qualitative ergodic assumptions! Indeed, the qualitative homogenization of~\eqref{e.secondorderlin} and thus the validity of~\eqref{e.thirdorderdervLbar} requires that $\mathbf{F}_{2,p,e_i,e_j}$ belong to~$L^2$, in the sense that $\E \left[ \int_{[0,1]^d} \left| \mathbf{F}_{2,p,e_i,e_j} (x) \right|^2\,dx \right] < \infty$, and due to the product of two first-order correctors we only have $L^1$--type integrability\footnote{to be pedantic, we actually have $L^{1+\delta}$-type integrability for a tiny $\delta>0$ by the Meyers estimate, but this does not help.} for $ \mathbf{F}_{2,p,e_i,e_j}$. This is a serious problem which can only be fixed using the large-scale regularity theory for the first-order linearized equation, and a suitable bound on the minimal scale~$\X$, thereby obtaining a bound in $L^\infty(\cu_0)$ and hence $L^4(\cu_0)$ for~$\nabla \psi^{(1)}_{p,e_i}$, with at least a fourth moment in expectation. Note that this also requires some regularity of the Lagrangian $L$ on the smallest scale. 

\smallskip

If we differentiate the equation once more in an effort to prove that $\overline{L}\in C^4$, we will be faced with similar difficulties, this time with the more complicated vector field
\begin{align*} \label{}
\mathbf{F}^{(3)}_{p,e_i,e_j,e_k} 
(x) &
:= D^3L(p+\nabla \phi_p(x),x) \left( e_i + \nabla \psi^{(1)}_{p,e_i}  \right) \nabla \psi^{(2)}_{p,e_j,e_k} \\
& \qquad 
+ D^3L(p+\nabla \phi_p(x),x) \left( e_j + \nabla \psi^{(1)}_{p,e_j}  \right) \nabla \psi^{(2)}_{p,e_i,e_k} \\
& \qquad 
+D^3L(p+\nabla \phi_p(x),x) \left( e_k + \nabla \psi^{(1)}_{p,e_k}  \right) \nabla \psi^{(2)}_{p,e_i,e_j} \\
& \qquad 
+ D^4L(p+\nabla \phi_p(x),x)
\left( e_i+ \nabla \psi^{(1)}_{p,e_i}  \right)
\left( e_j + \nabla \psi^{(1)}_{p,e_j}  \right)
\left( e_k + \nabla \psi^{(1)}_{p,e_k}  \right).
\end{align*}
Notice that the last term of which has three factors of the first-order linearized correctors instead of two, and is thus ``even further'' from being obviously $L^2$ than was~$\mathbf{F}^{(2)}_{p,e_i,e_j}$. Homogenizing the third-order linearized equation will therefore require large-scale regularity estimates for both the first-order and second-order linearized equations, and one can now see the situation will get worse as the order increases beyond three. Moreover, proving quantitative homogenization for these equations will also require some smoothness of the homogenized coefficients associated to the lower-order equations, due to the needs for the homogenized solutions to be smooth in quantitative two-scale expansion arguments. 

\smallskip

This suggets a bootstrap argument for progressively and simultaneously obtaining (i) the smoothness of $\overline{L}$; (ii) the higher-order large-scale regularity of solutions (and solutions of linearized equations); and (iii) the homogenization of the higher-order linearized equations and commutation of homogenization and higher-order linearization. The point of this paper is to formalize this idea and thereby give a proof of ``Hilbert's 19th problem for homogenization.'' Here is a rough schematic of the argument, as discussed above, which comes in three distinct steps, discussed in more detail below:
\begin{itemize}

\item Homogenization \& large-scale $C^{0,1}$ regularity for the linearized equations up to order $\mathsf{N}$  $\implies$  $\overline{L} \in C^{2+\mathsf{N}}$. 

\item $\overline{L} \in C^{2+\mathsf{N}}$ and large-scale $C^{0,1}$ regularity for the linearized equations up to order $\mathsf{N}$ 
$\implies$ homogenization for the linearized equations up to order $\mathsf{N}+1$.

\item $\overline{L} \in C^{2+\mathsf{N}}$, large-scale $C^{0,1}$ regularity for the linearized equations up to order $\mathsf{N}$ and homogenization for the linearized equations up to order $\mathsf{N}+1$
$\implies$ large-scale $C^{0,1}$ regularity for the linearized equations up to order $\mathsf{N}+1$.

\end{itemize}

The three implications above are the focus of most of the paper and their proofs are given in Sections~\ref{s.regLbar}--\ref{s.reglinerrors}. 

\smallskip

Once this induction argument is completed, we consequently obtain a full $C^{k,1}$--type large scale regularity estimate for solutions of the original nonlinear equation, generalizing~\eqref{e.Ck1introd.lin}. The main question becomes what the replacement for $\mathcal{A}_k$ should be, that is, what the ``polynomials'' should be. We show that these are certain solutions of the system of linearized equations (linearized around a first-order corrector) exhibiting polynomial-type growth, which we classify by providing a Liouville-type result which is part of the statement of the theorem (see the discussion between the statements of Theorems~\ref{t.C11estimate} and Theorem~\ref{t.regularityhigher}, below, for a definition of these spaces, which are denoted by~$\mathsf{W}_n^p$). The resulting theorem we obtain, which is a version of the statement of Hilbert's 19th problem in the context of homogenization, provides a very precise description of the solutions of~\eqref{e.pde} in terms of the first-order correctors and the correctors of a hierarchy of linear equations.

\smallskip
 
We also obtain, as a corollary, the improvement of the scaling of linearization errors---which is a closely related to the regularity of solutions. To motivate this result, suppose we are given two solutions $u,v\in H^1(B_R)$ of~\eqref{e.pde} in a large ball ($R\gg 1$) which are close to each other in the sense that 
\begin{equation}
\left\| \nabla u - \nabla v \right\|_{L^2(B_R)}
\ll \left\| \nabla u \right\|_{L^2(B_R)}. 
\end{equation}
Suppose that we attempt to approximate the difference $u-v$ by the solution~$w\in H^1(B_{R/2})$ of the linearized problem 
\begin{equation*}
\left\{
\begin{aligned}
& -\nabla \cdot \left( D^2_pL(\nabla u,x) \nabla w \right) = 0 & \mbox{in} & \ B_R, \\
& w = v-u & \mbox{on} & \ \partial B_R. 
\end{aligned}
\right. 
\end{equation*}
Then we may ask the question of how small we should expect the first-order linearization error to be. The best answer that we have from deterministic elliptic regularity estimates is that there exists a small exponent $\delta(d,\Lambda)>0$ such that
\begin{equation}
\label{e.linerrorintro}
\frac{\left\| u - v - w\right\|_{L^2(B_{R})}}{\left\| u \right\|_{L^2(B_R)}}
\leq 
C\left( 
\frac{\left\| u - v \right\|_{L^2(B_{R})}}{\left\| u \right\|_{L^2(B_R)}}
\right)^{1+\delta}.  
\end{equation}
This can be easily proved for instance using the Meyers gradient~$L^{2+\delta}$ estimate, and it is sharp in the sense that it is not possible to do better than the very small exponent~$\delta$. We can say roughly that the space of solutions of~\eqref{e.pde} is a $C^{1,\delta}$ manifold, but no better. Of course, if~$L$ does not depend on~$x$, or if~$R$ is of order one and~$L$ is smooth in both variables~$(p,x)$, then one expects to have the estimate above with $\delta=1$ and to be able to prove more precise estimates using higher-order linearized equations. In fact, this is essentially a reformulation of the statement of Hilbert's 19th problem (indeed---see Appendix~\ref{a.constantcoeff}, where we give a proof of Hilbert's 19th problem in its classical formulation by following this line of reasoning).  

\smallskip

In this paper we also prove a large-scale version of the quadratic response to first-order linearization in the context of homogenization, which states that~\eqref{e.linerrorintro} holds with $\delta=1$ whenever~$R$ is larger than a random minimal scale. Moreover, we prove a full slate of higher-order versions of this result: see Corollary~\ref{c.linerrors}. These results roughly assert that, with probability one, the large-scale structure of solutions of~\eqref{e.pde} resembles that of a smooth manifold.

\smallskip

In the following two subsections, we state our assumptions and give the precise statements of the results discussed above.

\subsection{Assumptions and notation}
\label{ss.assumptions}
In this subsection, we state the standing assumptions in force throughout the paper. 

\smallskip

We fix following global parameters: the dimension~$d\in\N$ with $d\geq 2$, a constant~$\Lambda \in [1,\infty)$ measuring the ellipticity, an integer $\mathsf{N}\in\N$ with $\mathsf{N}\geq 1$ measuring the smoothness of the Lagrangian, and constants $\mathsf{M}_0, \mathsf{K}_0 \in [1,\infty)$. For short, we denote 
\begin{equation*}
\data:= (d,\Lambda,\mathsf{N},\mathsf{M}_0,\mathsf{K}_0).
\end{equation*}
This allows us to, for instance, denote constants~$C$ which depend on $(d,\Lambda,\mathsf{N},\mathsf{M}_0,\mathsf{K}_0)$ by simply $C(\data)$ instead of $C(d,\Lambda,\mathsf{N},\mathsf{M}_0,\mathsf{K}_0)$.

\smallskip

The probability space is the set~$\Omega$ of all Lagrangians~$L:\Rd \times \Rd \to \R$, written as a function of $(p,x)\in\Rd\times\Rd$, satisfying the following conditions:
\begin{enumerate}
\item[(L1)] $L$ is $2+\mathsf{N}$ times differentiable in the variable~$p$ and, for every $k\in \{2,\ldots,2+\mathsf{N}\}$, the function $D^{k}_pL$ is uniformly Lipschitz in both variables and satisfies
\begin{equation} \label{e.gradLbndk}
\left[ D^{k}_pL \right]_{C^{0,1}(\Rd \times\Rd)} 
\leq 
\mathsf{K}_0. 
\end{equation}
For $k=1$, we assume that, for $z \in \R^d$,
\begin{equation} \label{e.gradLbnd}
\left[ D_pL(z,\cdot) \right]_{C^{0,1}(\Rd)} 
\leq 
\mathsf{K}_0 (1+|z|). 
\end{equation}

\smallskip

\item[(L2)] $L$ is uniformly convex in the variable~$p$: for every $p\in\Rd$ and~$x\in\Rd$, 
\begin{equation*} \label{}
I_d \leq D^2_p L(p,x) \leq \Lambda I_d.
\end{equation*}

\smallskip

\item[(L3)] $D_pL(0,\cdot)$ is uniformly bounded: 
\begin{equation*}
\left\| D_pL(0,\cdot) \right\|_{L^\infty(\Rd)} \leq \mathsf{M}_0.
\end{equation*}
\end{enumerate} 
We define~$\Omega$ to be the set of all such Lagrangians~$L$:
\begin{equation*} \label{}
\Omega := \left\{ L \,:\,  \mbox{$L$ satisfies~(L1),  (L2) and~(L3)} \right\}.
\end{equation*}
Note that $\Omega$ depends on the fixed parameters~$(d,\Lambda,\mathsf{N},\mathsf{M_0},\mathsf{K}_0)$. It is endowed with the following family of~$\sigma$--algebras: for each Borel subset~$U \subseteq \Rd$, define
\begin{multline*} \label{}
\F(U):= \mbox{the $\sigma$--algebra generated by the family of random variables}\\
L \mapsto L(p,x), \quad (p,x) \in\Rd \times U. 
\end{multline*}
The largest of these is denoted by $\F:= \F(\Rd)$. 

\smallskip

We assume that the law of the ``canonical Lagrangian''~$L$ is a probability measure~$\P$ on $(\Omega,\F)$ satisfying the following two assumptions:

\begin{enumerate}

\item[(P1)] $\P$ has a unit range of dependence: for all Borel subsets $U,V\subseteq \Rd$ such that $\dist(U,V) \geq 1$,
\begin{equation*} \label{}
\mbox{$\F(U)$ and $\F(V)$ are $\P$--independent.}
\end{equation*}

\item[(P2)] $\P$ is stationary with respect to $\Zd$--translations: for every $z\in \Zd$ and $E\in \F$,
\begin{equation*} \label{}
\P \left[ E \right] = \P \left[ T_z E \right],
\end{equation*}
where the translation group $\{T_z\}_{z\in\Zd}$ acts on $\Omega$ by $(T_zL)(p,x) = L(p,x+z)$.
\end{enumerate}
The expectation with respect to~$\P$ is denoted by~$\E$. 

\smallskip

Since we will be often concerned with measuring the stretched exponential moments of the random variables we encounter, the following notation is convenient: for every $\sigma \in(0,\infty)$, $\theta>0$, and random variable~$X$ on $\Omega$, we write 
\begin{equation*}
X \leq \O_\sigma\left( \theta \right) 
\iff 
\E \left[ \exp\left( \left( \frac{X_+}{\theta} \right)^\sigma \right)  \right] \leq 2. 
\end{equation*}
This is essentially notation for an Orlicz norm on $(\Omega,\P)$. Some basic properties of this notation is given in~\cite[Appendix A]{AKMbook}.

\subsection{Statement of the main results}
\label{ss.assump}
We begin by introducing the higher-order linearized equations. These can be computed by hand, as we did for the second and third linearized equations in Section~\ref{ss.naive}, but it is convenient to work with more compact formulas. Observe that, by Taylor's formula with remainder, we have, for every $n\in\{ 0,\ldots,\mathsf{N}+1 \}$, 
\begin{equation*} 
\left| D_p L(p_0 + h,x) - \sum_{k=0}^{n} \frac1{k!} D_p^{k+1} L(p_0,x) h^{\otimes k}  \right| \leq \frac{C\mathsf{K}}{(n+1)!}   \left| h \right|^{n+1}.
\end{equation*}
Define, for $p,x,h_1,\ldots,h_{\mathsf{N}} \in\Rd$ and $t \in\R$, 
\begin{equation*}
\mathbf{G}(t,p,h_1,\ldots,h_{\mathsf{N}},x) := \sum_{k=2}^{\mathsf{N}+1} \frac1{k!} D_p^{k+1} L(p,x) \left( \sum_{j=1}^{\mathsf{N}} \frac{t^j}{j!} h_j \right)^{\otimes k}.
\end{equation*}
Also define, for each $m\in\{ 1,\ldots,\mathsf{N}+1\}$ and $p,x,h_1,\ldots,h_{m-1} \in \Rd$,
\begin{equation} \label{e.defFm}
\mathbf{F}_m (p,h_1,\ldots,h_{m-1},x)
:=
\left( \partial_t^m \mathbf{G}\right)\left(0,p,x,h_1,\ldots,h_{m-1},0,\ldots,0 \right).
\end{equation}
Observe that $\mathbf{F}_1\equiv 0$ by definition (that is, the right side of the first linearized equation is zero, as we have already seen). 

\smallskip

Our first main result concerns the regularity of the effective Lagrangian~$\overline{L}$ and states that it has essentially the same regularity in~$p$ as we  assumed for~$L$.

\begin{theorem}
[{Regularity of $\overline{L}$}]
\label{t.regularity.Lbar}
For every~$\beta\in (0,1)$, the effective Lagrangian $\overline{L}$ belongs to $C^{2+\mathsf{N},\beta}_{\mathrm{loc}}(\Rd)$ and, for every~$\mathsf{M}\in [1,\infty)$, there exists~$C(\beta,\mathsf{M},\data)<\infty$ such that 
\begin{equation*}
\left\| D^2 \overline{L} \right\|_{C^{\mathsf{N},\beta}(B_{\mathsf{M}})} 
\leq C. 
\end{equation*}
\end{theorem}

In view of Theorem~\ref{t.regularity.Lbar}, we may introduce homogenized versions of the above functions. We define, for every~$p\in\Rd$, $\{h_i\}_{i=1}^{\mathsf{N}} \subseteq \R^d$ and $\delta \in\R$, 
\begin{equation*} 
\overline{\mathbf{G}}(t,p,h_1,\ldots,h_{\mathsf{N}}) := \sum_{k=2}^{\mathsf{N}+1} \frac1{k!} D_p^{k+1} \overline{L}(p) \left( \sum_{j=1}^{\mathsf{N}} \frac{t^j}{j!} h_j \right)^{\otimes k}
\end{equation*}
and then, for every $m\in\{ 1,\ldots,\mathsf{N}+1\}$ and $\{h_i\}_{i=1}^{m-1} \subseteq \R^d$,
\begin{equation}
\label{e.defbarFm}
\overline{\mathbf{F}}_m (p,h_1,\ldots,h_{m-1})
:=
\left( \partial_t^m \overline{\mathbf{G}}\right)\left(0,p,h_1,\ldots,h_{m-1},0,\ldots,0 \right).
\end{equation}
As above, we have that $\overline{\mathbf{F}}_1\equiv 0$ by definition.

\smallskip

In the next theorem, we present a statement concerning the commutability of homogenization and higher-order linearizations. It generalizes~\cite[Theorem 1.1]{AFK}, which proved the result in the case $\mathsf{N}=1$. 

\begin{theorem}[Homogenization of higher-order linearized equations]
\label{t.linearizehigher}
\emph{}\\
Let $n\in\{0,\ldots,\mathsf{N}\}$, $\delta\in \left(0,\tfrac12\right]$, $\mathsf{M}\in [1,\infty)$, and $U_1,\ldots,U_{n+1} \subseteq\Rd$ be a sequence of bounded Lipschitz domains satisfying
\begin{equation}
\label{e.Um.inclus}
\overline{U}_{m+1} \subseteq U_{m}, \quad \forall m\in\left\{ 1,\ldots, n \right\}.
\end{equation}
There exist~$\sigma(\data)>0$,~$\alpha\left(\{U_m\},\delta,\data\right)>0$,~$C\left(\{U_m\},\mathsf{M},\delta,\data\right)<\infty$, and a random variable $\X$ satisfying 
\begin{equation}
\label{e.X2}
\X = \O_\sigma \left( C \right)
\end{equation}
such that the following statement is valid. Let $\ep\in (0,1]$, $f \in W^{1,2+\delta}(U_0)$ be such that 
\begin{equation*}
\left\| \nabla f \right\|_{L^{2+\delta}(U_1)} \leq \mathsf{M},
\end{equation*}
and, for each $m\in\{1,\ldots, n+1\}$, fix $g_m\in W^{1,2+\delta}(U_m)$ and
let $u^\ep \in H^1(U_1)$ and the functions $w_1^\ep\in H^1(U_1),w_2^\ep\in H^1(U_2),\ldots,w_{n+1}^\ep \in H^1(U_{n+1})$ satisfy, for every $m\in\{1,\ldots,n+1\}$, the Dirichlet problems
\begin{equation}
\label{e.linearized}
\left\{ 
\begin{aligned}
& -\nabla \cdot  \left( D_pL\left( \nabla u^\ep,\tfrac x\ep \right) \right) = 0 & \mbox{in} & \ U_1, \\
& u^\ep = f,  & \mbox{on} & \ \partial U_1, \\
& -\nabla \cdot  \left( D^2_pL\left( \nabla u^\ep,\tfrac x\ep \right) \nabla w^\ep_m \right) = \nabla \cdot \left( \mathbf{F}_m(\nabla u^\ep,\nabla w^\ep_1,\ldots,\nabla w^\ep_{m-1},\tfrac x\ep)\right) & \mbox{in} & \ U_m,  \\
& w_m^\ep = g_m & \mbox{on} & \ \partial U_m.
\end{aligned} 
\right. 
\end{equation}
Finally, let $\bar{u} \in H^1(U_1)$ and, for every $m\in \{1,\ldots,n+1\}$, the function $\bar{w}_{m}\in H^1(U_m)$ satisfy the homogenized problems 
\begin{equation}
\label{e.homogenizedlinearized}
\left\{ 
\begin{aligned}
& -\nabla \cdot D\overline{L}\left( \nabla \bar{u} \right) = 0 & \mbox{in} & \ U_1, \\
& \bar{u} = f & \mbox{on} & \ \partial U_1,\\
& -\nabla \cdot \left( D^2\overline{L}\left( \nabla \bar{u} \right)  \nabla \bar{w}_m \right) = \nabla\cdot \left( \overline{\mathbf{F}}_m\left(\nabla\bar{u},\nabla \bar{w}_1,\ldots,\nabla \bar{w}_{m-1}\right) \right)& \mbox{in} & \ U_m, \\
& \bar{w}_m = g_m & \mbox{on} & \ \partial U_m.
\end{aligned} 
\right. 
\end{equation}
Then, for every $m\in \{1,\ldots,n+1\}$, we have the estimate
\begin{equation}
\label{e.homogenization.estimates}
\left\| \nabla w_m^\ep - \nabla \bar{w}_m \right\|_{H^{-1}(U_m)} 
\leq
\X \ep^{\alpha} 
\sum_{j=1}^m \left\|  \nabla g_{j} \right\|_{{L}^{2+\delta}(U_{j})}^{\frac{m}{j}}.
\end{equation}
\end{theorem}


Observe that, due to the assumed regularity of $L$ in the spatial variable and the Schauder estimates, the vector fields $\mathbf{F}_m(\nabla u^\ep,\nabla w^\ep_1,\ldots,\nabla w^\ep_{m-1},\tfrac x\ep)$ on the right side of the equations for $w_m$ in~\eqref{e.linearized} belong to $L^\infty(U_m)$. In particular, they belong to $L^2(U_m)$ and therefore the Dirichlet problems in~\eqref{e.linearized} are well-posed in the sense that the solutions~$w_m$ belong to $H^1(U_m)$. Of course, this regularity given by the application of the Schauder estimate depends on~$\ep$ (and indeed the constants blow up like a large power of $\ep^{-1}$) and therefore this remark is not very useful as a quantitative statement. To prove the homogenization result for the $m$th linearized equation, we will need to possess much better bounds on these vector fields, which amounts to better regularity on the solutions~$\nabla u, \nabla w_1,\ldots,\nabla w_{m-1}$. 

\smallskip

This is the reason we must prove Theorems~\ref{t.regularity.Lbar} and~\ref{t.linearizehigher} at the same time (in an induction on the order $m$ of the linearized equation and of the regularity of $D^2\overline{L}$) as the following result on the large-scale regularity of solutions of the linearized equations and of the linearization errors. Its statement is a generalization of the large-scale~$C^{0,1}$ estimates for linearized equation and \emph{differences} of solutions proved in~\cite{AFK}. 

\smallskip

As mentioned above, throughout the paper we use the following notation for volume-normalized $L^p$ norms: for each $p\in [1,\infty)$, $U \subseteq\Rd$ with $|U|<\infty$ and $f\in L^p(U)$,
\begin{equation*}
\left\| f \right\|_{\underline{L}^p(U)}:=
\left( \fint_U |f|^p\,dx \right)^{\frac1p} 
=
\left| U \right|^{-\frac1p} \left\| f \right\|_{L^p(U)}. 
\end{equation*}

\begin{theorem}[Large-scale $C^{0,1}$ estimates for the linearized equations]
\label{t.regularity.linerrors}
\emph{}\\
Let $n\in\{0,\ldots,\mathsf{N}\}$, $q\in [2,\infty)$, and $\mathsf{M}\in [1,\infty)$. Then there exist~$\sigma(q,\data)>0$, a constant~$C(q,\mathsf{M},\data) <\infty$ and a random variable~$\X$ satisfying~$\X \leq \O_{\sigma}\left( C \right)$ such that the following statement is valid. 
For $R \in \left[ 2\X,\infty\right)$ and $u,v,w_1,\ldots,w_{n+1}\in H^1(B_R)$ satisfying, for every $m\in\{1,\ldots,n+1\}$,
\begin{equation*}
\left\{
\begin{aligned}
&
\left\| \nabla u \right\|_{\underline{L}^2(B_R)} \vee 
\left\| \nabla v \right\|_{\underline{L}^2(B_R)} \leq \mathsf{M}
\\ & 
-\nabla \cdot \left( D_pL(\nabla u,x) \right) = 0
\quad \mbox{and} \quad -\nabla \cdot \left( D_pL(\nabla v,x) \right) = 0
\quad \mbox{in} \ B_R,\\
& 
-\nabla \cdot  \left( D^2_pL\left( \nabla u,x \right) \nabla w_m \right) = \nabla \cdot \left( \mathbf{F}_m(\nabla u,\nabla w_1,\ldots,\nabla w_{m-1},x)\right) \quad \mbox{in}  \ B_R,
\end{aligned}
\right.
\end{equation*}
and defining, for~$m\in \{0,\ldots,n\}$, the $m$th order linearization error $\xi_m \in H^1(B_R)$ by
\begin{equation*}
\xi_m:= v - u - \sum_{k=1}^m  \frac1{k!} w_k,
\end{equation*}
then we have, for every $r \in \left[ \X , \tfrac 12 R \right]$ and $m\in\{0,\ldots,n\}$, the estimates
\begin{align}
\label{e.C01linsols}
\left\| \nabla w_{m+1}   \right\|_{\underline{L}^q(B_r)}
& 
\leq 
C\sum_{i=1}^{m+1} 
\left( \frac1R \left\|  w_{i} - \left( w_{i} \right)_{B_R}   \right\|_{\underline{L}^2(B_R)} 
\right)^{\frac{m+1}i}
\end{align}
and
\begin{multline}
\label{e.C01linerror}
\left\| \nabla \xi_{m}   \right\|_{\underline{L}^{q} (B_r)} 
 \\
\leq
C \sum_{i=0}^{m}  \left( \frac{1}{R}\left\| \xi_{i} -  \left( \xi_{i} \right)_{B_{R}} \right\|_{\underline{L}^2(B_R)} \right)^{\frac{m+1}{i+1}} 
+  C \sum_{i=1}^{m}   \left( \frac1R \left\|  w_{i} - \left( w_{i} \right)_{B_R}   \right\|_{\underline{L}^2(B_R)}\right)^{\frac {m+1}{i}} 
.
\end{multline}
\end{theorem}

The main interest in the above theorem is the case of the exponent~$q=2$. However, we must consider arbitrarily large exponents $q\in [2,\infty)$ in order for the induction argument to work. In particular, in order to show that Theorem~\ref{t.regularity.linerrors} for some $n$ implies Theorem~\ref{t.linearizehigher} for $n+1$, we need to consider potentially very large~$q$ (depending on $n$). 

As mentioned above,  Theorems~\ref{t.regularity.Lbar},~\ref{t.linearizehigher} and~\ref{t.regularity.linerrors} are proved together in an induction argument. Each of the theorems has already been proved in the case~$\mathsf{N}=0$ and~$q=2$ in our previous paper~\cite{AFK}. The integrability in Theorem~\ref{t.regularity.linerrors} is upgraded to $q \in (2,\infty)$ in Propositions~\ref{p.w1higher} and~\ref{p.Lipxi0} below for $w_1$ and $\xi_0$, respectively. These serve as the base case of the induction. The main induction step is comprised of the following three implications: 
\begin{itemize}

\item \emph{Regularity of~$\overline{L}$} (Section~\ref{s.regLbar}). We show that if Theorems~\ref{t.regularity.Lbar},~\ref{t.linearizehigher} and~\ref{t.regularity.linerrors} are valid for some $n\in \{0,\ldots,\mathsf{N}-1\}$, then Theorems~\ref{t.regularity.Lbar} is valid for $n+1$. The argument essentially consists of differentiating the corrector equation for the~$n$th linearized equation in the parameter~$p$. However, the reader should not be misled into expecting a simple argument based on the implicit function theorem. Due to the lack of sufficient spatial integrability of the vector fields~$\mathbf{F}_{m}$, it is necessary to use the large-scale regularity theory (i.e., the assumed validity of Theorem~\ref{t.regularity.linerrors} for~$n$) to complete the argument. 

\smallskip

\item \emph{Homogenization of higher-order linearized equations} (Section~\ref{s.homogenization}). We argue, for $n\in \{0,\ldots,\mathsf{N}-1\}$, that if Theorem~\ref{t.regularity.Lbar} is valid for~$n+1$ and Theorems~\ref{t.linearizehigher} and~\ref{t.regularity.linerrors} are valid for~$n$, then Theorems~\ref{t.linearizehigher} is valid for~$n+1$. The regularity of~$\overline{L}$ allows us to write down the homogenized equation, while the homogenization and regularity estimates for the previous linearized equations allow us to \emph{localize} the heterogeneous equation; that is, approximate it with another equation which has a finite range of dependence and bounded coefficients. This allows us to apply homogenization estimates from~\cite{AKMbook}. 

\smallskip

\item \emph{Large-scale ~$C^{0,1}$ regularity of linearized solutions and linearization errors}
(Sections~\ref{s.reglineqs} and~\ref{s.reglinerrors}). We argue, for $n\in \{0,\ldots,\mathsf{N}-1\}$, that if Theorems~\ref{t.regularity.Lbar} and~\ref{t.linearizehigher} are valid for~$n+1$ and Theorem~\ref{t.regularity.linerrors} is valid for~$n$, then we may conclude that Theorem~\ref{t.regularity.linerrors} is also valid for~$n+1$. Here we use the method introduced in~\cite{AS} of applying a quantitative excess decay iteration, based on the ``harmonic'' approximation provided by the quantitative homogenization statement. This estimate controls the regularity of the $w_m$'s on ``large'' scales (i.e., larger than a multiple of the microscopic scale). To obtain $L^q$--type integrability for $\nabla w_m$, it is also necessary to control the small scales, and for this we apply deterministic Calder\'on-Zygmund-type estimates (this is our reason for assuming~$L$ possesses some small-sacle spatial regularity). The estimates for the linearization errors~$\xi_{m-1}$ are then obtained as a consequence by comparing them to~$w_m$. 

\end{itemize}

From a high-level point of view, the induction argument summarized above resembles the resolution of Hilbert's 19th problem on the regularity of minimizers of integral functionals with uniformly convex and smooth integrands. The previous three theorems allow us to prove the next two results, which can be considered as resolutions of Hilbert's 19th problem in the context of homogenization. 

\smallskip

The first is the following result on precision of the higher-order linearization approximations which matches the one we have in the constant-coefficient case, as discussed near the end of Subsection~\ref{ss.naive}. 

\begin{corollary}
[Large-scale estimates of linearization errors]
\label{c.linerrors}
Fix $n\in\{0,\ldots,\mathsf{N} \}$, $\mathsf{M}\in [1,\infty)$ and let $U_0,U_1,\ldots,U_{n} \subseteq\Rd$ be a sequence of bounded Lipschitz domains satisfying
\begin{equation}
\label{e.Um.inclus2}
\overline{U}_{m+1} \subseteq U_{m}, \quad \forall m\in\left\{ 1,\ldots, n-1 \right\}.
\end{equation}
There exist constants~$\sigma(\data) \in \left(0,\tfrac12 \right]$,~$C(\{U_m\},\mathsf{M},\data)<\infty$ and a random variable~$\X$ satisfying 
\begin{equation*}
\X = \O_{\sigma} \left( C \right)
\end{equation*}
such that the following statement is valid.
Let~$r \in \left[ \X,\infty \right)$, $n\in\{1,\ldots,\mathsf{N}\}$ and $u,v\in H^1(rU_0)$ satisfy
\begin{equation*}
\left\{
\begin{aligned}
& -\nabla \cdot \left( D_pL(\nabla u,x) \right) = 0
\quad \mbox{and} \quad -\nabla \cdot \left( D_pL(\nabla v,x) \right) = 0
\quad \mbox{in}  \ rU_0,\\
&
\left\| \nabla u \right\|_{\underline{L}^{2}(rU_0)} \vee 
\left\| \nabla v \right\|_{\underline{L}^{2}(rU_0)} \leq \mathsf{M},
\end{aligned}
\right.
\end{equation*}
and recursively define $w_m \in H^1(rU_m)$, for every~$m\in\{1,\ldots,n\}$, to be the solution of the  Dirichlet problem 
\begin{equation}
\label{e.linearized.cor}
\left\{ 
\begin{aligned}
& -\nabla \cdot  \left( D^2_pL\left( \nabla u,x \right) \nabla w_m \right) = \nabla \cdot \left( \mathbf{F}_m(\nabla u,\nabla w_1,\ldots,\nabla w_{m-1},x)\right) & \mbox{in} & \ rU_m, \\
& w_m = 
v - u - \sum_{k=1}^{m-1} \frac1{k!} w_k
& \mbox{on} & \ r\partial U_m.
\end{aligned} 
\right. 
\end{equation}
Then, for every $m\in\{1,\ldots,n\}$,
\begin{equation}
\label{e.homogenization.estimates.cor}
\left\| 
\nabla v - \nabla \left( u + \sum_{k=1}^m \frac1{k!} w_k \right)
 \right\|_{\underline{L}^2(rU_{m})} 
\leq 
C \left( \left\| \nabla u - \nabla v \right\|_{\underline{L}^{2}(rU_0)} \right)^{m+1}.  
\end{equation}
\end{corollary}
Corollary~\ref{c.linerrors} is an easy consequence of the previous theorems stated above. Its proof is presented in Section~\ref{s.linerrorsproof}. 

\smallskip

The analysis of the linearized equations presented in the theorems above allow us to develop a higher regularity theory for solutions of the nonlinear equation on large scales, in analogy to the role of the Schauder theory in the resolution of Hilbert's~19th problem on the regularity of solutions of nonlinear equations with smooth (or constant) coefficients. 
This result generalizes the large-scale $C^{1,1}$-type estimate proved in our previous paper~\cite{AFK} to higher-order regularity as well as the result in the linear case~\cite[Theorem 3.6]{AKMbook}.
\smallskip

Before giving the statement of this result, we introduce some additional notation and provide some motivational discussion. Given a domain $U\subseteq \Rd$, we define 
\begin{equation*}
\mathcal{L}(U):=
\left\{ 
u\in H^1_{\mathrm{loc}}(U) \,:\,
-\nabla \cdot D_pL(\nabla u,x)  = 0 \ \mbox{in} \ U
\right\}. 
\end{equation*}
This is the set of solutions of the nonlinear equation in the domain~$U$, which we note is a stochastic object. We next define 
$\mathcal{L}_1$ to be the set of global solutions of the nonlinear equation which exhibit at most linear growth at infinity:
\begin{equation*}
\mathcal{L}_1 :=
\left\{ 
u\in \mathcal{L}(\Rd) \,:\,
\limsup_{r\to \infty} r^{-1} \left\| u \right\|_{\underline{L}^2(B_r)}
< \infty
\right\}. 
\end{equation*}
For each $p\in\Rd$, we denote the affine function~$\ell_p$ by~$\ell_p(x):=p\cdot x$. Observe that if the difference of two elements of $\mathcal{L}_1$ has strictly sublinear growth at infinity, it must be constant, by the $C^{0,1}$-type estimate for differences (the estimate~\eqref{e.C01linerror} with $m=0$). Therefore the following theorem, which was proved in~\cite{AFK}, gives a complete classification of~$\mathcal{L}_1$.

\begin{theorem}
[{Large-scale $C^{1,1}$-type estimate~\cite[Theorem 1.3]{AFK}}]
\label{t.C11estimate} 
\emph{}\\
Fix $\sigma \in \left(0,d\right)$ and $\mathsf{M} \in [1,\infty)$. There exist~$\delta(\sigma,d,\Lambda)\in \left( 0, \frac12 \right]$, $C(\mathsf{M},\sigma,\data)<\infty$ and a random variable $\X_\sigma$ which satisfies the estimate
\begin{equation}
\label{e.X}
\X_\sigma \leq \O_{\sigma}\left(C\right)
\end{equation}
such that the following statements are valid. 
\begin{enumerate}
\item[{$\mathrm{(i)}$}] 
For every $u \in \mathcal{L}_1$ satisfying $\limsup_{r \to \infty} \frac1r \left\|  u - (u)_{B_r} \right\|_{\underline{L}^2(B_r)}   \leq \mathsf{M}$,
there exist an affine function $\ell$ such that, for every $R\geq \X_\sigma$,
\begin{equation*} 
\label{e.liouvillec0}
\left\| u - \ell \right\|_{\underline{L}^2(B_R)} \leq C R^{1-\delta} .
\end{equation*} 

\item[{$\mathrm{(ii)}$}]
For every $p\in B_\mathsf{M}$, there exists $u\in \mathcal{L}_1$ satisfying, for every $R\geq \X_\sigma$,
\begin{equation*} 
\label{e.liouvillec1}
\left\| u - \ell_p \right\|_{\underline{L}^2(B_R)} \leq C R^{1-\delta} .
\end{equation*}

\smallskip

\item[{$\mathrm{(iii)}$}]
For every $R\geq \X_s$ and $u\in \mathcal{L}(B_R)$ satisfying 
$\frac1R \left\| u - (u)_{B_R} \right\|_{\underline{L}^2 \left( B_{R} \right)}  \leq \mathsf{M}$,  there exists $\phi \in \mathcal{L}_1$ such that, for every $r \in \left[ \X_\sigma,  R \right]$, 
\begin{equation}
\label{e.C11}
\left\| u - \phi \right\|_{\underline{L}^2(B_r)} \leq C \left( \frac r R \right)^{2} \inf_{\psi\in\L_1} 
\left\| u - \psi \right\|_{\underline{L}^2(B_R)}.
\end{equation}
\end{enumerate}
\end{theorem}

Statements~(i) and~(ii) of the above theorem, which give the characterization of~$\L_1$, can be considered as a first-order Liouville-type theorem. In the case of deterministic, periodic coefficient fields, this result was proved by Moser and Struwe~\cite{MS}, who generalized the result of Avellaneda and Lin~\cite{AL4} in the linear case. Part~(iii) of the theorem is a quantitative version of this Liouville-type result, which we call a ``large-scale~$C^{1,1}$ estimate'' since it states that, on large scales, an arbitrary solution of the nonlinear equation can be approximated by an elements of~$\L_1$ with the same precision as harmonic functions can be approximated by affine functions. It can be compared to similar statements in the linear case (see for instance~\cite{AKMbook,GNO2}). 

\smallskip

In this paper we prove a higher-order version of Theorem~\ref{t.C11estimate}. We will show that, just as a harmonic function can be approximated locally by harmonic polynomials, we can approximate an arbitrary element of $\L(B_R)$ can be approximated by elements of a random set of functions which are the natural analogue of harmonic polynomials. In order to state this result, we must first define this space of functions. 

\smallskip

Let us first discuss the constant-coefficient case. 
If $\overline{L}$ is a smooth Lagrangian, we know from the resolution of Hilbert's 19th problem that solutions of $-\nabla \cdot D_p \overline{L}(\nabla \overline{u}) = 0$ are smooth and thus may be approximated by a Taylor expansion at each point. One may then ask, can we characterize the possible Taylor polynomials? In Appendix~\ref{a.constantcoeff} we provide such a characterization in terms of the linearized equations. The quadratic part is an $\ahom_p:=D^2\overline{L}(p)$-harmonic polynomial and the higher-order polynomials satisfy the equations the linearized equations, involving the $\overline{\mathbf{F}}_m$'s as right hand sides. More precisely, for each $p\in\Rd$ and $n\in\N$, we set  
\begin{align} \notag 
\overline{\mathsf{W}}_n^{p,\textrm{hom}}  := \bigg\{ &  (\overline{w}_1,\ldots,\overline{w}_n)  \in H_{\textrm{loc}}^1(\R^d;\R^n) \,  : \, \mbox{for } 
m \in \{1,\ldots,n\} \mbox{ we have } 
\\  \notag & \qquad 
\lim_{r \to 0} r^{-m} \left\|  \overline{w}_m \right\|_{\underline{L}^2 \left( B_{r} \right)} = 0 , \quad 
 \lim_{R\to \infty} R^{-1-m} \left\| \nabla \overline{w}_m \right\|_{\underline{L}^2 \left( B_{R} \right)} = 0 ,
 \\  \notag & \qquad
- \nabla \cdot \left( \ahom_p  \nabla \overline{w}_m \right) = \nabla \cdot   \overline{\mathbf{F}}_m\left(p , \nabla \overline{w}_1 ,         \ldots, \nabla \overline{w}_{m-1} \right) 
    \bigg\} . 
\end{align}
It is not too hard to show that $\overline{\mathsf{W}}_n^{p,\textrm{hom}}  \subset \mathcal{P}^{\textrm{hom}}_2 \times \ldots \times \mathcal{P}^{\textrm{hom}}_{n+1}$, where $\mathcal{P}^{\textrm{hom}}_j$ stands for homogeneous polynomials of degree $j$. Indeed, we see, by Liouville's theorem, that~$\overline{w}_1$ is an~$\ahom_p$-harmonic polynomial of degree two. More importantly, according to Appendix~\ref{a.constantcoeff}, we have that if~$\overline{u}$ solves  $-\nabla \cdot D_p \overline{L}(\nabla \overline{u}) = 0$ in the neigborhood of origin,  and we set $p = \nabla \overline{u}(0)$ and $\overline{w}_m(x) :=  \frac{1}{m+1} \nabla^{m+1} \overline{u}(0) \, x^{\otimes (m+1)}$, then 
\begin{equation*} 
 (\overline{w}_1,\ldots,\overline{w}_n) \in \overline{\mathsf{W}}_n^{p,\textrm{hom}} . 
\end{equation*}
In particular, $\overline{w}_m$ is a sum of a special solution of 
$$\nabla \cdot \left( \ahom_p  \nabla \overline{w}_m + \overline{\mathbf{F}}_m\left(p , \nabla \overline{w}_1 ,         \ldots, \nabla \overline{w}_{m-1} \right) \right) = 0 $$
in $\mathcal{P}^{\textrm{hom}}_{m+1}$ and an $\ahom_p$-harmonic polynomial in $\mathcal{P}^{\textrm{hom}}_{m+1}$ . 
For our purposes it is convenient to relax the growth condition at the origin and define simply
\begin{align} \notag 
\overline{\mathsf{W}}_n^p  := \bigg\{ &  (\overline{w}_1,\ldots,\overline{w}_n)  \in \mathcal{P}_2 \times \ldots \times \mathcal{P}_{n+1}  \,  : \, \mbox{for } 
m \in \{1,\ldots,n\} \mbox{ we have } 
\\  \notag & \qquad 
 \lim_{R\to \infty} R^{-1-m} \left\| \nabla \overline{w}_m \right\|_{\underline{L}^2 \left( B_{R} \right)} = 0 ,
 \\  \notag & \qquad
- \nabla \cdot \left( \ahom_p  \nabla \overline{w}_m \right) = \nabla \cdot   \overline{\mathbf{F}}_m\left(p , \nabla \overline{w}_1 ,       \ldots, \nabla \overline{w}_{m-1} \right) 
    \bigg\} . 
\end{align}
With this definition, we lose the homogeneity of polynomials. Following the approach in~\cite[Chapter 3]{AKMbook}, it is natural to define heterogeneous versions of these spaces by
\begin{align} \notag 
\mathsf{W}_n^p  := \bigg\{ &  (w_1,\ldots,w_n)  \in H_{\textrm{loc}}^1(\R^d;\R^n)  \,  : \, \mbox{for } 
m \in \{1,\ldots,n\} \mbox{ we have } 
\\  \notag & \qquad 
 \limsup_{R\to \infty} R^{-1-m} \left\| \nabla w_m \right\|_{\underline{L}^2 \left( B_{R} \right)} = 0 ,
\\  \notag & \qquad
- \nabla \cdot \left( D_p^2L( p + \nabla \phi_p,\cdot)  \nabla w_m \right) = \nabla \cdot  \mathbf{F}_m\left( p+ \nabla \phi_p , \nabla w_1 ,        \ldots, \nabla w_{m-1} ,\cdot \right) 
    \bigg\} ,
\end{align}
that is, the tuplets of heterogeneous solutions with prescribed growth. Here $\ell_p + \phi_p$ is the unique element of~$\L_1$ (up to additive constants) satisfying 
\begin{equation*} 
\lim_{r\to \infty} \frac1r \left\| \phi_p \right\|_{\underline{L}^2 \left( B_{r} \right)} = 0.
\end{equation*}
In other words, $\phi_p$ is the first-order corrector with slope $p$: see Lemma~\ref{l.corr.sublinearity}.

\smallskip

The next theorem gives a higher-order Liouville-type result which classifies the spaces~$\mathsf{W}_n^p$ and states that they may be used to approximate any solution of the nonlinear equation with the precision of a~$C^{n,1}$ estimate.  

\begin{theorem}[Large-scale regularity]
\label{t.regularityhigher}
Fix $n \in \{1,\ldots,\mathsf{N}\}$ and $\mathsf{M}  \in [1, \infty)$. There exist constants $\sigma(n,\mathsf{M},\data),\delta(n,\data)\in \left( 0, \frac12 \right]$ and a random variable $\X$ satisfying the estimate
\begin{equation}
\label{e.higherX}
\X \leq \O_\sigma \left(C(n,\mathsf{M}, d,\Lambda) \right)
\end{equation}
such that the following statements hold:
\begin{enumerate}
\item[{$\mathrm{(i)}_n$}] There exists a constant $C(n,\mathsf{M},\data) <\infty$ such that, for every $p \in B_{\mathsf{M}}$ and  $(w_1,\ldots,w_n) \in \mathsf{W}_n^p$, there exists $(\overline{w}_1,\ldots,\overline{w}_n) \in \overline{\mathsf{W}}_n^p $ such that, for every $R\geq \X$ and $m \in \{1,\ldots,n\}$,
\begin{equation} \label{e.liouvillec}
\left\| w_m - \overline{w}_m  \right\|_{\underline{L}^2(B_R)} \leq C R^{1-\delta}  \left( \frac{R}{\X} \right)^{m} \sum_{i=1}^m \left(  \frac1{\X} \left\| \overline{w}_i \right\|_{\underline{L}^2(B_{\X})}\right)^{\frac{m}{i}} .
\end{equation}

\item[{$\mathrm{(ii)}_n$}] For every  $p \in B_{\mathsf{M}}$ and $(\overline{w}_1,\ldots,\overline{w}_n) \in \overline{\mathsf{W}}_n^{p}$, there exists  $(w_1,\ldots,w_n) \in \mathsf{W}_n^p$ satisfying~\eqref{e.liouvillec} for every $R\geq \X$ and $m \in \{1,\ldots,n\}$. 

\item[{$\mathrm{(iii)}_n$}]
There exists $C(n,\mathsf{M},\data)<\infty$ such that, for every $R\geq \X$ and $v \in \mathcal{L}(B_R)$ satisfying 
$
\left\| \nabla v \right\|_{\underline{L}^2 \left( B_{R} \right)} \leq \mathsf{M}, 
$
 there exist $p \in B_{C}$ and $(w_1,\ldots,w_n) \in \mathsf{W}_n^p$ such that, defining 
 \begin{equation*} 
\xi_k(x) := v(x) - p\cdot x - \phi_p(x) - \sum_{i=1}^{k} \frac{w_i}{i!} 
\end{equation*}
for every $r \in \left[ \X,  \frac12 R \right]$, we have, for $k \in \{0,1,\ldots,n\},$ the following estimates:
\begin{equation}
\label{e.intrinsicreg}
\sum_{i=0}^{k}\left(  \left\| \nabla \xi_i  \right\|_{\underline{L}^{2}(B_r)} \right)^{\frac{k+1}{i+1}}
\leq 
C \left( \frac r R \right)^{k+1} \sum_{i=0}^{k}\left( \frac1R  \left\| \xi_i - (\xi_i)_{B_R} \right\|_{\underline{L}^{2}(B_R)} \right)^{\frac{k+1}{i+1}}
\end{equation}
and
\begin{equation}
\label{e.intrinsicreg2}
 \left\| \nabla \xi_k  \right\|_{\underline{L}^{2}(B_r)} 
\leq 
C \left( \frac r R \right)^{k+1} \frac1R  \inf_{\phi \in \mathcal{L}_1}\left\| v  - \phi \right\|_{\underline{L}^{2}(B_R)}.
\end{equation}
\end{enumerate}
\end{theorem}

The proof of Theorem~\ref{t.regularityhigher} is given in Section~\ref{s.regularityhigher}. 

\smallskip

Appendices~\ref{app.linerrors}--\ref{a.constantcoeff} of this paper contain estimates for constant-coefficient equations which are essentially known but not to our knowledge written anywhere. We also collect some auxiliary estimates and computations in Appendices~\ref{app.CZ} and~\ref{s.AppendixFm}.

\section{Regularity estimates for the effective Lagrangian}
\label{s.regLbar}

In this section, we suppose that~$n\in \{0,\ldots,\mathsf{N}-1 \}$ is such that   
\begin{equation}
\label{e.assumption.section3}
\mbox{the statements of Theorems~\ref{t.regularity.Lbar},~\ref{t.linearizehigher} and~\ref{t.regularity.linerrors} are valid for $n$.}
\end{equation}
The goal is to prove Theorem~\ref{t.regularity.Lbar} for $n+1$. 

\smallskip

We proceed by constructing the linearized correctors $\psi^{(m)}_{p,h}$ up to $m=n+2$ and relate the correctors of different orders to each other via differentiation in the parameter~$p$. We show that these results allow us to improve the regularity of~$D^2\overline{L}$ up to $C^{n+1,\beta}$ and obtain the statement of Theorem~\ref{t.regularity.Lbar} for $n+1$. In particular, this allows us to define the effective coefficient~$\overline{\mathbf{F}}_{n+1}$. We also give formulas for the derivatives of~$\overline{L}$ and for~$\overline{\mathbf{F}}_{m}$ in terms of the correctors, which allow us to relate them to each other and show that~\eqref{e.defbarFm} holds. 

\smallskip

\subsection{The first-order correctors and linearized correctors}
In this subsection we construct the linearized correctors up to order~$n+2$. 

\smallskip

For each $p\in\Rd$, we define~$\phi_p$ to be the \emph{first-order corrector} %
of the nonlinear equation, that is, the unique solution of
\begin{equation} \label{e.nonlinearcorrector}
\left\{
\begin{aligned}
& -\nabla \cdot \left( D_pL\left( p + \nabla\phi_p(x),x\right) \right)=0 \quad \mbox{in}  \ \Rd,\\
& \nabla\phi_p \ \ \mbox{is $\Zd$--stationary, \quad and}\quad \E \left[ \int_{\cu_0} \nabla \phi_p(x)\,dx \right] = 0. 
\end{aligned}
\right.
\end{equation}
The existence and uniqueness (up to additive constants) of the first-order corrector $\phi_p$ is classical: it can be obtained from a variational argument (applied to an appropriate function space of stationary functions. Alternatively,  
it can be shown (following the proof given in~\cite[Section 3.4]{AKMbook}) that the elements of~$\L_1$, which was characterized in Theorem~\ref{t.C11estimate} above (which was proved already in~\cite{AFK}), 
have stationary gradients.

\smallskip

We define the coefficient field $\a_p(x)$ to be the coefficients for the linearized equation around the solution $x\mapsto p\cdot x+\phi_p(x)$:
\begin{equation*}
\a_p(x):= D_p^2L\left( p + \nabla\phi_p(x),x\right).
\end{equation*}
Given $p,h\in\Rd$, we define the \emph{first linearized corrector} $\psi^{(1)}_{p,h}$ to satisfy
\begin{equation*}
\left\{
\begin{aligned}
& -\nabla \cdot \left( \a_p(x) \left( h+ \nabla \psi^{(1)}_{p,h} \right)\right)=0 \quad \mbox{in}  \ \Rd,\\
&\nabla\psi^{(1)}_{p,h} \ \ \mbox{is $\Zd$--stationary, \quad and} \quad
\E \left[ \int_{\cu_0} \nabla \psi^{(1)}_{p,h}(x)\,dx \right] = 0.
\end{aligned}
\right.
\end{equation*}
In other words, ~$\psi^{(1)}_{p,h}$ is the first-order corrector with slope~$h$ for the equation which is the linearization around the first-order corrector $x\mapsto p\cdot x + \phi_p(x)$.
For $m \in \{ 2,\ldots,\mathsf{N}+2 \}$, we define the \emph{$m$th linearized corrector} to be the unique (modulo additive constants) random field $\psi^{(m)}_{p,h}$ satisfying
\begin{equation}
\label{e.mthlinearized.corr}
\left\{
\begin{aligned}
& -\nabla \cdot \left( \a_p(x) \nabla \psi^{(m)}_{p,h} \right)
\\
& \quad 
=\nabla \cdot 
\mathbf{F}_m\left( p+\nabla\phi_p(x),h+\nabla\psi^{(1)}_{p,h},\nabla \psi^{(2)}_{p,h}(x), \ldots, \nabla \psi^{(m-1)}_{p,h}(x),x \right)
 & \mbox{in} & \ \Rd,\\
& \nabla\psi^{(m)}_{p,h} \quad \mbox{is $\Zd$--stationary, \quad and} \quad
\E \left[ \int_{\cu_0} \nabla \psi^{(m)}_{p,h}(x)\,dx \right] = 0. 
\end{aligned}
\right.
\end{equation}
In other words, $\psi^{(m)}_{p,h}$ is the corrector with slope zero for the $m$th linearized equation around $x\mapsto p\cdot x + \phi_p(x)$ and $x\mapsto h\cdot x + \psi^{(1)}_{p,h}(x)$, $x\mapsto \psi^{(2)}_{p,h}(x),\ldots, x\mapsto \psi^{(m-1)}_{p,h}(x)$. Notice that this gives us the complete collection of correctors for the latter equation, since by linearity we observe that $\psi^{(m)}_{p,h} + \psi^{(1)}_{p,h'}$ is the corrector with slope $h'$. Furthermore, by the linearity of the map $h \mapsto h + \nabla \psi^{(1)}_{p,h}(x)$, it is easy to see from the structure of the equations of $\psi^{(m)}_{p,h}$ that, for $p,h \in \R^d$, and $t \in \R$, 
\begin{equation} \label{e.psimhomogen}
\nabla \psi^{(m)}_{p,t h} = t^m \nabla \psi^{(m)}_{p,h}. 
\end{equation}
For $p,h \in \R^d$, we define
\begin{equation}
\label{e.corrcoeff}
\mathbf{f}^{(k)}_{p,h} := \mathbf{F}_{k} \left(p + \nabla \phi_p , h+ \nabla \psi^{(1)}_{p,h}, \nabla \psi^{(2)}_{p,h}, \ldots ,  \nabla \psi^{(k-1)}_{p,h},\cdot \right) .
\end{equation}
By~\eqref{e.psimhomogen}, we have that 
\begin{equation*} 
\mathbf{f}^{(k)}_{p,h}= |h|^k \mathbf{f}^{(k)}_{p,h/|h|}
\end{equation*}

\smallskip

We first show that the the problem~\eqref{e.mthlinearized.corr} for the $m$th linearized corrector is well-posed for $m\in \{2,\ldots,n+1\}$. This is accomplished by checking, inductively, using our hypothesis~\eqref{e.assumption.section3} (and in particular the validity of Theorem~\ref{t.regularity.linerrors} for $m\leq n$), that we have appropriate estimates on the vector fields~$\mathbf{F}_m(\cdots)$ on the right side. We have to make this argument at the same time that we obtain estimates on the smoothness of the correctors $\psi^{(m)}_{p,h}$ as functions of~$p$. In fact, we prove that $\nabla \psi^{(m+1)}_{p,h} = h\cdot D_p\nabla \psi^{(m)}_{p,h}$, which expressed in coordinates is  
\begin{equation*} 
\nabla \psi^{(m+1)}_{p,h} = \sum_{i=1}^d h_i\partial_{p_i} \nabla \psi^{(m)}_{p,h}.
\end{equation*}
We will also obtain $C^{0,1}$-type bounds on the linearized correctors, which together with the previous display yields good quantitative control on the smoothness of the correctors in~$p$.

\smallskip

 Before the main statements, let us collect a few preliminary elementary results needed in the proofs.  The following lemma is well-known and can be proved by the Lax-Milgram lemma (see for instance~\cite[Chapter 7]{JKO}).

\begin{lemma}
\label{l.abstractnonsense}
Let $\a(\cdot)$ be a $\Zd$--stationary random field valued in the symmetric matrices satisfying the ellipticity bound $I_d \leq \a \leq \Lambda I_d$. Suppose that~$\mathbf{f}$ is a $\Zd$--stationary, $\Rd$--valued random field satisfying 
\begin{equation*}
\E \left[ \left\| \mathbf{f} \right\|_{L^2(\cu_0)}^2 \right] 
< \infty. 
\end{equation*}
Then there exists a unique random potential field $\nabla z$ satisfying 
\begin{equation*}
\left\{ 
\begin{aligned}
& -\nabla \cdot \a \nabla z = \nabla \cdot \mathbf{f} \quad \mbox{in} \ \Rd
\\ & 
\nabla z \ \mbox{is $\Zd$--stationary, \quad and} \quad
\E \left[ \int_{\cu_0} \nabla z(x)\,dx \right] = 0.
\end{aligned}
\right.
\end{equation*}
and, for a constant~$C(d,\Lambda)<\infty$, the estimate
\begin{equation*}
\E \left[ \left\| \nabla z \right\|_{L^2(\cu_0)}^2 \right] 
\leq C\, \E \left[ \left\| \mathbf{f} \right\|_{L^2(\cu_0)}^2 \right].
\end{equation*}
\end{lemma}

Next, a central object in our analysis is the quantity $\mathbf{F}_m$ defined in~\eqref{e.defFm}. Fix $m\in \{2,\ldots, \mathsf{N}+1\}$ and  $x,p,h_1,\ldots,h_{m-1} \in \R^d$. One can easily read from the definition that we have
\begin{equation}  \label{e.Fmalt}
\mathbf{F}_m(p,h_1,\ldots,h_{m-1},x)
= m! \sum_{2 \leq j \leq m}\frac{1}{j!}D_{p}^{j+1}L(p,x)\sum_{\stackrel{ i_{1}+\cdots i_{j}= m}{ i_{1},\dots, i_{j}\geq 1} } \prod_{k=1}^j \frac{h_{i_k}^{\otimes 1}}{i_{k}!} .
\end{equation}
It then follows, by Young's inequality, that there exists $C(m,\data)<\infty$ such that
\begin{equation}
 \label{e.Fmbasic}
 \left| \mathbf{F}_m (p,h_1,\ldots,h_{m-1},x) \right| \leq C  \sum_{i=1}^{m-1} |h_i|^{\frac{m}{i}} .
\end{equation}
We have, similarly, that 
\begin{equation}  \label{e.Fmbasicgradp}
 \left[ \mathbf{F}_m (\cdot,h_1,\ldots,h_{m-1},x) \right]_{C^{0,1}(\R^n)}  \leq C  \sum_{i=1}^{m-1} |h_i|^{\frac{m}{i}}
\end{equation}
and, for $k \in \{1,\ldots,m-1\}$, 
\begin{equation} \label{e.Fmbasicgradh}
 \left| D_{h_k} \mathbf{F}_m (p ,h_1,\ldots,h_{m-1},x) \right|  \leq C  \sum_{i=1}^{m-k} |h_i|^{\frac{m-k}{i}}.
\end{equation}
Using these we get, for $x,p,\tilde p, h_1,\tilde h_1,\ldots, h_{m-1},\tilde h_{m-1} \in \R^d$,
\begin{align}  \label{e.Fmbasic2}
\lefteqn{\left| \mathbf{F}_m (p,h_1,\ldots,h_{m-1},x) -  \mathbf{F}_m (\tilde p,\tilde h_1,\ldots,\tilde h_{m-1},x)\right| } \quad &
\\ \notag &
\leq
C  |p  - \tilde p| \sum_{i=1}^{m-1} \left( \left| h_i \right| + \left| \tilde h_i \right| \right) ^{\frac{m}{i}}  
+  C \sum_{i=1}^{m-1}  \left| h_i - \tilde h_i \right|   \sum_{j=1}^{m-i} \left( \left| h_j \right| + \left| \tilde h_j \right| \right)^{\frac{m-i}{j}}  .     
\end{align}
Therefore, by Young's inequality, we get, for all $\delta>0$, 
\begin{align} \label{e.Fmbasic3}
\lefteqn{\left| \mathbf{F}_m (p,h_1,\ldots,h_{m-1},x) -  \mathbf{F}_m (\tilde p,\tilde h_1,\ldots,\tilde h_{m-1},x)\right| } \quad &
\\ \notag &
\leq
C \left(   |p  - \tilde p| + \delta \right)  \sum_{i=1}^{m-1} \left( \left| h_i \right| + \left| \tilde h_i \right| \right)^{\frac{m}{i}}  
+  C \delta \sum_{i=1}^{m-1}  \left( \delta^{-1} \left| h_i - \tilde h_i \right| \right)^{\frac mi}    .     
\end{align}
Furthermore,  as we will be employing an induction argument in $m$, it is useful to notice that the leading term in $\mathbf{F}_m$
has a simple form, and we have 
\begin{equation}  \label{e.Fmdecomposed}
\mathbf{F}_m (p,h_1,\ldots,h_{m-1},x) = m D_p^{3} L(p,x) h_{m-1}^{\otimes 1} h_1^{\otimes 1}  +  \tilde{\mathbf{F}}_m (p,h_1,\ldots,h_{m-2},x) ,
\end{equation}
with $\tilde{\mathbf{F}}_2 = 0$. 
Moreover, as is shown in Appendix~\ref{s.AppendixFm}, if, for $p,h \in \R^d$, $t \mapsto \mathbf{g}(p+th)$ is~$m$ times differentiable at $t=0$ and noticing that since $m \leq \mathsf{N}+1$, $p\mapsto L(p,x)$ is in $C^{m+2}$, then 
\begin{align}  \label{e.Fmrelation}
\lefteqn{
\mathbf{F}_{m+1} (\mathbf{g}(p) , D_p \mathbf{g}(p) h^{\otimes 1} ,\ldots,D_p^{m} \mathbf{g}(p) h^{\otimes m},x) 
} \quad &
\\ \notag &
=  D_p \left( \mathbf{F}_{m} (\mathbf{g}(p),D_p \mathbf{g}(p) h^{\otimes 1},\ldots,D_p^{m-1} \mathbf{g}(p) h^{\otimes (m-1)},x) \right) \cdot h   
\\ \notag & \quad 
+ D_p \left( D_p^{2} L(\mathbf{g}(p),x) \right) h^{\otimes 1} \left( D_p^{m} \mathbf{g}(p) h^{\otimes m}  \right)^{\otimes 1}.
\end{align}

\smallskip

Our first result in this section gives direct consequences of~\eqref{e.assumption.section3} for estimates on the first-order correctors and linearized correctors. 

\begin{theorem}[Quantitative estimates on linearized correctors]
\label{t.correctorestimates}

Assume~\eqref{e.assumption.section3} is valid. Fix $\mathsf{M}\in [1,\infty)$. For every $m\in \{ 2,\ldots, n+1 \}$ and $p,h\in\Rd$, there exists a function $\psi^{(m)}_{p,h}$ satisfying~\eqref{e.mthlinearized.corr}. 
Moreover, there exist constants $C(\mathsf{M},\data)<\infty$ and $\sigma(n,d,\Lambda) \in \left(0,\tfrac12\right]$ and a random variable $\X$ satisfying $\X \leq \O_\sigma(C)$ such that the following statement is valid. 
For every~$p\in B_{\mathsf{M}}$, $h \in \overline{B}_1$, $m\in \{1,\ldots,n\}$,  and $r\geq \X$, 
\begin{equation}
\label{e.linearization.corr}
\left\|  
\nabla \phi_{p+h} 
- \left( \nabla \phi_p  + h + \sum_{k=1}^m \frac1{k!} \nabla \psi^{(k)}_{p,h} \right)
\right\|_{\underline{L}^2(B_r)} 
\leq C |h|^{m+1}
\end{equation}
and, for every~$p\in B_{\mathsf{M}}$, $h \in  \overline{B}_1$, $m\in \{1,\ldots,n+1\}$ and $r\geq \X$, 
\begin{equation}
\label{e.sec3corrbnd0}
\left\|  
\nabla \psi^{(m)}_{p,h}
\right\|_{\underline{L}^2(B_r)} 
\leq C |h|^{m}.
\end{equation}
Finally, for $q \in [2,\infty)$and $m\in \{1,\ldots,n+1\}$,  there exist constants~$\delta(m,d,\Lambda)\in \left(0,\tfrac12\right]$ and ~$C(q,m,\mathsf{M},\data)<\infty$ such that, for every
~$p\in B_{\mathsf{M}}$ and $h \in   \overline{B}_1$, 
\begin{equation} 
\label{e.sec3corrbnd}
\left\|
 \nabla \psi^{(m)}_{p,h}
\right\|_{L^q(\cu_0)} 
\leq \O_{\delta}  
\left(
C |h|^m  
\right).
\end{equation}
\end{theorem}
\begin{proof}
Set,  for $m \in \{1,\ldots,n+1\}$, 
\begin{equation*} 
\xi_0  :=  (p+h) \cdot x + \phi_{p+h}(x)  - (p \cdot x + \phi_{p}(x))     \quad \mbox{and} \quad 
\xi_m := \xi_0 -  \sum_{k=1}^m \frac1{k!} \psi^{(k)}_{p,h} .
\end{equation*}
We first collect two consequences of  Theorem~\ref{t.regularity.linerrors} assumed for $n$. Fix $q \in [2,\infty)$. Theorem~\ref{t.regularity.linerrors} implies that there is a minimal scale  $\X$ such that~\eqref{e.C01linsols} and~\eqref{e.C01linerror} are valid with $q (n+1)$ instead of $q$ and for every $r \in \left[ \X , \tfrac 12 R \right]$. Hence, 
for every $r \in \left[ \X , \tfrac 12 R \right]$ and $m\in\{0,\ldots,n\}$, we get the estimates
\begin{align}
\label{e.C01linsolsapplied1}
\left\| \nabla \psi^{(m+1)}_{p,h}   \right\|_{\underline{L}^{ q (n+1)}(B_r)}
& 
\leq 
C\sum_{i=1}^{m+1} 
\left(  \left\|  \nabla \psi^{(i)}_{p,h}    \right\|_{\underline{L}^2(B_R)} 
\right)^{\frac{m+1}i}
\end{align}
and 
\begin{align}
\label{e.C01linerrorapplied1}
\left\| \nabla \xi_{m}   \right\|_{\underline{L}^{q (n+1) }(B_r)} 
\leq  
 C \sum_{i=0}^{m-1} 
\left( \left\| \nabla \xi_{i} \right\|_{\underline{L}^2(B_R)} +   \left\|  \nabla \psi^{(i+1)}_{p,h}   \right\|_{\underline{L}^2(B_R)}  \right)^{\frac{m+1}{i+1}} .
\end{align}
On the other hand, using~\eqref{e.Fmbasic} we obtain
\begin{equation*} 
\left|  \mathbf{f}^{(m)}_{p,h} \right| \leq C \sum_{i=1}^{m-1} \left|  \nabla \psi^{(i)}_{p,h} \right|^{\frac{m}{i}}.
\end{equation*}
In particular, it follows by~\eqref{e.C01linsolsapplied1} that, for $m \in \{1,\ldots,n+2\}$, $R>2\X$ and $r \in [\X,\tfrac12 R)$, 
\begin{equation*} 
\left\|  \mathbf{f}^{(m)}_{p,h} \right\|_{\underline{L}^q(B_r)} \leq  C \sum_{i=1}^{m-1} \left\|  \nabla \psi^{(i)}_{p,h}  \right\|_{\underline{L}^2(B_R)}^{\frac{m}{i}}.
\end{equation*}
Since $\nabla \psi^{(i)}_{p,h}  $ is $\Z^d$-stationary random field, we have by the ergodic theorem, after sending $R\to \infty$, that a.s.
\begin{equation*} 
\left\|  \mathbf{f}^{(m)}_{p,h} \right\|_{\underline{L}^q(\X \cu_0)} \leq  C \sum_{i=1}^{m-1} \E \left[  \left\|  \nabla \psi^{(i)}_{p,h}  \right\|_{L^2(\cu_0)}^2 \right]^{\frac{m}{2i}}.
\end{equation*}
Furthermore, by Lemma~\ref{l.abstractnonsense} and the previous display we get, for $m \in \{1,\ldots,n+2\}$, that
\begin{align} \notag 
 \E \left[  \left\|  \nabla \psi^{(m)}_{p,h} \right\|_{L^2(\cu_0)}^2 \right] 
 &
\leq 
C  \E \left[  \left\| \mathbf{f}^{(m)}_{p,h} \right\|_{L^2(\cu_0)}^2 \right]
\\ \notag & 
\leq  
C \E \left[ \X^d \left\| \mathbf{f}^{(m)}_{p,h} \right\|_{\underline{L}^2(\X_m \cu_0)}^2 \right]
\leq 
C  \sum_{i=1}^{m-1} \E \left[  \left\| \nabla \psi^{(i)}_{p,h} \right\|_{L^2(\cu_0)}^2 \right]^{\frac{m}{i}} .
\end{align}
Observe that the limiting behavior of $\xi_0$ and $\psi^{(1)}_{p,h}$ can be identified via their equations
\begin{equation*} 
-\nabla \cdot \left(\tilde{\a}_p \left(h + \nabla (\phi_{p+h} - \phi_{p}) \right)  \right) = 0  
\quad \mbox{and} \quad 
- \nabla \cdot \left( \a_p \left(h + \nabla \psi^{(1)}_{p,h}) \right)  \right) = 0
\end{equation*}
respectively,  where
\begin{equation*} 
\tilde{\a}_p := \int_0^1 D_p^2L\left( p + th + t \nabla \phi_{p+h} + (1-t)\nabla \phi_{p} ,\cdot \right) \, dt  .
\end{equation*}
By $\Z^d$-stationary of $\nabla \phi_p$ and $\nabla \phi_{p+h}$, implying $\Z^d$-stationarity of $\a_p$ and $\tilde{\a}_p$, we may apply Lemma~\ref{l.abstractnonsense} to obtain
\begin{equation*} 
 \E \left[ \left\|  \nabla \xi_{0}    \right\|_{L^2(\cu_0)}^2 +  \left\|  \nabla \psi^{(1)}_{p,h}   \right\|_{L^2(\cu_0)}^2  \right]^\frac12  \leq C |h| .
\end{equation*}
It then follows inductively that, for $m \in \{1,\ldots,n+2\}$,
\begin{equation*} 
 \E \left[  \left\|  \nabla \psi^{(m)}_{p,h} \right\|_{L^2(\cu_0)}^2 \right]^\frac12  \leq C |h|^m .
\end{equation*}
Using this together with the ergodic theorem and~\eqref{e.C01linsolsapplied1},~\eqref{e.C01linerrorapplied1}, we obtain inductively, for $q\in [2,\infty)$, 
$m \in \{0,\ldots,n+1\}$ and $r \geq \X$, 
\begin{equation*} 
\left\| \nabla \psi^{(m)}_{p,h}   \right\|_{\underline{L}^q(B_r)}  \leq C_q |h|^{m} 
\quad \mbox{and} \quad 
\left\| \nabla \xi_{m-1}   \right\|_{\underline{L}^{q}(B_r)}  \leq C |h|^{m} .
\end{equation*}
Now~\eqref{e.sec3corrbnd} follows by giving up a volume factor. The proof is complete. 
\end{proof}

We next show, again using~\eqref{e.assumption.section3}, that the corrector $\psi^{(n+2)}_{p,h}$ satisfies an~$L^2$--type gradient estimate. 

\begin{lemma} \label{l.psinplus2}
Assume~\eqref{e.assumption.section3} is valid. Let $\mathsf{M} \in [1,\infty)$. Suppose that~$p \in B_{\mathsf{M}}$ and~$h \in \overline{B}_1$. 
There exists~$\psi^{(n+2)}_{p,h}$ satisfying~\eqref{e.mthlinearized.corr} for $m=n+2$ and a constant $C(n,\mathsf{M},\data)$ such that
\begin{equation*} 
 \E \left[  \left\|  \nabla \psi^{(n+2)}_{p,h} \right\|_{L^2(\cu_0)}^2 \right]^{\frac12}  \leq C |h|^{n+2}. 
\end{equation*}
\end{lemma}

\begin{proof}
The result follows directly by Lemma~\ref{l.abstractnonsense} and~\eqref{e.sec3corrbnd} using~\eqref{e.Fmbasic}. 
\end{proof}

\begin{lemma} \label{l.stupidpsi}
Assume~\eqref{e.assumption.section3} is valid. Suppose that $p \in B_{\mathsf{M}}$ and $h \in \overline{B}_1$. Then, for $m\in \{1,\ldots,n+1\}$, we have, a.s. and a.e.,
\begin{equation} \label{e.gradpsiformula}
\nabla \psi^{(m+1)}_{p,h} = \sum_{i=1}^d h_i\partial_{p_i} \nabla \psi^{(m)}_{p,h} .
\end{equation}
Moreover, for $\beta \in (0,1)$ and $m \in \{1,\ldots,n+1\}$, there exists~$C(\beta,m,\mathsf{M},\data)<\infty$ such that, for $t \in (-1,1)$, 
\begin{equation}  \label{e.sec3contest}
\E \left[ \left\| 
\nabla \psi_{p + th,h}^{(m)} -   \nabla \psi_{p,h}^{(m)} - t  \nabla \psi_{p,h}^{(m+1)}
\right\|_{L^2(\cu_0)}^2 \right]^\frac12  \leq C |h|^{m+1} |t|^{1+\beta}.
\end{equation}
For $m \in \{1,\ldots,n\}$, we can take $\beta =1$ in~\eqref{e.sec3contest}. Finally, we have, for $m \in \{1,\ldots,n+2\}$, that, a.s. and a.e.,
\begin{equation}  \label{e.Fmrelationapplied}
\mathbf{f}^{m+1}_{p,h} = \left( D_p \a_p \cdot h \right) \nabla \psi_{p,h}^{(m)} + D_p \mathbf{f}^{m}_{p,h} \cdot h.
\end{equation}
\end{lemma}

\begin{proof} \
Fix $p \in \R^d$ and, without loss of generality, $h \in \partial B_1$.  By~\eqref{e.linearization.corr} we have that $p \mapsto p + \nabla \phi_p(x)$ is $C^n$ and $D_p^j \nabla \phi_p(x) h^{\otimes j} = \psi_{p,h}^{(j)}$ for every $j \in \{2,\ldots,n\}$ almost surely for almost every $x$. Thus~\eqref{e.sec3contest} is valid with $\beta=1$ for $m\in \{1,\ldots,n-1\}$
and, by~\eqref{e.Fmrelation}, 
\begin{equation}  \label{e.e.Fmrelationapplied1}
\mathbf{f}_{p,h}^{n+1}  = \left( D_p \mathbf{f}_{p,h}^{n}  + D_p \a_p \left(  \nabla \psi_{p,h}^{(n)} \right)^{\otimes 1} \right)\cdot h.
\end{equation}
We denote, in short, for $t \neq 0$, 
\begin{equation*} 
\zeta_{p,h,t}^{(n)} := \frac1t \left(  \psi_{p + th,h}^{(n)} -   \psi_{p,h}^{(n)} - t \psi_{p,h}^{(n+1)} \right). 
\end{equation*}
Observe that, by~\eqref{e.sec3corrbnd}, $\E \left[ \left\| 
\nabla \zeta_{p,h,t}^{(n)}
\right\|_{L^2(\cu_0)}^2 \right] < \infty$ for $t\neq 0$.  

\smallskip

\emph{Step 1}. We prove that, for $t\in (-1,1)$, $t\neq 0$,  
\begin{align}  \label{e.sec3contest1}
\E \left[ \left\| 
\nabla \zeta_{p,h,t}^{(n)}
\right\|_{L^2(\cu_0)}^2 \right] 
 &
\leq 
C \E \left[ \left\| D_p \a_{p} \cdot h \left( \nabla \psi_{p,h}^{(n)} - \nabla \psi_{p+th,h}^{(n)} \right) \right\|_{L^2(\cu_0)}^2 \right]
\\ \notag & \quad 
+
C \E \left[ \left\| \frac1t   \left( \a_{p+th} - \a_p - D_p \a_{p} \cdot h  \right) \nabla \psi_{p+th,h}^{(n)} \right\|_{L^2(\cu_0)}^2 \right]
\\ \notag & \quad 
+
C \E \left[ \left\|  \frac1t  \left( \mathbf{f}_{p+th,h}^{n} -  \mathbf{f}_{p,h}^{n} - D_p \mathbf{f}_{p,h}^{n}  \cdot h \right)  \right\|_{L^2(\cu_0)}^2 \right] .
\end{align}
To show~\eqref{e.sec3contest1}, we first claim that the difference quotient solves the equation
\begin{multline}  \label{e.sec3conteqn}
\nabla \cdot \left( \a_p \nabla \zeta_{p,h,t}^{(n)} \right)
=
 t \nabla \cdot \left( D_p \a_{p} \cdot h \left( \nabla \psi_{p,h}^{(n)} - \nabla \psi_{p+th,h}^{(n)} \right) \right)
\\ - \nabla \cdot \left( \left( \a_{p+th} - \a_p - D_p \a_{p} \cdot h  \right) \nabla \psi_{p+th,h}^{(n)} - \left(   \mathbf{f}_{p+th,h}^{n} -  \mathbf{f}_{p,h}^{n} - D_p \mathbf{f}_{p,h}^{n}  \cdot h \right) \right).
\end{multline}
Indeed, then~\eqref{e.sec3contest1} follows by Lemma~\ref{l.abstractnonsense}.  Rewriting 
\begin{align} \notag 
\a_p \nabla \zeta_{p,h,t}^{(n)}&
= \left( \a_{p+th} \nabla \psi_{p+th,h}^{(n)} + \mathbf{f}_{p+th,h}^{n} \right) -  \left( \a_{p} \nabla \psi_{p,h}^{(n)} + \mathbf{f}_{p,h}^{n} \right) - t 
\left( \a_{p} \nabla \psi_{p,h}^{(n+1)} + \mathbf{f}_{p,h}^{n+1} \right)
\\ \notag & \quad
+  t \left( \mathbf{f}_{p,h}^{n+1}  - D_p \a_{p} \cdot h \psi_{p+th,h}^{(n)}   - D_p \mathbf{f}_{p,h}^{n}  \cdot h \right)
\\ \notag & \quad 
 + t D_p \a_{p} \cdot h \left( \nabla \psi_{p,h}^{(n)} - \nabla \psi_{p+th,h}^{(n)} \right)
\\ \notag & \quad
- \left( \a_{p+th} - \a_p - D_p \a_{p} \cdot h  \right) \nabla \psi_{p+th,h}^{(n)} - \left(   \mathbf{f}_{p+th,h}^{n} -  \mathbf{f}_{p,h}^{n} - D_p \mathbf{f}_{p,h}^{n}  \cdot h \right),
\end{align}
we observe that the first three terms on the right are solenoidal by the equations of~$\psi_{p+th,h}^{(n)}$,~$\psi_{p,h}^{(n)}$ and~$\psi_{p,h}^{(n+1)}$, respectively, and the fourth term on the right is zero by~\eqref{e.e.Fmrelationapplied1}. We thus obtain~\eqref{e.sec3conteqn}.

\smallskip

\emph{Step 2}. We will estimate the terms on the right in~\eqref{e.sec3contest1} separately, and in this step we first show that, for $t\in (-1,1)$, $t\neq 0$, 
\begin{equation}  \label{e.sec3contest2}
\E \left[ \left\| D_p \a_{p} \cdot h \left( \nabla \psi_{p,h}^{(n)} - \nabla \psi_{p+th,h}^{(n)} \right) \right\|_{L^2(\cu_0)}^2 \right] \leq 
C |t|^{2\beta}\left( 1 + \E \left[   \left\| \nabla \zeta_{p,h,t}^{(n)} \right\|^2_{L^2(\cu_0)} \right]  \right)^{\beta} .
\end{equation}
By the triangle inequality we have 
\begin{equation*} 
\left| \nabla \psi_{p + th,h}^{(n)} - \nabla \psi_{p,h}^{(n+1)} \right| 
\leq  |t|^{\beta}
\left( \left| \nabla \psi_{p + th,h}^{(n)}\right| + \left|\nabla \psi_{p,h}^{(n)} \right| \right)^{1-\beta} 
\left( \left|  \nabla \psi_{p,h}^{(n+1)} \right|  +   \left| \nabla \zeta_{p,h,t}^{(n)} \right|  \right)^\beta.
\end{equation*}
Therefore, by H\"older's inequality and~\eqref{e.sec3corrbnd},
\begin{align} \notag 
\lefteqn{
 \E \left[ 
\left\| D_p \a_{p} h^{\otimes 1} \left(  \nabla \psi_{p + th,h}^{(n)} - \nabla \psi_{p,h}^{(n)}  \right) \right\|^2_{L^2(\cu_0)}
  \right]
} \quad &
\\ \notag &
 \leq C |t|^{2\beta} \E \left[ 
\left\|  D_p \a_{p} \right\|_{L^{\frac{4}{1-\beta}} (\cu_0) }^2  \left( \left\| \nabla \psi_{p + th,h}^{(n)}\right\|_{L^4(\cu_0)} + \left\| \nabla \psi_{p,h}^{(n)}\right\|_{L^4(\cu_0)}\right)^2
  \right]^{1-\beta}
  \\ \notag &
  \qquad \times 
  \left( \E \left[  \left\| \nabla \psi_{p,h}^{(n+1)} \right\|^2_{L^2(\cu_0)} \right] +
  \E \left[   \left\| \nabla \zeta_{p,h,t}^{(n)}  \right\|^2_{L^2(\cu_0)} \right]  \right)^{\beta}
  \\ \notag &
 \leq
 C  |t|^{2\beta}  \left( 1 +
   \E \left[   \left\| \nabla \zeta_{p,h,t}^{(n)} \right\|^2_{L^2(\cu_0)} \right]  \right)^{\beta},
\end{align}
which is~\eqref{e.sec3contest2}. 

\smallskip

\emph{Step 3}. 
We show that 
\begin{align}  \label{e.sec3contest3}
 \E \left[ \left\| \left( \a_{p+th} - \a_p - D_p \a_{p} \cdot h  \right) \nabla \psi_{p+th,h}^{(n)} \right\|_{L^2(\cu_0)}^2 \right] \leq C t^4 .
\end{align}
We have that 
\begin{equation*} 
\a_{p+th} - \a_p - D_p \a_{p} \cdot h = t^2 \int_0^1 s_1 \int_0^1 D_p^2 \a_{p+s_1s_2 th}  h^{\otimes 2} \, ds_1 \,ds_2 ,
\end{equation*}
and since 
\begin{equation*} 
 D_p^2 \a_{p}\bigg|_{p=z}  h^{\otimes2}  
 = D_p^3 L(z+ \phi_z,\cdot) \left( \nabla \psi_{z,h}^{(2)} \right)^{\otimes 1}  + D_p^4 L(p+ \phi_p,\cdot) \left( h + \nabla \psi_{p,h}^{(1)} \right)^{\otimes 2},
\end{equation*}
we obtain~\eqref{e.sec3contest3} by~\eqref{e.sec3corrbnd}. 

\smallskip

\emph{Step 4}. We then prove that 
 \begin{align}  \label{e.sec3contest4}
 \E \left[ \left\|  \mathbf{f}_{p+th,h}^{n} -  \mathbf{f}_{p,h}^{n} - D_p \mathbf{f}_{p,h}^{n}  \cdot h  \right\|_{L^2(\cu_0)}^2 \right] \leq C |t|^{2(1+\beta)}\left( 1 + \E \left[   \left\| \nabla \zeta_{p,h,t}^{(n)} \right\|^2_{L^2(\cu_0)} \right]  \right)^{\beta} .
\end{align}
Using decomposition~\eqref{e.Fmdecomposed}, we have that 
\begin{align} \notag 
D_p \mathbf{f}_{p,h}^{n} \cdot h 
& 
= D_p \left(  n  D_p \a_p  \cdot h \nabla \psi_{p,h}^{(n-1)}  +    \tilde{\mathbf{f}}_{p,h}^{n}  \right) \cdot h 
\\ & \notag 
= n D_p \a_p  \cdot h \nabla \psi_{p,h}^{(n)} + n D_p^2 \a_p h^{\otimes 2} \nabla \psi_{p,h}^{(n-1)} + D_p \tilde{\mathbf{f}}_{p,h}^{n}
\\  & \notag 
= n D_p \a_p  \cdot h \nabla \psi_{p,h}^{(n)}  +   \mathbf{g}_{p,h}^{n-1},
\end{align}
where 
\begin{equation*} 
\mathbf{g}_{p,h}^{n-1} :=  n D_p^2 \a_p h^{\otimes 2} \nabla \psi_{p,h}^{(n-1)} + D_p \tilde{\mathbf{f}}_{p,h}^{n} 
\end{equation*}
is a function of $\nabla \phi_p,\nabla \psi_{p,h}^{(1)}, \ldots,\nabla \psi_{p,h}^{(n-1)}$ for $n \geq 2$, and thus differentiable in~$p$. In particular, by~\eqref{e.sec3corrbnd},
for every $q \in [2,\infty)$ there is $\delta(q,n,\data)>0$ and $C(q,n,\mathsf{M},\data)<\infty$ such that 
\begin{equation*} 
\left\|          \mathbf{g}_{p,h}^{n-1}  \right\|_{L^q(\cu_0)} + \left\|         D_p \mathbf{g}_{p,h}^{n-1} \right\|_{L^q(\cu_0)}
\leq \O_\delta(C). 
\end{equation*}
Using the above decomposition for $D_p \mathbf{f}_{p,h}^{n} \cdot h$, we compute
\begin{align} \notag 
\lefteqn{
\mathbf{f}_{p+th,h}^{n} -  \mathbf{f}_{p,h}^{n} - D_p \mathbf{f}_{p,h}^{n}  \cdot h 
} \quad &
\\ \notag &
= t \int_0^1 \left( D_p \mathbf{f}_{p+s_1 t h,h}^{n} - D_p \mathbf{f}_{p,h}^{n}  \right) \cdot h \, ds_1
\\ \notag &
= 
n t \int_0^1  D_p \a_{p}  \cdot h \left(  \nabla \psi_{p+s t h,h}^{(n)} - \nabla \psi_{p,h}^{(n)} \right)
 \, ds
  \\ \notag & \quad
+  t^2 \int_0^1 s_1 \int_{0}^1 \left( n D_p^2 \a_{p+s_1s_2 t h} h^{\otimes 2}  \nabla \psi_{p+s_1 t h,h}^{(n)}  +  \left(    D_p \mathbf{g}_{p + s_1s_2 th ,h}^{n-1}   \right) \cdot h  \right) \, ds_1 \, ds_2 .
\end{align}
We have by~\eqref{e.sec3corrbnd} that 
\begin{equation*} 
\left\| D_p^2 \a_{p+s_1s_2 t h} h^{\otimes 2}  \nabla \psi_{p+s_1 t h,h}^{(n)}  +  \left(    D_p \mathbf{g}_{p + s_1s_2 th ,h}^{n-1}   \right) \cdot h \right\|_{L^q(\cu_0)} \leq \O_\delta(C),
\end{equation*}
and therefore 
\begin{equation*} 
\mathbf{f}_{p+th,h}^{n} -  \mathbf{f}_{p,h}^{n} - D_p \mathbf{f}_{p,h}^{n}  \cdot h  =  t \int_0^1  D_p \a_{p}  \cdot h \left(  \nabla \psi_{p+s t h,h}^{(n)} - \nabla \psi_{p,h}^{(n)} \right) + \O_\delta(Ct^2).
\end{equation*}
The right hand side can be estimated with the aid of~\eqref{e.sec3contest2} to obtain~\eqref{e.sec3contest4}.

\smallskip

\emph{Step 5}. Conclusion. Combining~\eqref{e.sec3contest1} with~\eqref{e.sec3contest2},~\eqref{e.sec3contest3} and~\eqref{e.sec3contest4} yields~\eqref{e.sec3contest} by Young's inequality. Now,~\eqref{e.sec3contest} implies~\eqref{e.gradpsiformula} for $m=n$. Therefore, we may replace $n$ by $n+1$ in Steps 1-4 above, and conclude that~\eqref{e.sec3contest} is valid for $m=n+1$ as well, which then gives~\eqref{e.gradpsiformula} for $m=n+1$. Using obtained formula, it is straightforward to show that~\eqref{e.sec3contest} is valid for $m=n$ with $\beta=1$. Indeed, we notice that
\begin{equation*} 
\E \left[ \left\| 
\nabla \psi_{p + th,h}^{(m)} -   \nabla \psi_{p,h}^{(m)} - t  \nabla \psi_{p,h}^{(m+1)} - \frac{t^2}2  \nabla \psi_{p,h}^{(m+2)}
\right\|_{L^2(\cu_0)}^2 \right]^\frac12 \leq C |t|^{2+\beta}, 
\end{equation*}
from which we get~\eqref{e.sec3contest} for $m=n$ with $\beta=1$. Finally,~\eqref{e.sec3contest} implies that $t \mapsto \nabla \phi_{p+th}$ is in $C^{n+2,\beta}$ close to $t=0$, and thus we have that~\eqref{e.Fmrelationapplied} is valid by~\eqref{e.Fmrelation}. The proof is complete.
\end{proof}

\begin{lemma} \label{l.polarizationD_pphi}
Assume~\eqref{e.assumption.section3} is valid. Let $\mathsf{M} \in [1,\infty)$ and $p \in B_{\mathsf{M}}$. Then $p \mapsto p + \nabla \phi_p$ is $(n+2)$ times differentiable with respect to $p$ and, for $q \in [2,\infty)$, there are constants $\delta(q,n,\data) \in \left(0, \tfrac12  \right]$ and $C(q,n,\mathsf{M},\data)<\infty$ such that, for $m \{1,\ldots,n+1\}$, 
\begin{equation*} 
\left\|  D_p^{m} \phi_p \right\|_{L^q(\cu_0)} \leq \O_\delta(C)
\quad \mbox{and} \quad 
\E \left[ \left\| D_p^{n+2} \phi_p \right\|_{L^2(\cu_0)}^2 \right] \leq C.
\end{equation*}
\end{lemma}

The proof of Lemma~\ref{l.polarizationD_pphi} relies on a general principle in polarization based on multilinear analysis, and it is formalized in the following lemma. 

\begin{lemma}
\label{l.polarization}
Let $V$ be a real, finite-dimensional vector space, and let $\Phi:V^{n}\to\mathbb{R}$ be a multilinear, symmetric form, that is,
for all $v_1,\ldots,v_n  \in V$ and any permutation $\sigma$ of $\{1,\ldots,n\}$, we have
\begin{equation*} 
\Phi(v_1,\ldots,v_{n}) = \Phi(v_{\sigma(1)},\ldots,v_{\sigma(n)}) . 
\end{equation*}
For $v\in V$, define $\phi(v):=\Phi(v, v, \dots, v)$.
Then, for $v_{1}, \dots, v_{n}\in V$, the polarization formula
\[
\Phi(v_{1}, \dots, v_{n})=\frac{1}{n!}\sum_{A\subseteq\{1,\dots, n\}}(-1)^{n-|A|}\phi\left(\sum_{j\in A}v_{j}\right)
\]
holds, where the leftmost summation is over all non-empty subsets $A\subseteq \{1,2,\dots, n\}$ and $|A|$ is the number of elements in $A$.
\end{lemma}

\begin{proof}
We show the equivalent statement that
\[
\sum_{A\subseteq\{1,\dots, n\}}(-1)^{n-|A|}\phi\left(\sum_{j\in A}v_{j}\right)=n!\Phi(v_{1}, \dots, v_{n}),
\]
where the sum on the left is over all non-empty subsets $A$ of $\{1,2,\dots, n\}$.
For this, we begin by expanding each summand $\phi\left(\sum_{j\in A}v_{j}\right)=\Phi\left(\sum_{j\in A}v_{j},\dots, \sum_{j\in A}v_{j}\right)$ fully, as a sum of terms of the form $\Phi(v_{j_{1}}, \dots, v_{j_{n}})$ with $j_1,\dots, j_n\in A$.
Using the symmetry of $\Phi$, each such term can be written as $\Phi(v_{f(1)}, \dots, v_{f(n)})$, with non-decreasing indices $f(1)\leq f(2)\leq \cdots \leq f(n)$ in $A$.
Denote
\begin{equation*} 
\mathcal{M} := \{ f: \{1,\ldots,n\} \to \{1,\ldots,n\} \, : \, f(1) \leq f(2) \leq \ldots \leq f(n) \} 
\end{equation*}
and, for $f \in \mathcal{M}$, 
\begin{equation*} 
 \text{im}\, f := \bigcup_{j \in \{1,\ldots,n\}}\{ f(j)  \} .
\end{equation*}
Letting $c_{A}(f)$ denote the number of ordered $n$-tuples $(j_{1}, \dots, j_{n})$ of elements of $A$ which can be reordered to form $(f(1),\dots, f(n))$, it follows that the expression $\sum_{A\subseteq\{1,\dots, n\}}(-1)^{n-|A|}\phi\left(\sum_{j\in A}v_{j}\right)$ can be expanded to give
\[
\sum_{A\subseteq\{1,\dots, n\}}(-1)^{n-|A|}\sum_{f\in \mathcal{M} , \; \text{ im} \, f \subset  A}c_{A}(f)\Phi(v_{f(1)}, \dots, v_{f(n)}).
\]
Changing the order of summation gives
\[
\sum_{ f\in \mathcal{M} }c_{\{1,\dots, n\}}(f)\Phi(v_{f(1)},\dots, v_{f(n)})\sum_{A\supseteq \text{im}\,f}(-1)^{n-|A|},
\]
where the sum on the right is over all subsets $A$ of $\{1,2,\dots, n\}$ which contain $\text{im}\,f$. 
Each such subset $A$ can be written as $\{f(1), \dots, f(n)\}\cup B$, for some set $B$, possibly empty, satisfying $B\subseteq \{1,2,\dots, n\}\setminus \text{im}\,f $.
Hence, as $|A|=|\text{im}\, f|+|B|$ for $B$ defined in this way, we can write our expression as
\[
\sum_{ f\in \mathcal{M} } c_{\{1,\dots, n\}}(f)\Phi(v_{f(1)},\dots, v_{f(n)})s_{f},
\]
where 
\[
s_{f}=(-1)^{n-|\text{im}\, f| }\sum_{B\subseteq ( \{1,\dots, n\}\setminus \text{im}\, f)}(-1)^{|B|}.
\]
Finally, it is a well-known combinatorial fact that for every nonempty finite set $S$, we have
\begin{equation}
\sum_{B\subseteq S}(-1)^{|B|}= \sum_{j=0}^{|S|} (-1)^j { \binom{|S|}{j}} = 0,
\end{equation}
where the sum on the left is over all subsets $B$ of $S$. Thus, above, we have
\[
s_{f}=0
\]
unless $f(1)< \cdots < f(n)$. Therefore $f(1)=1, f(2)=2, \dots, f(n)=n$, and in this case $s_{f}=1$ and $c_{\{1,\dots, n\}}(f)=n!$. It follows that
\[
\sum_{A\subseteq\{1,\dots, n\}}(-1)^{n-|A|}\phi\left(\sum_{j\in A}v_{j}\right)=n!\Phi(v_{1},\dots, v_{n}),
\]
as was to be shown; this proves the polarization formula.
\end{proof}

\begin{proof}[Proof of Lemma~\ref{l.polarizationD_pphi}]
The lemma follows from Lemmas~\ref{l.psinplus2},~\ref{l.stupidpsi},~\ref{l.polarization}, and~\eqref{e.sec3corrbnd}. 
\end{proof}

We next prove H\"older continuity of $p\mapsto D_p^{n+2} \phi_p$ and $p \mapsto \mathbf{f}^{n+2}_{p,h}$.
\begin{lemma} \label{l.continuityinp}
Assume~\eqref{e.assumption.section3} is valid. Let $\mathsf{M} \in [1,\infty)$ and $\beta \in (0,1)$. Then there is a constant $C(\beta,n,\mathsf{M},\data)<\infty$ such that, for all $p,p' \in B_{\mathsf{M}}$ and $h \in \overline{B}_1\setminus\{0\}$, 
\begin{equation}  \label{e.continuityinp}
\E \left[ \left\| \nabla D_p^{n+2} \phi_p - \nabla D_p^{n+2} \phi_{p'} \right\|_{L^2(\cu_0)}^2 \right]^\frac12 + |h|^{-n-2} \E \left[ \left\| \mathbf{f}^{n+2}_{p',h} - \mathbf{f}^{n+2}_{p,h}   \right\|_{L^2(\cu_0)}^2 \right]^\frac12 \leq C \left| p- p' \right|^\beta.
\end{equation}
\end{lemma}

\begin{proof}
Fix $\mathsf{M} \in [1,\infty)$ and $\beta \in (0,1)$, and fix $p,p' \in B_{\mathsf{M}}$ and $h \in \overline{B}_1\setminus\{0\}$.
By Lemma~\ref{l.abstractnonsense} and equations of $\psi_{p,h}^{(n+2)}$ and $\psi_{p',h}^{(n+2)}$ we have that 
\begin{equation*} 
\E \left[ \left\|  \nabla \psi_{p,h}^{(n+2)} - \nabla \psi_{p',h}^{(n+2)}\right\|_{L^2(\cu_0)}^2 \right] \leq 
C \E \left[ \left\| \mathbf{f}^{n+2}_{p',h} - \mathbf{f}^{n+2}_{p,h}   \right\|_{L^2(\cu_0)}^2 \right] .
\end{equation*}
Therefore, in view of Lemma~\ref{l.polarization}, it is enough to show that 
\begin{equation}  \label{e.continuityinppre}
\E \left[ \left\| \mathbf{f}^{n+2}_{p',h} - \mathbf{f}^{n+2}_{p,h}   \right\|_{L^2(\cu_0)}^2 \right]^\frac12 \leq C|h|^{n+2} |p-p'|^\beta .
\end{equation}
By~\eqref{e.Fmdecomposed} we may decompose $\mathbf{f}^{n+2}_{p,h}$ as
\begin{align}  \label{e.fnplus2decomposed}
\mathbf{f}^{n+2}_{p,h} & =  (n+2) D_p^{3} L(p + \nabla \phi_p, \cdot)\left((h + \nabla (D_p \phi_{p} h^{\otimes 1})\right)^{\otimes 1} \left( \nabla \psi_{p,h}^{(n+1)} \right)^{\otimes 1}  
\\ \notag & 
\quad 
+  \tilde{\mathbf{F}}_{n+2} \left(p + \nabla \phi_p ,h + \nabla  \left(D_p \phi_{p} h^{\otimes 1}\right) ,\ldots,  \nabla \left( D_p^n \phi_{p} h^{\otimes n}\right) ,\cdot\right) .
\end{align}
Since~$\tilde{\mathbf{F}}_{n+2}$ is $C^{0,1}$ in its first argument and polynomial in its last $n$ arguments, we obtain by homogeneity, the chain rule and Lemma~\ref{l.polarizationD_pphi} that 
\begin{align} \notag 
 & \E \bigg[ 
\bigg\|   \tilde{\mathbf{F}}_{n+2} \left(p {+}\nabla \phi_p ,h {+} \nabla  \left(D_p \phi_{p} h^{\otimes 1}\right)  ,{\ldots},\nabla \left( D_p^n \phi_{p} h^{\otimes n}\right) ,\cdot \right)
\\ \notag & \ \
 - \tilde{\mathbf{F}}_{n+2} \left(p' {+}\nabla \phi_{p'} ,h {+} \nabla \left(D_p \phi_{p'} h^{\otimes 1}\right),{\ldots},\nabla \left( D_p^n \phi_{p'} h^{\otimes n}\right) ,\cdot \right)
\bigg\|_{L^2(\cu_0)}^2 \bigg]^\frac12  
\leq 
C|h|^{n+2} |p-p'| . 
\end{align}
Therefore, the leading term is the first one on the right in~\eqref{e.fnplus2decomposed}. Indeed, we observe by the above display that
\begin{multline*} 
\E \left[ \left\| \mathbf{f}^{n+2}_{p',h} - \mathbf{f}^{n+2}_{p,h}   \right\|_{L^2(\cu_0)}^2 \right]^\frac12 
\\
\leq C \E \left[ \left\| \left| h + \nabla (D_p \phi_{p} h^{\otimes 1})\right| \left| \nabla \psi_{p',h}^{(n+1)} - \nabla \psi_{p,h}^{(n+1)} \right| \right\|_{L^2(\cu_0)}^2 \right]^\frac12 
+ C  |h|^{n+2} |p-p'| 
\end{multline*}
To conclude, we need to estimate the first term on the right. To this end, we use the triangle inequality to get, for any $\beta \in [0,1]$, 
\begin{equation*} 
\left| \nabla \psi_{p',h}^{(n+1)} - \nabla \psi_{p,h}^{(n+1)} \right|
\leq 
\left( 
\left| \nabla \psi_{p',h}^{(n+1)}\right|  + \left| \nabla \psi_{p,h}^{(n+1)} \right| 
\right)^{1-\beta} 
\left( 
\left| \nabla \psi_{p',h}^{(n+1)} - \nabla \psi_{p,h}^{(n+1)} \right| 
\right)^{\beta}.
\end{equation*}
It follows, by H\"older's inequality, that 
\begin{align} \notag 
\lefteqn{
\E \left[ \left\| \mathbf{f}^{n+2}_{p',h} - \mathbf{f}^{n+2}_{p,h}   \right\|_{L^2(\cu_0)}^2 \right]^\frac12 
} \quad &
\\ \notag &
\leq
C \E \left[ \left\| h + \nabla (D_p \phi_{p} h^{\otimes 1}) \right\|_{L^{\frac{4}{1-\beta}}(\cu_0)}^{\frac{4}{1-\beta}} \right]^{\frac {1-\beta}{4} }
\E \left[ \left\| \left|  \nabla \psi_{p',h}^{(n+1)} \right| + \left| \nabla \psi_{p,h}^{(n+1)}\right|   \right\|_{L^4(\cu_0)}^4 \right]^{\frac {1-\beta}{4} } 
\\ \notag & \qquad \times  \E \left[ \left\|  \nabla \psi_{p',h}^{(n+1)} - \nabla \psi_{p,h}^{(n+1)}  \right\|_{L^2(\cu_0)}^2 \right]^{\frac{\beta}{2}} .
\end{align}
By Lemma~\ref{l.polarizationD_pphi} and H\"older's inequality, we get
\begin{align} \notag 
\lefteqn{
 \E \left[ \left\|  \nabla \psi_{p',h}^{(n+1)} - \nabla \psi_{p,h}^{(n+1)}  \right\|_{L^2(\cu_0)}^2 \right]^{\frac{\beta}{2}}
} \quad &
\\ \notag &
 \leq 
  C \E \left[ \left\|  \nabla D_p^{n+1} \phi_{p'}  - \nabla D_p^{n+1} \phi_{p} \right\|_{L^2(\cu_0)}^2 \right]^{\frac{\beta}{2}}  |h|^{\beta(n+1)}
\\ \notag &
\leq 
  C \E \left[ \left\|  \nabla \int_0^1 D_p^{n+2} \phi_{t p' + (1-t)p} \,dt  \right\|_{L^2(\cu_0)}^2 \right]^{\frac{\beta}{2}}  |h|^{\beta(n+1)}|p-p'|^\beta
  \\ \notag &
\leq 
  C \left( \int_0^1 \E \left[ \left\|  \nabla  D_p^{n+2} \phi_{t p' + (1-t)p}  \right\|_{L^2(\cu_0)}^2 \right]^{\frac{1}{2}} \right)^\beta  \,dt  |h|^{\beta(n+1)}|p-p'|^\beta
   \\ \notag &
\leq  
   C |h|^{\beta(n+1)}|p-p'|^\beta,
\end{align}
together with
\begin{equation*} 
\E \left[ \left\| h + \nabla (D_p \phi_{p} h^{\otimes 1}) \right\|_{L^{\frac{4}{1-\beta}}(\cu_0)}^{\frac{4}{1-\beta}} \right]^{\frac {1-\beta}{4} } \leq C|h|
\end{equation*}
and
\begin{equation*} 
\E \left[ \left\| \left|  \nabla \psi_{p',h}^{(n+1)} \right| + \left| \nabla \psi_{p,h}^{(n+1)}\right|   \right\|_{L^2(\cu_0)}^4 \right]^{\frac {1-\beta}{4} }  \leq C|h|^{(1-\beta)(n+1)}. 
\end{equation*}
Consequently, combining above displays yields~\eqref{e.continuityinppre}, finishing the proof.
\end{proof}

\subsection{Smoothness of~\texorpdfstring{$\overline{L}$}{L-bar}}

It is a well-known and easy consequence of qualitative homogenization that $D\overline{L}$ is given by the formula
\begin{equation*}
D\overline{L}(p) 
=
\E \left[ 
\int_{\cu_0}  
D_pL \left( p + \nabla \phi_p(x),x\right)\,dx
\right].
\end{equation*}
In~\cite{AFK}, we proved that $\overline{L}\in C^2$ and $D^2\overline{L}$ is given by the formula
\begin{equation} \label{e.D2pbarL}
D^2\overline{L}(p) h
=
\E \left[ 
\int_{\cu_0}  
\a_p
\left( h + \nabla \psi_{p,h}^{(1)}\right)
\,dx
\right].
\end{equation}
We next generalize this result to higher derivatives of $\overline{L}$ by essentially differentiating the previous formula many times and applying the results of the previous subsection. Moreover, we validate the statement of Theorem~\ref{t.regularity.Lbar} for $n+1$. 

\begin{proposition}
\label{e.Lbar.reg.qualitative}
Assume that~\eqref{e.assumption.section3} is valid. 
Then Theorem~\ref{t.regularity.Lbar} is valid for $n+1$. Moreover, for every $m\in\{ 1,\ldots,n+1\}$ and $h\in\overline{B}_1$, we have the formula 
\begin{equation}
\label{e.formulaDervsL}
D_p^{m+2} \overline{L}(p) h^{\otimes (m+1)} =  \E\left[ \int_{\cu_0}   \left( \a_p \nabla \psi^{(m+1)}_{p,h} + \mathbf{f}^{(m+1)}_{p,h} \right) \right].
\end{equation}
\end{proposition}
\begin{proof} 
Fix $\mathsf{M} \in [1,\infty)$ and $p \in B_{\mathsf{M}}$. We begin by showing~\eqref{e.formulaDervsL}. Starting from~\eqref{e.D2pbarL}, and observing that since
\begin{equation*} 
\mathbf{f}^{(2)}_{p,h} = D_p^3L(p+\nabla \phi_p,\cdot) \left(h + \nabla\psi_{p,h }^{(1)} \right)^{\otimes 2}  ,
\end{equation*}
we have that 
\begin{equation*} 
D_p \left(\a_p \left( h + \nabla \psi_{p,h}^{(1)}\right) \right) \cdot h =  \a_p \nabla \psi_{p,h}^{(2)} + \mathbf{f}^{(2)}_{p,h}. 
\end{equation*}
This implies that
\begin{equation*} 
D_p^3 \overline{L}(p) h^{\otimes 2} = \E\left[ \int_{\cu_0}   \left( \a_p \nabla \psi_{p,h}^{(2)} + \mathbf{f}^{(2)}_{p,h}\right) \right] .
\end{equation*}
Assume then, inductively, that for some $m \in \{3,\ldots,n+1\}$ we have that,  for all $ k \in \{2,\ldots,m\}$,  
\begin{equation} 
\label{e.diffL}
D_p^{k+1} \overline{L}(p) h^{\otimes k} =  \E\left[ \int_{\cu_0}   \left( \a_p \nabla \psi^{(k)}_{p,h} + \mathbf{f}^{(k)}_{p,h} \right) \right].
\end{equation}
We prove that~\eqref{e.diffL} is valid for $k=m+1$ as well.  Differentiating with respect to $p$ yields, using~\eqref{e.gradpsiformula} and~\eqref{e.Fmrelationapplied}, that 
\begin{align} \notag 
D_p^{m+2} \overline{L}(p) h^{\otimes (m+1)} & =  \E\left[ \int_{\cu_0}   \left( D_p \left( \a_p \nabla \psi^{(m)}_{p,h} + \mathbf{f}^{(m)}_{p,h} \right) \cdot h\right) \right]
\\ \notag &
= 
\E\left[ \int_{\cu_0}   \left( \a_p \nabla \psi^{(m+1)}_{p,h} +  \left( D_p \a_p \cdot h \nabla \psi^{(m)}_{p,h} + D_p \mathbf{f}^{(m)}_{p,h} \cdot h \right) \right) \right]
\\ \notag &
= 
\E\left[ \int_{\cu_0}   \left( \a_p \nabla \psi^{(m+1)}_{p,h} +    \mathbf{f}^{(m+1)}_{p,h}  \right) \right],
\end{align}
proving the induction step. This validates~\eqref{e.formulaDervsL}.

\smallskip

To show the regularity of $\overline{L}$, we first observe that Theorem~\ref{t.correctorestimates} and Lemma~\ref{l.psinplus2}, together with~\eqref{e.diffL} and Lemma~\ref{l.polarization}, yield that 
\begin{equation}  \label{e.DpLreg1}
\max_{k \in \{2,\ldots,n+3\} }  \left\| D_p^{k} \overline{L}(p) \right\|_{L^\infty(B_{\mathsf{M}})} \leq C. 
\end{equation}
Fix then $p' \in B_{\mathsf{M}}$. Since  
\begin{equation*} 
\nabla \psi^{(n+2)}_{p,h}  = \nabla \left(D_p^{n+2} \phi_p h^{\otimes (n+2)} \right),
\end{equation*}
decomposing 
\begin{align} \notag 
\lefteqn{
\a_p \nabla \psi^{(m+1)}_{p,h} +    \mathbf{f}^{(m+1)}_{p,h}  - (\a_{p'} \nabla \psi^{(m+1)}_{p',h} +    \mathbf{f}^{(m+1)}_{p',h} )
} \quad &
\\ \notag &
=
\a_p \nabla \left(( D_p^{n+2} \phi_p -   D_p^{n+2} \phi_{p'})  h^{\otimes (n+2)} \right)  
\\ \notag & \quad
+ (\a_p - \a_{p'}) \nabla \left(D_p^{n+2} \phi_{p'})  h^{\otimes (n+2)} \right)   + \left( \mathbf{f}^{(m+1)}_{p,h} - \mathbf{f}^{(m+1)}_{p',h} \right),
\end{align}
and noticing that 
\begin{align} \notag 
 | \a_p(x) - \a_{p'}(x)| & \leq  \left[D_p^2 L(\cdot,x) \right]_{C^{0,1}} \left(|p-p'| +  \left| \nabla \phi_p(x) - \nabla \phi_{p'}(x)\right| \right)
\\ \notag &
\leq C\int_0^1 (1+ \left| \nabla D_p \phi_{t p + (1-t)p'}(x) \right|) \,dt \,  |p-p'|,
\end{align}
we obtain by  Lemmas~\ref{l.polarizationD_pphi} and~\ref{l.continuityinp}, together with H\"older's inequality, that 
\begin{equation*} 
\left| D_p^{n+3} \overline{L}(p) h^{\otimes (n+3)} - D_p^{n+3} \overline{L}(p') h^{\otimes (n+3)} \right| \leq C|h|^{n+3}|p-p'|^\beta. 
\end{equation*}
In view of Lemma~\ref{l.polarization} this yields
\begin{equation*} 
\left| D_p^{n+3} \overline{L}(p)  - D_p^{n+3} \overline{L}(p') \right| \leq C |p-p'|^\beta,
\end{equation*}
proving that $[D_p^{n+3} \overline{L}]_{C^{0,\beta}(B_{\mathsf{M}})} \leq C$. Together with~\eqref{e.DpLreg1} we thus get
\begin{equation*} 
\left\| D_p^2 \overline{L} \right\|_{C^{n+1,\beta}(B_{\mathsf{M}})} \leq C,
\end{equation*}
which is the statement of Theorem~\ref{t.regularity.Lbar} for $n+1$. The proof is complete. 
\end{proof}
\begin{remark} \label{r.barFregularity}
Since we now have that Theorem~\ref{t.regularity.Lbar} is valid for $n+1$, we also have that $p \mapsto \overline{F}_{n+2}(p,\cdot,\ldots,\cdot)$ belongs to $C^{0,\beta}$ for all~$\beta \in (0,1)$.  In particular, for~$\beta \in (0,1)$ and $\mathsf{M}\in [1,\infty)$, there exists $C(n,\beta,\mathsf{M},\data)<\infty$ such that, for every tuplet $(h_1,\ldots,h_{n+1}) \in \R^d \times \ldots \R^d$, we have that 
\begin{equation*} 
\left[ F_{n+2}(\cdot ,h_1,\ldots,h_{n+1})  \right]_{C^{0,\beta}(B_{\mathsf{M}} )} \leq C \sum_{i=1}^{n+1} |h_i|^{\frac{n+2}{i}}.
\end{equation*}
\end{remark}


\subsection{Sublinearity of correctors}

By the ergodic theorem, we have that, for every $p,h\in\Rd$ and $m\in\{1,\ldots,n+1\}$, the correctors and linearized correctors are (qualitatively) sublinear at infinity:
\begin{equation*} \label{}
\left\{
\begin{aligned}
&
\limsup_{r\to \infty} 
\frac1r \left\| \phi_{p} - \left( \phi_p \right)_{B_r} \right\|_{\underline{L}^2(B_r)} = 0, 
\\ & 
\limsup_{r\to \infty} 
\frac1r \left\| \psi_{p,h}^{(m)} - \left( \psi_{p,h}^{(m)} \right)_{B_r} \right\|_{\underline{L}^2(B_r)} = 0, 
\end{aligned}
\right.
\qquad \mbox{$\P$--a.s.}
\end{equation*}
The assumption~\eqref{e.assumption.section3} allows us to give a quantitative estimate of this sublinearity.

\begin{lemma}[Sublinearity of correctors]
\label{l.corr.sublinearity}
Assume~\eqref{e.assumption.section3} is valid. Let $\mathsf{M} \in [1,\infty)$. There exist~$\alpha(\data),\delta(\data) >0$, $C(\mathsf{M},\data)<\infty$ and a random variable $\X$ satisfying $\X \leq \O_\delta(C)$ such that, for every $r\geq \X$,~$p \in B_{\mathsf{M}}$, $h\in \overline{B}_1$ and  $m\in\{1,\ldots,n+1\}$, 
\begin{equation} 
\label{e.correctorsublin}
\left\| \phi_{p} - \left( \phi_p \right)_{B_r} \right\|_{\underline{L}^2(B_r)}
+
\left\| \psi_{p,h}^{(m)} - \left( \psi_{p,h}^{(m)} \right)_{B_r} \right\|_{\underline{L}^2(B_r)}
\leq 
Cr^{1-\alpha}. 
\end{equation}
\end{lemma}
\begin{proof}
Fix $p \in B_\mathsf{M}$ and $h \in \overline{B}_1$. Clearly $x \mapsto p\cdot x + \phi_p(x)$ belongs to $\mathcal{L}_1$, and thus the result follows from~\cite[Theorem 1.3]{AFK} as in~\cite[Section 3.4]{AKMbook}. Hence we are left to show that $\psi_{p,h}^{(m)} $ satisfies the estimate in the statement. For $m=1$ the result follows by~\cite[Theorem 5.2]{AFK} and~\cite[Section 3.4]{AKMbook}. We thus proceed inductively. Assume that~\eqref{e.correctorsublin} is valid for $m \in \{1,\ldots,k\}$ for some $k \in \{1,\ldots, n\}$. We then show that it continues to hold for $m=k+1$. Since~\eqref{e.assumption.section3} is valid, by taking $\alpha$ and $\X$ as in Theorem~\ref{t.linearizehigher}, and setting $\Y := \X^{2/\alpha}$, we have that if $R:=\ep^{-1}\geq \Y$, then $\ep^\alpha \X \leq R^{-\frac \alpha2}$. We relabel $\Y$ to $\X$ and $\frac \alpha2$ to $\alpha$. We further take $\X$ larger, if necessary, so that~\cite[Proposition 4.3]{AFK} is at our disposal.  Suppressing both $p$ and $h$ from the notation, we let $\phi_{r}$ and $\psi_{r}^{(m)}$, for $m \in \{1,\ldots,n+1\}$,  solve 
 \begin{equation}
\label{e.corr.hom000}
\left\{
\begin{aligned}
& -\nabla \cdot \left( D_pL(p + \nabla \phi_r ,\cdot )\right)  = 0   & \mbox{in} & \ B_{2^{n+1}r},\\
&-\nabla \cdot \left( D_p^2L(p + \nabla \phi ) \nabla \psi^{(m)}_{r} \right) = \nabla \cdot \f^{(k+1)}    & \mbox{in} & \ B_{2^{n+1-m}r},\\
&-\nabla \cdot \left( D_p^2L(p + \nabla \phi_r ) \nabla \tilde \psi^{(m)}_{r} \right) = \nabla \cdot \f^{(k+1)}_r    & \mbox{in} & \ B_{2^{n+1-m}r},\\
& \phi_r = 0   & \mbox{on} & \ \partial B_{2^{n+1} r}, \\
& \psi^{(m)}_{r} = 0   & \mbox{on} & \ \partial B_{2^{n+1-m}r},
\end{aligned}
\right.
\end{equation}
where
\begin{align} \notag 
\f^{(m)} & := \mathbf{F}_m\left( p+\nabla\phi,h+\nabla\psi^{(1)}, \nabla\psi^{(2)}, \ldots,  \nabla\psi^{(m-1)},\cdot \right) ,
\\ \notag
\f^{(m)}_r & := \mathbf{F}_m\left( p+\nabla\phi,h+\nabla\psi^{(1)}_{r}, \nabla\psi^{(2)}_{r}, \ldots,  \nabla\psi^{(m-1)}_{r},\cdot \right). 
\end{align}
Now, $\phi_r, \psi_r^{(1)},\ldots, \psi_r^{(m-1)}$ all homogenize to zero and we get, by Theorem~\ref{t.linearizehigher}, that, for~$r \geq \X$, 
\begin{equation*} 
\left\| \phi_{r}  \right\|_{\underline{L}^2(B_r)}
+
\left\| \tilde \psi_{r}^{(m)} \right\|_{\underline{L}^2(B_{r})}
\leq 
Cr^{1-\alpha}
. 
\end{equation*}
This and the induction assumption, i.e. that~\eqref{e.correctorsublin} is valid for $m \in \{1,\ldots,k\}$,  together with Lemma~\ref{l.diff.linearizedsystem} below, imply that 
\begin{equation*} 
\frac1r \left\|  \psi_r^{(k+1)}  - \tilde \psi_r^{(k+1)}  \right\|_{\underline{L}^2 \left( B_{r} \right)} + \left\| \f^{(k+1)}  - \f_r^{(k+1)}  \right\|_{\underline{L}^2 \left( B_{r} \right)} \leq C r^{-\alpha}.
\end{equation*}
Combining the previous two displays yields
\begin{equation*} 
\left\| \phi_{r}  \right\|_{\underline{L}^2(B_r)}
+
\left\| \psi_{r}^{(k+1)} \right\|_{\underline{L}^2(B_{r})}
\leq 
Cr^{1-\alpha}
. 
\end{equation*}
Now, since $\psi_{2r}^{(k+1)}- \psi_{r}^{(k+1)}$ is $\a_p$-harmonic in $B_r$, we have by the Lipschitz estimate~\cite[Proposition 4.3]{AFK} that, for $r \geq \X$ and $t \in [\X,r]$, 
\begin{equation*} 
\left\|\nabla \psi_{2r}^{(k+1)} - \nabla \psi_{r}^{(k+1)}  \right\|_{\underline{L}^2(B_{t})} 
\leq 
\frac{C}{r} \left\| \psi_{2r}^{(k+1)} -  \psi_{r}^{(k+1)}  \right\|_{\underline{L}^2(B_{r})}
\leq 
Cr^{-\alpha}
.
\end{equation*}
Therefore, by compactness, there exists $\hat \psi^{(k+1)}$ such that, for $t \in [\X,r]$,  
\begin{equation*} 
\frac1r \left\| \hat  \psi^{(k+1)}  \right\|_{\underline{L}^2(B_{r})}  + \left\|\nabla \psi_{r}^{(k+1)} - \nabla \hat  \psi^{(k+1)}  \right\|_{\underline{L}^2(B_{t})} \leq C r^{-\alpha}
.
\end{equation*}
Proceeding now as in~\cite[Section 3.4]{AKMbook} proves that $\nabla \hat \psi^{(k+1)}$ is $\Z^d$-stationary. Finally, by integration by parts we also obtain that $\E \left[ \int_{\cu_0} \nabla \hat \psi^{(k+1)} \right] = 0$ and, therefore, since $\hat \psi^{(k+1)}$ solves the same equation as $\psi^{(k+1)}$, by the uniqueness we have that 
$\hat \psi^{(k+1)} = \psi^{(k+1)}$ up to a constant. The proof is hence complete by the previous display. 
\end{proof}

Above we made use of a lemma, which roughly states that if two solutions of the systems of linearized equations are close in $L^2$ then their gradients are also close. This  lemma will also be applied repeatedly in the following sections. 

\begin{lemma}
\label{l.diff.linearizedsystem}
Suppose that~\eqref{e.assumption.section4} holds. Let~$q\in [2,\infty)$ and $\mathsf{M}\in [1,\infty)$. There exist~$\delta(q,\data)>0$,~$C(q,\mathsf{M},\data) <\infty$ and a random variable~$\X \leq \O_{\delta}\left( C \right)$ such that the following holds. Let $R\geq \mathcal{X}$, $N\leq n+1$ and $\left(u,w_1,\ldots,w_{N}\right),\left(\tilde{u},\tilde{w}_1,\ldots,\tilde{w}_{N} \right) \in (H^1(B_R))^{N+1}$ each be a solution of the system of the linearized equations, that is, for every
$m\in\{1,\ldots,N\}$, we have 
\begin{equation*}
\left\{
\begin{aligned}
&
\left\| \nabla u \right\|_{\underline{L}^2(B_R)} \vee 
\left\| \nabla v \right\|_{\underline{L}^2(B_R)} \leq \mathsf{M}
\\ & 
-\nabla \cdot \left( D_pL(\nabla u,x) \right) = 0
\quad \mbox{and} \quad -\nabla \cdot \left( D_pL(\nabla v,x) \right) = 0
\quad \mbox{in} \ B_R,\\
& 
-\nabla \cdot  \left( D^2_pL\left( \nabla u,x \right) \nabla w_m \right) = \nabla \cdot \left( \mathbf{F}_m(\nabla u,\nabla w_1,\ldots,\nabla w_{m-1},x)\right) \quad \mbox{in}  \ B_R,
\end{aligned}
\right.
\end{equation*}
and the same holds with $\left(\tilde{u},\tilde{w}_1,\ldots,\tilde{w}_{N} \right)$ in place of~$\left(u,w_1,\ldots,w_{N}\right)$. Then 
\begin{equation} 
\label{e.diff.utildeu}
\left\| \nabla u- \nabla\tilde{u} \right\|_{\underline{L}^q(B_{R/2})} 
\leq  
\frac CR\left\|  u- \tilde{u} \right\|_{\underline{L}^2(B_R)}
\end{equation}
and, denoting 
\begin{equation*} \label{}
\mathsf{h}_i
:=
 \frac 1R\left\|  w_i - (w_i)_{B_R} \right\|_{\underline{L}^2(B_R)} + \frac 1R\left\|  \tilde{w}_i - (\tilde{w}_i)_{B_R} \right\|_{\underline{L}^2(B_R)},
\end{equation*}
we have, for every $m\in\{1,\ldots,N\}$,
\begin{align} 
\label{e.diff.linearizedsystem.wm}
\lefteqn{
\left\| \nabla w_m - \nabla\tilde{w}_m \right\|_{\underline{L}^2(B_{R/2})} 
}
\\ & \qquad   \notag
\leq 
C \left( \frac 1R\left\|  u- \tilde{u} \right\|_{\underline{L}^2(B_R)} \right) 
\sum_{i=1}^{m-1} \mathsf{h}_i^{\frac mi}
+
C\sum_{i=1}^{m}
\frac1R \left\| w_i - \tilde{w}_i \right\|_{\underline{L}^{2}(B_{R})}
\sum_{j=1}^{m-i} 
\mathsf{h}_j^{\frac{m-i}{j}}.
\end{align}
and
\begin{align} 
\label{e.diff.linearizedsystem.Fm}
\lefteqn{
\left\| \mathbf{F}_m\left(\nabla u, \nabla w_{1} ,\ldots,\nabla w_{m-1},\cdot\right) - \mathbf{F}_m\left(\nabla \tilde{u}, \nabla \tilde{w}_{1} ,\ldots,\nabla \tilde{w}_{m-1},\cdot\right) \right\|_{\underline{L}^2(B_{R/2})} 
}
\\ & \qquad   \notag
\leq 
C \left( \frac 1R\left\|  u- \tilde{u} \right\|_{\underline{L}^2(B_R)} \right) 
\sum_{i=1}^{m-1} \mathsf{h}_i^{\frac mi}
+
C\sum_{i=1}^{m-1}
\frac1R \left\| w_i - \tilde{w}_i \right\|_{\underline{L}^{2}(B_{R})}
\sum_{j=1}^{m-i} 
\mathsf{h}_j^{\frac{m-i}{j}}.
\end{align}
\end{lemma}
\begin{proof}
Let us define $\a(x):= D^2_pL(\nabla u,x)$ and $\mathbf{f}_m(x):= \mathbf{F}_m(\nabla u,\nabla w_1,\ldots,\nabla w_{m-1},x)$ and analogously define $\tilde{\a}$ and~$\tilde{\mathbf{f}}_m$. We assume that $R \geq 2^{m+2}\X$, where $\X$ is as in Theorem~\ref{t.regularity.linerrors} for $n$, valid by the assumption of~\eqref{e.assumption.section4}.

\smallskip

The estimate~\eqref{e.diff.utildeu} is just the estimate for $\xi_0$ in Theorem~\ref{t.regularity.linerrors}. It also implies
\begin{equation} 
\label{e.estaatilde}
\left\| \a - \tilde{\a} \right\|_{\underline{L}^q(B_{R/2})} \leq 
\frac CR\left\|  u- \tilde{u} \right\|_{\underline{L}^2(B_R)}.
\end{equation}
By~\eqref{e.Fmbasic2}, we have that 
\begin{align*} 
&
\left| \f_m - \tilde{\f}_m  \right|
\\ & \quad\notag
\leq 
C \left| \nabla u - \nabla \tilde{u} \right| 
\sum_{i=1}^{m-1}
\left( \left| \nabla w_i \right| \vee\left| \nabla \tilde{w}_i \right| \right)^{\frac mi}
+
C \sum_{i=1}^{m-1} 
\left| \nabla w_i - \nabla \tilde{w}_i \right| 
\sum_{j=1}^{m-i} 
\left( \left| \nabla w_j \right| \vee\left| \nabla \tilde{w}_j \right| \right)^{\frac{m-i}{j}}.
\end{align*}
Using H\"older's inequality and applying Theorem~\ref{t.regularity.linerrors} for~$n$, we obtain, for any $p\in [2,\infty)$ and $\delta>0$, 
\begin{align}
\label{w.fmdfifs} 
\lefteqn{
\left\| \f_m - \tilde{\f_m} \right\|_{L^{p}(B_{2^{-m}R})}
} \quad & 
\\ &\notag
\leq 
C \left( \frac 1R\left\|  u- \tilde{u} \right\|_{\underline{L}^2(B_R)} \right) 
\sum_{i=1}^{m-1} \mathsf{h}_i^{\frac mi}
+ 
C\sum_{i=1}^{m-1}
\left\| \nabla w_i - \nabla \tilde{w}_i \right\|_{\underline{L}^{p+\delta}(B_{2^{-m}R})}
\sum_{j=1}^{m-i} 
\mathsf{h}_j^{\frac{m-i}{j}}. 
\end{align}
We observe that $\zeta:= w_m-\tilde{w}_m$ satisfies the equation
\begin{equation} 
\label{e.differe.wm}
-\nabla \cdot \a\nabla \zeta = 
\nabla \cdot \left( \f_m - \tilde{\f}_m \right) + \nabla \cdot \left( \left( \a-\tilde{\a} \right) \nabla \tilde{w}_m\right) \quad \mbox{in} \ B_R. 
\end{equation}
By Meyer's estimate, if $\delta >0$ is small enough, then 
\begin{align*} \label{}
\left\| \nabla w_m - \nabla \tilde{w}_m \right\|_{\underline{L}^{2+\delta}(B_{2^{-m-1}R})}
&
\leq 
C\left\| \f_m - \tilde{\f_m} \right\|_{L^{2+\delta}(B_{2^{-m}R})}
+
\frac CR \left\|   w_m -   \tilde{w}_m \right\|_{\underline{L}^{2}(B_{R})}
\\ & \qquad
+ \left\| \a-\tilde{\a} \right\|_{\underline{L}^5(B_{R/2})} 
\left\| \nabla \tilde{w}_m \right\|_{\underline{L}^5(B_{R/2})}.
\end{align*}
Combining these and using~\eqref{e.estaatilde} and the validity of Theorem~\ref{t.regularity.linerrors}, we get 
\begin{align*} \label{}
\left\| \nabla w_m - \nabla \tilde{w}_m \right\|_{\underline{L}^{2+\delta}(B_{2^{-m-1}R})}
&
\leq 
\frac CR \left\|   w_m -   \tilde{w}_m \right\|_{\underline{L}^{2}(B_{R})}
+
C \left( \frac 1R\left\|  u- \tilde{u} \right\|_{\underline{L}^2(B_R)} \right) 
\sum_{i=1}^{m-1} \mathsf{h}_i^{\frac mi}
\\ 
&\qquad 
+ 
C\sum_{i=1}^{m-1}
\left\| \nabla w_i - \nabla \tilde{w}_i \right\|_{\underline{L}^{2+2\delta}(B_{2^{-m}R})}
\sum_{j=1}^{m-i} 
\mathsf{h}_j^{\frac{m-i}{j}}.
\end{align*}
Taking $\delta_0$ sufficiently small and putting $\delta_m:=2^{-m}\delta_0$, we get by induction (using Young's inequality and rearranging several sums) that, for every $m\in\{1,\ldots,N\}$, 
\begin{align*} \label{}
& \left\| \nabla w_m - \nabla \tilde{w}_m \right\|_{\underline{L}^{2+\delta_m}(B_{2^{-m-1}R})}
\\ & \qquad
\leq 
C \left( \frac 1R\left\|  u- \tilde{u} \right\|_{\underline{L}^2(B_R)} \right) 
\sum_{i=1}^{m-1} \mathsf{h}_i^{\frac mi}
+
C\sum_{i=1}^{m}
\frac1R \left\| w_i - \tilde{w}_i \right\|_{\underline{L}^{2}(B_{R})}
\sum_{j=1}^{m-i} 
\mathsf{h}_j^{\frac{m-i}{j}}. 
\end{align*}
Combining this with~\eqref{w.fmdfifs}, we get 
\begin{align*} \label{}
& \left\| \f_m - \tilde{\f_m} \right\|_{L^{2}(B_{2^{-m-1}R})}
\\ & \qquad
\leq 
C \left( \frac 1R\left\|  u- \tilde{u} \right\|_{\underline{L}^2(B_R)} \right) 
\sum_{i=1}^{m-1} \mathsf{h}_i^{\frac mi}
+
C\sum_{i=1}^{m-1}
\frac1R \left\| w_i - \tilde{w}_i \right\|_{\underline{L}^{2}(B_{R})}
\sum_{j=1}^{m-i} 
\mathsf{h}_j^{\frac{m-i}{j}}. 
\end{align*}
These imply~\eqref{e.diff.linearizedsystem.wm} and~\eqref{e.diff.linearizedsystem.Fm} after a covering argument. 
\end{proof}


\section{Quantitative homogenization of the linearized equations}
\label{s.homogenization}

In this section, we suppose that~$n\in \{0,\ldots,\mathsf{N}-1 \}$ is such that   
\begin{equation}
\label{e.assumption.section4}
\left\{
\begin{aligned}
& \  \mbox{Theorem~\ref{t.regularity.Lbar} is valid with $n+1$ in place of $\mathsf{N}$,} \\
& \  \mbox{Theorems~\ref{t.linearizehigher} and~\ref{t.regularity.linerrors} are valid for $n$.}
\end{aligned}\right.
\end{equation}
The goal is to prove that Theorem~\ref{t.linearizehigher} is also valid for~ $n+1$ in place of~$n$. That is, we need to homogenize the $(n+2)$th linearized equation. 

\smallskip

In order to prove homogenization for the $(n+2)$th linearized equation, we follow the procedure used in~\cite{AFK} for homogenizing the first linearized equation. We first show, using the induction hypothesis~\eqref{e.assumption.section4}, that the coefficients $D^2_pL\left(\nabla u^\ep,\tfrac x\ep \right)$ and $\mathbf{F}_{n+1}\left( \tfrac x\ep, \nabla u^\ep, \nabla w_1^\ep, \ldots,\nabla w_{n}^\ep \right)$ can be approximated by random fields which are \emph{local} (they satisfy a finite range of dependence condition) and \emph{locally stationary} (they are stationary up to dependence on a slowly varying macroscopic variable). This then allows us to apply known quantitative homogenization results for linear elliptic equations which can be found for instance in~\cite{AKMbook}. 

\smallskip

\subsection{Stationary, finite range coefficient fields}

\label{ss.stat}
We proceed in a similar fashion as in~\cite[Section 3.4]{AFK} by introducing approximating equations which are stationary and localized.

We fix an integer $k\in\N$ which represents the scale on which we localize. 
Let~$v_{p,z}^{(k)}$ denote the solution of the Dirichlet problem
\begin{equation*}
\left\{
\begin{aligned}
& -\nabla \cdot \left( D_{p}L(\nabla v_{p,z}^{(k)}, x)\right)=0 & \mbox{in} & \ z+\cu_{k+1}, \\
& v_{p,z}^{(k)} = \ell_p & \mbox{on} & \ \partial (z+\cu_{k+1}),
\end{aligned}
\right.
\end{equation*}
where~$\ell_p$ is the affine function $\ell_p(x):=p\cdot x$. 
We then define, for each~$z\in 3^k\Zd$, a coefficient field $\tilde\a_{p,z}^{(k)}$ in $z+\cu_{k+1}$ by
\begin{equation*}
\tilde\a_{p,z}^{(k)}(x) := D_{p}^{2}L(\nabla v_{p,z}^{(k)}(x),x), \quad 
x\in z+\cu_{k+1}
\end{equation*}
and then recursively define, for each $\Theta = (p, h_{1}, \dots, h_{n+1})\in\left(\mathbb{R}^{d}\right)^{n+2}$, $m\in \{1,\ldots,n+1\}$ and $z\in3^k\Zd$, the functions $w_{m,\Theta,z}^{(k)} \in H^1(z+(1+2^{-m})\cu_{k})$ to be the solutions of the sequence Dirichlet problems
\begin{equation*}
\left\{
\begin{aligned}
& -\nabla\cdot\left( \tilde\a_{p,z}^{(k)} \nabla w_{m,\Theta,z}^{(k)} \right)= \nabla \cdot \tilde{\mathbf{F}}_{m,\Theta,z}^{(k)}  & \mbox{in} & \ z+(1+2^{-m})\cu_k, \\
& w_{m,\Theta,z}^{(k)} =\ell_{h_m} & \mbox{on} & \ \partial (z+(1+2^{-m})\cu_k),
\end{aligned}
\right.
\end{equation*}
where $\tilde{\mathbf{F}}_{m,\Theta,z}^{(k)} \in L^2(z+(1+2^{-(m-1)})\cu_k)$ is defined for $m\in \{1,\ldots,n+2\}$ by
\begin{equation}
\tilde{\mathbf{F}}_{m,\Theta,z}^{(k)}(x)
 := 
\mathbf{F}_{m}(\nabla v^{(k)}_{p,z}(x, z+\cu_{k+1}, p), \nabla w^{(k)}_{1,\Theta,z}(x), \ldots, \nabla w^{(k)}_{m-1,\Theta,z}(x),x).
\end{equation}
Finally, we create $3^k\Zd$--stationary fields by gluing the above functions together: for each $x\in\Rd$, we define
\begin{equation*}
\left\{
\begin{aligned}
& 
v^{(k)}_p(x):= v^{(k)}_{p,z}(x) 
\\ & 
w^{(k)}_{m,\Theta} (x),
:=
w^{(k)}_{m,\Theta,z}(x) 
\\ &
\a^{(k)}_p(x):= \tilde\a_{p,z}^{(k)}(x),
\\ & 
\mathbf{F}^{(k)}_{m,\Theta}(x)
:=
\tilde{\mathbf{F}}^{(k)}_{m,\Theta,z}(x),
\end{aligned}
\right.
\qquad \mbox{$z\in 3^k\Zd$ is such that $x\in z+\cu_k$.}
\end{equation*}
Notice that $v^{(k)}_p$ and $w^{(k)}_{m,\Theta}$ might not be $H^1$ functions globally, but we can nevertheless definer their gradients locally in $z+\cu_k$. The $\R^{d(n+3)}$--valued random field $\left( \nabla v^{(k)}_p, \nabla w^{(k)}_{1,\Theta},\ldots,\nabla w^{(k)}_{n+2,\Theta} \right)$ is  $3^{k}\mathbb{Z}^{d}$-stationary and has a range of dependence of at most $3^k\sqrt{15+d}$, by construction. The same is also true of the corresponding coefficient fields, since these are local functions of this random field:
\begin{equation}
\label{e.klocalize}
\left\{ 
\begin{aligned}
& \  \left(\a^{(k)}_p, \mathbf{F}^{(k)}_{2,\Theta},\ldots,\mathbf{F}^{(k)}_{n+2,\Theta} \right) 
\ \  \mbox{is $3^{k}\mathbb{Z}^{d}$-stationary and }
\\ &  \ \ \mbox{has a range of dependence of at most $3^k\sqrt{15+d}$}.
\end{aligned}
\right.
\end{equation}
In the next subsection, we will apply some known quantitative homogenization estimates to the linear equation
\begin{equation}
\label{e.statpde}
-\nabla\cdot\left( \a^{(k)}_p \nabla v \right) = \nabla \cdot \mathbf{F}^{(k)}_{n+2,\Theta}. 
\end{equation}
In order to anchor these estimates, we require some deterministic bounds on the vector field~$\mathbf{F}^{(k)}_{n+2,\Theta}$ and thus on the gradients of the functions $w^{(k)}_{m,\Theta}$ defined above. These bounds are exploding as a large power of~$3^k$, but we will eventually choose $3^k$ to be relatively small compared with the macroscopic scale (so these large powers of $3^k$ will be absorbed).

\begin{lemma}
\label{l.detbounds.wm}
Fix $\mathsf{M}\in [1,\infty)$. 
There exist exponents~$\beta(\data)>0$,~$Q(\data)<\infty$ and a constant~$C(\mathsf{M},\data)<\infty$ such that, for every~$k\in\N$,~$m\in\{ 1,\ldots,n+2\}$ and $\Theta=(p,h_1,\ldots,h_{n+1})\in \R^{d(n+2)}$ with $|p|\leq \mathsf{M}$, we have
\begin{equation}
\label{e.detbounds.Fm}
\left\| \mathbf{F}^{(k)}_{m,\Theta} \right\|_{C^{0,\beta}(\cu_k)} 
\leq 
C  3^{Qk} \sum_{j=1}^{m-1} \left| h_j \right|^{\frac mj}.
\end{equation}
\end{lemma}
\begin{proof}
By~\cite[Proposition A.3]{AFK}, there exist~$\beta(d,\Lambda)\in (0,1)$ and $C(\mathsf{M},\data)<\infty$ such that, for every $z\in 3^k\Zd$ and $p\in B_{\mathsf{M}}$, 
\begin{align*}
\left[ \nabla v_{p,z}^{(k)} \right]_{C^{0,\beta}(z+2\cu_{k})} 
\leq 
C + C \sup_{x\in z+2\cu_{k}} \left\| \nabla v_{p,z}^{(k)} \right\|_{L^2(B_1(x))} 
&
\leq 
C + C \left\| \nabla v_{p,z}^{(k)} \right\|_{L^2(z+\cu_{k+1})}
\\ & 
\leq 
C + C (1+|p| ) 3^{\frac{kd}2}.
\end{align*}
We deduce the existence of~$C(\mathsf{M},\data)<\infty$ such that
\begin{equation}
\label{e.detbounds.ak}
\left[ \tilde{\a}^{(k)}_{p,z}  \right]_{C^{0,\beta}(z+2\cu_{k})} \leq C  3^{\frac{kd}2}.
\end{equation}
We will argue by induction in~$m\in \{1,\ldots,n+2\}$ that there exist~$Q(\data)<\infty$ and $C(\mathsf{M},\data)<\infty$ such that 
\begin{equation}
\label{e.detbounds.wm.indy}
\left\| \nabla w_{m,\Theta,z}^{(k)} \right\|_{C^{0,\beta}(z+(1+\frac23 2^{-m})\cu_k)} 
\leq 
C  3^{Qk} \sum_{j=1}^{m} \left| h_j \right|^{\frac mj}.
\end{equation}
For $m=1$ the claim follows from Proposition~\ref{p.schauder} and~\eqref{e.detbounds.ak}. Suppose now that there exists $M\in \{2,\ldots,n+2\}$ such that the bound~\eqref{e.detbounds.wm.indy} holds for each $m \in \{1,\ldots,M-1\}$. Then we obtain that, for some $Q(M,\data)<\infty$,  
\begin{equation*}
\left[ \tilde{\mathbf{F}}^{(k)}_{M,\Theta,z} \right]_{C^{0,\beta}(z+(1+2^{-M})\cu_k)} 
\leq 
C  3^{Qk} \sum_{j=1}^{M} \left| h_j \right|^{\frac Mj}
\end{equation*}
By the Caccioppoli inequality, we get 
\begin{equation*}
\left\| \nabla w_{M,\Theta,z}^{(k)} \right\|_{\underline{L}^2(z+(1+\frac34 \cdot 2^{-M})\cu_k)}
\leq 
C  3^{Qk} \sum_{j=1}^{M} \left| h_j \right|^{\frac Mj}.
\end{equation*}
In view of~\eqref{e.detbounds.ak} and the previous two displays,
another application of Proposition~\ref{p.schauder} yields, after enlarging~$Q$ and~$C$, the bound~\eqref{e.detbounds.wm.indy} for~$m=M$. This completes the induction argument and the proof of~\eqref{e.detbounds.wm.indy}. The bound~\eqref{e.detbounds.Fm} immediately follows. 
\end{proof}

By the assumed validity of Theorem~\ref{t.regularity.linerrors} for $n$, we also have that, for each $\mathsf{M},q \in[2,\infty)$ there exist $\delta(q,\data)>0$ and $C(\mathsf{M},q,\data)<\infty$ and a random variable~$\X = \O_\delta(C)$ such that, for every $k\in \N$ with $3^k\geq \X$ and every $m\in\{1,\ldots,n+1\}$ and $\Theta = (p,h_1,\ldots,h_{n+1})\in\R^{d(n+2)}$ with $|p| \leq \mathsf{M}$, we have
\begin{equation} 
\label{e.wmTheta0bounds.X}
\left\| \nabla w^{(k)}_{m,\Theta,0}
\right\|_{\underline{L}^q((1+2^{-m-1})\cu_{k})}
\leq 
C \sum_{j=1}^{m-1} \left| h_j \right|^{\frac mj}
\end{equation}
and hence, for such~$k$ and $\Theta$ and $m\in\{1,\ldots,n+2\}$, 
\begin{equation} 
\label{e.FmTheta0bounds.X}
\left\| \tilde{\mathbf{F}}_{m,\Theta,0}^{(k)}
\right\|_{\underline{L}^q((1+2^{-m-1})\cu_{k})}
\leq 
C \sum_{j=1}^{m-1} \left| h_j \right|^{\frac mj}.
\end{equation}
Observe that~\eqref{e.detbounds.Fm} and~\eqref{e.FmTheta0bounds.X} together imply that, for $\delta$ and $C$ as above and every $m\in\{ 1,\ldots,n+2\}$,
\begin{equation} 
\label{e.FmTheta0bounds.X2}
\left\| \tilde{\mathbf{F}}_{m,\Theta,0}^{(k)}
\right\|_{\underline{L}^q( \cu_{k})}
\leq 
\O_\delta\left( C \sum_{j=1}^{m-1} \left| h_j \right|^{\frac mj} \right).
\end{equation}

We next study the continuity of $\a_p^{(k)}$ and~$\mathbf{F}^{(k)}_{m,\Theta}$ in the parameter~$\Theta$. 

\begin{lemma}[{Continuity of $\a_p^{(k)}$ and~$\mathbf{F}^{(k)}_{m,\Theta}$ in $\Theta$}]
\label{l.contdep.coeffs}
Fix $q\in [2,\infty)$ and $\mathsf{M}\in [1,\infty)$. 
There exist constants~$\delta(q,\data)>0$ and~$Q(q,\data),C(q,\mathsf{M},\data)<\infty$ and a random variable $\X = \O_\delta(C)$ such that, for every~$k\in\N$ with $3^k\geq \X$, $\Theta=(p,h_1,\ldots,h_{n+1})\in \R^{d(n+2)}$ and $\Theta'=(p',h_{1}',\ldots,h_{n+1}')\in \R^{d(n+2)}$
with $|p|,|p'|\leq \mathsf{M}$, 
\begin{equation}
\label{e.contdep.ap}
\left\| \a_p^{(k)} -  \a_{p'}^{(k)} \right\|_{\underline{L}^q(\cu_k)} 
\leq 
C  \left|p-p'\right|
\end{equation}
and, for every~$m\in\{ 1,\ldots,n+2\}$, 
\begin{align} 
\label{e.contdep.Fm}
\lefteqn{
\left\| \mathbf{F}^{(k)}_{m,\Theta} - \mathbf{F}^{(k)}_{m,\Theta'}  \right\|_{\underline{L}^2(\cu_k)} 
} \qquad & 
\\ &  \notag
\leq 
C |p-p'| \sum_{i=1}^{m-1} \left( \left| h_i \right| \vee \left| h_i' \right| \right)^{\frac mi}
+C \sum_{i=1}^{m-1} |h_i-h_i'| \sum_{j=1}^{m-i}
\left( \left| h_j \right| \vee \left| h_j' \right| \right)^{\frac {m-i}j}.
\end{align}
\end{lemma}
\begin{proof}
We take $\X$ to be larger than the random variables in the statements of  Lemma~\ref{l.diff.linearizedsystem} and Theorems~\ref{t.linearizehigher} and~\ref{t.regularity.linerrors} for~$n$. The bound~\eqref{e.contdep.ap} is then an immediate consequence of~\eqref{e.diff.utildeu} and the obvious fact that 
\begin{equation*} \label{}
\left\| \nabla v^{(k)}_{p,0} - \nabla v^{(k)}_{p',0} \right\|_{\underline{L}^2(z+\cu_{k+1})}
\leq C |p-p'|. 
\end{equation*}
We then use the equation for the difference~$w^{(k)}_{m,\Theta} -  w^{(k)}_{m,\Theta'}$ (see~\eqref{e.differe.wm}) and then apply the result of Lemma~\ref{l.diff.linearizedsystem} to obtain, for every $m\in\{1,\ldots,n+1\}$,
\begin{align*} \label{}
\lefteqn{
\left\| \nabla w^{(k)}_{m,\Theta,0} -\nabla  w^{(k)}_{m,\Theta',0}  \right\|_{\underline{L}^2((1+2^{-m})\cu_k)} 
} \quad & 
\\ &
\leq 
C \left| h_m - h_m' \right| 
+
C
\left\| \tilde{\mathbf{F}}^{(k)}_{m,\Theta,0} - \tilde{\mathbf{F}}^{(k)}_{m,\Theta',0}  \right\|_{\underline{L}^2((1+2^{-m})\cu_k)} 
+
C|p-p'| \sum_{i=1}^m \left| h_i \right|^{\frac mi} 
\\ & 
\leq 
C \left| h_m - h_m' \right| 
+
C |p-p'|
\sum_{i=1}^{m-1} \left( \left| h_i \right| \vee \left| h_i' \right| \right)^{\frac mi}
\\ & \quad
+
C\sum_{i=1}^{m-1}
\frac1R \left\|  w^{(k)}_{i,\Theta,0} -  w^{(k)}_{i,\Theta',0} \right\|_{\underline{L}^{2}((1+2^{-i})\cu_k)}
\sum_{j=1}^{m-i} 
\left( \left| h_i \right| \vee \left| h_i' \right| \right)^{\frac{m-i}{j}}.
\end{align*}
By induction we now obtain, for every $m\in\{1,\ldots,n+1\}$, 
\begin{align*} \label{}
\lefteqn{
\left\| \nabla w^{(k)}_{m,\Theta,0} -\nabla  w^{(k)}_{m,\Theta',0}  \right\|_{\underline{L}^2((1+2^{-m})\cu_k)} 
} \quad & 
\\ & 
\leq
C |p-p'| \sum_{i=1}^{m-1} \left( \left| h_i \right| \vee \left| h_i' \right| \right)^{\frac mi}
+C \sum_{i=1}^{m-1} |h_i-h_i'| \sum_{j=1}^{m-i}
\left( \left| h_j \right| \vee \left| h_j' \right| \right)^{\frac {m-i}j}.
\end{align*}
This implies~\eqref{e.contdep.Fm}. 
\end{proof}

By combining~\eqref{e.detbounds.Fm} and~\eqref{e.contdep.Fm} and using interpolation, we obtain the existence of~$\alpha(\data)>0$,~$Q(\data)<\infty$ and $C(\data)>0$ such that, with~$\X$ as in the statement of Lemma~\ref{l.contdep.coeffs}, then for every~$k\in\N$ with $3^k\geq \X$, $m\in\{1,\ldots,n+2\}$, and $\Theta=(p,h_1,\ldots,h_{n+1})\in \R^{d(n+2)}$ and $\Theta'=(p',h_{1}',\ldots,h_{n+1}')\in \R^{d(n+2)}$
with $|p|,|p'|\leq \mathsf{M}$, we have 
\begin{align} 
\label{e.contdep.Fm.Linfty}
\lefteqn{
\left\| \mathbf{F}^{(k)}_{m,\Theta} - \mathbf{F}^{(k)}_{m,\Theta'}  \right\|_{L^\infty(\cu_k)} 
} \quad & 
\\ &  \notag
\leq 
C3^{Qk} |p-p'|^\alpha \sum_{i=1}^{m-1} \left( \left| h_i \right| \vee \left| h_i' \right| \right)^{\frac mi}
+C3^{Qk} \sum_{i=1}^{m-1} |h_i-h_i'|^\alpha \sum_{j=1}^{m-i}
\left( \left| h_j \right| \vee \left| h_j' \right| \right)^{\frac {m-i}j}.
\end{align}
Likewise, we can use~\eqref{e.detbounds.ak} and~\eqref{e.contdep.ap} to obtain 
\begin{equation}
\label{e.contdep.ap.Linfty}
\left\| \a_p^{(k)} -  \a_{p'}^{(k)} \right\|_{\underline{L}^q(\cu_k)} 
\leq 
C 3^{Qk} \left|p-p'\right|^\alpha.
\end{equation}
This variation of Lemma~\ref{l.contdep.coeffs} will be needed below.

\subsection{Setup of the proof of Theorem~\ref{t.linearizehigher}}

We are now ready to begin the proof of the implication
\begin{equation*} \label{}
\mbox{
~\eqref{e.assumption.section4} 
\, $\implies$ \, 
the statement of Theorem~\ref{t.linearizehigher}\,  with~$n+1$ in place of~$n$. 
}
\end{equation*}
We fix parameters $\mathsf{M} \in [1,\infty)$, $\delta>0$ and $\ep\in (0,1)$, 
a sequence of Lipschitz domains $U_1,U_2,\ldots,U_{n+2} \subseteq \cu_0$ satisfying
\begin{equation}
\label{e.domainsdescending}
\overline{U}_{m+1} \subseteq U_m, \quad \forall m\in\{1,\ldots,n+1\}.
\end{equation}
a function $f\in W^{1,2+\delta}(U_1)$ satisfying
\begin{equation}
\label{e.fboundedass}
\left\| \nabla f \right\|_{L^{2+\delta}(U_1)} \leq \mathsf{M},
\end{equation}
and a sequence of boundary conditions 
$g_1\in W^{1,2+\delta}(U_1), \ldots, g_{n+2} \in W^{1,2+\delta}(U_{n+2})$. We let $u^\ep \in f+H^1_0(U_1)$ and $w^\ep_1\in g_1+H^1_0(U_1),\ldots,w^\ep_{n+2} \in g_{n+2}+H^1_0(U_{n+2})$ as well as $\overline{u} \in f+H^1_0(U)$ and $\overline{w}_1\in g_1+H^1_0(U_1),\ldots,\overline{w}_{n+2} \in g_{n+2}+H^1_0(U_{n+2})$ be as in the statement of Theorem~\ref{t.linearizehigher} for~$n+1$ in place of~$n$. 

\smallskip

We denote
\begin{equation}
\label{e.aepbars}
\left\{ 
\begin{aligned}
& \a^\ep(x):= D^2_p L \left( \nabla u^\ep(x),\tfrac x\ep \right), \\
& \ahom(x):= D^2_p \overline{L} \left( \nabla \overline{u}(x) \right).
\end{aligned}
\right.
\end{equation}
and
\begin{equation}
\label{e.Fepbars}
\left\{ 
\begin{aligned}
& \mathbf{F}^\ep_m(x):= \mathbf{F}_{m}\left(\nabla u^\ep(x), \nabla w^\ep_{1}(x), \ldots, \nabla w^\ep_{m-1}(x),\tfrac x\ep \right), \\
& \overline{\mathbf{F}}_m(x):= \overline{\mathbf{F}}_{m}\left(\nabla \overline{u}(x), \nabla \overline{w}_{1}(x), \ldots, \nabla \overline{w}_{m-1}(x)\right).
\end{aligned}
\right.
\end{equation}
We also choose $K\in\N$ to be the unique positive integer satisfying $3^{-K-1}< \ep \leq 3^{-K}$. We write 
\begin{equation} 
\label{e.overlineTheta}
\overline\Theta(x):= \left( \nabla \overline{u}(x),\nabla \overline{w}_{1}(x), \ldots, \nabla \overline{w}_{n+1}(x) \right).
\end{equation}
By the assumption~\eqref{e.assumption.section4}, we only need to homogenize the linearized equation for $m=n+2$. As we have already proved a homogenization result in~\cite[Theorem 1.1]{AFK} for the linearized equation with zero right-hand side, it suffices to prove~\eqref{e.homogenization.estimates} for $m=n+2$ under the assumption that the boundary condition vanishes:
\begin{equation}
\label{e.gvanishass}
g_{n+2}= 0 \quad \mbox{in} \ U_{n+2}.
\end{equation}
We can write the equations for $w_{n+2}^\ep$ and $\overline{w}_{n+2}$ respectively as 
\begin{equation} 
\label{e.wn+2}
\left\{
\begin{aligned}
& -\nabla \cdot  \left( \a^\ep\nabla w_{n+2}^\ep \right) = \nabla \cdot \mathbf{F}^\ep_{n+2}  & \mbox{in} & \ U_{n+2}, \\
& w_{n+2}= 0,  & \mbox{on} & \ \partial U_{n+2},
\end{aligned} 
\right. 
\end{equation}
and
\begin{equation} 
\label{e.wn+2.bar}
\left\{
\begin{aligned}
& -\nabla \cdot  \left( \overline{\a}(x) \nabla \overline{w}_{n+2}^\ep \right) = \nabla \cdot \overline{\mathbf{F}}_{n+2}  & \mbox{in} & \ U_{n+2}, \\
& \overline{w}_{n+2}= 0,  & \mbox{on} & \ \partial U_{n+2}.
\end{aligned} 
\right. 
\end{equation}
Our goal is to prove the following estimate: there exists a constant $C(\mathsf{M},\data)<\infty$, exponents~$\sigma(\data)$ and $\alpha(\data)>0$ and random variable~$\X$ satisfying 
\begin{equation}
\label{e.Xint.wts}
\X = \O_\sigma(C),
\end{equation}
as in the statement of the theorem:
\begin{equation}
\label{e.homogenization.estimates.wts}
\left\| \nabla w_{n+2}^\ep - \nabla \bar{w}_{n+2} \right\|_{H^{-1}(U_m)} 
\leq
\X \ep^{\alpha}   
\sum_{j=1}^{n+1} \left\|  \nabla g_{j} \right\|_{{L}^{2+\delta}(U_{j})}^{\frac{n+2}{j}}.
\end{equation}
To prove~\eqref{e.homogenization.estimates.wts}, we first compare the solution~$w_{n+2}$ to the solution~$\tilde{w}_{n+2}^\ep$ of a second heterogeneous problem, namely
\begin{equation} 
\label{e.wn+2.tilde}
\left\{
\begin{aligned}
& -\nabla \cdot  \left( \tilde{\a}^\ep\nabla \tilde{w}_{n+2}^\ep \right) = \nabla \cdot \tilde{\mathbf{F}}^\ep_{n+2}  & \mbox{in} & \ U_{n+2}, \\
& \tilde{w}_{n+2}^\ep = 0,  & \mbox{on} & \ \partial U_{n+2},
\end{aligned} 
\right. 
\end{equation}
where the coefficient fields~$\tilde{\a}^\ep$ and~$\tilde{\mathbf{F}}^\ep_{n+2}$ are defined in terms of the localized, stationary approximating coefficients (introduced above in Subsection~\ref{ss.stat}) by 
\begin{equation}
\label{e.tildecoeffs}
\left\{ 
\begin{aligned}
& \tilde{\a}^\ep(x)
:= 
\a^{(k)}_{\nabla \overline{u}(x)}\left(\tfrac x\ep\right), \\
& \tilde{\mathbf{F}}^\ep_{n+2} 
:= 
\mathbf{F}^{(k)}_{n+2,\overline\Theta(x)}
\left(\tfrac x\ep\right).
\end{aligned}
\right.
\end{equation}
The parameter~$k\in\N$ will be chosen below in such a way that~$1\ll 3^{k} \ll \ep^{-1}$. We also need to declare a second mesoscopic scale by taking $l\in\N$ such that~$1\ll 3^k \ll 3^l \ll \ep^{-1}$ and, 
for every $m\in\{1,\ldots,n+1\}$,
\begin{equation} 
\label{e.l.mesofitting}
U_{m+1} + \ep\cu_l \subseteq U_m.
\end{equation}
Like~$k$, the parameter~$l$ will be declared later in this section. For convenience we will also take a slightly smaller domain $U_{n+3}$ than $U_{n+2}$, which also depends on~$\ep$ and~$l$ and is defined by 
\begin{equation}
\label{e.defUn+3}
U_{n+3}:= \left\{ x\in U_{n+2} \,:\, x + \ep\cu_l \subseteq U_{n+2}  \right\}.
\end{equation}
Thus we have~\eqref{e.l.mesofitting} for every $m\in\{1,\ldots,n+2\}$. See Subection~\ref{ss.meso} below for more on the choices of the parameters~$k$ and~$l$.

\smallskip

The estimate~\eqref{e.homogenization.estimates.wts} follows from the following two estimates, which are proved separately below (here~$\X$ denotes a random variable as in the statement of the theorem):
\begin{equation} \label{e.homogenization.estimates.wts.1}
\left\| \nabla w_{n+2}^\ep - \nabla \tilde{w}_{n+2}^\ep \right\|_{L^2(U_m)} 
\leq
\X \ep^{\alpha}   
\sum_{j=1}^{n+1} \left\|  \nabla g_{j} \right\|_{{L}^{2+\delta}(U_{j})}^{\frac{n+2}{j}},
\end{equation}
and
\begin{equation} \label{e.homogenization.estimates.wts.2}
\left\| \nabla \tilde{w}_{n+2}^\ep - \nabla \bar{w}_{n+2} \right\|_{H^{-1}(U_m)} 
\leq
\X \ep^{\alpha}   
\sum_{j=1}^{n+1} \left\|  \nabla g_{j} \right\|_{{L}^{2+\delta}(U_{j})}^{\frac{n+2}{j}}.
\end{equation}

\subsection{Estimates on size of the $w_m^\ep$ and $\overline{w}_m$}

To prepare for the proofs of~\eqref{e.homogenization.estimates.wts.1} and~\eqref{e.homogenization.estimates.wts.2}, we present some preliminary bounds on the size of the $w_m^\ep$'s. The first is a set of deterministic bounds representing the ``worst-case scenario'' in which nothing has homogenized up to the current scale. 

\begin{lemma}
\label{l.detbounds.het}
There exist exponents~$\beta(d,\Lambda)\in (0,1)$ and $Q(\data)<\infty$ and a constant~$C(\{ U_m \}, \mathsf{M},\mathsf{K}_0,\data)<\infty$ such that, for every $m\in \{1,\ldots,n+2\}$, 
\begin{equation}
\label{e.detbounds.het}
\left\| \nabla w^\ep_{m} \right\|_{C^{0,\beta}(U_{m+1})} 
\leq 
C \ep^{-Q} \sum_{j=1}^{m} \left\|  \nabla g_{j} \right\|_{{L}^{2+\delta}(U_{j})}^{\frac{m}{j}}.
\end{equation}
\end{lemma}
\begin{proof}
By~\cite[Proposition A.3]{AFK}, there exist~$\beta(d,\Lambda)\in (0,1)$, $Q(\data)<\infty$ and $C(\{ U_m\},\mathsf{M},\mathsf{K}_0,\data)<\infty$ such that 
\begin{align*}
\left[ \nabla u^\ep \right]_{C^{0,\beta}(U_1)} 
\leq 
C \ep^{-Q}
\end{align*}
and hence 
\begin{equation}
\label{e.detbounds.aep}
\left[ \a^\ep  \right]_{C^{0,\beta}(U_1)} \leq C  \ep^{-Q}.
\end{equation}
We will argue by induction in~$m\in \{1,\ldots,n+2\}$ that there exists $Q(m,\data)<\infty$ and $C(m,\{U_k\}, \mathsf{M},\mathsf{K}_0,\data)<\infty$ such that 
\begin{equation}
\label{e.detbounds.wm.het.indy}
\left\| \nabla w^\ep_{m} \right\|_{C^{0,\beta}(U_{m+1})} 
\leq 
C \ep^{-Q} \sum_{j=1}^{m} \left| h_j \right|^{\frac mj}.
\end{equation}
For $m=1$ the claim follows from Proposition~\ref{p.schauder} and~\eqref{e.detbounds.aep}. Suppose now that there exists $M\in \{2,\ldots,n+2\}$ such that the bound~\eqref{e.detbounds.wm.het.indy} holds for each $m \in \{1,\ldots,M-1\}$. Then we obtain that, for some $Q(\mathsf{M},\data)<\infty$,  
\begin{equation*}
\left\| \mathbf{F}^{\ep}_M \right\|_{L^2(U_M)} 
\leq 
C  \ep^{-Q} \sum_{j=1}^{M} \left| h_j \right|^{\frac Mj}
\end{equation*}
Then by the basic energy estimate, we get 
\begin{equation*}
\left\| \nabla w^\ep_{M} \right\|_{{L}^2(U_M)}
\leq 
C \ep^{-Q} \sum_{j=1}^{M} \left| h_j \right|^{\frac Mj}.
\end{equation*}
In view of~\eqref{e.detbounds.ak} and the previous two displays,
another application of Proposition~\ref{p.schauder} yields, after enlarging~$Q$ and~$C$, the bound~\eqref{e.detbounds.wm.het.indy} for~$m=M$. This completes the induction argument and the proof of the lemma. 
\end{proof}

The \emph{typical} size of $|\nabla w_m^\ep|$ is much better than~\eqref{e.detbounds.het} gives.

\begin{lemma}
\label{l.goodbounds.wm}
For each $q\in [2,\infty)$, there exist $\sigma(q,data)>0$, $C(q,\mathsf{M},\mathsf{K}_0,\data)<\infty$ and a random variable $\X$ satisfying 
\begin{equation*}
\X = \O_\sigma(C)
\end{equation*}
such that, for each $m\in \{1,\ldots,n+1\}$, 
\begin{equation}
\label{e.goodbounds.wm}
\left\| \nabla w^\ep_{m} \right\|_{L^q(U_{m+1})} 
\leq 
\X \sum_{j=1}^{m} \left\|  \nabla g_{j} \right\|_{{L}^{2+\delta}(U_{j})}^{\frac{m}{j}}.
\end{equation}
\end{lemma}
\begin{proof}
We argue by induction in~$m$. Assume that~\eqref{e.goodbounds.wm} holds for $m\in \{1,\ldots,M-1\}$ for some $M\leq n+1$. 
By~\eqref{e.assumption.section4}, in particular the assumed validity of Theorem~\ref{t.regularity.linerrors} for $m\leq n+1$, 
it suffices to show that 
\begin{equation}
\label{e.mishitswts}
\left\| \nabla  w^\ep_{M} \right\|_{{L}^2(U_M)}
\leq 
C\X \sum_{j=1}^{M} \left\|  \nabla g_{j} \right\|_{{L}^{2}(U_{j})}^{\frac{M}{j}}.
\end{equation}
The induction hypothesis yields that 
\begin{equation}
\left\| \mathbf{F}_{M}^\ep \right\|_{L^2(U_M)}
\leq 
C\X \sum_{j=1}^{M-1} \left\|  \nabla g_{j} \right\|_{{L}^{2}(U_{j})}^{\frac{M}{j}}
\end{equation}
and then the basic energy estimate yields~\eqref{e.mishitswts}.
\end{proof}

We also require bounds on the homogenized solutions~$\overline{u}$ and~$\overline{w}_1,\ldots,\overline{w}_{n+2}$. These are consequences of elliptic regularity estimates presented in Appendix~\ref{s.appendixconstant}, namely Lemma~\ref{l.appC.C1alphabarwn}, which is applicable here because of  the assumption~\eqref{e.assumption.section4} which ensures that~$\overline{L} \in C^{3+\mathsf{N},\beta}$  for every $\beta\in (0,1)$. We obtain the existence of $C(\mathsf{M},\data)<\infty$ such that, for every~$m\in\{1,\ldots,n+1\}$,
\begin{equation} 
\label{e.detanchors}
\left\{
\begin{aligned}
& \left\| \nabla \overline{u} \right\|_{C^{0,1}(\overline{U}_2)}
\leq C, 
\\ & 
\left\| \nabla \overline{w}_m \right\|_{C^{0,1}(\overline{U}_{n+2})}
\leq 
C \sum_{i=1}^{m} \left\| g_i \right\|_{{L}^2(U_i)}^{\frac mi}.
\end{aligned}
\right.
\end{equation}
In particular, the function~$\overline\Theta$ defined in~\eqref{e.overlineTheta} is Lipschitz continuous. By the global Meyers estimate, we also have, for some~$\delta(d,\Lambda)>0$ and~$C(\mathsf{M},\data)<\infty$, the bound
\begin{equation}
\label{e.wn2meyers}
\left\| \nabla \overline{w}_{n+2} \right\|_{L^{2+\delta}(\overline{U}_{n+2})}
\leq 
C \left\| \overline{\mathbf{F}}_{n+2} \right\|_{L^{2+\delta}(U_{n+2})} 
\leq 
C \sum_{i=1}^{n+1} \left\| g_i \right\|_{{L}^2(U_i)}^{\frac{n+2}i}.
\end{equation}
We may also apply Lemma~\ref{l.appC.C1alphabarwn} to get a bound on $w_{n+2}$. In view of the merely mesoscopic gap between $\partial U_{n+2}$ and $U_{n+3}$, which is much smaller than the macroscopic gaps between $\partial U_{m}$ and $U_{m+1}$ for $m\in\{1,\ldots,n+1\}$, cf.~\eqref{e.domainsdescending} and~\eqref{e.defUn+3}, the estimate we obtain is, for every $\beta \in (0,1)$,  
\begin{equation} 
\label{e.detanchors2}
\left\| \nabla \overline{w}_{n+2} \right\|_{C^{0,\beta}(\overline{U}_{n+3})}
\leq 
C \left( 3^l\ep \right)^{-Q} \sum_{i=1}^{n+1} \left\| g_i \right\|_{{L}^2(U_i)}^{\frac {n+2}i},
\end{equation}
for an exponent~$Q(\beta,\data)<\infty$ which can be explicitly computed rather easily (but for our purposes is not  worth the bother) and a constant~$C(\beta,\mathsf{M},\data)<\infty$.

\subsection{The mesoscopic parameters~$k$ and~$l$}
\label{ss.meso}

Here we enter into a discussion regarding the choice of the parameters~$k$ and~$l$ which specify the mesoscopic scales. Recall that $1 \ll 3^k \ll 3^l \ll \ep^{-1}$. As we will eventually see later in this section, we will estimate the left sides of~\eqref{e.homogenization.estimates.wts.1} and~\eqref{e.homogenization.estimates.wts.2} by expressions of the form
\begin{equation}
\label{e.expression}
C\left( \ep^\alpha \left( \ep 3^l \right)^{-Q} + 3^{-l\alpha} \left( \ep 3^l \right)^{-1} + 3^{-k\alpha} + \left( \ep 3^l \right)^\alpha 3^{kQ} \right)
\sum_{j=1}^{n+1} \left\|  \nabla g_{j} \right\|_{{L}^{2+\delta}(U_{j})}^{\frac{n+2}{j}},
\end{equation}
where~$Q(\data)<\infty$ is a large exponent and~$\alpha(\data)>0$ is a small exponent. We then need to choose $k$ and $l$ so that the expression in parentheses is a positive power of~$\ep$. We may do this as follows. First, we may assume throughout for convenience that $Q\geq 1 \geq 4\alpha$. Then we may take care of the first two terms by choosing $l$ so that $\ep3^l$ is a ``very large'' mesoscale: we let $l$ be defined so that 
\begin{equation*}
\ep 3^l \leq \ep^{\frac\alpha{4Q}} < \ep 3^{l+1} .
\end{equation*}
Then we see that, for some $\beta>0$,  
\begin{equation*}
\ep^\alpha \left( \ep 3^l \right)^{-Q} + 3^{-l\alpha} \left( \ep 3^l \right)^{-1}
\leq C \ep^{\beta}.
\end{equation*}
Next, we take care of the last term: since, for some $\beta>0$, 
\begin{equation*}
\left( \ep 3^l \right)^\alpha 3^{kQ}
\leq 
C\ep^{\beta} 3^{kQ}, 
\end{equation*}
we can make this smaller than $\ep^{\frac\beta2}$ by taking $\ep 3^k$ to be a ``very small'' mesoscale. We take $k$ so that, for~$\beta$ and~$Q$ as in the previous display, 
\begin{equation*}
3^k \leq \ep^{-\frac{\beta}{2Q}} < 3^{k+1}. 
\end{equation*}
From this we deduce that $\left( \ep 3^l \right)^\alpha 3^{kQ}
\leq 
C\ep^{\frac\beta2}$. We see that this choice of $k$ also makes the third term inside the parentheses on the right side of~\eqref{e.expression} smaller than a positive power of~$\ep$. With these choices, we obtain that, for some $\beta>0$, 
\begin{equation}
\label{e.goodchoices}
C\left( \ep^\alpha \left( \ep 3^l \right)^{-Q} + 3^{-l\alpha} \left( \ep 3^l \right)^{-1} + 3^{-k\alpha} + \left( \ep 3^l \right)^\alpha 3^{kQ} \right)
\sum_{j=1}^{n+1} \left\|  \nabla g_{j} \right\|_{{L}^{2+\delta}(U_{j})}^{\frac{n+2}{j}}
\leq C\ep^{\beta}.
\end{equation}
Throughout the rest of this section, we will allow the exponents~$\alpha \in \left(0,\tfrac14\right]$ and~$Q\in [1,\infty)$ to vary in each occurrence, but will always depend only on~$\data$. Similarly, we let~$c$ and~$C$ denote positive constants which may vary in each occurrence and whose dependence will be clear from the context. 

\subsection{The minimal scales}
\label{ss.mscales}

Many of the estimates we will use in the proof of Theorem~\ref{t.linearizehigher} are deterministic estimates (i.e., the constants~$C$ in the estimate are not random) but which are valid only above a minimal scale~$\X$ which satisfies~$\X = \O_\delta(C)$ for some $\delta(\data)>0$ and $C(\mathsf{M},\data)<\infty$. This includes, for instance, the estimates of Theorems~\ref{t.regularity.Lbar},~\ref{t.linearizehigher} and~\ref{t.regularity.linerrors} assumed to be valid by our assumption~\eqref{e.assumption.section4}, as well as Lemmas~\ref{l.diff.linearizedsystem},~\ref{l.detbounds.wm},~\ref{l.contdep.coeffs} and~\ref{l.goodbounds.wm}, some estimates like~\eqref{e.contdep.Fm.Linfty} which appear in the text, and future estimates such as those of Lemmas~\ref{l.coeffclose},~\ref{l.fmsclose} and~\ref{l.barsclose}. We stress that this list is not exhaustive. 

\smallskip

In our proofs of~\eqref{e.homogenization.estimates.wts.1} and~\eqref{e.homogenization.estimates.wts.2}, we may always suppose that $3^k \geq \X$, where $\X$ is the maximum of all of these minimal scales. In fact, we may suppose that $3^k$ is larger than the stationary translation~$T_z\X$ of $\X$ by any element of $z\in  \Zd \cap \ep^{-1} U_1$. To see why, first we remark that if $\X$ is any random variable satisfying $\X= \O_\delta(C)$, then a union bounds gives that, for any $Q<\infty$ the random variable 
\begin{equation}
\tilde{\X} := \sup\left\{ 3^{j+1} \,:\, \sup_{z\in \Zd \cap 3^{Qj} U_1} T_z\X \geq 3^j \right\}
\end{equation}
satisfies, for a possibly smaller $\delta$ and larger $C<\infty$, depending also on $Q$, the estimate $\tilde{\X} = \O_{\delta/2}(C)$. Since, as explained in the previous subsection,~$3^k$ is a small but positive power of $\ep^{-1}$, we see that by choosing $Q$ suitably we have that $3^k \geq \tilde{\X}$ implies the validity of all of our estimates. Finally, in the event that $3^k < \tilde{\X}$, we can use the deterministic bounds obtained in Lemma~\ref{l.detbounds.het} and~\eqref{e.wn2meyers} very crudely as follows:
\begin{align*}
\left\| 
\nabla {w}_{n+2}^\ep
-
\nabla \overline{w}_{n+2}
\right\|_{H^{-1}(U_{n+2})} 
& 
\leq 
C\left\| 
\nabla {w}_{n+2}^\ep
-
\nabla \overline{w}_{n+2}
\right\|_{L^2(U_{n+2})} 
\\ & 
\leq
C\left\| 
\nabla {w}_{n+2}^\ep\right\|_{L^2(U_{n+2})} 
+
C\left\| 
\nabla \overline{w}_{n+2}
\right\|_{L^2(U_{n+2})} 
\\ & 
\leq
C\left( 1+ \ep^{-Q} \right) \sum_{j=1}^{n+1} \left\|  \nabla g_{j} \right\|_{{L}^{2+\delta}(U_{j})}^{\frac{n+2}{j}}.
\end{align*}
Then we use that
\begin{equation}
(1+ \ep^{-Q})\indc_{\{3^k < \tilde{\X}\}}
\leq 
\ep {\tilde{\X}}^Q = \O_{\delta/Q}(C\ep),
\end{equation}
where in the previous display the exponent $Q$ is larger in the second instance than in the first. 
This yields~\eqref{e.homogenization.estimates.wts.1} and~\eqref{e.homogenization.estimates.wts.2} in the event $\{ 3^k > \tilde{\X} \}$, so we do not have to worry about this event.

\smallskip

Therefore, throughout the rest of this section, we let $\X$ denote a minimal scale satisfying $\X=\O_\delta(C)$ which, in addition to both~$\delta$ and $C$, may vary in each occurrence. 

\subsection{The proof of the local stationary approximation}

We now turn to the proof of~\eqref{e.homogenization.estimates.wts.1}, which amounts to showing that the difference between the coefficient fields $(\a^\ep,\mathbf{F}^\ep_{n+2})$ and $(\tilde{\a}^\ep,\tilde{\mathbf{F}}^\ep_{n+2})$ is small. This is accomplished by an application of Lemma~\ref{l.diff.linearizedsystem}.

\begin{lemma}
\label{l.coeffclose}
There exist~$\alpha(\data)>0$, $Q(\data)<\infty$ and~$C(\mathsf{M},\data)<\infty$ and a minimal scale $\X=\O_\delta(C)$ such that, if $3^k\geq \X$, then 
\begin{align} \label{e.blargh.1}
&
\left\| \nabla w_{n+2}^\ep - \nabla \tilde{w}_{n+2}^\ep \right\|_{L^2(U_m)}
\\ & \quad \notag
\leq
C\left( \ep^\alpha \left( \ep 3^l \right)^{-\frac d2-1} + 3^{-l\alpha} \left( \ep 3^l \right)^{-1} + 3^{-k\alpha} + \left( \ep 3^l \right)^\alpha 3^{kQ} \right)
\sum_{j=1}^{n+1} \left\|  \nabla g_{j} \right\|_{{L}^{2+\delta}(U_{j})}^{\frac{n+2}{j}}.
\end{align}
\end{lemma}
\begin{proof}
Throughout $\X$ denotes a random scale which may change from line to line and satisfies $\X=\O_\delta(C)$. 

\smallskip

For each $z\in \ep3^l \Zd$ with $z+\ep\cu_{l+1}\subseteq U_{n+1}$, we compare $\nabla u$ and $\nabla w_m^\ep$ to the functions $\nabla v^{(l)}_{\nabla \overline{u}(z),\frac z\ep}\left( \tfrac \cdot \ep\right)$ and $\nabla w^{(l)}_{m,\overline{\Theta}(z),\frac z\ep}\left( \tfrac \cdot \ep\right)$, using Lemma~\ref{l.diff.linearizedsystem} and the assumed validity of Theorems~\ref{t.linearizehigher} and~\ref{t.regularity.linerrors} for~$n$. The latter yields that, if $\ep^{-1}\geq \X$, then for every such~$z$, and every~$m\in\{1,\ldots,n+1\}$
\begin{equation*}
\left\{ 
\begin{aligned}
& \left\| u^\ep - \overline{u} \right\|_{\underline{L}^2(z+\ep \cu_{l+1})} 
\leq C\ep^\alpha \left( \ep 3^l \right)^{-\frac d2},
\\ & 
\left\| w^\ep_m - \overline{w}_m \right\|_{\underline{L}^2(z+\ep\cu_{l+1})} 
C\ep^\alpha \left( \ep 3^l \right)^{-\frac d2} 
\sum_{j=1}^{m} \left\|  \nabla g_{j} \right\|_{{L}^{2+\delta}(U_{j})}^{\frac{m}{j}}.
\end{aligned} 
\right.
\end{equation*}
Likewise, if $3^l \geq \X$,  then for every~$z$ as above, and every~$m\in\{1,\ldots,n+1\}$,
\begin{equation*}
\left\{ 
\begin{aligned}
& \left\| v^{(l)}_{\nabla \overline{u}(z),\frac z\ep}\left( \tfrac \cdot \ep\right) - \ell_{\nabla \overline{u}(z)} \right\|_{\underline{L}^2(z+\ep \cu_{l+1})} 
\leq C3^{-l\alpha},
\\ & 
\left\| w^{(l)}_{m,\overline{\Theta}(z),\frac z\ep}\left( \tfrac \cdot \ep\right) - \ell_{\nabla \overline{w}_m(z)} \right\|_{\underline{L}^2(z+\ep \left(1+2^{-m} \right) \cu_l)} 
\leq 
C3^{-l\alpha} 
\sum_{j=1}^{m} \left\|  \nabla g_{j} \right\|_{{L}^{2+\delta}(U_{j})}^{\frac{m}{j}}.
\end{aligned} 
\right.
\end{equation*}
Here we used~\eqref{e.detanchors} in order that the~$C$ not depend implicitly on $\left| \nabla u(z) \right|$ and in order that the prefactor on the right side of the second line be is as it is, rather than~$\sum_{j=1}^m \left| \nabla \overline{w}_j(z) \right|^{\frac mj}$. Using~\eqref{e.detanchors}  again, we see that 
\begin{equation*}
\left\{ 
\begin{aligned}
& \left\| \overline{u} -\left( \overline{u}(z) + \ell_{\nabla \overline{u}(z)} \right) \right\|_{\underline{L}^2(z+\ep \cu_{l+1})} 
\leq C\ep^2 3^{2l},
\\ & 
\left\| \overline{w}_m - \left( \overline{w}_m(z) + \ell_{\nabla \overline{w}_m(z)} \right) \right\|_{\underline{L}^2(z+\ep \cu_{l+1})} 
\leq 
C\ep^2 3^{2l}
\sum_{j=1}^{m} \left\|  \nabla g_{j} \right\|_{{L}^{2+\delta}(U_{j})}^{\frac{m}{j}}.
\end{aligned} 
\right.
\end{equation*}
Combining the three previous displays with the triangle inequality, we get, for all such~$m$ and~$z$ and provided that~$3^l\geq \X$,
\begin{equation*}
\left\{ 
\begin{aligned}
& 
\left( \ep 3^{l} \right)^{-1}
\left\| u^\ep - v^{(l)}_{\nabla \overline{u}(z),\frac z\ep}\left( \tfrac \cdot \ep\right) - \overline{u}(z) 
\right\|_{\underline{L}^2(z+\ep \cu_{l+1})} 
\leq 
C\left( \ep^\alpha \left( \ep 3^l \right)^{-\frac d2-1} + 3^{-l\alpha} \left( \ep 3^l \right)^{-1} + \ep 3^{l} \right),
\\ & 
\left( \ep 3^{l} \right)^{-1} \left\| w^\ep_m - w^{(l)}_{m,\overline{\Theta}(z),\frac z\ep}\left( \tfrac \cdot \ep\right) - \overline{w}_m(z) 
\right\|_{\underline{L}^2(z+\ep\left(1+2^{-m} \right) \cu_l)} 
\\ & \qquad 
\leq 
C\left( \ep^\alpha \left( \ep 3^l \right)^{-\frac d2-1} + 3^{-l\alpha} \left( \ep 3^l \right)^{-1} + \ep 3^{l} \right)
\sum_{j=1}^{m} \left\|  \nabla g_{j} \right\|_{{L}^{2+\delta}(U_{j})}^{\frac{m}{j}}.
\end{aligned} 
\right.
\end{equation*}
An application of Lemma~\ref{l.diff.linearizedsystem} then yields
\begin{equation*}
\left\{ 
\begin{aligned}
& 
\left\| \nabla u^\ep - \nabla v^{(l)}_{\nabla \overline{u}(z),\frac z\ep}\left( \tfrac \cdot \ep\right) 
\right\|_{\underline{L}^2(z+\ep \cu_{l})} 
\leq 
C\left( \ep^\alpha \left( \ep 3^l \right)^{-\frac d2-1} + 3^{-l\alpha} \left( \ep 3^l \right)^{-1} + \ep 3^{l} \right),
\\ & 
\left\| \nabla w^\ep_m - \nabla w^{(l)}_{m,\overline{\Theta}(z),\frac z\ep}\left( \tfrac \cdot \ep\right) 
\right\|_{\underline{L}^2(z+\ep \cu_l)} 
\\ & \qquad 
\leq 
C\left( \ep^\alpha \left( \ep 3^l \right)^{-\frac d2-1} + 3^{-l\alpha} \left( \ep 3^l \right)^{-1} + \ep 3^{l} \right)
\sum_{j=1}^{m} \left\|  \nabla g_{j} \right\|_{{L}^{2+\delta}(U_{j})}^{\frac{m}{j}}.
\end{aligned} 
\right.
\end{equation*}
By~$L^p$ interpolation and the bounds~\eqref{e.wmTheta0bounds.X} and~\eqref{e.goodbounds.wm}, we obtain, for every $q\in[2,\infty)$, an~$\X$ such that $3^l\geq \X$ implies that, for every~$m$ and~$z$ as above, 
\begin{equation*}
\left\{ 
\begin{aligned}
& 
\left\| \nabla u^\ep - \nabla v^{(l)}_{\nabla \overline{u}(z),\frac z\ep}\left( \tfrac \cdot \ep\right) 
\right\|_{\underline{L}^q(z+\ep \cu_{l})} 
\leq 
C\left( \ep^\alpha \left( \ep 3^l \right)^{-\frac d2-1} + 3^{-l\alpha} \left( \ep 3^l \right)^{-1} + \ep 3^{l} \right),
\\ & 
\left\| \nabla w^\ep_m - \nabla w^{(l)}_{m,\overline{\Theta}(z),\frac z\ep}\left( \tfrac \cdot \ep\right) 
\right\|_{\underline{L}^q(z+\ep \cu_l)} 
\\ & \qquad 
\leq 
C\left( \ep^\alpha \left( \ep 3^l \right)^{-\frac d2-1} + 3^{-l\alpha} \left( \ep 3^l \right)^{-1} + \ep 3^{l} \right)
\sum_{j=1}^{m} \left\|  \nabla g_{j} \right\|_{{L}^{2+\delta}(U_{j})}^{\frac{m}{j}}.
\end{aligned} 
\right.
\end{equation*}
This estimate implies
\begin{equation*} \label{}
\left\{ 
\begin{aligned}
& \left\| \a^\ep - \a^{(l)}_{\nabla \overline{u}(z)}\left(\tfrac\cdot\ep\right) \right\|_{\underline{L}^q(z+\ep\cu_l)} 
\leq 
C\left( \ep^\alpha \left( \ep 3^l \right)^{-\frac d2-1} + 3^{-l\alpha} \left( \ep 3^l \right)^{-1} + \ep 3^{l} \right),
\\ & 
\left\| \mathbf{F}^\ep_{n+2} - \mathbf{F}^{(l)}_{n+2,\overline{\Theta}(z)}\left( \tfrac\cdot\ep\right) \right\|_{\underline{L}^q(z+\ep\cu_l)} 
\\ & \qquad 
\leq 
C\left( \ep^\alpha \left( \ep 3^l \right)^{-\frac d2-1} + 3^{-l\alpha} \left( \ep 3^l \right)^{-1} \!+\! \ep 3^{l} \right)
\sum_{j=1}^{n+1} \left\|  \nabla g_{j} \right\|_{{L}^{2+\delta}(U_{j})}^{\frac{n+2}{j}}.
\end{aligned} 
\right.
\end{equation*}
By a very similar argument, comparing the functions the functions $\nabla v^{(l)}_{\nabla \overline{u}(z),\frac z\ep}\left( \tfrac \cdot \ep\right)$ and $\nabla w^{(l)}_{m,\overline{\Theta}(z),\frac z\ep}\left( \tfrac \cdot \ep\right)$ to $\nabla v^{(k)}_{\nabla \overline{u}(z),\frac z\ep}\left( \tfrac \cdot \ep\right)$ and $\nabla w^{(k)}_{m,\overline{\Theta}(z),\frac z\ep}\left( \tfrac \cdot \ep\right)$, also using Lemma~\ref{l.diff.linearizedsystem} and the assumed validity of Theorems~\ref{t.linearizehigher} and~\ref{t.regularity.linerrors} for~$n$, we obtain, for all $m$ and $z$ as above,
\begin{equation*} \label{}
\left\{ 
\begin{aligned}
& \left\| \a^{(l)}_{\nabla \overline{u}(z)}\left(\tfrac\cdot\ep\right) - \a^{(k)}_{\nabla \overline{u}(z)}\left(\tfrac\cdot\ep\right) \right\|_{\underline{L}^q(z+\ep\cu_l)} 
\leq 
C 3^{-k\alpha} ,
\\ & 
\left\| \mathbf{F}^{(l)}_{n+2,\overline{\Theta}(z)}\left( \tfrac\cdot\ep\right)
- \mathbf{F}^{(k)}_{n+2,\overline{\Theta}(z)}\left( \tfrac\cdot\ep\right)
 \right\|_{\underline{L}^q(z+\ep\cu_l)} 
\leq 
C3^{-k\alpha}
\sum_{j=1}^{n+2} \left\|  \nabla g_{j} \right\|_{{L}^{2+\delta}(U_{j})}^{\frac{n+1}{j}}.
\end{aligned} 
\right.
\end{equation*}
Using~\eqref{e.contdep.Fm.Linfty} and~\eqref{e.contdep.ap.Linfty}, in view of~\eqref{e.detanchors}, we find an exponent~$Q(\data)<\infty$ such that, for every~$m$ and~$z$ as above,
\begin{equation*} \label{}
\left\{ 
\begin{aligned}
& \left\| \tilde{\a}^\ep - \a^{(k)}_{\nabla \overline{u}(z)}\left(\tfrac\cdot\ep\right) \right\|_{\underline{L}^q(z+\ep\cu_l)} 
\leq 
C \left(\ep 3^l \right) 3^{Qk} ,
\\ & 
\left\| 
\tilde{\mathbf{F}}^\ep_{n+2}
- \mathbf{F}^{(k)}_{n+2,\overline{\Theta}(z)}\left( \tfrac\cdot\ep\right)
 \right\|_{\underline{L}^q(z+\ep\cu_l)} 
\leq 
C\left(\ep 3^l \right)  3^{Qk}
\sum_{j=1}^{n+2} \left\|  \nabla g_{j} \right\|_{{L}^{2+\delta}(U_{j})}^{\frac{n+1}{j}}.
\end{aligned} 
\right.
\end{equation*}
Combining the above estimates, we finally obtain that, for every~$z$ and $m$ as above,
\begin{equation*} \label{}
\left\{ 
\begin{aligned}
& \left\| \a^\ep - \tilde{\a}^\ep \right\|_{\underline{L}^q(z+\ep\cu_l)} 
\leq 
C\left( \ep^\alpha \left( \ep 3^l \right)^{-\frac d2-1} + 3^{-l\alpha} \left( \ep 3^l \right)^{-1} + 3^{-k\alpha} + \ep 3^{l+kQ} \right),
\\ & 
\left\| \mathbf{F}^\ep_{n+2} - \tilde{\mathbf{F}}^\ep_{n+2} \right\|_{\underline{L}^q(z+\ep\cu_l)} 
\\ & \qquad 
\leq
C\left( \ep^\alpha \left( \ep 3^l \right)^{-\frac d2-1} + 3^{-l\alpha} \left( \ep 3^l \right)^{-1} + 3^{-k\alpha} + \ep 3^{l+kQ} \right)
\sum_{j=1}^{n+1} \left\|  \nabla g_{j} \right\|_{{L}^{2+\delta}(U_{j})}^{\frac{n+2}{j}},
\end{aligned} 
\right.
\end{equation*}
Using that these coefficient fields are bounded and that the boundary layer (i.e., the set of points which lie in $U_{n+2}$ but not in any cube of the form $z+\ep\cu_{l}$ with $z$ as above) has thickness~$O(\ep3^l)$, we get from the previous estimate that
\begin{equation*} \label{}
\left\{ 
\begin{aligned}
& \left\| \a^\ep - \tilde{\a}^\ep \right\|_{\underline{L}^q(U_{n+2})} 
\leq 
C\mathsf{A},
\\ & 
\left\| \mathbf{F}^\ep_{n+2} - \tilde{\mathbf{F}}^\ep_{n+2} \right\|_{\underline{L}^q(U_{n+2})} 
\leq 
C\mathsf{A}
\sum_{j=1}^{n+1} \left\|  \nabla g_{j} \right\|_{{L}^{2+\delta}(U_{j})}^{\frac{n+2}{j}},
\end{aligned} 
\right.
\end{equation*}
where (in order to shorten the expressions) we denote
\begin{equation*} \label{}
\mathsf{A}:=\left( \ep^\alpha \left( \ep 3^l \right)^{-\frac d2-1} + 3^{-l\alpha} \left( \ep 3^l \right)^{-1} + 3^{-k\alpha} + \left( \ep 3^l \right)^\alpha 3^{kQ} \right).
\end{equation*}
To complete the proof of the lemma, we observe that $\zeta:= \tilde{w}_{n+2}^\ep-w^\ep_{n+2}\in H^1_0(U_{n+2})$ satisfies the equation
\begin{equation}
-\nabla \cdot \tilde{\a}^\ep \nabla \zeta = \nabla \cdot \left( \tilde{\mathbf{F}}^\ep_{n+2} - \mathbf{F}^\ep_{n+2} \right) + \nabla \cdot \left( \left( \tilde{\a}^\ep - {\a}^\ep \right) \nabla w^\ep_{n+2} \right).
\end{equation}
By the basic energy estimate (test the equation for~$\zeta$ with~$\zeta$), we obtain
\begin{align*}
\left\| \nabla \zeta \right\|_{{L}^{2}(U_{n+2})}
\leq 
\left( 
\left\|  \tilde{\mathbf{F}}^\ep_{n+2} - \mathbf{F}^\ep_{n+2}  \right\|_{L^{2 }(U_{n+2})} 
+ 
\left\| \left( \tilde{\a}^\ep - {\a}^\ep \right) \nabla w_{n+2}^\ep  \right\|_{L^{2}(U_{n+2})}
\right).
\end{align*} 
The first term on the right side is already estimated above. For the second term, we use the Meyers estimate and~\eqref{e.goodbounds.wm}, without forgetting~\eqref{e.gvanishass}, to find~$\delta(d,\Lambda)>0$ such that
\begin{align*}
\left\| \nabla {w}^\ep_{n+2} \right\|_{\underline{L}^{2+\delta}(U_{n+2})}
\leq 
C \left\| {\mathbf{F}}^\ep_{n+2} \right\|_{L^{2+\delta}(U_{n+2})} 
\leq 
C
\sum_{j=1}^{n+1} \left\|  \nabla g_{j} \right\|_{{L}^{2+\delta}(U_{j})}^{\frac{n+2}{j}}.
\end{align*}
Therefore, by H\"older's inequality and the above estimate on~$\left\| \a^\ep - \tilde{\a}^\ep \right\|_{\underline{L}^q(z+\ep\cu_l)} $ with exponent~$q:=\frac{4+2\delta}{\delta} $, we obtain 
\begin{align*}
\lefteqn{
\left\| \left( \tilde{\a}^\ep - \a^\ep \right) \nabla w^\ep_{n+2}   \right\|_{L^{2}(U_{n+2})}
} \qquad & 
\\ & 
\leq 
\left\| \tilde{\a}^\ep - \a^\ep \right\|_{L^{\frac{4+2\delta}{\delta}}(U_{1})} \left\| \nabla w^\ep_{n+2} \right\|_{L^{2+\delta}(U_{n+2})}
\leq
C\mathsf{A}
\sum_{j=1}^{n+1} \left\|  \nabla g_{j} \right\|_{{L}^{2+\delta}(U_{j})}^{\frac{n+2}{j}}.
\end{align*}
This completes the proof of~\eqref{e.blargh.1}. 
\end{proof}

In view of the discussion in Subsections~\ref{ss.meso} and~\ref{ss.mscales}, 
Lemma~\ref{l.coeffclose} implies~\eqref{e.homogenization.estimates.wts.1}.

\subsection{Homogenization estimates for the approximating equations}
\label{ss.homogapp}
To prepare for the proof of~\eqref{e.homogenization.estimates.wts.1},
we apply the quantitative homogenization estimates proved in~\cite{AKMbook} for the linear equation~\eqref{e.statpde}. We denote by $\ahom_p^{(k)}$ and $\overline{\mathbf{F}}^{(k)}_{n+2,\Theta}$ the homogenized coefficients. 

\smallskip

By an application of the results of~\cite[Chapter 11]{AKMbook}, the solutions $v(\cdot,U,e,\Theta)$ of the family of Dirichlet problems 
\begin{equation}
\label{e.linearizequantities}
\left\{ 
\begin{aligned}
& -\nabla\cdot\left( \a^{(k)}_p \nabla v(\cdot,U,e,\Theta) \right) = \nabla \cdot \mathbf{F}^{(k)}_{n+2,\Theta} & \mbox{in} & \ U, \\
& v(\cdot,U,e,\Theta) = \ell_e & \mbox{on} & \ \partial U,
\end{aligned}
\right.
\end{equation}
indexed over bounded Lipschitz domains $U\subseteq\Rd$ and $e\in B_1$ and~$\Theta\in\R^{d(n+2)}$,
satisfy the following quantitative homogenization estimate: there exist~$\alpha(s,d,\Lambda)>0$, $\beta(\data)>0$, $\delta(\data)>0$ and~$Q(\data)<\infty$ and~$C(s,\mathsf{M},\data)<\infty$ and a random variable~$\X = \O_\delta(C)$  such that, for every $\Theta=(p,h_1,\ldots,h_{n+1})\in\R^{d(n+2)}$ with $|p|\leq \mathsf{M}$, $e\in\Rd$ and every $l\in\N$ with $3^{l-k} \geq \X$,
\begin{align}
\label{e.appcorrbounds}
& 
3^{-l} \left\| v(\cdot,\cu_{l+1},e,\Theta) - \ell_e \right\|_{\underline{L}^2(\cu_l)} +
3^{-l} \left\| \nabla v(\cdot,\cu_{l+1},e,\Theta) - e \right\|_{\underline{H}^{-1}(\cu_l)} 
\\ &  \ \  \notag
+
3^{-l} \left\| \a_p^{(k)} \nabla v(\cdot,\cu_{l+1},e,\Theta) + \mathbf{F}^{(k)}_{n+2,\Theta} - \left( \ahom_p^{(k)} e+\overline{\mathbf{F}}^{(k)}_{n+2,\Theta} \right)\right\|_{\underline{H}^{-1}(\cu_l)}
\\ & \notag
\qquad\qquad\qquad\qquad
\qquad\qquad \qquad 
\leq
C3^{-\alpha(l-k)} 3^{Qk} \left( |e| +  \sum_{j=1}^{n+1} \left| h_j \right|^{\frac {n+2}j}    \right) 
\end{align}
as well as
\begin{equation}
\label{e.detbound.wU}
\left\| \nabla v(\cdot,\cu_{l+1},e,\Theta) \right\|_{C^{0,\beta}(\cu_l)} 
\leq 
Cl^Q3^{Qk} \left( |e| + \sum_{j=1}^{n+1} \left| h_j \right|^{\frac {n+2}j}  \right).
\end{equation}
To obtain~\eqref{e.appcorrbounds}, we change the scale by performing a dilation $x\mapsto 3^k \left\lceil \left(15+d\right)^{\frac12} \right\rceil  x$ so that the resulting coefficient fields have a unit range of dependence and are still~$\Zd$--stationary, cf.~\eqref{e.klocalize}. We then apply~\cite[Lemma 11.11]{AKMbook} to obtain~\eqref{e.appcorrbounds}, using Lemma~\ref{l.detbounds.wm}, namely~\eqref{e.detbounds.Fm}, to give a bound on the parameter~$\mathsf{K}$ in the application of the former. The reason we quote results for general nonlinear equations, despite the fact that~\eqref{e.linearizequantities} is linear, is because~\eqref{e.linearizequantities} has a nonzero right-hand side and the results for linear equations have not been formalized in that generality. 

\smallskip

The bound~\eqref{e.detbound.wU} is a consequence the large-scale $C^{0,1}$--type estimate for the equation~\eqref{e.statpde} (see \cite[Theorem 11.13]{AKMbook}) which, together with an routine union bound, yields~\eqref{e.detbound.wU} with the left side replaced by
$\sup_{z \in \Zd \cap \cu_l} \left\| \nabla v \right\|_{\underline{L}^2(z+(l-k)\cu_k)}$. Indeed, if we denote by $\X_z$ the minimal scale in the large-scale $C^{0,1}$ estimate, then a union bound yields (for any $s\in (0,d)$, so in particular we can take $s>1$), 
\begin{equation*} \label{}
\max\left\{ 
k' \in \N \,:\, \max_{z\in \Zd \cap \cu_{k'} } \X_z > \left( \log 3^{k'-k} \right)^{\frac 1s} 
\right\}
\leq \O_\delta (C). 
\end{equation*}
We obtain~\eqref{e.detbound.wU} by combining this bound for $\sup_{z \in \Zd \cap \cu_l} \left\| \nabla v \right\|_{\underline{L}^2(z+(l-k)\cu_k)}$ with the deterministic bounds~\eqref{e.detbounds.Fm}, ~\eqref{e.detbounds.ak} and an application of the Schauder estimates (see Proposition~\ref{p.schauder}).

\smallskip

We also need the following estimate for the difference of~$v$'s on overlapping cubes, which can be inferred from~\eqref{e.appcorrbounds} and the Caccioppoli inequality: 
there exist~$\alpha(s,d,\Lambda)>0$, $\beta(\data)>0$, $\delta(\data)>0$ and~$Q(\data)<\infty$ and~$C(s,\mathsf{M},\data)<\infty$ and a random variable~$\X = \O_\delta(C)$  such that, for every $\Theta=(p,h_1,\ldots,h_{n+1})\in\R^{d(n+2)}$ with $|p|\leq \mathsf{M}$, $e\in\Rd$ and every $l\in\N$ with $3^{l-k} \geq \X$,
\begin{align}
\label{e.appcorr.overlap}
\lefteqn{
\sum_{z'\in 3^l\Zd\cap \cu_{l+1}}
\left\| 
\nabla v\left(\cdot,z+\cu_{l+1},e,\Theta\right) 
-
\nabla v\left(\cdot,z'+\cu_{l},e,\Theta\right) 
\right\|_{\underline{L}^2(z'+\cu_{l})}
} 
\qquad\qquad\qquad\qquad
\qquad\qquad\qquad 
& 
\\ & \notag
\leq
C3^{-\alpha(l-k)}  \left( |e| + 3^{Qk} \sum_{j=1}^{n+1} \left| h_j \right|^{\frac {n+2}j} \right) .
\end{align}
Finally, we mention that we have the following deterministic estimate on the functions $v(\cdot,z+\cu_{l+1},e,\Theta)$ which states that
\begin{align}
\label{e.appcorr.bndd}
\left\| \nabla v(\cdot,z+\cu_{l+1},e,\Theta) \right\|_{\underline{L}^2(z+\cu_{l+1})}
&
\leq
C \left( |e| + \left\| \mathbf{F}^{(k)}_{n+2,\Theta} \right\|_{\underline{L}^2(z+\cu_{l+1})} \right)
\\ & \notag
\leq
C\left( |e| + \sum_{j=1}^{n+1} \left| h_j \right|^{\frac {n+2}j}  \right).
\end{align}
This follows from testing the problem~\eqref{e.linearizequantities} with $U=z+\cu_{l+1}$ with the test function $v(\cdot,z+\cu_{l+1},e,\Theta) - \ell_e\in H^1_0(z+\cu_{l+1})$ and then using~\eqref{e.FmTheta0bounds.X} with~ $m=n+2$. 

Our next goal is to compare the homogenized coefficients~$\overline{\mathbf{F}}^{(k)}_{m,\Theta}$ for the approximating equations to the functions~$\overline{\mathbf{F}}_m (\Theta)$ defined in~\eqref{e.defbarFm}. In view of the results of Section~\ref{s.regLbar}, it is natural to proceed by comparing the vector fields~$\mathbf{F}_{m,\Theta}^{(k)}$ to~$\mathbf{f}^{(m)}_{p,h}$ defined in~\eqref{e.corrcoeff}. This is accomplished by invoking Lemma~\ref{l.diff.linearizedsystem} again.

\begin{lemma}
\label{l.fmsclose}
Fix $q\in [2,\infty)$ and $\mathsf{M}\in [1,\infty)$. 
There exist~$\delta(q,\data),\alpha(q,\data)\in (0,\frac{1}{2}]$,~$C(q,\mathsf{M}, \text{data})<\infty$ such that, for every~$k\in\N$ with $3^k\geq \X$,~$m\in\{ 1,\ldots,n+2\}$ and~$p,h\in\Rd$ with $|p|\leq \mathsf{M}$, we have
\begin{equation}
\label{e.comcoeff.ap}
\left\| \a_p - \a_p^{(k)} \right\|_{\underline{L}^q(\cu_k)}
\leq 
\O_\delta\left( C 3^{-k\alpha} \right)
\end{equation}
and, for~$\Theta:=(p,h,0,\cdots,0)\in \R^{d(n+2)}$,
\begin{equation}
\label{e.comcoeff.Fm}
\left\| \mathbf{F}_{m,\Theta}^{(k)} - \mathbf{f}^{(m)}_{p,h} \right\|_{\underline{L}^q(\cu_k)}
\leq 
\O_\delta\left( C 3^{-k\alpha} |h|^{m} \right).
\end{equation}
\end{lemma}
\begin{proof}
We take~$\X$ to be the maximum of the random scales in Theorem~\ref{t.correctorestimates}, Lemmas~\ref{l.corr.sublinearity} and~\ref{l.diff.linearizedsystem} and Theorems~\ref{t.linearizehigher} and~\ref{t.regularity.linerrors} for $n$, the latter two being valid by assumption~\eqref{e.assumption.section4}. Assume $3^k \geq \X$. 

\smallskip

By Lemma~\ref{l.corr.sublinearity} and the assumed validity of Theorem~\ref{t.linearizehigher} for $n$, we have  
\begin{equation*} \label{}
3^{-k} \left\| v^{(k)}_{p,0} - \left( \ell_p + \phi_p - \left(\phi_p \right)_{\cu_{k+1}} \right) \right\|_{\underline{L}^2(\cu_{k+1})}
\leq 
C 3^{-k\alpha},
\end{equation*} 
\begin{equation*} \label{}
3^{-k} \left\| w^{(k)}_{1,\Theta,0}  - \left( \ell_h+\psi^{(1)}_{p,h} - \left(\psi^{(1)}_{p,h} \right)_{\frac32\cu_{k}} \right) \right\|_{\underline{L}^2(\frac32\cu_{k})}
\leq 
C 3^{-k\alpha}|h|
\end{equation*}
and, for $m\in\{2,\ldots,n+1\}$,
\begin{equation*} \label{}
3^{-k} \left\| w^{(k)}_{m,\Theta,0} - \left( \psi^{(m)}_{p,h} - \left(\psi^{(m)}_{p,h} \right)_{\frac32\cu_{k}} \right) \right\|_{\underline{L}^2((1+2^{-m})\cu_{k})}
\leq 
C 3^{-k\alpha} |h|^m.
\end{equation*}
By Theorem~\ref{t.correctorestimates} and the assumed validity of Theorem~\ref{t.regularity.linerrors} for $n$, we also have that, for every $m\in\{1,\ldots,n+1\}$,  
\begin{equation} 
\label{e.wmTheta0bounds}
\left\| \nabla w^{(k)}_{m,\Theta,0}
\right\|_{\underline{L}^2((1+2^{-m})\cu_{k})}
+
\left\| 
\nabla \psi^{(m)}_{p,h}
\right\|_{\underline{L}^2((1+2^{-m})\cu_{k})}
\leq 
C|h|^m. 
\end{equation}
Using these estimates and Lemma~\ref{l.diff.linearizedsystem}, we obtain, for $m\in\{2,\ldots,n+1\}$,
\begin{equation} 
\label{e.closenessofcorr}
\left\| \nabla w^{(k)}_{m,\Theta,0} - \nabla \psi^{(m)}_{p,h} 
\right\|_{\underline{L}^2(\cu_{k})}
\leq 
C 3^{-k\alpha} |h|^m.
\end{equation}
The previous display holds for $3^k\geq \X$. Combining this estimate with~\eqref{e.sec3corrbnd} and~\eqref{e.FmTheta0bounds.X2} and using $L^p$ interpolation, we obtain
\begin{equation} 
\label{e.closenessofcorr.O}
\left\| \nabla w^{(k)}_{m,\Theta,0} - \nabla \psi^{(m)}_{p,h} 
\right\|_{\underline{L}^2(\cu_{k})}
\leq 
\O_\delta\left( C 3^{-k\alpha} |h|^m \right).
\end{equation}
The conclusion of the lemma now follows. 
\end{proof}

We next observe that the homogenized coefficients~$\ahom^{(k)}_p$ and $\overline{\mathbf{F}}^{(k)}_{n+2,\Theta}$ agree, up to a small error, with $\ahom_p:=D^2\overline{L}(p)$ and $\overline{\mathbf{F}}_{n+2}(\Theta)$. This will help us eventually to prove that the homogenized coefficients for the linearized equations are the linearized coefficients of the homogenized equation.

\begin{lemma}
\label{l.barsclose}
Fix $\mathsf{M}\in [1,\infty)$.  There exist~$\alpha(\data)>0$ and~ $C(n,\mathsf{M},\data)<\infty$ such that, for every $\Theta =(p,h_1,\ldots,h_{n+1}) \in B_\mathsf{M} \times \R^{d(n+1)}$, 
\begin{equation}
\label{e.abarsclose}
\left| \ahom_p - \ahom_p^{(k)}\right| 
\leq 
C 3^{-k\alpha}
\end{equation}
and
\begin{equation}
\label{e.barsclose}
\left| \overline{\mathbf{F}}^{(k)}_{n+2,\Theta} - \overline{\mathbf{F}}_{n+2}(\Theta)\right| 
\leq 
C 3^{-k\alpha}
\left( \sum_{j=1}^{n+1} \left| h_j \right|^{\frac {n+1}j}  \right).
\end{equation}
\end{lemma}
\begin{proof}
We first give the argument in the case that $\Theta=(p,h,0,\ldots,0)$.
From the definition~\eqref{e.defbarFm} and Proposition~\ref{e.Lbar.reg.qualitative}, we see that $\left( \ahom_p,\overline{\mathbf{F}}_{n+2}(\Theta) \right)$ are the homogenized coefficients for the stationary coefficient fields~$\left( \a_p, \f_{p,h}^{(n+2)} \right)$. By Lemma~\ref{l.fmsclose}, and the bounds~\eqref{e.detbounds.Fm} and~\eqref{e.sec3corrbnd}, we find that, for every $q\in [2,\infty)$, 
\begin{equation}
\left\| \a_p - \a_p^{(k)} \right\|_{\underline{L}^q(\cu_k)}
\leq \O_\delta\left( 3^{-k\alpha} \right)
\end{equation}
and
\begin{equation}
\left\| \mathbf{F}_{m,\Theta}^{(k)} - \mathbf{f}^{(m)}_{p,h} \right\|_{\underline{L}^q(\cu_k)} 
\leq 
\O_\delta\left( C 3^{-k\alpha} |h|^{m} \right),
\end{equation}
with the constants $C$ depending additionally on~$q$. The result now follows easily by the Meyers estimate in the case~$\Theta=(p,h,0,\ldots,0)$. Indeed, by subtracting the equations and applying H\"older's inequality and the Meyers estimate, we can show that the first-order correctors for the two linear problems are close in the $\underline{L}^2(\cu_k)$ norm (with a second moment in expectation). Therefore their fluxes are also close, and taking the expectations of the fluxes gives the homogenized coefficients. This argument can also be performed in a finite volume box~$\cu_M$ with the result obtained after sending $M\to \infty$. See~\cite[Lemma 3.2]{AFK} for a similar argument. 

\smallskip

For general~$\Theta$ the result follows by polarization (see Lemma~\ref{l.polarization}) and the multilinear structure shared by the functions $\Theta\mapsto \overline{\mathbf{F}}^{(k)}_{n+2,\Theta}$ and $\Theta \mapsto \overline{\mathbf{F}}_{n+2}(\Theta)$.
\end{proof}

We also require the following continuity estimate for the functions $w(\cdot,\cu_l,e,\Theta)$ in~$(e,\Theta)$. Following the proof of Lemma~\ref{l.contdep.coeffs} for~$L^2$ dependence in $\Theta$ (we need need to do one more step of the iteration described there), using~\eqref{e.appcorrbounds} for ~$L^2$ dependence in~$e$, and then interpolating this result with~\eqref{e.detbound.wU}, we obtain exponents $\alpha(\data)>0$ and $Q(\data)<\infty$ and a random variable~$\X = \O_\delta(C)$ such that, for every $e,e'\in\Rd$ and~$\Theta=(p,h_1,\ldots,h_{n+1})\in \R^{d(n+2)}$ and $\Theta'=(p',h_{1}',\ldots,h_{n+1}')\in \R^{d(n+2)}$
with $|p|,|p'|\leq \mathsf{M}$, if $3^k\geq \X$, then
\begin{align} 
\label{e.contdep.corrw}
&
\left\| \nabla v(\cdot,\cu_{l+1},e,\Theta) - \nabla v(\cdot,\cu_{l+1},e',\Theta' )\right\|_{L^\infty(\cu_l)}
\\
& \qquad \notag
\leq
Cl^Q3^{kQ} |p-p'|^\alpha \sum_{i=1}^{m-1} \left( \left| h_i \right| \vee \left| h_i' \right| \right)^{\frac mi}
+Cl^Q 3^{kQ} \sum_{i=1}^{m-1} |h_i-h_i'|^\alpha \sum_{j=1}^{m-i}
\left( \left| h_j \right| \vee \left| h_j' \right| \right)^{\frac {m-i}j}.
\end{align}

\subsection{Homogenization of the locally stationary equation}

We next present the proof of~\eqref{e.homogenization.estimates.wts.2}, which is the final step in the proof of Theorem~\ref{t.linearizehigher}. The argument follows a fairly routine (albeit technical) two-scale expansion argument and requires the homogenization estimates given in Subsection~\ref{ss.homogapp}.

\begin{proof}[{The proof of~\eqref{e.homogenization.estimates.wts.2}}]
In view of the discussion in Subsection~\ref{ss.mscales}, we may assume throughout that $3^k\geq \X$ where~$\X$ is any minimal scale described above. 
We begin by building an approximation of the solution~$\tilde{w}_{n+2}^\ep$ of~\eqref{e.wn+2.tilde} built out of the solutions of~\eqref{e.linearizequantities}. We select a smooth function $\zeta \in C^\infty_c( \Rd )$ such that, for some $C(d)<\infty$ and each~$i\in\{1,2\}$,  
\begin{equation} 
\label{e.cutoffzeta}
\indc_{\ep \cu_l} \leq \zeta \leq \indc_{\ep \cu_{l+1}}
\quad 
\left\| \nabla^i \zeta \right\|_{L^\infty}
\leq C\left(\ep3^{l}\right)^{-i}, 
\quad \mbox{and} \quad  
\sum_{z\in \ep 3^l\Zd} \zeta(\cdot-z) = 1 \quad \mbox{in} \ \Rd.
\end{equation}
We can construct such a~$\zeta$ by a suitable mollification of $\indc_{\ep\cu_{l}}$.
We write $\zeta_z:= \zeta(\cdot-z)$ for short. We next define a function $\tilde{v}^\ep_{n+2} \in H^1_0(U_{n+2})$ by 
\begin{align*} \label{}
& 
\tilde{v}_{n+2}^\ep(x)
\\ & \quad 
:=
\overline{w}_{n+2}(x) 
+
\sum_{z\in \ep3^l\Zd \cap U_{n+3}}
\zeta_z(x) \left( \ep v\left( \tfrac x\ep , \tfrac{z}\ep + \cu_{l+1},\nabla \overline{w}_{n+2}(z),\overline\Theta(z) \right) - \ell_{\nabla \overline{w}_{n+2}(z)}(x) \right).
\end{align*}
Since $\tilde{v}_{n+2}^\ep - \tilde{w}_{n+2}^\ep \in H^1_0(U_{n+2})$, 
We have that 
\begin{align} 
\label{e.captcha}
\left\| \nabla \tilde{v}^\ep_{n+2} - \nabla \tilde{w}^\ep_{n+2} 
\right\|_{L^2(U_{n+2})}
&
\leq 
C 
\left\| \nabla \cdot \tilde{\a}^\ep \nabla \tilde{v}^\ep_{n+2} 
-  \nabla \cdot \tilde{\a}^\ep\nabla \tilde{w}^\ep_{n+2} 
\right\|_{H^{-1}(U_{n+2})}
\\ & \notag
= 
C  
\left\| 
\nabla \cdot \left(  \tilde{\a}^\ep \nabla \tilde{v}_{n+2}^\ep
+ \tilde{\mathbf{F}}^\ep_{n+2} \right) 
\right\|_{H^{-1}(U_{n+2})}.
\end{align}
A preliminary goal is therefore to prove the estimate
\begin{align} 
\label{e.eqyes.wts}
&
\left\| 
\nabla \cdot \left(  \tilde{\a}^\ep \nabla \tilde{v}_{n+2}^\ep
+ \tilde{\mathbf{F}}^\ep_{n+2} \right) 
\right\|_{H^{-1}(U_{n+2})}
\\ & \qquad \notag
\leq
C\left( \ep^\alpha \left( \ep 3^l \right)^{-Q} + 3^{-l\alpha} \left( \ep 3^l \right)^{-1} + 3^{-k\alpha} + \left( \ep 3^l \right)^\alpha 3^{kQ} \right)
\sum_{j=1}^{n+1} \left\|  \nabla g_{j} \right\|_{{L}^{2+\delta}(U_{j})}^{\frac{n+2}{j}}.
\end{align}
The proof of~\eqref{e.eqyes.wts} is essentially completed in Steps~1--5 below: see~\eqref{e.yesstep6} and~\eqref{e.captcha.indeed}. After obtaining this estimate, we will prove that 
\begin{equation} 
\label{e.throwyes.wts}
\left\| \nabla \tilde{v}^\ep_{n+2} - \nabla \overline{w}_{n+2} \right\|_{H^{-1}(U_{n+2})}
\leq 
C 3^{-\alpha(l-k)} 3^{Qk} \sum_{i=1}^{n+1} \left\| g_i \right\|_{{L}^2(U_i)}^{\frac {n+2}i}.
\end{equation}
Together these inequalities imply~\eqref{e.homogenization.estimates.wts.2}. 

\smallskip

Denote $v_z:=  v\left( \cdot , \tfrac{z}\ep + \cu_{l+1},\nabla \overline{w}_{n+2}(z),\overline\Theta(z) \right)$ for short and  compute 
\begin{align}
\label{e.gradtildevep}
\nabla \tilde{v}_{n+2}^\ep(x)
&
=
\sum_{z\in \ep3^l\Zd \cap U_{n+3}} \left(
\nabla \zeta_z(x) \left( \ep v_z(\tfrac x\ep ) - \ell_{\nabla \overline{w}_{n+2}(z) }(x)\right)
+ \zeta_z(x) \nabla v_z\left( \tfrac x\ep \right)
\right)
\\ & \qquad \notag
+
\left( \nabla \overline{w}_{n+2}(x) - \sum_{z\in \ep3^l\Zd \cap U_{n+3}} \zeta_z(x) \nabla \overline{w}_{n+2}(z) \right) .
\end{align}
Thus
\begin{align} 
\label{e.fluxtildevep}
\tilde{\a}^\ep \nabla \tilde{v}^{\ep}_{n+2} 
&
=
\sum_{z\in \ep3^l\Zd \cap U_{n+3}}
\zeta_z
\tilde{\a}^{(k)}_{\nabla \overline{u}(z)}
\nabla v_z\left( \tfrac \cdot\ep \right)
\\ & \qquad \notag
+
 \sum_{z\in \ep3^l\Zd \cap U_{n+3}}
\tilde{\a}^\ep\nabla \zeta_z \left( \ep v_z(\tfrac \cdot\ep ) - \ell_{\nabla \overline{w}_{n+2}(z)}(x) \right)
\\ & \qquad \notag
+
\sum_{z\in \ep3^l\Zd \cap U_{n+3}}
\zeta_z(x) 
\left( \tilde{\a}^\ep - \tilde{\a}^{(k)}_{\nabla \overline{u}(z)}\left(\tfrac\cdot\ep \right)\right) 
\nabla v_z\left( \tfrac \cdot\ep \right)
\\ & \qquad \notag
+
 \tilde{\a}^\ep\left( \nabla \overline{w}_{n+2}(x) - \sum_{z\in \ep3^l\Zd \cap U_{n+3}} \zeta_z(x) \nabla \overline{w}_{n+2}(z) \right) .
\end{align}
Using the equation for $v_z$, we find that 
\begin{align*} \label{}
\nabla \cdot \left(  \tilde{\a}^\ep \nabla \tilde{v}^{\ep}_{n+2} + \tilde{\mathbf{F}}^\ep_{n+2} \right)
&
=
\sum_{z\in \ep3^l\Zd \cap U_{n+3}}
\nabla \zeta_z \cdot \left( 
\tilde{\a}^{(k)}_{\nabla \overline{u}(z)}\left(\tfrac \cdot\ep\right)
\nabla v_z\left( \tfrac \cdot\ep \right) + \tilde{\mathbf{F}}^{(k)}_{n+2,\overline{\Theta}(z)}\left(\tfrac \cdot\ep\right) \right)
\\ & \qquad 
+\nabla \cdot \sum_{z\in \ep3^l\Zd \cap U_{n+3}}
\zeta_z(x) 
\left( \tilde{\a}^\ep - \tilde{\a}^{(k)}_{\nabla \overline{u}(z)}\left(\tfrac\cdot\ep \right)\right) 
\nabla v_z\left( \tfrac \cdot\ep \right)
\\ & \qquad 
+
\nabla \cdot \left( \tilde{\mathbf{F}}^\ep_{n+2} -  \sum_{z\in \ep3^l\Zd \cap U_{n+3}}
\zeta_z  
\tilde{\mathbf{F}}^{(k)}_{n+2,\overline{\Theta}(z)}\left(\tfrac \cdot\ep\right) \right)
\\ & \qquad 
+
\nabla \cdot\left( \sum_{z\in \ep3^l\Zd \cap U_{n+3}}
\tilde{\a}^\ep\nabla \zeta_z \left( \ep v_z(\tfrac \cdot\ep ) - \ell_{\nabla \overline{w}_{n+2}(z)}(x) \right) \right)
\\ & \qquad 
+
\nabla \cdot 
\left(
\tilde{\a}^\ep\left( \nabla \overline{w}_{n+2}(x) - \sum_{z\in \ep3^l\Zd \cap U_{n+3}} \zeta_z(x) \nabla \overline{w}_{n+2}(z) \right)
\right).
\end{align*}
Therefore, we have that 
\begin{align} 
\label{e.bigestwthreetterms}
\lefteqn{
\left\| 
\nabla \cdot \left(  \tilde{\a}^\ep \nabla \tilde{v}_{n+2}^\ep
+ \tilde{\mathbf{F}}^\ep_{n+2} \right) 
\right\|_{H^{-1}(U_{n+2})}
} \quad & 
\\ & \notag
\leq 
C\left\| 
\sum_{z\in \ep3^l\Zd \cap U_{n+3}}
\nabla \zeta_z \cdot \left( 
\tilde{\a}^{(k)}_{\nabla \overline{u}(z)}\left(\tfrac \cdot\ep\right)
\nabla v_z\left( \tfrac \cdot\ep \right) + \tilde{\mathbf{F}}^{(k)}_{n+2,\overline{\Theta}(z)}\left(\tfrac \cdot\ep\right) \right)
\right\|_{H^{-1}(U_{n+2})} 
\\ & \qquad \notag
+C\left\|  \sum_{z\in \ep3^l\Zd \cap U_{n+3}}
\zeta_z(x) 
\left( \tilde{\a}^\ep - \tilde{\a}^{(k)}_{\nabla \overline{u}(z)}\left(\tfrac\cdot\ep \right)\right) 
\nabla v_z\left( \tfrac \cdot\ep \right)
\right\|_{L^2(U_{n+2})}
\\ & \qquad \notag
+
C\left\|  
\tilde{\mathbf{F}}^\ep_{n+2} -  \sum_{z\in \ep3^l\Zd \cap U_{n+3}}
\zeta_z  
\tilde{\mathbf{F}}^{(k)}_{n+2,\overline{\Theta}(z)}\left(\tfrac \cdot\ep\right)
\right\|_{L^2(U_{n+2})}
\\ & \qquad \notag
+
C\left\| \sum_{z\in \ep3^l\Zd \cap U_{n+3}}
\nabla \zeta_z \left( \ep v_z(\tfrac \cdot\ep ) - \ell_{\nabla \overline{w}_{n+2}(z)}(x) \right) 
\right\|_{L^2(U_{n+2})}
\\ & \qquad \notag
+
C\left\| \nabla \overline{w}_{n+2}(x) - \sum_{z\in \ep3^l\Zd \cap U_{n+3}} \zeta_z(x) \nabla \overline{w}_{n+2}(z)
\right\|_{L^2(U_{n+2})} .
\end{align}
We next estimate each of the five terms on the right side of the previous display. 

\smallskip

\emph{Step 1.} The estimate of the first term on the right side of~\eqref{e.bigestwthreetterms}. 
This is where we use the homogenized equation for $\overline{w}_{n+2}$. 
Fix  $h\in H^1_0(U_{n+2})$ with $\left\| h \right\|_{H^1(U_{n+2})}=1$. 
By the equation for $\overline{w}_{n+2}$ and integration by parts that
\begin{align*}
0
&
=
\int_{U_{n+2}} 
\nabla h(x) \cdot \left( \ahom(x) \nabla \overline{w}_{n+2}(x) + \overline{\mathbf{F}}_{n+2}(x) \right)
\,dx 
\\ &  
=
- \int_{U_{n+2}} 
h(x)  \sum_{z\in3^l\Zd\cap U_{n+3}} \nabla \zeta_z(x) \cdot \left( \ahom(x) \nabla \overline{w}_{n+2}(x) + \overline{\mathbf{F}}_{n+2}(x) \right)
\,dx 
\\ & \qquad 
+
\int_{U_{n+2}} 
\nabla h(x) \cdot \sum_{z\in3^l\Zd\setminus U_{n+3}} \zeta_z(x) \left( \ahom(x) \nabla \overline{w}_{n+2}(x) + \overline{\mathbf{F}}_{n+2}(x) \right)
\,dx .
\end{align*}
Therefore we may have 
\begin{align*}
\lefteqn{
\left| 
\int_{U_{n+2}} 
h(x) 
\sum_{z\in \ep3^l\Zd \cap U_{n+3}}
\nabla \zeta_z(x) \cdot \left( 
\tilde{\a}^{(k)}_{\nabla \overline{u}(z)}\left(\tfrac  x\ep\right)
\nabla v_z\left( \tfrac x\ep \right) + \tilde{\mathbf{F}}^{(k)}_{n+2,\overline{\Theta}(z)}\left(\tfrac x\ep\right) \right) \,dx
\right|
}  & 
\\ & 
\leq 
\left| 
\int_{U_{n+2}} 
h \!\!
\sum_{z\in \ep3^l\Zd \cap U_{n+3}}
\nabla \zeta_z \cdot \left( 
\tilde{\a}^{(k)}_{\nabla \overline{u}(z)}\left(\tfrac  \cdot \ep\right)
\nabla v_z\left( \tfrac \cdot\ep \right) + \tilde{\mathbf{F}}^{(k)}_{n+2,\overline{\Theta}(z)}\left(\tfrac \cdot\ep\right) 
-
\ahom \nabla \overline{w}_{n+2} + \overline{\mathbf{F}}_{n+2}
\right) \,dx
\right|
\\ & \qquad 
+
\int_{U_{n+2}} 
\left| \nabla h(x) \right| 
\left| \, \sum_{z\in3^l\Zd\setminus U_{n+3}} \zeta_z(x) \left( \ahom(x) \nabla \overline{w}_{n+2}(x) + \overline{\mathbf{F}}_{n+2}(x) \right)
\, \right|
\,dx .
\end{align*}
For the first term on the right, we use~\eqref{e.detanchors},~\eqref{e.appcorrbounds},~\eqref{e.abarsclose} and~\eqref{e.barsclose} to obtain
\begin{multline*}
\left| 
\int_{U_{n+2}} 
h \!\!
\sum_{z\in \ep3^l\Zd \cap U_{n+3}}
\nabla \zeta_z \cdot \left( 
\tilde{\a}^{(k)}_{\nabla \overline{u}(z)}\left(\tfrac  \cdot \ep\right)
\nabla v_z\left( \tfrac \cdot\ep \right) + \tilde{\mathbf{F}}^{(k)}_{n+2,\overline{\Theta}(z)}\left(\tfrac \cdot\ep\right) 
-
\ahom \nabla \overline{w}_{n+2} + \overline{\mathbf{F}}_{n+2}
\right) \,dx \,
\right|
\\
\leq
C \left\| h \right\|_{H^1(U_{n+2})}
\left( 3^{-\alpha(l-k)}3^{Qk} + 3^{-k\alpha} \right)
 \sum_{i=1}^{n+1} \left\| g_i \right\|_{{L}^2(U_i)}^{\frac {n+2}i}.
\end{multline*}
For the second term, we denote 
\begin{equation}
\label{e.defUn4}
U_{n+4}:= \left\{ x\in U_{n+3} \,:\, x+\ep\cu_{l+1} \subseteq U_{n+3} \right\},
\end{equation}
observe that $\left|U_{n+2}\setminus U_{n+4} \right| \leq C \ep 3^{l}$, and compute, using the H\"older inequality, to obtain, using also~\eqref{e.wn2meyers},
\begin{align*}
\lefteqn{
\int_{U_{n+2}} 
\left| \nabla h(x) \right| \left|  \sum_{z\in3^l\Zd\setminus U_{n+3}} \zeta_z(x) \left( \ahom(x) \nabla \overline{w}_{n+2}(x) + \overline{\mathbf{F}}_{n+2}(x) \right) \right|
\,dx
} \qquad & 
\\ & 
\leq 
C \left\| \nabla h \right\|_{L^2(U_{n+2})}
\left\| \ahom  \nabla \overline{w}_{n+2}  + \overline{\mathbf{F}}_{n+2} \right\|_{L^2(U_{n+2} \setminus U_{n+4})}
\\ & 
\leq 
C \left\| \nabla h \right\|_{L^2(U_{n+2})}
\left( 
(\ep3^l)^{\frac{\delta}{4+2\delta}} 
\left\| \nabla \overline{w}_{n+2} \right\|_{L^{2+\delta}(U_{n+2})} + \ep 3^l \left\| \overline{\mathbf{F}}_{n+2} \right\|_{L^\infty(U_{n+2})}  \right)
\\ & 
\leq 
C \left\| \nabla h \right\|_{L^2(U_{n+2})} (\ep3^l)^{\alpha} \sum_{i=1}^{n+1} \left\| g_i \right\|_{L^2(U_i)}^{\frac{n+2}i} .
\end{align*}
Combining the above estimates and taking the supremum over~$h\in H^1_0(U_{n+2})$ with~$\| h \|_{H^1(U_{n+2})}=1$ yields 
\begin{align*}
&
\left\| 
\sum_{z\in \ep3^l\Zd \cap U_{n+3}}
\nabla \zeta_z \cdot \left( 
\tilde{\a}^{(k)}_{\nabla \overline{u}(z)}\left(\tfrac \cdot\ep\right)
\nabla v_z\left( \tfrac \cdot\ep \right) + \tilde{\mathbf{F}}^{(k)}_{n+2,\overline{\Theta}(z)}\left(\tfrac \cdot\ep\right) \right)
\right\|_{H^{-1}(U_{n+2})} 
\\ & \qquad\qquad\qquad\qquad
\notag
\leq 
C \left( 3^{-\alpha(l-k)}3^{Qk} + 3^{-k\alpha} + (\ep3^l)^{\alpha} \right) \sum_{i=1}^{n+1} \left\| g_i \right\|_{L^2(U_i)}^{\frac{n+2}i} .
\end{align*}

\smallskip

\emph{Step 2.} The estimate of the second term on the right side of~\eqref{e.bigestwthreetterms}. We use~\eqref{e.contdep.ap.Linfty},~\eqref{e.detanchors} and~\eqref{e.appcorr.bndd} to find
\begin{align*}
\lefteqn{
\left\|  \sum_{z\in \ep3^l\Zd \cap U_{n+3}}
\zeta_z(x) 
\left( \tilde{\a}^\ep - \tilde{\a}^{(k)}_{\nabla \overline{u}(z)}\left(\tfrac\cdot\ep \right)\right) 
\nabla v_z\left( \tfrac \cdot\ep \right)
\right\|_{L^2(U_{n+2})}
} \qquad & 
\\ &
\leq 
\sum_{z\in \ep3^l\Zd \cap U_{n+3}}
\left\|
\left( \tilde{\a}^\ep - \tilde{\a}^{(k)}_{\nabla \overline{u}(z)}\left(\tfrac\cdot\ep \right)\right) 
\nabla v_z\left( \tfrac \cdot\ep \right)
\right\|_{L^2(z+\ep\cu_{l+1})}
\\ & 
\leq 
C
\sum_{z\in \ep3^l\Zd \cap U_{n+3}}
\left\|
\tilde{\a}^\ep - \tilde{\a}^{(k)}_{\nabla \overline{u}(z)}\left(\tfrac\cdot\ep \right)
\right\|_{L^\infty(z+\ep\cu_{l+1})}
\left\| \nabla v_z\left( \tfrac \cdot\ep \right)
\right\|_{L^2(z+\ep\cu_{l+1})}
\\ & 
\leq 
C 3^{Qk} \left( \ep 3^{l} \right)^\alpha \sum_{i=1}^{n+1} \left\| g_i \right\|_{{L}^2(U_i)}^{\frac {n+2}i}. 
\end{align*}

\smallskip

\emph{Step 3.} The estimate of the third term on the right side of~\eqref{e.bigestwthreetterms}. We have by~\eqref{e.contdep.Fm.Linfty} and~\eqref{e.detanchors} that 
\begin{align*} \label{}
\lefteqn{
\left\|  \tilde{\mathbf{F}}^\ep_{n+2} -  \sum_{z\in \ep3^l\Zd \cap U_{n+3}}
\zeta_z  
\tilde{\mathbf{F}}_{n+2,\overline{\Theta}(z)}\left(\tfrac \cdot\ep\right) \right\|_{L^2(U_{n+2})}
} \qquad & 
\\ & 
\leq 
C \left\| \sum_{z\in \ep3^l\Zd \cap U_{n+3}} \zeta_z \left(  \tilde{\mathbf{F}}^\ep_{n+2} -  
\tilde{\mathbf{F}}_{n+2,\overline{\Theta}(z)}\left(\tfrac \cdot\ep\right) \right) \right\|_{L^2(U_{n+2})}
\\ & 
\leq 
C\sum_{z\in \ep3^l\Zd \cap U_{n+3}}  \left\|  \tilde{\mathbf{F}}^\ep_{n+2} -  
\tilde{\mathbf{F}}_{n+2,\overline{\Theta}(z)}\left(\tfrac \cdot\ep\right) \right\|_{L^2(z+\ep\cu_{l+1})}
\\ & 
\leq 
C 3^{Qk} \left( \ep 3^l \right)^\alpha
\sum_{i=1}^{n+1} \left\| g_i \right\|_{{L}^2(U_i)}^{\frac {n+2}i}.
\end{align*}

\smallskip

\emph{Step 4.} The estimate of the fourth term on the right side of~\eqref{e.bigestwthreetterms}. For convenience, extend $v_z$ so that it is equal to the affine function $\ell_{\nabla \overline{w}_{n+2}(z)}$ outside of the cube~$z+\ep\cu_{l+1}$. 
By the triangle inequality, we have 
\begin{align}
\label{e.termfour}
\lefteqn{
\left\| \sum_{z\in \ep3^l\Zd \cap U_{n+3}}
\nabla \zeta_z \left( \ep v_z(\tfrac \cdot\ep ) - \ell_{\nabla \overline{w}_{n+2}(z)}(x) \right) 
\right\|_{L^2(U_{n+2})}
} \ \qquad & 
\\ & \notag
\leq
\left\|
\sum_{y,z\in \ep3^l\Zd \cap U_{n+3}}
\nabla \zeta_z \left( \ep v_y(\tfrac \cdot\ep ) 
-
\ell_{\nabla \overline{w}_{n+2}(y)} \right)
\right\|_{L^2(U_{n+2})}
\\ & \qquad \notag
+
\left\|\ep 
\sum_{y,z\in \ep3^l\Zd \cap U_{n+3}}\nabla \zeta_z
\left( v_z(\tfrac \cdot\ep ) - v_y(\tfrac \cdot\ep ) \right)
\right\|_{L^2(U_{n+2})}
\\ & \qquad \notag
+
\left\|
\sum_{y,z\in \ep3^l\Zd \cap U_{n+3},\,y\sim z}
\nabla \zeta_z \ep \left| \nabla \overline{w}_{n+2}(z) - \nabla \overline{w}_{n+2}(y) \right|
\right\|_{L^2(U_{n+2})},
\end{align}
where here and below we use the notation $y\sim z$ to mean that $y+\cu_{l+1}$ and $z+\cu_{l+1}$ have nonempty intersection. 
Using~\eqref{e.cutoffzeta} and the fact that $\sum_{z\in \ep3^l\Zd } \nabla \zeta_z = 0$, we have that
\begin{equation} 
\label{e.zetatrim}
\left| \sum_{z\in \ep3^l\Zd \cap U_{n+3}} \nabla \zeta_z \right| 
\leq 
C \left( \ep 3^{l}\right)^{-1} \indc_{V}.
\end{equation}
where 
\begin{equation}
V:= \left\{ x\in \Rd\,:\, (x+\ep\cu_{l+1}) \cap \partial U_{n+3} \neq \emptyset \right\} 
\end{equation}
is a boundary layer of $U_{n+3}$ of thickness $\ep 3^l$. Observe that $|V| \leq C\ep3^l$. 
Thus by~\eqref{e.appcorrbounds} we find that 
\begin{align*}
\lefteqn{
\left\|
\sum_{y,z\in \ep3^l\Zd \cap U_{n+3}}
\nabla \zeta_z \left( \ep v_y(\tfrac \cdot\ep ) 
-
\ell_{\nabla \overline{w}_{n+2}(y)} \right)
\right\|_{L^2(U_{n+2})}
} \qquad & 
\\ & 
\leq 
C \left( \ep 3^l\right)^{-1} 
\ep\left\| 
\indc_{V} \sum_{y\in\ep3^l\Zd \cap U_{n+3}}
 \left( \ep v_y(\tfrac \cdot\ep ) 
-
\ell_{\nabla \overline{w}_{n+2}(y)} \right) \right\|_{L^2(U_{n+2})}
\\ & 
\leq 
C3^{-\alpha(l-k)} 3^{Qk} \sum_{i=1}^{n+1} \left\| g_i \right\|_{{L}^2(U_i)}^{\frac {n+2}i}
.
\end{align*}
Combining~\eqref{e.cutoffzeta} with~\eqref{e.appcorr.overlap},~\eqref{e.detanchors} and the triangle inequality to compare the $v_z$'s on overlapping cubes, 
we find that
\begin{align*}
\lefteqn{
\left\|\ep 
\sum_{y,z\in \ep3^l\Zd \cap U_{n+3}}\nabla \zeta_z
\left( v_z(\tfrac \cdot\ep ) - v_y(\tfrac \cdot\ep ) \right)
\right\|_{L^2(U_{n+2})}
} \quad & 
\\ & 
\leq
\sum_{y,z\in \ep3^l\Zd \cap U_{n+3}, \, y\sim z}
\left\| \nabla \zeta_z \right\|_{L^\infty(\Rd)}
\cdot \ep \left\| v_y - v_z \right\|_{L^2((y+\cu_{l+1})\cap (z+\cu_{l+1}))} 
\\ &
\leq 
C3^{-\alpha(l-k)}3^{Qk} \sum_{i=1}^{n+1} \left\| g_i \right\|_{{L}^2(U_i)}^{\frac {n+2}i}.
\end{align*}
For the last term, we compute, using~\eqref{e.detanchors2},
\begin{align*}
\lefteqn{
\left\|
\sum_{y,z\in \ep3^l\Zd \cap U_{n+3},\,y\sim z}
\nabla \zeta_z \ep \left| \nabla \overline{w}_{n+2}(z) - \nabla \overline{w}_{n+2}(y) \right|
\right\|_{L^2(U_{n+2})}
} \qquad & 
\\ & 
\leq 
C\ep^{1+\beta} 3^{l} \left\| \nabla \overline{w}_{n+2} \right\|_{C^{0,\beta}(U_{n+3})} \left\| \sum_{y,z\in \ep3^l\Zd \cap U_{n+3},\,y\sim z}
\left| \nabla \zeta_z \right|\right\|_{L^2(U_{n+2})}
\\ & 
\leq 
C \ep^\alpha \left( 3^l\ep \right)^{-Q} 
\sum_{i=1}^{n+1} \left\| g_i \right\|_{{L}^2(U_i)}^{\frac {n+2}i}.
\end{align*}
Putting the above inequalities together, we find that 
\begin{align*}
\lefteqn{
\left\| \sum_{z\in \ep3^l\Zd \cap U_{n+3}}
\nabla \zeta_z \left( \ep v_z(\tfrac \cdot\ep ) - \ell_{\nabla \overline{w}_{n+2}(z)}(x) \right) 
\right\|_{L^2(U_{n+2})}
} 
\qquad\qquad\qquad\qquad\qquad
& 
\\ & 
\leq 
C\left(  
3^{-\alpha(l-k)}3^{Qk}
+
\ep^\alpha \left( 3^l\ep \right)^{-Q} 
\right)
\sum_{i=1}^{n+1} \left\| g_i \right\|_{ {L}^2(U_i)}^{\frac {n+2}i}.
\end{align*}

\smallskip

\emph{Step 5.} The estimate of the fifth term on the right side of~\eqref{e.bigestwthreetterms}. Recall that~$U_{n+4}$ is defined in~\eqref{e.defUn4}.
We have that, using~\eqref{e.wn2meyers} and~\eqref{e.detanchors2},
\begin{align*}
\lefteqn{
\left\| \nabla \overline{w}_{n+2} - \sum_{z\in \ep3^l\Zd \cap U_{n+3}} \zeta_z \nabla \overline{w}_{n+2}(z)
\right\|_{L^2(U_{n+2})} 
}  & 
\\ & 
\leq
\left\| \nabla \overline{w}_{n+2} \left( 1 - \sum_{z\in \ep3^l\Zd \cap U_{n+3}} \zeta_z \right)
\right\|_{L^2(U_{n+2})} \!\!\!
+
\left\| \sum_{z\in \ep3^l\Zd \cap U_{n+3}} \zeta_z \left( \nabla \overline{w}_{n+2} -\nabla \overline{w}_{n+2}(z) \right)
\right\|_{L^2(U_{n+2})} 
\\ & 
\leq 
\left\| \nabla \overline{w}_{n+2} 
\right\|_{L^2(U_{n+2}\setminus U_{n+4})} 
+ C \ep 3^{l} \left\| \nabla \overline{w}_{n+2} \right\|_{C^{0,\beta}(U_{n+3})} 
\\ & 
\leq 
C\left| U_{n+2}\setminus U_{n+4} \right|^{\frac{\delta}{4+2\delta}} 
\left\| \nabla \overline{w}_{n+2} 
\right\|_{L^{2+\delta}(U_{n+2})} 
+
C \ep^\alpha \left( \ep 3^{l} \right)^{-Q}\sum_{i=1}^{n+1} \left\| g_i \right\|_{{L}^2(U_i)}^{\frac {n+2} i}
\\ & 
\leq 
C \left( (\ep 3^{l})^{\alpha} + \ep^\alpha \left( \ep 3^{l} \right)^{-Q} \right) \sum_{i=1}^{n+1} \left\| g_i \right\|_{{L}^2(U_i)}^{\frac {n+2}i}.
\end{align*}

\emph{Step 6.} Let us summarize what we have shown so far. Substituting the results of the five previous steps into~\eqref{e.bigestwthreetterms}, we obtain that 
\begin{align}
\label{e.yesstep6}
\lefteqn{
\left\| 
\nabla \cdot \left(  \tilde{\a}^\ep \nabla \tilde{v}_{n+2}^\ep
+ \tilde{\mathbf{F}}^\ep_{n+2} \right) 
\right\|_{H^{-1}(U_{n+2})}
} \qquad & 
\\ & \notag
\leq 
C\left( \ep^\alpha \left( \ep 3^l \right)^{-Q} + 3^{-l\alpha} \left( \ep 3^l \right)^{-1} + 3^{-k\alpha} + \left( \ep 3^l \right)^\alpha 3^{kQ} \right)
\sum_{j=1}^{n+1} \left\|  \nabla g_{j} \right\|_{{L}^{2+\delta}(U_{j})}^{\frac{n+2}{j}}.
\end{align}
This implies by~\eqref{e.captcha} that 
\begin{align}
\label{e.captcha.indeed}
\lefteqn{
\left\| \nabla \tilde{v}^\ep_{n+2} - \nabla \tilde{w}^\ep_{n+2} 
\right\|_{L^2(U_{n+2})}
} \qquad & 
\\ & \notag
\leq 
C\left( \ep^\alpha \left( \ep 3^l \right)^{-Q} + 3^{-l\alpha} \left( \ep 3^l \right)^{-1} + 3^{-k\alpha} + \left( \ep 3^l \right)^\alpha 3^{kQ} \right)
\sum_{j=1}^{n+1} \left\|  \nabla g_{j} \right\|_{{L}^{2+\delta}(U_{j})}^{\frac{n+2}{j}}.
\end{align}
Therefore to obtain the lemma it suffices to prove~\eqref{e.throwyes.wts}. 
To prove this bound, we use the identity 
\begin{align*}
\lefteqn{
\nabla \tilde{v}_{n+2}^\ep(x)
-
\nabla \overline{w}_{n+2}(x)
} \qquad &
\\ &
=
\nabla \left( \sum_{z\in \ep3^l\Zd \cap U_{n+3}}
\zeta_z(x) \left( \ep v\left( \tfrac x\ep , \tfrac{z}\ep + \cu_{l+1},\nabla \overline{w}_{n+2}(z),\overline\Theta(z) \right) - \ell_{\nabla \overline{w}_{n+2}(z)}(x) \right) \right)
\end{align*}
which, combined with~\eqref{e.appcorrbounds}, and~\eqref{e.detanchors}, yields
\begin{align*}
\lefteqn{
\left\| \nabla \tilde{v}_{n+2}^\ep(x)
-
\nabla \overline{w}_{n+2}(x)
\right\|_{H^{-1}(U_{n+2})}
} \ \ & 
\\ & 
\leq 
\left\| 
\sum_{z\in \ep3^l\Zd \cap U_{n+3}}
\zeta_z(x) \left( \ep v\left( \tfrac x\ep , \tfrac{z}\ep + \cu_{l+1},\nabla \overline{w}_{n+2}(z),\overline\Theta(z) \right) - \ell_{\nabla \overline{w}_{n+2}(z)}(x) \right)
\right\|_{L^2(U_{n+2})}
\\ & 
\leq 
C\ep 3^{l}  3^{-\alpha(l-k)} 3^{Qk} \sum_{i=1}^{n+1} \left\| g_i \right\|_{{L}^2(U_i)}^{\frac {n+2}i}
\leq 
C 3^{-\alpha(l-k)} 3^{Qk} \sum_{i=1}^{n+1} \left\| g_i \right\|_{{L}^2(U_i)}^{\frac {n+2}i}.
\end{align*}
This is~\eqref{e.throwyes.wts}. 

\smallskip

In view of the selection of the mesoscopic parameters~$k$ and~$l$ in Section~\ref{ss.meso}, which gives us~\eqref{e.goodchoices}, the proof of~\eqref{e.homogenization.estimates.wts.2} is now complete. 
\end{proof}

\section{Large-scale 
\texorpdfstring{$C^{0,1}$}{C01}-type estimates for linearized equations}
\label{s.reglineqs}

In the next two sections, we suppose that~$n\in \{0,\ldots,\mathsf{N}-1 \}$ is such that   
\begin{equation}
\label{e.assumption.section5}
\left\{
\begin{aligned}
& \  \mbox{Theorems~\ref{t.regularity.Lbar} and~\ref{t.linearizehigher} are valid for $n+1$ in place of $n$,} \\
& \ \mbox{and Theorem~\ref{t.regularity.linerrors} is valid for $n$.}
\end{aligned}\right.
\end{equation}
The goal is to prove that Theorem~\ref{t.regularity.linerrors} is also valid for~ $n+1$. Combined with the results of the previous two sections and an induction argument, this completes the proof of Theorems~\ref{t.regularity.Lbar},~\ref{t.linearizehigher} and~\ref{t.regularity.linerrors}. 

\smallskip

The goal of this section is to prove the half of Theorem~\ref{t.regularity.linerrors}, namely estimate~\eqref{e.C01linsols} for~$w_{n+2}$. The estimate~\eqref{e.C01linerror} will be the focus of Section~\ref{s.reglinerrors}. 

\smallskip

Both of the estimate in the conditional proof of Theorem~\ref{t.regularity.linerrors} are obtained by ``harmonic'' approximation: homogenization says that on large scales the heterogeneous equations behave like the homogenized equations, and therefore we may expect the former to inherit some of the better regularity estimates of the latter. The quantitative version of the homogenization statement provided by Theorem~\ref{t.linearizehigher} allows us to prove a~$C^{0,1}$--type estimate, following a well-known excess decay argument originally introduced in~\cite{AS}. 

\smallskip

\subsection{Approximation of 
\texorpdfstring{$w_{n+2}$}{w n+2}
by smooth functions}

The large-scale regularity estimates are obtained by approximating the solutions of the linearized equations for the heterogeneous Lagrangian~$L$, as well as the linearization errors, by the solutions of the linearized equations for the homogenized Lagrangian~$\overline{L}$. Since the latter possess better smoothness properties, this allows us to implement an excess decay iteration for the former. 

\smallskip

We begin by giving a quantitative statement concerning the smoothness of solutions to the linearized equations for the homogenized operator~$\overline{L}$. This is essentially well-known, but we need a more precise statement than we could find in the literature. 

\smallskip

We next present the statement concerning the approximating of the solutions~$w_m$ of the linearized equations for~$L$, as well as the linearized errors~$\xi_m$, by solutions of the linearized equations for~$\overline{L}$. For the~$w_m$'s, this is essentially a rephrasing of the assumed validity of Theorem~\ref{t.linearizehigher} for $n+1$ in place of~$\mathsf{N}$, see~\eqref{e.assumption.section5}. For the $\xi_m$'s, this can be thought of as a homogenization result.

\begin{lemma}[{Smooth approximation of $w_{n+2}$}]
\label{l.smoothapprox.w}
Assume that~\eqref{e.assumption.section5} holds. Fix $\mathsf{M} \in [1,\infty)$. There exist $\delta(n,d,\Lambda)\in(0,d)$, $\alpha(d,\Lambda) \in \left(0,\tfrac 12\right]$, $C(n,\mathsf{M},d,\Lambda)<\infty$ and a random variable~$\X$ satisfying
\begin{equation*}
\X = \O_{\delta}(C)
\end{equation*}
such that the following statement is valid. 
For every $R\geq \X$ and $u,v,w_1,\ldots,w_{n+2} \in H^1(B_R)$ satisfying, for each $m\in \{1,\ldots,n+2\}$, 
\begin{equation}
\label{e.linearized.approx}
\left\{ 
\begin{aligned}
& -\nabla \cdot  \left( D_pL\left( \nabla u,x \right) \right) = 0 
\quad \mbox{and} \quad -\nabla \cdot \left( D_pL(\nabla v,x) \right) = 0
& \mbox{in} &  \ B_R,\\
& -\nabla \cdot  \left( D^2_pL\left( \nabla u,x \right) \nabla w_m \right) = \nabla \cdot \left( \mathbf{F}_m(x, \nabla u,\nabla w_1,\ldots,\nabla w_{m-1})\right) & \mbox{in} & \ B_{R},\\
&
\left\| \nabla u \right\|_{\underline{L}^2(B_R)} \vee 
\left\| \nabla v \right\|_{\underline{L}^2(B_R)} \leq \mathsf{M},
\end{aligned} 
\right. 
\end{equation}
if we let $\overline{u},\overline{v},\overline{w}_1,\ldots,\overline{w}_{n+2} \in H^1(B_{R/2})$ be the solutions of the Dirichlet problems 
\begin{equation}
\label{e.linearized.approx.hom}
\left\{ 
\begin{aligned}
& -\nabla \cdot  \left( D_p\overline{L}\left( \nabla \overline{u} \right) \right) = 0 
\quad \mbox{and} \quad -\nabla \cdot \left( D_p\overline{L}(\nabla \overline{v}) \right) = 0
& \mbox{in} & \ B_{R},\\
& -\nabla \cdot  \left( D^2_p\overline{L}\left( \nabla \overline{u} \right) \nabla \overline{w}_m \right) = \nabla \cdot \left( \overline{\mathbf{F}}_m(\nabla \overline{u},\nabla \overline{w}_1,\ldots,\nabla \overline{w}_{m-1})\right) & \mbox{in} & \ B_{\frac12(1+2^{-m}) R},\\
& \overline{u} = u, \ \overline{v} = v, \ \overline{w}_m = w_m & \mbox{on} & \ \partial B_{\frac12 (1+2^{-m})R},
\end{aligned} 
\right. 
\end{equation}
then we have, for each $m \in \{1,\ldots,n+2\}$, the estimate
\begin{equation}
\label{e.approx.w}
\left\| w_{m} - \bar{w}_{m} \right\|_{\underline{L}^2(B_{\frac12 (1+2^{-m}) R})} 
\leq 
CR^{-\alpha} 
\sum_{i=1}^{m} \left( \frac1R \left\|  w_{i} - \left( w_{i} \right)_{B_R}   \right\|_{\underline{L}^2(B_R)} \right)^{\frac{m}{i}}.
\end{equation}
\end{lemma}
\begin{proof}
Denote, in short, $R_m := \frac12 (1+2^{-m}) R$ and  $r_m := \frac14 (R_m - R_{m-1})$. Since we assume~\eqref{e.assumption.section5}, 
we have that Theorem~\ref{t.linearizehigher} applied for $n+1$ instead of $\mathsf{N}$ and Theorem~\ref{t.regularity.linerrors}, assumed now for $n$ in place of $\mathsf{N}$, are both valid. In particular, there is $\sigma(n,\data) \in (0,1)$ and $C(n,\mathsf{M},\data)<\infty$, a random variable    
$\tilde \X \leq \O_\sigma(C)$ such that, for $R \geq \tilde \X$ and $m \in \{1,\ldots,n+2\}$, Theorem~\ref{t.regularity.linerrors} gives
\begin{equation} \label{e.C01linsols.applied000}
\sum_{i=1}^{m}  \left\|  \nabla w_i \right\|_{L^{\frac{m}{i} (2+\delta)}(B_{R_i})}^{\frac mi} \leq C
\sum_{i=1}^{m} \left( \frac1R \left\|  w_{i} - \left( w_{i} \right)_{B_R}   \right\|_{\underline{L}^2(B_R)} \right)^{\frac{m}{i}}.
\end{equation}
and Theorem~\ref{t.linearizehigher} yields
\begin{equation}  \label{e.C01linsols.applied0000}
\left\| w_m - \bar{w}_m \right\|_{\underline{L}^2(B_{R_m} )} 
\leq
 \tilde{\X} R^{-\alpha}
\sum_{i=1}^m \left\|  \nabla w_{j} \right\|_{\underline{L}^{2+\delta}(B_{R_i} )} ^{\frac{m}{i}} .
\end{equation}
We set 
\begin{equation*} 
\X := \left( 1 \vee \tilde{\X} \right)^{\frac{2}{\alpha}}.
\end{equation*}
Clearly $\X = \O_{\frac12 \alpha \sigma}(C)$ and $\X \geq \tilde \X$, and if $R \geq \X$, then $ \tilde{\X} R^{-\alpha} \leq R^{-\frac12 \alpha}$. In conclusion, by taking $\alpha$ smaller if necessary, we obtain by~\eqref{e.C01linsols.applied0000} that, for $m \in \{1,\ldots,n+2\}$ and~$R \geq \X$, 
\begin{equation}  \label{e.C01linsols.applied001}
\left\| w_m - \bar{w}_m \right\|_{\underline{L}^2(B_{R_m} )} 
\leq
C R^{-\alpha} 
\sum_{i=1}^m \left\|  \nabla w_{i} \right\|_{\underline{L}^{2+\delta}(B_{R_i})}^{\frac{m}{i}}. 
\end{equation}
Furthermore,  we notice that~\eqref{e.Fmbasic} yields
\begin{equation*} 
\left| \mathbf{F}_{m}(x, \nabla u,\nabla w_1,\ldots,\nabla w_{n+1}) \right| \leq C \sum_{i=1}^{n+1} \left|\nabla w_i \right|^{\frac{m}{i}},
\end{equation*}
and thus we get by~\eqref{e.C01linsols.applied000} and~\eqref{e.C01linsols.applied001} that 
\begin{equation*} 
\left\| \mathbf{F}_{m}(x, \nabla u,\nabla w_1,\ldots,\nabla w_{m-1}) \right\|_{\underline{L}^{2+\delta}(B_{R_{m-1}})} \leq C
\sum_{i=1}^{m-1} \left( \frac1R \left\|  w_{i} - \left( w_{i} \right)_{B_R}   \right\|_{\underline{L}^2(B_R)}\right)^{\frac {m}{i}}.
\end{equation*}
It then follows by the Meyers estimate and equation of $w_m$ that 
\begin{equation*} 
\left\| \nabla w_{m}  \right\|_{\underline{L}^{2+\delta}(B_{R_m} )}  \leq C \sum_{i=1}^{m} \left( \frac1R \left\|  w_{i} - \left( w_{i} \right)_{B_R}   \right\|_{\underline{L}^2(B_R)}\right)^{\frac {m}{i}},
\end{equation*}
finishing the proof of~\eqref{e.approx.w} by~\eqref{e.C01linsols.applied001}. 
\end{proof}

\subsection{Excess decay iteration for \texorpdfstring{$w_{n+2}$}{w n+2}}

In this subsection we conditionally prove the statement of Theorem~\ref{t.regularity.linerrors} for $n+1$ and for $q=2$. The restriction on $q$ will be removed in the next subsection. 
The proof is by a decay of excess iteration, following along similar lines as the argument from~\cite{AS}, using ``harmonic'' approximation. The statement we prove is summarized in the following proposition.

\begin{proposition}
\label{p.wLip}
Assume that~\eqref{e.assumption.section5} holds. Fix $\mathsf{M} \in [1,\infty)$. Then there exist constants $\sigma(\data),\alpha(d,\Lambda) \in \left(0,\frac12 \right]$, $C(\mathsf{M},\data)<\infty$ and a random variable~$\X$ satisfying
\begin{equation*}
\X \leq \O_{\sigma}(C)
\end{equation*}
such that the following statement is valid. Let~$R \in [\X , \infty)$ and
$u,w_1,\ldots,w_{n}$ satisfy, for each $m\in \{1,\ldots,n+1\}$, 
\begin{equation}
\label{e.linearized.c1alpha}
\left\{ 
\begin{aligned}
& -\nabla \cdot  \left( D_pL\left( \nabla u,x \right) \right) = 0 & \mbox{in} & \ B_R, \\
& -\nabla \cdot  \left( D^2_pL\left( \nabla u,x \right) \nabla w_m \right) = \nabla \cdot \left( \mathbf{F}_m(\nabla u,\nabla w_1,\ldots,\nabla w_{m-1})\right) & \mbox{in} & \ B_{R},
\end{aligned} 
\right. 
\end{equation}
where $u$ satisfy the normalization 
\begin{equation*} 
\frac1R \left\|  u - (u)_{B_R} \right\|_{\underline{L}^2\left( B_{R} \right)} \leq \mathsf{M}  . 
\end{equation*}
Then, for $m \in \{1,\ldots,n+2\}$ and for every $r \in [\X,\tfrac12 R]$, we have
\begin{multline} \label{e.Lipw}
 \left( \frac{r}{R} + \frac1{r} \right)^{-\alpha} \inf_{\ell \in \mathcal{P}_1} \frac{1}{r} \left\|  w_{m} - \ell   \right\|_{\underline{L}^2(B_{r})}  
 + \left\| \nabla w_{m}   \right\|_{\underline{L}^2(B_r)} 
 \\ \leq C \sum_{i=1}^{m} \left( \frac1R \left\|  w_{i} - \left( w_{i} \right)_{B_R}   \right\|_{\underline{L}^2(B_R)}\right)^{\frac {m}{i}} .
\end{multline}
\end{proposition}

\begin{remark} \label{r.w1reg}
Proposition~\ref{p.wLip} holds for $n=0$ by~\cite[Proposition 4.3]{AFK} without assuming~\eqref{e.assumption.section5}.
\end{remark}

\begin{proof}
The proof is based on combination of harmonic approximation and decay estimates for homogenized solutions presented in Appendix~\ref{s.appendixconstant}. The necessary estimate is~\eqref{e.appC.C1alphabarw}. We take the minimal scale $\X$ as the maximum of the minimal scale in Lemma~\ref{l.smoothapprox.w} and in Theorem~\ref{t.regularity.linerrors}, which is valid with $n$ in place of $\mathsf{N}$, corresponding $q  =2 $, and a constant $\mathsf{R}(\mathsf{M},\data) \in [1,\infty)$ to be fixed in the course of the proof.  This choice, in particular, implies that there exist constants $C(\mathsf{M},\data)<\infty$ and $\sigma(\data) \in \left(0 ,\frac12 \right]$ such that $\X \leq \O_{\sigma}(C)$ and $\X \geq \mathsf{R}$.
\smallskip

We will prove the statement using induction in shrinking radii. Indeed, we set, for $j \in \N$ and $\theta \in (0,\frac12]$, $r_j := \theta^{j} r_0 $, where $r_0  \in [\X, \delta R]$ and $\delta \in \left(0 , \tfrac12 \right]$. Parameters~$\theta$,~$\delta$ and~$\mathsf{R}$ will all be fixed in the course of the proof. Having fixed~$\theta$,~$\delta$ and~$\mathsf{R}$, we assume that there is $J \in \N$ such that  $r_{J+1} <  \X \leq r_J$ for some $J \in \N$.  If there is no such $J$ or $\X \geq \delta R$, the result will follow simply by giving up a volume factor.  Furthermore, we device the notation of this proof in such a way that it will also allow us to prove the result of the next lemma, Lemma~\ref{l.C1alphanew}. 

\smallskip

We denote, in short, for $m \in \{1,\ldots,n+2\}$ and $\gamma \in [0,1)$ to be fixed, 
\begin{align} \label{e.wLipDm}
\mathsf{D}_{m} 
&
:=   \frac{1}{r_0} \left\| w_{m} - \left( w_{m} \right)_{B_{r_0}}     \right\|_{\underline{L}^2(B_{r_0})} 
+ \sum_{i=1}^{m-1} \left( \frac1{R} \left\|  w_{i} - \left( w_{i} \right)_{B_{R}}   \right\|_{\underline{L}^2(B_{R})}\right)^{\frac{m}i} 
\end{align}
and
\begin{align} \label{e.wLipEm}
\mathsf{E}_{m} 
&
:=   
 \inf_{\ell \in \mathcal{P}_1} \frac{1}{r_0} \left\|  w_{m} - \ell   \right\|_{\underline{L}^2(B_{r_0})} 
 +\left( \frac{r_0}{R} + \frac1{r_0} \right)^{\gamma}  \mathsf{D}_{m}  .
\end{align}
Theorem~\ref{t.regularity.linerrors} implies that, for $r \in [\X , \tfrac12 R]$ and $m \in \{1,\ldots,n+1\}$, we have that 
\begin{align}
\label{e.C01linsols.applied1}
\left\| \nabla w_{m}   \right\|_{\underline{L}^2(B_r)}
& 
\leq 
C \sum_{i=1}^{m} \left( \frac1R \left\|  w_{i} - \left( w_{i} \right)_{B_R}   \right\|_{\underline{L}^2(B_R)}\right)^{\frac mi} .
\end{align}
Notice then that, by~\eqref{e.C01linsols.applied1} and Poincar\'e's inequality, we have, for $r \in [\X , \tfrac12 R]$ and $m\in \{1,\ldots,n+2\}$,  that 
\begin{equation} \label{e.wLipEmest}
 \sum_{i=1}^{m-1} \left( \frac1r \left\|  w_{i} - \left( w_{i} \right)_{B_r}   \right\|_{\underline{L}^2(B_r)}\right)^{\frac mi} \leq C\mathsf{D}_{m} .
\end{equation}

\smallskip

\emph{Step 1.} 
Letting $\overline{u}_j$ solve
\begin{equation*}
\left\{ 
\begin{aligned}
& -\nabla \cdot  \left( D_p\overline{L}\left( \nabla \overline{u}_j \right) \right) = 0 & \mbox{in} & \ B_{2r_j}, \\
&  \overline{u}_j = u & \mbox{on} & \ \partial B_{2r_j},
\end{aligned} 
\right. 
\end{equation*}
we show that there exist, for $\eta \in (0,1]$, constants $ \alpha_1(d,\Lambda) \in (0,\frac12)$, $\delta(\eta,\mathsf{M},\data)<\infty$ and~$\mathsf{R}(\eta,\mathsf{M},\data)<\infty$ such that, for $j \in \{0,\ldots,J\}$, 
\begin{equation}  \label{e.witerbarureg}
\frac1{2r_j} \inf_{\ell \in \mathcal{P}_1}\left\| u -\ell \right\|_{\underline{L}^2 \left( B_{2r_j} \right)}
+
\frac1{2r_j} \inf_{\ell \in \mathcal{P}_1} \left\| \overline{u}_j -\ell \right\|_{\underline{L}^2 \left( B_{2r_j} \right)} \leq \eta \left( \frac{r_j}{R} +  \frac{1}{r_j}\right)^{\alpha_1} .
\end{equation}
The parameter $\eta$ shall be fixed later in~\eqref{e.fixetainw}.  On the one hand, we have from~\cite[Theorem 2.3]{AFK} that there exist constants~$\beta_1(d,\Lambda) \in (0,1)$ and $C(\mathsf{M},d,\Lambda)<\infty$ such that 
\begin{equation*} 
\frac1{2r_0}  \inf_{\ell \in \mathcal{P}_1}\left\| u -\ell \right\|_{\underline{L}^2 \left( B_{2r_0} \right)} \leq C \left( \frac{r_0}{R} +  \frac{1}{r_0}\right)^{\beta_1} .
\end{equation*}
On the other hand, by harmonic approximation (\cite[Corollary 2.2]{AFK}) and Lipschitz estimate for $u$  (\cite[Theorem 2.3]{AFK})  we get that there exist constants~$\beta_2(d,\Lambda) \in (0,1)$ and $C(\mathsf{M},d,\Lambda)<\infty$ such that
\begin{equation*} 
\left\| u - \overline{u}_j  \right\|_{\underline{L}^2 \left( B_{r_j} \right)} \leq C r_j^{-\beta_2} .
\end{equation*}
Thus~\eqref{e.witerbarureg} follows by the triangle inequality by taking $\alpha_1 := \tfrac12 (\beta_1 \wedge \beta_2)$, and choosing $\delta$ small enough and $\mathsf{R}$ large enough so that 
\begin{equation}  \label{e.wLipetaR1}
C \left( \delta +  \frac{1}{\mathsf{R}}\right)^{\frac12 \beta_1} + C \mathsf{R}^{-\frac12 \beta_2}  \leq \eta .
\end{equation}
We assume, from now on,  that $\delta$ and $\mathsf{R}$ are such that~\eqref{e.wLipetaR1} is valid. 

\smallskip

\emph{Step 2.} 
Letting $j \in \{0,\ldots,J\}$ and $m\in \{1,\ldots,n+2\}$,  and $\bar{w}_{1,j},\ldots,\bar{w}_{m,j}$ to solve, with $\overline{u}_j$ as in Step 2, equations
\begin{equation*}
\left\{ 
\begin{aligned}
& -\nabla \cdot  \left( D^2_p\overline{L}\left( \nabla \overline{u}_j\right) \nabla  \overline{w}_{m,j} \right) = \nabla \cdot \overline{\mathbf{F}}_{m}\left(\nabla \bar{u}_j , \bar{w}_{1,j} , \ldots, \bar{w}_{m-1,j} \right) & \mbox{in} & \ B_{\frac12 (1+2^{-m})r_{j}}
, \\
& \bar{w}_{m,j} = w_m & \mbox{on} & \ \partial B_{\frac12 (1+2^{-m})r_{j}},
\end{aligned} 
\right. 
\end{equation*}
we show that then there is $\alpha_2(\data) \in \left(0,\tfrac 12\right]$ such that 
\begin{align}  \label{e.wLipcomp}
 \frac1{r_j} \left\| w_{m} -  \bar{w}_{m,j} \right\|_{L^2 \left(B_{  r_j/2} \right)}
&
\leq 
 Cr_j^{-\alpha_2}  \frac1{r_j} \left\| w_{m} - (w_{m})_{B_{r_j}}   \right\|_{\underline{L}^{2}(B_{r_j})}  
 \\ \notag & \quad 
 +
 Cr_j^{-\alpha_2}  \sum_{i=1}^{m-1} \left( \frac1{R} \left\|  w_{i} - \left( w_{i} \right)_{B_{R}}   \right\|_{\underline{L}^2(B_{R})}\right)^{\frac{m}i} .
\end{align}
This, however,  is a direct consequence of~\eqref{e.C01linsols.applied1} and Lemma~\ref{l.smoothapprox.w}. 

\smallskip

\emph{Step 3}. Induction assumption. Set $\alpha := \frac \beta2 (\alpha_1 \wedge \alpha_2)$, where~$\alpha_1$ and~$\alpha_2$ are as in Steps 1 and 2, respectively,  and $\beta$ comes from the $C^{n+2,\beta}$ regularity of $\overline{L}$. Let~$\delta_j$ be defined as
\begin{equation} \label{e.wLipdeltaj}
\delta_j :=  \left( \frac {r_j}{r_0}  + \frac 1{r_j^{d-\sigma}} \right)^{\alpha} .
\end{equation}
We assume inductively that, for $j^* \in \{1,\ldots,J\}$, $j \in \{0,\ldots,j^*\}$, and $m \in \{1,\ldots,n+2\}$, we have, for a constant $\mathsf{C} \in [1,\infty)$ to be fixed in Step 5, that
\begin{equation} \label{e.wLipinductionass}
 \inf_{\ell \in \mathcal{P}_1} \frac1{r_j}\left\|  w_{m} - \ell   \right\|_{\underline{L}^2(B_{r_j})} \leq \delta_j  \mathsf{E}_m 
 \quad \mbox{and} \quad  \frac1{r_j} \left\|  w_{m} - (w_{m})_{B_{r_j}}   \right\|_{\underline{L}^2(B_{r_j})} \leq \mathsf{C}  \mathsf{D}_{m}.
\end{equation}
Here~$\mathsf{D}_m $ and~$\mathsf{E}_m $ are defined in~\eqref{e.wLipDm} and~\eqref{e.wLipEm}, respectively. 
Obviously ~\eqref{e.wLipinductionass} is valid for $j=0$ by the definitions of~$\mathsf{D}_m $ and~$\mathsf{E}_m $. Fixing 
\begin{equation}  \label{e.fixetainw}
\eta := (1 \vee \mathsf{C})^{1/\beta},
\end{equation}
we have that~\eqref{e.witerbarureg} implies 
\begin{equation}  \label{e.witerbarureg2}
\mathsf{C}  \left( \frac1{r_j} \inf_{\ell \in \mathcal{P}_1}  \left\|  \overline{u}_j  - \ell \right\|_{\underline{L}^{2} \left( B_{r_j} \right)}  \right)^\beta \leq  \delta_{j} \left( \frac{r_0}{R} + \frac1{r_0} \right)^{\alpha} .
\end{equation}
Using also~\eqref{e.wLipcomp} and~\eqref{e.wLipinductionass}, we obtain, for $m \in \{1,\ldots,n+1\}$, 
\begin{equation}  \label{e.wLipcompwstupid}
 \frac1{r_j} \left\| w_{m} -  \bar{w}_{m,j} \right\|_{L^2 \left(B_{  r_j/2} \right)}   \leq  C \mathsf{C} r_j^{-\alpha_2 } \mathsf{D}_m 
 \leq 
 \frac12 \theta^{\frac d2 +1}\delta_{j+1} \mathsf{E}_{m}
\end{equation}
provided that 
\begin{equation} \label{e.wLipetaR2}
 C \mathsf{C} \theta^{-\frac d2 - 2}  \mathsf{R}^{- \frac12  \alpha_2} \leq \frac12 .
\end{equation}
We assume, from now on, that $\mathsf{R}$ is such that both~\eqref{e.wLipetaR1} and~\eqref{e.wLipetaR2} are valid. 

\smallskip

\emph{Step 4}.  
We show that the first inequality in~\eqref{e.wLipinductionass} continues to hold for $j = j^*+1$.  
First, by the triangle inequality,~\eqref{e.wLipcompwstupid} and~\eqref{e.wLipinductionass}, we have that
\begin{align*}  
\lefteqn{
 \frac1{\delta_j r_j} \inf_{\ell \in \mathcal{P}_1} \left\|  \bar{w}_{m,j} - \ell \right\|_{L^2 \left(B_{  r_j/2} \right)}  
} \quad &
\\ \notag &
\leq 
  \frac1{ \delta_j r_j} \inf_{\ell \in \mathcal{P}_1} \left\| w_{m} - \ell \right\|_{L^2 \left(B_{  r_j/2} \right)}  + 
   \frac1{ \delta_j r_j} \left\| w_{m} -  \bar{w}_{m,j} \right\|_{L^2 \left(B_{  r_j/2} \right)} 
\\ \notag &
\leq 
2 \mathsf{E}_{m} 
\end{align*}
and
\begin{equation*} 
\frac1{ \delta_{j+1} r_{j+1}}\inf_{\ell \in \mathcal{P}_1 } \left\|  w_{m} - \ell  \right\|_{\underline{L}^2 \left( B_{r_{j+1}} \right)}
\leq 
\frac1{ \delta_{j+1}   r_{j+1}}\inf_{\ell \in \mathcal{P}_1 } \left\|  \overline{w}_{m,j} - \ell  \right\|_{\underline{L}^2 \left( B_{r_{j+1}} \right)}
+  \frac12  \mathsf{E}_{m} .
\end{equation*}
By a similar computation, using also the induction assumption~\eqref{e.wLipinductionass}, 
\begin{equation*} 
\sum_{i=1}^{m} \left(  \frac1{r_j} \left\|  \overline{w}_{i} - \left( \overline{w}_{i} \right)_{B_{r_j/2}}   \right\|_{\underline{L}^2(B_{r_j/2})}\right)^{\frac{m}{i}} 
\leq C \mathsf{C} \mathsf{D}_{m} .
\end{equation*}
Now, applying~\eqref{e.appC.C1alphabarw} we obtain by the previous three displays,~\eqref{e.witerbarureg2} and~\eqref{e.wLipEmest}, for each $m \in \{1,\ldots,n+2\}$, that
\begin{align}
\label{e.wLipbarwiter}
\lefteqn{
\left(  \frac{r_j}{r_{j+1}} \right)^{\beta}  \frac1{ r_{j+1}}\inf_{\ell \in \mathcal{P}_1 } \left\|  \overline{w}_{m,j} - \ell  \right\|_{\underline{L}^2 \left( B_{r_{j+1}} \right)}
} \quad &
\\ \notag &
\leq 
C \delta_j \sum_{i=1}^{m} \left( \frac1{\delta_j r_j} \inf_{\ell \in \mathcal{P}_1} \left\|  \overline{w}_{i,j}  - \ell \right\|_{\underline{L}^2(B_{r_j/2})} \right)^{\frac{m}{i}}
\\ \notag & \quad 
+
C \delta_j \sum_{i=1}^{m-1} \left(  \frac1{r_j} \left\|  \overline{w}_{i} - \left( \overline{w}_{i} \right)_{B_{r_j/2}}   \right\|_{\underline{L}^2(B_{r_j/2})} \right)^{\frac{m}{i}}
\\ \notag & \quad
+  C \left( \frac1{ r_j} \inf_{\ell \in \mathcal{P}_1} \left\|  \bar{u}_j - \ell \right\|_{\underline{L}^2(B_{r_j/2})}\right)^\beta 
\sum_{i=1}^{m} \left( \frac1{r_j} \left\|  \overline{w}_{i} - \left( \overline{w}_{i} \right)_{B_{r_j/2}}   \right\|_{\underline{L}^2(B_{r_j/2})} \right)^{\frac{m}{i}} 
 \\ \notag &
\leq    C \delta_j \mathsf{E}_{m}.
\end{align}
and, consequently,
\begin{equation*} 
  \frac1{\delta_{j+1} r_{j+1}}\inf_{\ell \in \mathcal{P}_1 } \left\|  w_{m} - \ell  \right\|_{\underline{L}^2 \left( B_{r_{j+1}} \right)}
  \leq \frac12  \mathsf{E}_{m} + 
  C \left( \frac{\delta_j}{\delta_{j+1}}  \right) \left(  \frac{r_{j+1}}{r_j} \right)^{\beta} \mathsf{E}_{m} \leq \left(\frac12 +  C \theta^{\frac \beta2} \right) \mathsf{E}_{m} .
\end{equation*}
Thus, choosing $\theta$ small enough so that  $C \theta^{\frac \beta2} \leq \frac12$, we obtain that the first inequality in~\eqref{e.wLipinductionass} is valid for $j= j^*+1$.  

\smallskip

\emph{Step 5}. The last step in the proof is to show the second inequality in~\eqref{e.wLipinductionass}  for $j = j^*+1$. Let $\ell_j$ be the minimizing affine function in $ \inf_{\ell \in \mathcal{P}_1} \left\|  w_{m} - \ell   \right\|_{\underline{L}^2(B_{r_j})}$. Then, by the triangle inequality and the first inequality in~\eqref{e.wLipinductionass} valid for $j \in \{0,\ldots,j^*\}$,
\begin{equation*} 
\left| \nabla \ell_{j+1} - \nabla  \ell_j  \right|  
\leq C  (\delta_{j+1} + \delta_j)  \mathsf{E}_m 
\end{equation*}
Thus, summation gives that 
\begin{equation*} 
|\nabla \ell_{j^*+1} - \nabla \ell_0 | \leq  C \mathsf{E}_m   \sum_{j=0}^{j^* +1} \delta_j 
.
\end{equation*}
Therefore,
 \begin{align*} 
\frac{1}{r_{j^*+1}} \left\| w_{m} - (w_{m})_{B_{r_{j^*+1}}} \right\|_{\underline{L}^2 \left( B_{r_{j^*+1}} \right)} 
 & 
\leq 
\frac{1}{r_{j^*+1}} \left\| w_{m} - \ell_{j^*+1} \right\|_{\underline{L}^2 \left( B_{r_{j^*+1}} \right)} + |\nabla \ell_{j^*+1} | 
\\ &
\leq  
 C \mathsf{E}_m   \sum_{j=0}^{j^* +1} \delta_j   + |\nabla \ell_{0} | .
\end{align*}
By the triangle inequality we have that 
\begin{equation*} 
|\nabla \ell_{0} | \leq  \frac{8}{r_{0}} \left\| w_{m} - \ell_{0} \right\|_{\underline{L}^2 \left( B_{r_{0}} \right)} + 
 \frac{8}{r_{0}} \left\| w_{m} - (w_m)_{B_{r_0}} \right\|_{\underline{L}^2 \left( B_{r_{0}} \right)} \leq 8 \mathsf{E}_m + 8 \mathsf{D}_m 
\end{equation*}
and hence 
\begin{equation*} 
\frac{1}{r_{j^*+1}} \left\| w_{m} - (w_{m})_{B_{r_{j^*+1}}} \right\|_{\underline{L}^2 \left( B_{r_{j^*+1}} \right)}  \leq C  \mathsf{D}_m   \sum_{j=0}^{j^* +1} \delta_j .
\end{equation*}
Choosing $\mathsf{C} = C$, where $C$ is as in the above inequality, verifies the second inequality in~\eqref{e.wLipinductionass}  for $j = j^*+1$. This finishes the proof of the induction step, and thus completes the proof.
\end{proof}

To show Lipschitz estimates for the linearization errors in the next section, we need a small variant of the previous proposition. 

\begin{lemma} \label{l.C1alphanew}
Assume that~\eqref{e.assumption.section5} holds. Fix $\mathsf{M} \in [1,\infty)$. Then there exist constants $\alpha(\data),\sigma(n,\mathsf{M},\data), \theta(n,\mathsf{M},\data) \in \left( 0, \frac12  \right]$ and $C(\sigma,\mathsf{M},\data)<\infty$, and a random variable~$\X$ satisfying
\begin{equation*}
\X \leq \O_{\sigma}(C)
\end{equation*}
such that the following statement is valid. Let~$R \in [\X , \infty)$ and
$u,w_1,\ldots,w_{n}$ satisfy, for each $m\in \{1,\ldots,n+1\}$, 
\begin{equation}
\label{e.linearized.c1alphanew}
\left\{ 
\begin{aligned}
& -\nabla \cdot  \left( D_pL\left( \nabla u,x \right) \right) = 0 & \mbox{in} & \ B_R, \\
& -\nabla \cdot  \left( D^2_pL\left( \nabla u,x \right) \nabla w_m \right) = \nabla \cdot \left( \mathbf{F}_m(\nabla u,\nabla w_1,\ldots,\nabla w_{m-1})\right) & \mbox{in} & \ B_{R},
\end{aligned} 
\right. 
\end{equation}
and, for $r \in [\X , \frac12 R]$, 
\begin{equation*} 
 -\nabla \cdot  \left( D^2_pL\left( \nabla u,x \right) \nabla w_{n+2} \right) = \nabla \cdot \left( \mathbf{F}_m(\nabla u,\nabla w_1,\ldots,\nabla w_{n+1})\right) \quad \mbox{in}  \ B_{r},
\end{equation*}
where $u$ satisfy the normalization
\begin{equation*} 
\frac1R \left\|  u - (u)_{B_R} \right\|_{\underline{L}^2\left( B_{R} \right)} \leq \mathsf{M}.
\end{equation*}
Then 
\begin{align} 
\label{e.Lipw.pre}
\lefteqn{
 \inf_{\ell \in \mathcal{P}_1} \frac 1{\theta r}\left\|  w_{n+2} - \ell   \right\|_{\underline{L}^2(B_{\theta r})}
} \qquad &
\\ & \notag
\leq
   \frac12  \inf_{\ell \in \mathcal{P}_1} \frac 1r \left\|  w_{n+2} - \ell   \right\|_{\underline{L}^2(B_r)} 
 + C    \left(  \frac{r}{R} + \frac1{r} \right)^{\alpha} \frac{1}{r} \left\|  w_{n+2} - \left( w_{n+2} \right)_{B_r}   \right\|_{\underline{L}^2(B_r)} 
 \\ \notag & \qquad 
 +  C\left(  \frac{r}{R} + \frac1{r}  \right)^{\alpha }  \ \sum_{i=1}^{n+1} \left( \frac1R \left\|  w_{i} - \left( w_{i} \right)_{B_R}   \right\|_{\underline{L}^2(B_R)}\right)^{\frac {n+2}{i}} .
\end{align}
\end{lemma}

\begin{proof}
The proof is a rearrangement of the argument in the proof of Proposition~\ref{p.wLip}. Indeed, we take $r_0=r$ and combine the first inequality of~\eqref{e.wLipbarwiter} with~\eqref{e.wLipcomp} and~\eqref{e.Lipw}, which is valid for $m\in \{1,\ldots,n+1\}$.
\end{proof}
 
\subsection{Improvement of spatial integrability} \label{s.spatintw} 

We next complete the conditional proof of~\eqref{e.C01linsols} by improving the spatial integrability of~\eqref{e.Lipw.pre} from $L^2$ to $L^q$ for every $q\in [2,\infty)$. To do this, we use the estimate~\eqref{e.Lipw.pre} to pass from the large scale~$R$ to the microscopic, random scale~$\X$. We then use deterministic estimates from classical elliptic regularity theory to obtain local $L^q$ estimates in balls of radius one, picking up a volume factor---which is power of~$\X$---as a price to pay. The first formalize the latter step in the following lemma. 

\begin{lemma}
\label{l.toqandbeyond}
Assume~\eqref{e.assumption.section5} and the hypotheses of Theorem~\ref{t.regularity.linerrors}. Let $\beta\in (0,1)$ and $q\in (2,\infty)$. Then there exist~$\sigma(q,\data) \in (0,d)$, $C(q,\mathsf{M},\data)<\infty$ and a random variable~$\X$ satisfying~$\X = \O_\sigma(C)$ such that, for every $r\geq \X$ and $m\in\{1,\ldots,n+2\}$,
\begin{equation}
\label{e.smallscalezz}
\left\| \nabla w_m \right\|_{\underline{L}^q(B_{r/2})}
\leq
C \left( 1 + r^{\frac{d^2(q-2)}{4q\beta} + \frac{d(q-2)}{2q}} \right) 
\sum_{i=1}^{m} 
\left( \frac1R \left\|  w_{i} - \left( w_{i} \right)_{B_R}   \right\|_{\underline{L}^2(B_R)} 
\right)^{\frac{m}i}.
\end{equation}
\end{lemma}
\begin{proof}
In view of the assumption of~\eqref{e.assumption.section5} and thus the validity of Theorem~\ref{t.regularity.linerrors} for~$n$, we only need to prove~\eqref{e.smallscalezz} for $m=n+2$. 
Fix~$q\in (2,\infty)$, $r\in [2,\infty)$, $\beta\in (0,1)$ to be selected below. The $C^{1,\beta}$-estimate in Proposition~\ref{p.schauder}, together with a covering argument, yields 
\begin{equation*}
\left\| \nabla u \right\|_{L^\infty(B_r)}
+
\left[ \nabla u \right]_{C^{0,\beta}(B_r)}
\leq 
Cr^{\frac d2} \left\| \nabla u \right\|_{\underline{L}^2(B_{2r})}.
\end{equation*}
Setting $\a:=D^2_pL(\nabla u,x)$, we deduce by the assumption of~(L1) that
\begin{equation*}
\left[ \a \right]_{C^{0,\beta}(B_r)} 
\leq
C \left( 1 + \left[ \nabla u \right]_{C^{0,\beta}(B_r)} \right)
\leq 
C \left(1 + r^{\frac d2} \left\| \nabla u \right\|_{\underline{L}^2(B_{2r})} \right). 
\end{equation*}
Applying Proposition~\ref{p.gradientLq} we find that, for each~$x\in B_{r/2}$ and $q\in (2,\infty]$,
\begin{align*}
\lefteqn{
\left\| \nabla w_{n+2} \right\|_{L^q(B_1(x))}
} \quad & 
\\ &
\leq
C \left[ \a \right]_{C^{0,\beta}(B_r)}^{\frac{d(q-2)}{2q\beta}} \left\| \nabla w_{n+2} \right\|_{L^2(B_2(x))} 
+
C\left\| \mathbf{f}_m\left( \nabla u,\nabla w_1,\ldots,\nabla w_{n+1}\right) \right\|_{L^q(B_2(x))}
\\ & 
\leq 
C \left( 1 + r^{\frac{d^2(q-2)}{4q\beta}}  \left\| \nabla u \right\|_{\underline{L}^2(B_{2r})} \right)
\left\| \nabla w_{n+2} \right\|_{L^2(B_2(x))} 
\\ & \qquad
+
C\left\| \mathbf{f}_{n+2}\left( \nabla u,\nabla w_1,\ldots,\nabla w_{n+1}\right) \right\|_{L^q(B_2(x))}.
\end{align*}
By a covering argument, we therefore obtain
\begin{align}
\label{e.readyforqbeezes}
\left\| \nabla w_{n+2} \right\|_{L^q(B_{r/2})}
& \leq
C \left( 1 + r^{\frac{d^2(q-2)}{4q\beta}}  \left\| \nabla u \right\|_{\underline{L}^2(B_{2r})} \right)
\left\| \nabla w_{n+2} \right\|_{L^2(B_r)} 
\\ & \qquad\notag
+
C\left\| \mathbf{f}_{n+2}\left( \nabla u,\nabla w_1,\ldots,\nabla w_{n+1}\right) \right\|_{L^q(B_r)}.
\end{align}
If we now take~$\X$ to be the maximum of the minimal scales in the statements of:
\begin{enumerate}
\item[(1)]
\cite[Theorem 11.13]{AKMbook};
\item[(2)] Theorem~\ref{t.regularity.linerrors} for~$n$ in place of $\mathsf{N}$ and with a sufficiently large exponent of spatial integrability~$q'$ in place of~$q$ (which can be computed explicitly in terms of our~$q$ using the H\"older inequality, although we omit this computation)---the validity of which is given by assumption~\eqref{e.assumption.section5};
\item[(3)] Proposition~\ref{p.wLip};
\end{enumerate}
then we have that $\X = \O_\sigma(C)$ as stated in the lemma and that~$r \geq \X$ implies the following estimates:
\begin{equation*}
\left\| \nabla u \right\|_{\underline{L}^2(B_{2r})}  \leq C,
\end{equation*}
\begin{equation*}
\left\| \mathbf{f}_{n+2}\left( \nabla u,\nabla w_1,\ldots,\nabla w_{n+1}\right) \right\|_{\underline{L}^q(B_r)}
\leq 
C\sum_{i=1}^{n+1} 
\left( \frac1R \left\|  w_{i} - \left( w_{i} \right)_{B_R}   \right\|_{\underline{L}^2(B_R)} 
\right)^{\frac{n+2}i}
\end{equation*}
and
\begin{equation*} \label{}
\left\| \nabla w_{n+2} \right\|_{\underline{L}^2(B_r)} 
\leq 
C\sum_{i=1}^{n+2} 
\left( \frac1R \left\|  w_{i} - \left( w_{i} \right)_{B_R}   \right\|_{\underline{L}^2(B_R)} 
\right)^{\frac{n+2}i}.
\end{equation*}
Combining these with~\eqref{e.readyforqbeezes}, we obtain
\begin{equation*}
\left\| \nabla w_{n+2} \right\|_{\underline{L}^q(B_{r/2})}
\leq
C \left( 1 + r^{\frac{d^2(q-2)}{4q\beta} + \frac{d(q-2)}{2q}} \right) 
\sum_{i=1}^{n+2} 
\left( \frac1R \left\|  w_{i} - \left( w_{i} \right)_{B_R}   \right\|_{\underline{L}^2(B_R)} 
\right)^{\frac{n+2}i}.
\end{equation*}
This completes the proof of the lemma. 
\end{proof}

In the next lemma we finally achieve the goal of this section, which is to show that~\eqref{e.assumption.section5} implies~\eqref{e.C01linsols} for $m=n+1$. 

\begin{lemma} \label{l.toqandbeyond2}
Assume that~\eqref{e.assumption.section5} holds. Fix $\mathsf{M} \in [1,\infty)$ and $q\in (2,\infty)$. Then there exist constants $\sigma(q,\data) \in (0,d)$, $C(q,\mathsf{M},\data)<\infty$ and a random variable~$\X$ satisfying
\begin{equation*}
\X \leq \O_{\sigma}(C)
\end{equation*}
such that the following is valid. Suppose that~$R \in [\X , \infty)$ and~$u,w_1,\ldots,w_{n} \in H^1(B_R)$ such that, for every $m\in \{1,\ldots,n+2\}$, 
\begin{equation}
\label{e.linearized.2toq}
\left\{ 
\begin{aligned}
& \frac1R \left\|  u - (u)_{B_R} \right\|_{\underline{L}^2\left( B_{R} \right)} \leq \mathsf{M},
\\
& -\nabla \cdot  \left( D_pL\left( \nabla u,x \right) \right) = 0 & \mbox{in} & \ B_R, \\
& -\nabla \cdot  \left( D^2_pL\left( \nabla u,x \right) \nabla w_m \right) = \nabla \cdot \left( \mathbf{F}_m(\nabla u,\nabla w_1,\ldots,\nabla w_{m-1})\right) & \mbox{in} & \ B_{R},
\end{aligned} 
\right. 
\end{equation}
Then, for every $r \in [\X,\tfrac12 R]$, we have
\begin{equation} 
\label{e.Lipw.q}
\left\| \nabla w_{n+2}   \right\|_{\underline{L}^q(B_r)} 
\leq C \sum_{i=1}^{{n+2}} \left( \frac1R \left\|  w_{i} - \left( w_{i} \right)_{B_R}   \right\|_{\underline{L}^2(B_R)}\right)^{\frac {{n+2}}{i}} .
\end{equation}
\end{lemma}
\begin{proof}
Fix $q\in (2,\infty)$. 
Select a parameter~$\theta\in (0,1)$ which will denote a mesoscopic scale. 
For each $z\in\Rd$, we take $\X_z$ to be the random variable $\X$ in the statement of Proposition~\ref{p.wLip}, centered at the point~$z$. Define another random variable (a minimal scale) by 
\begin{equation*} \label{}
\mathcal{Y} 
:=
\sup \left\{ 3^k \,:\, k\in\N, \, \sup_{z\in \Zd\cap B_{3^k}} \X_z \geq 3^{k \theta} \right\}.
\end{equation*}
It is clear from Proposition~\ref{p.wLip} and an routine union bound argument that 
\begin{equation*} \label{}
\Y \leq \O_\sigma(C). 
\end{equation*}
Next, for every $k\in\N$ and $z\in \Rd$ we let $\mathcal{Z}_{k,z}$ denote the random variable 
\begin{equation*} \label{}
\mathcal{Z}_{k,z}:=
\sup_{(u,w_1,\ldots,w_{n+2})}
\sup_{m\in\{1,\ldots,n+2\}}
\frac{\left\| \nabla w_m \right\|_{\underline{L}^q(z+\cu_k)}}{\sum_{i=1}^{m} 
\left( 3^{-k} \left\|  w_{i} - \left( w_{i} \right)_{z+\cu_{k+1}}   \right\|_{\underline{L}^2(z+\cu_{k+1})} 
\right)^{\frac{m}i}},
\end{equation*}
where the supremum is taken over all $(u,w_1,\ldots,w_{n+2})\in \left( H^1(z+\cu_{k+1})\right)^{n+3}$  satisfying, for every $m\in \{1,\ldots,n+2\}$, 
\begin{equation} 
\label{e.linsreffss}
\left\{
\begin{aligned}
& \left\| \nabla u \right\|_{\underline{L}^2(z+\cu_{k+1})} \leq \mathsf{M}, 
\\ & 
-\nabla \cdot D_pL\left( \nabla u,x \right) = 0 &  \mbox{in} & \ z+\cu_{k+1},
\\ & 
-\nabla \cdot  \left( D^2_pL\left( \nabla u,x \right) \nabla w_m \right) = \nabla \cdot \left( \mathbf{F}_m(\nabla u,\nabla w_1,\ldots,\nabla w_{m-1},x)\right) &  \mbox{in} & \ z+\cu_{k+1}.
\end{aligned}
\right.
\end{equation}
Observe that $\mathcal{Z}_{k,z}$ is $\mathcal{F}(z+\cu_{k+1})$--measurable and, by Lemma~\ref{l.toqandbeyond} and an easy covering argument, it satisfies the estimate
\begin{equation} 
\label{e.Zkzest}
\mathcal{Z}_{k,z} \leq \O_\sigma(C). 
\end{equation}
Fix $\mathsf{A}\in [1,\infty)$ and define another random variable (a minimal scale) $\mathcal{Z}$ by 
\begin{equation*} \label{}
\mathcal{Z} :=
\sup\left\{ 
3^k \,:\,
\left( \left| 3^{\lceil \theta k \rceil}\Zd \cap \cu_{k+1} \right|^{-1} \!\!\!\!
\sum_{z\in 3^{\lceil \theta k \rceil}\Zd \cap \cu_{k+1} }
\mathcal{Z}_{\lceil \theta k \rceil,z}^q \right)^{\frac1q}
\geq \mathsf{A}
\right\}.
\end{equation*}
We will show below that, if~$\mathsf{A}$ is chosen sufficiently large (depending of course on the appropriate parameters) then 
\begin{equation} 
\label{e.claimedZest}
\mathcal{Z} \leq \O_\sigma(C). 
\end{equation}
Assuming that~\eqref{e.claimedZest} holds for the moment, let us finish the proof of the lemma. Suppose now that $k\in\N$ satisfies $\mathcal{Y} \vee \mathcal{Z} \leq 3^k \leq 3^{k+1}\leq R$. Let $(u,w_1,\ldots,w_{n+2})\in \left( H^1(B_R)\right)^{n+3}$ satisfy~\eqref{e.linearized.2toq}. Then we have that 
\begin{align*} \label{}
\lefteqn{
\left\| \nabla w_{n+2} \right\|_{\underline{L}^q(\cu_k)} 
}  & 
\\ &
=
\left(\left| 3^{\lceil \theta k \rceil}\Zd \cap \cu_k \right|^{-1} \!\!\!\!
\sum_{z\in 3^{\lceil \theta k \rceil}\Zd \cap \cu_k}
\left\| \nabla w_{n+2} \right\|_{\underline{L}^q\left(z+\cu_{\lceil \theta k\rceil}\right)}^q 
\right)^{\frac1q} 
\\ & 
\leq 
\left(\left| 3^{\lceil \theta k \rceil}\Zd \cap \cu_k \right|^{-1} \!\!\!\!
\sum_{z\in 3^{\lceil \theta k \rceil}\Zd \cap \cu_k}
\sum_{i=1}^{{n+2}} 
\left( 3^{-\theta k} \left\|  w_{i} - \left( w_{i} \right)_{z+\cu_{\lceil \theta k\rceil +1}}   \right\|_{\underline{L}^2(z+\cu_{\lceil \theta k\rceil +1})} 
\right)^{\frac{(n+2)q}i}
\mathcal{Z}_{\lceil \theta k \rceil,z}^q 
\right)^{\frac1q} 
\\ & 
\leq 
C\sum_{i=1}^{{n+2}} 
\left(3^{-k} \left\|  w_{i} - \left( w_{i} \right)_{z+\cu_{k +1}}   \right\|_{\underline{L}^2(\cu_{k +1})} 
\right)^{\frac{n+2}i}
\left(\left| 3^{\lceil \theta k \rceil}\Zd \cap \cu_k \right|^{-1}\!\!\!\!
\sum_{z\in 3^{\lceil \theta k \rceil}\Zd \cap \cu_k}
\mathcal{Z}_{\lceil \theta k \rceil,z}^q 
\right)^{\frac1q} 
\\ & 
\leq 
C\mathsf{A}\sum_{i=1}^{{n+2}} 
\left( 3^{-k} \left\|  w_{i} - \left( w_{i} \right)_{z+\cu_{k +1}}   \right\|_{\underline{L}^2(\cu_{k +1})} 
\right)^{\frac{{n+2}}i}
\\ & 
\leq 
C\mathsf{A}\sum_{i=1}^{{n+2}} 
\left( \frac1R \left\|  w_{i} - \left( w_{i} \right)_{z+\cu_{k +1}}   \right\|_{\underline{L}^2(B_R)} 
\right)^{\frac{{n+2}}i}.
\end{align*}
Note that in the third and final lines we used that $3^k\geq \Y$, that is, we used the result of Proposition~\ref{p.wLip}. This is the desired estimate for $\X = \Y \vee \mathcal{Z}$, and so the proof of the lemma is complete subject to the verification of~\eqref{e.claimedZest}. 

\smallskip

Turning to the proof of~\eqref{e.claimedZest}, we notice first that it suffices to show, for~$\mathsf{A}$ sufficiently large, the existence of~$\sigma(q,\data)>0$ and $C(q,\mathsf{M},\data)<\infty$ such that, for every~$k\in\N$,  
\begin{equation} 
\label{e.claimedZest.red}
\P\left[ \left| 3^{\lceil \theta k \rceil}\Zd \cap \cu_{k+1} \right|^{-1} \!\!\!\!
\sum_{z\in 3^{\lceil \theta k \rceil}\Zd \cap \cu_{k+1} }
\mathcal{Z}_{\lceil \theta k \rceil,z}^q
\geq \mathsf{A}^q  \right]
\leq 
C \exp\left( -c3^{k\sigma } \right).
\end{equation}
Indeed, we can see that~\eqref{e.claimedZest.red} implies~\eqref{e.claimedZest} using a simple union bound. 
Fix then a parameter~$\lambda\in [1,\infty)$ and compute, using~\eqref{e.Zkzest} and the Chebyshev inequality,
\begin{equation} 
\label{e.Zest1}
\P \left[  \sup_{z\in 3^{\lceil \theta k \rceil}\Zd \cap \cu_{k+1} } \mathcal{Z}_{\lceil \theta k \rceil \Zd\cap \cu_{k+1}}
> \lambda
\right]
\leq 
\exp\left( -c \lambda^{\sigma} \right). 
\end{equation}
Using the simple large deviations bound for sums of bounded, independent random variables, we have
\begin{align}
\label{e.Zest2}
&
\P \left[ 
 \left| 3^{\lceil \theta k \rceil}\Zd \cap \cu_{k+1} \right|^{-1} \!\!\!\!
\sum_{z\in 3^{\lceil \theta k \rceil}\Zd \cap \cu_{k+1} }
\left( \mathcal{Z}_{\lceil \theta k \rceil,z}^q \wedge \lambda \right)
\geq \E \left[ \mathcal{Z}_{\lceil \theta k \rceil,z}^q \wedge \lambda \right]  + 1
\right]
\\ & \qquad\qquad\qquad \notag
\leq 3^d \exp\left( -c\lambda^{-2} \left| 3^{\lceil \theta k \rceil}\Zd \cap \cu_{k+1} \right| \right)
\leq 
3^d \exp\left( -c3^{d(1-\theta)k} \lambda^{-2} \right).
\end{align}
Here we are careful to notice that, while the collection~$\{ \mathcal{Z}_{\lceil \theta k \rceil,z}^q \wedge \lambda : z\in 3^{\lceil \theta k \rceil}\Zd \cap \cu_{k+1} \}$ of random variables is not independent (since adjacent cubes are touching and thus not separated by a unit distance), we can partition this collection into $3^d$ many subcollections which have an equal number of elements and each of which is independent. The large deviations estimate can then be applied to each subcollection, and then a union bound yields~\eqref{e.Zest2}. Combining~\eqref{e.Zest1},~\eqref{e.Zest2} and the observation that $\E \left[ \mathcal{Z}_{\lceil \theta k \rceil,z}^q \wedge \lambda \right] \leq \E \left[ \mathcal{Z}_{\lceil \theta k \rceil,z}^q\right] \leq C$ by~\eqref{e.Zkzest}, we obtain 
\begin{equation*}
\P\left[ \left| 3^{\lceil \theta k \rceil}\Zd \cap \cu_{k+1} \right|^{-1} \!\!\!\!
\sum_{z\in 3^{\lceil \theta k \rceil}\Zd \cap \cu_{k+1} }
\mathcal{Z}_{\lceil \theta k \rceil,z}^q
\geq  
C + 1
\right]
\leq 
C \exp\left( -c \left( \lambda^\sigma \wedge 3^{d(1-\theta)k} \lambda^{-2} \right) \right).
\end{equation*}
Taking $\lambda := 3^{\frac d4(1-\theta)k}$ and $\mathsf{A}^q:=C+1$ yields~\eqref{e.claimedZest.red}. 
\end{proof}

The above proof, in view of Remark~\ref{r.w1reg}, gives the following result without assuming~\eqref{e.assumption.section5}. This, together with~\eqref{p.Lipxi0} below, serves as the base case for the induction.  

\begin{proposition} \label{p.w1higher}
Let $q\in [2,\infty)$, $\mathsf{M}\in [1,\infty)$. Then there exist~$\sigma(q,\data)>0$ and~$C(q,\mathsf{M},\data) <\infty$ and a random variable~$\X$ satisfying~$\X \leq \O_{\sigma}\left( C \right)$ such that the following statement is valid. 
For $R \in \left[ 2\X,\infty\right)$ and $u,w_1 \in H^1(B_R)$ satisfying $\left\| \nabla u \right\|_{\underline{L}^2(B_R)} \leq \mathsf{M}$ and
\begin{equation*}
\left\{
\begin{aligned}
&
-\nabla \cdot \left( D_pL(\nabla u,x) \right) = 0 \quad \mbox{in} \ B_R,\\
& 
-\nabla \cdot  \left( D^2_pL\left( \nabla u,x \right) \nabla w_1 \right) = 0 \quad \mbox{in}  \ B_R,
\end{aligned}
\right.
\end{equation*}
we have, for all $r\in [\X ,\tfrac12 R]$, 
\begin{equation} \label{e.w1higher}
\left\| \nabla w_1 \right\|_{\underline{L}^q \left( B_{r} \right)} \leq \frac{C}{R}  \left\| w_1  -(w_1)_{B_R} \right\|_{\underline{L}^2 \left( B_{r} \right)}.
\end{equation}
\end{proposition}

\section{Large-scale 
\texorpdfstring{$C^{0,1}$}{{C01}}-type estimate for linearization errors}
\label{s.reglinerrors}

In this section we continue to suppose that~$n\in \{0,\ldots,\mathsf{N}-1 \}$ is such that~\eqref{e.assumption.section5} holds. 
The goal is to complete the proof that Theorem~\ref{t.regularity.linerrors} is also valid for~ $n+1$. Combined with the results of the previous three sections and an induction argument, this completes the proof of Theorems~\ref{t.regularity.Lbar},~\ref{t.linearizehigher} and~\ref{t.regularity.linerrors}.

\subsection{Excess decay iteration for $\xi_{n+1}$}

We start by proving higher integrability for a difference of two solutions. This, together with Proposition~\ref{p.w1higher}, yields the base case for the induction. 

\begin{proposition}  
\label{p.Lipxi0}
Fix $\mathsf{M} \in [1,\infty)$ and $q \in [2,\infty)$. 
There exist $\alpha(\data),\sigma(q,\data)  \in \left(0,\tfrac 12\right]$, $C(q,\mathsf{M},\data)<\infty$ and a random variable~$\X$ satisfying
\begin{equation*}
\X = \O_{\sigma}(C)
\end{equation*}
such that the following statement is valid. 
For every $R\geq \X$ and $u,v \in H^1(B_R)$ satisfying, for each $m\in \{1,\ldots,n+1\}$, 
\begin{equation}
\label{e.linearized.approx.L.0}
\left\{ 
\begin{aligned}
& -\nabla \cdot  \left( D_pL\left( \nabla u,x \right) \right) = 0 
\quad \mbox{and} \quad -\nabla \cdot \left( D_pL(\nabla v,x) \right) = 0
& \mbox{in} &  \ B_R,\\
&
\left\| \nabla u \right\|_{\underline{L}^2(B_R)} \vee 
\left\| \nabla v \right\|_{\underline{L}^2(B_R)} \leq \mathsf{M},
\end{aligned} 
\right. 
\end{equation}
Then, for $m\in \{1,\ldots,n+1\}$, $\xi_0 = u-v$ and $r \in [\X , \tfrac12 R]$, we have
\begin{equation} \label{e.Lipxi0}
  \left\|  \nabla \xi_{0}     \right\|_{\underline{L}^{q}(B_r)}  
  + \left( \frac{r}{R} + \frac1{r} \right)^{-\alpha} \inf_{\ell \in \mathcal{P}_1} \frac{1}{r} \left\|  \xi_{0} - \ell   \right\|_{\underline{L}^2(B_{r})}   
 \leq 
C \frac{1}{R}\left\| \xi_{0} -  \left( \xi_{0} \right)_{B_{R}} \right\|_{\underline{L}^2(B_R)} .
\end{equation}
\end{proposition}

\begin{proof}
On the one hand, the estimate
\begin{equation} \label{e.Lipxi0pre}
 \left\|  \nabla \xi_{0}     \right\|_{\underline{L}^{2}(B_r)}  +  \left( \frac{r}{R} + \frac1{r} \right)^{-\alpha} \inf_{\ell \in \mathcal{P}_1} \frac{1}{r} \left\|  \xi_{0} - \ell   \right\|_{\underline{L}^2(B_{r})}  \leq 
C \frac{1}{R}\left\| \xi_{0} -  \left( \xi_{0} \right)_{B_{R}} \right\|_{\underline{L}^2(B_R)} 
\end{equation}
follows by~\cite[Proposition 4.2]{AFK}. On the other hand, the proof of 
\begin{equation}  \label{e.Lipxi0pre2}
 \left\|  \nabla \xi_{0}     \right\|_{\underline{L}^{q}(B_r)} 
 \leq 
C \frac{1}{R}\left\| \xi_{0} -  \left( \xi_{0} \right)_{B_{R}} \right\|_{\underline{L}^2(B_R)} 
\end{equation}
is similar to the proof of $L^q$-integrability of $w_1$ presented in Section~\ref{s.spatintw}.  Indeed, noticing that $\xi_0$ satisfies the equation 
\begin{equation*} 
-\nabla \cdot \tilde \a  \nabla \xi_0 = 0, \qquad \tilde \a(x) := \int_0^1 D_p^2F(t \nabla u(x) + (1-t) \nabla v(x),x) \, dt,
\end{equation*}
we have by the normalization in~\eqref{e.linearized.approx.L.0} and $C^{1,\alpha}$ regularity of $u$ and $v$ that we may replace $w_1$ with $\xi_0$ in Lemma~\ref{l.toqandbeyond} applied with $m=1$. Using this together with the Lipschitz estimate~\eqref{e.Lipxi0pre} for  $\xi_0$ as in the proof of Lemma~\ref{l.toqandbeyond2}, concludes the proof of~\eqref{e.Lipxi0pre2}. We omit further details.
\end{proof}

\begin{proposition} \label{p.Lipxi}
Assume that~\eqref{e.assumption.section5} holds. Fix $\mathsf{M} \in [1,\infty)$. There exist constants $\alpha(n,\data),\sigma(n,\data)  \in \left(0,\tfrac 12\right]$, $C(n,\mathsf{M},\data)<\infty$ and a random variable~$\X$ satisfying
\begin{equation*}
\X = \O_{\sigma}(C)
\end{equation*}
such that the following statement is valid. 
For every $R\geq \X$ and $u,v,w_1,\ldots,w_{n+1} \in H^1(B_R)$ solving, for each $m\in \{1,\ldots,n+1\}$, 
\begin{equation}
\label{e.linearized.approx.L}
\left\{ 
\begin{aligned}
& -\nabla \cdot  \left( D_pL\left( \nabla u,x \right) \right) = 0 
\quad \mbox{and} \quad -\nabla \cdot \left( D_pL(\nabla v,x) \right) = 0
& \mbox{in} &  \ B_R,\\
& -\nabla \cdot  \left( D^2_pL\left( \nabla u,x \right) \nabla w_m \right) = \nabla \cdot \left( \mathbf{F}_m(x, \nabla u,\nabla w_1,\ldots,\nabla w_{m-1})\right) & \mbox{in} & \ B_{R},\\
&
\left\| \nabla u \right\|_{\underline{L}^2(B_R)} \vee 
\left\| \nabla v \right\|_{\underline{L}^2(B_R)} \leq \mathsf{M},
\end{aligned} 
\right. 
\end{equation}
we have, for $m\in \{1,\ldots,n+1\}$ and $r \in [\X , \tfrac12 R]$, the estimate
\begin{align} 
\label{e.Lipxi}
\lefteqn{ \left( \frac{r}{R} + \frac1{r} \right)^{-\alpha} \inf_{\ell \in \mathcal{P}_1} \frac{1}{r} \left\|  \nabla \xi_{m} - \ell   \right\|_{\underline{L}^2(B_{r})}  
 +  \left\|  \nabla \xi_{m}     \right\|_{\underline{L}^{2}(B_r)} 
 } \qquad & 
\\ & \leq \notag
C \sum_{i=0}^{m} 
\left(\frac{1}{R}\left\| \xi_{i} -  \left( \xi_{i} \right)_{B_{R}} \right\|_{\underline{L}^2(B_R)}^{\frac{1}{i+1}} +  \frac1R \left\|  w_{i} - \left( w_{i} \right)_{B_R}   \right\|_{\underline{L}^2(B_R)}^{\frac 1i}  \right)^{m+1} .
\end{align}
\end{proposition}
\begin{proof}
We start by fixing some notation. Let $q(n,d,\Lambda)$ be as in Lemma~\ref{l.lineq}, applied for $n+1$ instead of $n$. Corresponding this $q$, choose $\X $ to be the maximum of minimum scales appearing in Proposition~\ref{p.Lipxi0}, Theorem~\ref{t.regularity.linerrors} and  Lemma~\ref{l.C1alphanew}, of which last two are valid for $n+1$ in place of $\mathsf{N}$ by~\eqref{e.assumption.section5}. We also assume that $\X \geq \mathsf{R}$, by taking $\X \vee \mathsf{R}$ instead of $\X$, where $\mathsf{R}$ will be fixed in the course of the proof to depend on parameters $(n,\mathsf{M},\data)$. Chosen this way,~$\X$ satisfies $\X \leq \O_\sigma(C)$ for some constants $\sigma(n,\data)>0$ and $C(n,\mathsf{M},\data)<\infty$.

\smallskip
Let $r_j = \theta^{j} \eta R$, where $\theta$ is as in Lemma~\ref{l.C1alphanew} and $\eta \in \left(0,\frac12 \right]$. The constant $\eta$, as well as $\mathsf{R}$, will be fixed in the course of the proof, so that $\eta$ is small and $\mathsf{R}$ is large. We track down the dependencies on $\eta$ and  $\mathsf{R}$ carefully and, in particular, constants denoted by $C$ below do not depend on them. We may assume, without loss of generality, that $\eta R \geq \X$.

Set, for $k \in \{0,1,\ldots,n\}$, 
\begin{align} \label{e.xiregEk}
\mathsf{E}_{k} & 
:= \frac { 1}{2 r_0}\left\| \xi_{k+1} -  \left( \xi_{k+1} \right)_{B_{2 r_0}}  \right\|_{\underline{L}^2(B_{2 r_0})}
\\ & \notag 
\qquad 
+
 \sum_{i=0}^{k} \left(  \frac{1}{R}\left\| \xi_{i} -  \left( \xi_{i} \right)_{B_{R}} \right\|_{\underline{L}^2(B_{R})} 
+
\frac1R \left\|  w_{i+1} - \left( w_{i+1} \right)_{B_R}   \right\|_{\underline{L}^2(B_R)}  \right)^{\frac{k+2}{i+1}} .
\end{align}
We denote
\begin{equation*} 
R_k := \frac12 \left(1 + 2^{-k} \right)R .
\end{equation*}

\smallskip

\emph{Step 1.} Induction on degree. We assume that, for fixed $k \in \{0,\ldots,n\}$, we have, for every $m \in \{0,\ldots,k\}$ and every 
$r \in [\X ,  R_m]$, 
\begin{multline} \label{e.Lipxiind}
 \left( \frac{r}{R} + \frac1{r} \right)^{-\alpha} \inf_{\ell \in \mathcal{P}_1} \frac{1}{r} \left\|  \nabla \xi_{m} - \ell   \right\|_{\underline{L}^2(B_{r})}  
 +  \left\|  \nabla \xi_{m}     \right\|_{\underline{L}^{2}(B_r)} 
 \\ \leq 
C \sum_{i=0}^{m} 
\left(\frac{1}{R}\left\| \xi_{i} -  \left( \xi_{i} \right)_{B_{R}} \right\|_{\underline{L}^2(B_R)}  \right)^{\frac{m+1}{i+1}}  + \sum_{i=1}^{m}  \left( \frac1R \left\|  w_{i} - \left( w_{i} \right)_{B_R}   \right\|_{\underline{L}^2(B_R)}  \right)^{\frac{m+1}{i}} .
\end{multline}
Notice that if $k=0$, then~\eqref{e.Lipxiind} follows by Proposition~\ref{p.Lipxi0}.

\smallskip

\emph{Step 2.} Cacciopppoli estimate for $\xi_{k+1}$. 
We show that, for all $r \in [\X, R_{k}]$, 
\begin{equation}  
\label{e.ximplus1bnd}
\left\| \nabla \xi_{k+1} \right\|_{\underline{L}^{2} \left( B_{(1-2^{-k-2})r} \right)} 
\leq 
\frac{C}{r} \left\|  \xi_{k+1}  - ( \xi_{k+1})_{B_{r}} \right\|_{\underline{L}^{2} \left( B_{r} \right)}
+
  C \mathsf{E}_{k}.
\end{equation}
We first have by~\eqref{e.EnL2} that
\begin{align*} 
\lefteqn{
\left\| \nabla \xi_{k+1} \right\|_{\underline{L}^{2} \left( B_{(1-2^{-k-3})r} \right)} 
} \qquad & 
\\ &
\leq 
 C  \sum_{i=1}^{k} \left( \frac1{r}\left\|  \xi_{i} - (\xi_{i})_{B_{r}}\right\|_{\underline{L}^{2} \left( B_{r} \right)} \right)^{\frac{k+2}{i+1}} 
+
C \left\| \nabla \xi_{0} \right\|_{\underline{L}^{q}(B_{r})}^{k+2}
 +
C \sum_{i=1}^{k+1}  \left\| \nabla \frac{w_i}{i!} \right\|_{\underline{L}^{q}(B_{r})}^{\frac{k+2}{i}}  .
\end{align*}
By Theorem~\ref{t.regularity.linerrors} and the choice of $q$  in the beginning of the proof,  we obtain, for every $r \in \left[ \X , \tfrac 12 R \right]$ and $m\in\{1,\ldots,k+1\}$, the estimates
\begin{align} \label{e.ximplus1bndw}
\left\| \nabla w_{m}   \right\|_{\underline{L}^q(B_r)}^{\frac{1}{m}}
& 
\leq 
C\sum_{i=1}^{m} 
\left( \frac1R \left\|  w_{i} - \left( w_{i} \right)_{B_R}   \right\|_{\underline{L}^2(B_R)} 
\right)^{\frac{1}i} 
\end{align}
and, by Proposition~\ref{p.Lipxi0},
\begin{equation*} 
\left\| \nabla \xi_0   \right\|_{\underline{L}^q(B_r)} \leq  \frac CR \left\|  \xi_{0} - \left( \xi_{0} \right)_{B_R}   \right\|_{\underline{L}^2(B_R)} .
\end{equation*}
Combining above displays yields~\eqref{e.ximplus1bnd} in view of the induction assumption, i.e., that~\eqref{e.Lipxiind} holds for $m \in \{0,\ldots,k\}$.

\smallskip

\emph{Step 3.} 
We prove that there is a constant $C<\infty$ independent of $\eta$ and $\mathsf{R}$ such that, for $j \in \N_0$ such that $r_j \geq 2 (\X \vee \mathsf{R})$, we have
\begin{align} \label{e.Lnplusonedecay.case2res}
\lefteqn{
 \inf_{\ell \in \mathcal{P}_1} \frac 1{r_{j+1}}\left\|  \xi_{k+1} - \ell   \right\|_{\underline{L}^2(B_{r_{j+1}})}
} \quad & \\ \notag 
&
\leq
   \frac12  \inf_{\ell \in \mathcal{P}_1} \frac 1{r_j} \left\|  \xi_{k+1} - \ell   \right\|_{\underline{L}^2(B_{r_j})} 
 + C\frac{\ep_j  }{r_j} \left\|  \xi_{k+1} - \left( \xi_{k+1} \right)_{B_{r_j}}   \right\|_{\underline{L}^2(B_{r_j})} 
 +  C  \ep_j^{\frac{1}{k+1}}    \mathsf{E}_{k} ,
\end{align}
where 
\begin{equation}  \label{e.Lnplusonedecay.epr}
\ep_j := \frac12 \left( \frac{r_j}{R} + \frac1{r_j} \right)^{\frac{\alpha}{2}},
\end{equation}
and $\alpha(\data)$ is the minimum of parameter $\alpha$ appearing in statements of Proposition~\ref{p.Lipxi0} and Lemma~\ref{l.C1alphanew}. 

\smallskip

Let us fix $j \in \N_0$ such that $r_j \geq 2 (\X \vee \mathsf{R})$. We argue in two cases, namely, either~\eqref{e.Lnplusonedecay.case1}
or~\eqref{e.Lnplusonedecay.case2} below is valid. We prove that in both cases~\eqref{e.Lnplusonedecay.case2res} follows.

\smallskip

\emph{Step 3.1.} We first assume that 
\begin{equation}
\label{e.Lnplusonedecay.case1}
\mathsf{E}_{k} 
\\ > \ep_j
\left( 
\frac1R \left\| \xi_0  - ( \xi_0  )_{B_R} \right\|_{\underline{L}^2(B_R)} 
+ \sum_{i=1}^{k+1} \left( \frac1R  \left\| w_i - (w_i)_{B_R} \right\|_{\underline{L}^2(B_R)} \right)^{\frac {k+1}i} 
\right),
\end{equation}
where $\ep_j $ has been fixed in~\eqref{e.Lnplusonedecay.epr}. We show that this implies
\begin{equation}  \label{e.Lnplusonedecay.case1res}
\inf_{\ell \in \mathcal{P}_1 }\frac1r \left\| \xi_{k+1}  - \ell \right\|_{\underline{L}^2 \left( B_{r} \right)} \leq C \ep_j\mathsf{E}_{k} .
\end{equation}
Notice that this gives~\eqref{e.Lnplusonedecay.case2res}. To show~\eqref{e.Lnplusonedecay.case1res}, 
we have by the triangle inequality that
\begin{align} \notag 
\inf_{\ell \in \mathcal{P}_1 }\frac1{r_j} \left\| \xi_{k+1}  - \ell \right\|_{\underline{L}^2 \left( B_{r_j} \right)}
&
\leq
\inf_{\ell \in \mathcal{P}_1 }\frac1{r_j} \left\| \xi_{0}  - \ell \right\|_{\underline{L}^2 \left( B_{r_j} \right)}
+ 
\sum_{i=1}^{k+1} \inf_{\ell \in \mathcal{P}_1 }  \frac1{r_j}  \left\| w_i - \ell  \right\|_{\underline{L}^2(B_{r_j})} .
\end{align}
By the choice of $\alpha$, we get  by Proposition~\ref{p.Lipxi0} that 
\begin{equation*} 
\inf_{\ell \in \mathcal{P}_1 }\frac1{r_j} \left\| \xi_{0}  - \ell \right\|_{\underline{L}^2 \left( B_{r_j} \right)} 
\leq 
C \left( \frac{r_j}{R} + \frac1{r_j} \right)^{\alpha}  \frac1R \left\| \xi_0  - ( \xi_0  )_{B_R} \right\|_{\underline{L}^2(B_R)} 
\end{equation*}
and, by Proposition~\ref{p.wLip}, 
\begin{equation*} 
\inf_{\ell \in \mathcal{P}_1 }  \frac1{r_j}  \left\| w_i - \ell  \right\|_{\underline{L}^2(B_r)} 
\leq 
C \left( \frac{r_j}{R} + \frac1{r_j} \right)^{\alpha}  \sum_{h=1}^i \left( \frac1R  \left\| w_h - (w_h)_{B_R} \right\|_{\underline{L}^2(B_R)} \right)^{\frac ih} .
\end{equation*}
Combining the estimates and using~\eqref{e.Lnplusonedecay.case1}, we have that 
\begin{equation*} 
 \inf_{\ell \in \mathcal{P}_1 }\frac1r \left\| \xi_{k+1}  - \ell \right\|_{\underline{L}^2 \left( B_{r} \right)} \leq C \left( \frac{r_j}{R} + \frac1{r_j} \right)^{\alpha}  \ep_j^{-1}  \mathsf{E}_{k}  . 
\end{equation*}
We then obtain~\eqref{e.Lnplusonedecay.case1res} by the choice of $\ep_j$ in~\eqref{e.Lnplusonedecay.epr}, provided that~\eqref{e.Lnplusonedecay.case1} is valid.

\smallskip

\emph{Step 3.2.} We then assume that~\eqref{e.Lnplusonedecay.case1} is not the case, that is, 
\begin{equation}
\label{e.Lnplusonedecay.case2}
\mathsf{E}_{k} 
\leq \ep_j 
\left( 
\frac1R \left\| \xi_0  - ( \xi_0  )_{B_R} \right\|_{\underline{L}^2(B_R)} 
+ \sum_{i=1}^{k+1} \left( \frac1R  \left\| w_i - (w_i)_{B_R} \right\|_{\underline{L}^2(B_R)} \right)^{\frac {k+1}{i}} 
\right)
\end{equation}
holds with~$\ep_j$ defined in~\eqref{e.Lnplusonedecay.epr}. We validate~\eqref{e.Lnplusonedecay.case2res} also in this case. To this end, let us first observe an immediate consequence of~\eqref{e.Lnplusonedecay.case2}. By Young's inequality, we have
\begin{equation*} 
 \ep_j \sum_{i=1}^{k+1} \left( \frac1R  \left\| w_i - (w_i)_{B_R} \right\|_{\underline{L}^2(B_R)} \right)^{\frac{k+1}{i}}  
\leq 
C  \ep_j^{\frac{k+2}{k+1}}
+ \frac14 \sum_{i=0}^{k} \left( \frac1R  \left\| w_{i+1} - (w_{i+1})_{B_R} \right\|_{\underline{L}^2(B_R)} \right)^{\frac {k+2}{i+1}} 
\end{equation*}
and
\begin{equation*} 
\ep_j   \frac1R \left\| \xi_0  - ( \xi_0  )_{B_R} \right\|_{\underline{L}^2(B_R)} \leq  4 \ep_j^{\frac{k+2}{k+1}} +  
\frac14 \left(\frac1R \left\| \xi_0  - ( \xi_0  )_{B_R} \right\|_{\underline{L}^2(B_R)}  \right)^{k+2} .
\end{equation*}
Hence we get, by the definition of $\mathsf{E}_{k}$ in~\eqref{e.xiregEk} and reabsorption,  that 
\begin{equation}  \label{e.Eksmallincasetwo}
\mathsf{E}_{k}  \leq C  \ep_j^{\frac{k+2}{k+1}}. 
\end{equation}
Let then ${w}_{m+2,j}$ solve
\begin{equation*}
\left\{ 
\begin{aligned}
& -\nabla \cdot  \left( D^2_pL\left( \nabla u ,x\right) \nabla {w}_{m+2,j} \right) = \nabla \cdot \left({\mathbf{F}}_{m+2}(\nabla {u},\nabla {w}_1,\ldots,\nabla {w}_{m+1}, \cdot )\right) & \mbox{in} & \ B_{r_j},\\
& {w}_{m+2,j} = \xi_{m+1}  & \mbox{on} & \ \partial B_{r_j}.
\end{aligned} 
\right. 
\end{equation*}
It follows from Lemma~\ref{l.lineq}, together with~\eqref{e.Lipxi0},~\eqref{e.ximplus1bndw} and~\eqref{e.Lipxiind}, assumed inductively for $m \in \{1,\ldots,k\}$, that 
\begin{align} \notag 
\lefteqn{\left\| \nabla \cdot \left( D^2 L(\nabla u,\cdot) \nabla\left( \xi_{k+1} - w_{k+2,j} \right) \right) \right\|_{\underline{H}^{-1}(B_{r_j})} } \qquad &
\\ \notag &
\leq 
C\sum_{i=1}^{k+1} 
\left( \left\|  \nabla \xi_{i-1}  \right\|_{\underline{L}^{2+\delta}(B_{r_j})}^{\frac{1}{i}}   
+ 
\left\|  \nabla \xi_{0}   \right\|_{\underline{L}^{\frac{8}{\delta}(k+3)}(B_{r_j})} 
+
\left\|  \nabla w_{i}   \right\|_{\underline{L}^{\frac{8}{\delta}(k+3)}(B_{r_j})}^{\frac{1}{i}} 
\right)^{k+3}
\\ \notag &
\leq 
C \mathsf{E}_{k}^{\frac{k+3}{k+2}}.
\end{align}
Consequently, since $ \xi_{k+1} - w_{k+2,j }\in H_0^1(B_{r_j})$,~\eqref{e.Eksmallincasetwo} yields
\begin{equation}  \label{e.Lnplusonecomp}
\frac1{r_j} \left\|  \xi_{k+1} - w_{k+2,j}\right\|_{\underline{L}^2 ( B_{r_j} )} \leq C \ep_j^{\frac{1}{k+1}} \mathsf{E}_{k} .
\end{equation}
By Lemma~\ref{l.C1alphanew} we have 
\begin{align} \label{e.Lipw.pre.applied}
\lefteqn{
 \inf_{\ell \in \mathcal{P}_1} \frac 1{r_{j+1}}\left\|  w_{k+2,r} - \ell   \right\|_{\underline{L}^2(B_{r_{j+1}})}
} \qquad & \\ \notag 
&
\leq
   \frac12  \inf_{\ell \in \mathcal{P}_1} \frac 1{r_j} \left\|  w_{k+2,j} - \ell   \right\|_{\underline{L}^2(B_{r_j})} 
 + C \ep_j  \frac{1}{r_j} \left\|  w_{k+2,j} - \left( w_{k+2,j} \right)_{B_{r_j}}   \right\|_{\underline{L}^2(B_{r_j})} 
 \\ \notag & \qquad
 +  C \ep_j   \sum_{i=1}^{k+1} \left( \frac1R \left\|  w_{i} - \left( w_{i} \right)_{B_R}   \right\|_{\underline{L}^2(B_R)}\right)^{\frac {k+2}{i}} .
\end{align}
Combining this with~\eqref{e.Lnplusonecomp} and the triangle inequality yields~\eqref{e.Lnplusonedecay.case2res}.

\smallskip

\emph{Step 4.} We show that, for $r \in [ 2 (\X_\sigma \vee \mathsf{R}),r_0]$ we have 
\begin{equation} \label{e.Lipxipre}
\frac{1}{r} \left\|  \xi_{k+1} - \left( \xi_{k+1} \right)_{B_{r}}   \right\|_{\underline{L}^2(B_{r})} \leq C \mathsf{E}_{k} . 
\end{equation}
We proceed inductively, and assume $J \in \N_0$ is such that  $r_J \geq 2(\X \vee \mathsf{R})$ and for all $j \in \{0,\ldots,J\}$ we assume that there exists a constant $\mathsf{C}(d,\theta) \in [1,\infty)$, independent of $\eta$ and $\mathsf{R}$, such that 
\begin{equation*} 
\frac{1}{r_j} \left\|  \xi_{k+1} - \left( \xi_{k+1} \right)_{B_{r_j}}   \right\|_{\underline{L}^2(B_{r_j})}  \leq \mathsf{C} \mathsf{E}_{k} .
\end{equation*}
This is true for $J=0$ by the definition of $\mathsf{E}_{k}$.  We claim that it continues to hold for $j=J+1$ as well. Combining~\eqref{e.Lnplusonedecay.case1res} and~\eqref{e.Lnplusonedecay.case2res} with the induction assumption  we have, for $j \in \{1,\ldots,J\}$,  that 
\begin{align*} 
 \inf_{\ell \in \mathcal{P}_1} \frac 1{r_{j+1}}\left\|  \xi_{k+1} - \ell   \right\|_{\underline{L}^2(B_{r_{j+1}})}
\leq
   \frac12  \inf_{\ell \in \mathcal{P}_1} \frac 1{r_j} \left\|  \xi_{k+1} - \ell   \right\|_{\underline{L}^2(B_{r_j})}  +  C \left( \mathsf{C} \ep_j +    \ep_j^{\frac{1}{k+1}}  \right)  \mathsf{E}_{k} .
\end{align*}
Since $r_J \geq \mathsf{R}$, we may take $\mathsf{R}$ large and $\eta$ small enough so that 
\begin{equation*} 
C \sum_{j=0}^J \left( \mathsf{C} \ep_j +    \ep_j^{\frac{1}{k+1}}  \right)
\leq 
C \mathsf{C} \left( 1 - \theta^{\frac{\alpha}{n+1}})\right)^{-1} \left(\eta + \frac1{\mathsf{R}} \right)^{\frac{1}{k+1}} 
 \leq 
 \frac12. 
\end{equation*}
Thus, by summing and reabsorbing, 
\begin{equation} \label{e.iterxikplusone}
\sum_{j=0}^{J+1}  \inf_{\ell \in \mathcal{P}_1} \frac 1{r_{j}}\left\|  \xi_{k+1} - \ell   \right\|_{\underline{L}^2(B_{r_{j}})} \leq  \inf_{\ell \in \mathcal{P}_1} \frac 1{r_0} \left\|  \xi_{k+1} - \ell   \right\|_{\underline{L}^2(B_{r_0})} +  \mathsf{E}_{k} \leq 2 \mathsf{E}_{k}.
\end{equation}
Letting $\ell_j$ being the minimizing affine function in $ \inf_{\ell \in \mathcal{P}_1} \left\|  \xi_{k+1} - \ell   \right\|_{\underline{L}^2(B_{r_{j}})}$, we obtain by the above display and the triangle inequality that 
\begin{equation*} 
\left| \nabla \ell_{j+1}          - \nabla \ell_0 \right| \leq C(d) \theta^{-\frac d2-1} \mathsf{E}_{k}. 
\end{equation*}
By the triangle inequality again, we obtain that 
\begin{equation*} 
\left| \nabla \ell_0 \right| \leq  C(d) \left( \frac{1}{r_0}\left\|  \xi_{k+1} - \ell_0   \right\|_{\underline{L}^2(B_{r_{0}})} + \frac{1}{r_0}\left\|  \xi_{k+1} - (\xi_{k+1})_{B_{r_0}}   \right\|_{\underline{L}^2(B_{r_{0}})} \right) \leq 2 \mathsf{E}_{k}
\end{equation*}
and, consequently, for $C(d,\theta) <\infty$,
\begin{equation*} 
\left| \nabla \ell_{J+1} \right| \leq C \mathsf{E}_{k} .
\end{equation*}
We thus obtain by the triangle inequality, together with~\eqref{e.iterxikplusone} and the above display, that, there exists $C(d,\theta) <\infty$ such that
\begin{equation*} 
\frac{1}{r_{J+1}}\left\|  \xi_{k+1} - (\xi_{k+1})_{B_{r_{J+1}}}   \right\|_{\underline{L}^2(B_{r_{J+1}})}   \leq C  \mathsf{E}_{k} .
\end{equation*}
Hence we can take $\mathsf{C} = C$, proving the induction step. Letting then $J$ be such that $r \in (r_{J+1},r_{J}]$, we obtain~\eqref{e.Lipxipre} by giving up a volume factor. 

\smallskip

\emph{Step 5.} Conclusion.
To conclude, we obtain from~\eqref{e.Lnplusonedecay.case2res} and~\eqref{e.Lipxipre} by iterating that
\begin{equation*} 
 \inf_{\ell \in \mathcal{P}_1} \frac 1{r_{j}}\left\|  \xi_{k+1} - \ell   \right\|_{\underline{L}^2(B_{r_{j}})} \leq C \left(2^{-j} + \sum_{i=0}^j 2^{i-j} \ep_j \right)  \mathsf{E}_{k} 
\end{equation*}
We hence find $\alpha$ such that, for all $r \in [2 (\X \wedge \mathsf{R} , r_0]$, 
\begin{equation*} 
 \left( \frac{r}{R} + \frac1{r} \right)^{-\alpha} \inf_{\ell \in \mathcal{P}_1} \frac{1}{r} \left\|  \nabla \xi_{k+1} - \ell   \right\|_{\underline{L}^2(B_{r})}  
 \leq 
 C \mathsf{E}_{k}.
\end{equation*}
Moreover, by~\eqref{e.Lipxipre} and~\eqref{e.ximplus1bnd} we deduce that, for all $r \in [2 (\X \wedge \mathsf{R}) , R_{k+1}]$, 
\begin{equation*} 
\left\| \nabla \xi_{k+1} \right\|_{\underline{L}^{2} \left( B_{r} \right)} \leq  C \sum_{i=0}^{k+1} 
\left(\frac{1}{R}\left\| \xi_{i} -  \left( \xi_{i} \right)_{B_{R}} \right\|_{\underline{L}^2(B_R)} +  \frac1R \left\|  w_{i+1} - \left( w_{i+1} \right)_{B_R}   \right\|_{\underline{L}^2(B_R)}  \right)^{\frac{m+1}{i+1}} .
\end{equation*}
Therefore,~\eqref{e.Lipxiind} is valid for $m= k+1$, by giving up a volume factor. This proves the induction step and completes the proof. 
\end{proof}

\subsection{Improvement of spatial integrability} \label{s.spatintxi}

Following the strategy in Section~\ref{s.spatintw}, we improve the spatial integrability of~\eqref{e.Lipxi} from $L^2$ to $L^q$ for every $q\in [2,\infty)$. Fix $q \in [2,\infty)$. 
Now, $\xi_{n+1}$ solves
\begin{equation*} 
-\nabla \cdot \left( D_p^2 L(\nabla u,\cdot) \nabla \xi_{n+1}  \right) = \nabla \cdot \mathbf{E}_{n+1} ,
\end{equation*}
where $ \mathbf{E}_{n+1} $ satisfies the estimate~\eqref{e.EnLq} for $\delta = q$ and $n+1$ instead of $n$. Recall that both~\eqref{e.C01linsols} and~\eqref{e.C01linerror} are valid for $m \in \{1,\ldots,n\}$ with $2nq$ instead of $q$. These, together with Proposition~\ref{p.Lipxi0}, yields by~\eqref{e.EnLq} that, for $R \geq \X$, 
\begin{align*}
\left\| \mathbf{E}_{n+1}   \right\|_{\underline{L}^{q} \left( B_{\X/2} \right)} 
& 
\leq 
C \sum_{i=0}^{n}  \left( \frac{1}{R}\left\| \xi_{i} -  \left( \xi_{i} \right)_{B_{R}} \right\|_{\underline{L}^2(B_R)} \right)^{\frac{n+2}{i+1}} 
 \\ & \qquad 
 +  C \sum_{i=1}^{n+1}   \left( \frac1R \left\|  w_{i} - \left( w_{i} \right)_{B_R}   \right\|_{\underline{L}^2(B_R)}\right)^{\frac {n+2}{i}}  .
\end{align*}
Having this at our disposal, we may repeat the proof of Section~\ref{s.spatintw} to conclude~\eqref{e.Lipxi} for $q \in [2,\infty)$.

\section{Sharp estimates of linearization errors}
\label{s.linerrorsproof}

Here we show that Corollary~\ref{c.linerrors} is a consequence of Theorems~\ref{t.regularity.Lbar},~\ref{t.linearizehigher} and~\ref{t.regularity.linerrors}. 

\begin{proof}[{Proof of Corollary~\ref{c.linerrors}}]
For each $k\in \{0,\ldots,n-1\}$, let $\{ B^{(k)}_j \}$ be a finite family of balls such that 
\begin{equation}
U_k \subseteq \bigcup_j B^{(k)}_j 
\quad \mbox{and} \quad 
\bigcup_j 4 B^{(k)}_j  \subseteq U_{k+1}.
\end{equation}
By the Vitali covering lemma, we may further assume that $\frac13 B^{(k)}_j \cap \frac13 B^{(k)}_i=\emptyset$ whenever $i\neq j$. Let $\mathcal{Z}$ be the finite set consisting of the centers of these balls. 
The size of $\mathcal{Z}$ depends only on the geometry of the sequence of domains~$U_0,U_1,\ldots,U_{n}$. Let~$\X$ be the maximum of the random variables~$\X$ given in Theorem~\ref{t.regularity.linerrors}, centered at  elements of~$\mathcal{Z}$, divided by the radius of the smallest ball. We assume that $r\geq \X$. This ensures the validity of Theorem~\ref{t.regularity.linerrors} in each of the balls $B^{(k)}_j$: that is, for every $q\in [2,\infty)$ and $m\in\{1,\ldots,n\}$, we have the estimate
\begin{align}
\label{e.C01linsols2}
\left\| \nabla w_{m}   \right\|_{\underline{L}^q(B^{(m)}_j)}
& 
\leq 
C\sum_{i=1}^{m} 
\left\| \nabla w_{i}  \right\|_{\underline{L}^2(2B^{(m)}_j)}^{\frac{m}i}
\end{align}
and hence, by the covering, 
\begin{align}
\label{e.C01linsols3}
\left\| \nabla w_{m}   \right\|_{\underline{L}^q(U_m)}
& 
\leq 
C\sum_{i=1}^{m} 
\left\| \nabla w_{i}  \right\|_{\underline{L}^2(U_{i})}^{\frac{m}i}.
\end{align}
Proceeding with the proof of the corollary, we define, as usual,
\begin{equation*}
\xi_m := v - u - \sum_{k=1}^m \frac1{k!} w_k.  
\end{equation*}
Arguing by induction, let us suppose that~$m\in\N$ with $m\geq 1$ and $\theta>0$ are such that, for every $j \in \left\{ 0,\ldots,m-1\right\}$ and $q\in [2,\infty)$, there exists $C(q,\data)<\infty$ such that 
\begin{equation}
\label{e.inductgoatass}
\left\| \nabla \xi_{j}  \right\|_{\underline{L}^{2+\theta}(U_j)} 
+
\left\| \nabla w_{j+1}  \right\|_{\underline{L}^q(U_{j+1})} 
\leq 
C \left( \left\| \nabla u - \nabla v \right\|_{\underline{L}^{2}(rU_0)} \right)^{j+1}.
\end{equation}
This obviously holds for~$m=1$ and some~$\theta (d,\Lambda)>0$ by the Meyers estimate and Theorem~\ref{t.regularity.linerrors}. We will show that it must also hold for~$m+1$ and some other (possibly smaller) exponent~$\theta>0$. 

\smallskip

\emph{Step 1.} We show that 
\begin{equation}
\label{e.corowts1}
\left\| \nabla w_{m+1} \right\|_{\underline{L}^2(U_{m+1})}
\leq 
C \left( \left\| \nabla u - \nabla v \right\|_{\underline{L}^{2}(rU_0)} \right)^{m+1}.
\end{equation}
By the basic energy estimate,
\begin{align*}
\left\| \nabla w_{m+1} \right\|_{\underline{L}^2(U_{m+1})}
\leq
C\cdot
\left\{ 
\begin{aligned}
& \left\| \nabla u - \nabla v \right\|_{\underline{L}^{2}(rU_0)}  & \mbox{if} & \ m=0,\\
& \left\| \mathbf{F}_{m+1} (\cdot,\nabla u,\nabla w_1,\ldots,\nabla w_{m})  \right\|_{L^2(U_{m+1})}
& \mbox{if} & \ m\geq 1,
\end{aligned}
\right.
\end{align*}
where 
\begin{equation*}
\left| \mathbf{F}_{m+1}  (\cdot,\nabla u,\nabla w_1,\ldots,\nabla w_{m}) \right|
\leq 
C \sum_{k=1}^{m}  \left| \nabla w_k \right|^{\frac{m+1}{k}}.
\end{equation*}
By Theorem~\ref{t.regularity.linerrors} (using our definition of~$\X$ and the fact that $r\geq \X$) and the induction hypothesis, we have, for every $k\in\{ 1,\ldots,m\}$, 
\begin{align*}
\left\| \nabla w_k \right\|_{\underline{L}^{\frac{2(m+1)}{k}}(U_{m})}^{\frac{m+1}{k}}
\leq 
C\sum_{i=1}^{k-1}
\left( \frac1R
\left\|  w_{i}  \right\|_{\underline{L}^2(U_{m-1})} 
\right)^{\frac{m+1}i}
\leq 
C\left( \left\| \nabla u - \nabla v \right\|_{\underline{L}^{2}(rU_0)} \right)^{m+1}.
\end{align*}
This completes the proof of~\eqref{e.corowts1}.

\smallskip

\emph{Step 2.} We show that 
\begin{equation}
\label{e.corowts2}
\left\| \nabla \xi_m \right\|_{\underline{L}^{2+\theta/2}(U_m)}
\leq 
C \left( \left\| \nabla u - \nabla v \right\|_{\underline{L}^{2}(rU_0)} \right)^{m+1}.
\end{equation}
Observe that, since $m\geq 1$, $\xi_m \in H^1_0(U_m)$, that is, $\xi_m$ vanishes on $\partial U_m$. Therefore, by Meyers estimate and Lemma~\ref{l.lineq}, in particular~\eqref{e.EnLq} with $q=2+\frac12\theta$ and $\delta=\frac12\theta$, we get
\begin{align*}
\left\| \nabla \xi_m \right\|_{\underline{L}^{2+\theta/2}(U_m)}
&
\leq 
C\left( \sum_{i=1}^{m-1} \left\| \nabla \xi_{i} \right\|_{\underline{L}^{2+\theta} \left( B_{R} \right)}^{\frac{m+1}{i+1}}  
+ 
 \left\| \nabla \xi_{0} \right\|_{\underline{L}^{\frac{2m(4+\theta)}{\theta}}(B_R)}^{m+1}
+
\sum_{i=1}^{m-1}  \left\| \nabla w_i \right\|_{\underline{L}^{\frac{2m(4+\theta)}{\theta}}(B_R)}^{\frac{m+1}{i}} \right)
\\ & 
\leq 
C \left( \left\| \nabla u - \nabla v \right\|_{\underline{L}^{2}(rU_0)} \right)^{m+1}.
\end{align*}
This completes the proof of~\eqref{e.corowts2}. 

\smallskip

The corollary now follows by induction. 
\end{proof}


\section{Liouville theorems and higher regularity}
\label{s.regularityhigher}

In this section we prove Theorem~\ref{t.regularityhigher} by an induction in the degree~$n$. The initial step $n=1$ has been already established in~\cite{AFK}. Indeed, $\mathrm{(i)}_1$ and $\mathrm{(ii)}_1$ are consequences of~\cite[Theorem 5.2]{AFK}, and $\mathrm{(iii)}_1$ follows from Theorem~\ref{t.C11estimate} which is~\cite[Theorem 1.3]{AFK}. Moreover, these estimates hold with optimal stochastic integrability, namely we may take any $\sigma \in (0,d)$ for $n=1$ (with the constant~$C$ then depending additionally on~$\sigma$). 

\smallskip

Throughout the section we will use the following notation. Given $p \in \R^d$ and $k \in \N$, we denote
\begin{equation*} 
\A_k^p := \left\{ u \in H_{\textrm{loc}}^1(\R^d) \; : \; -\nabla \cdot \left( \a_p \nabla u \right) = 0 , \; \lim_{r \to \infty} 
r^{-k-1} \left\| u \right\|_{\underline{L}^2 \left( B_{r} \right)} = 0 \right\}
\end{equation*}
and
\begin{equation*} 
\Ahom_k^p := \left\{  \overline{u} \in H_{\textrm{loc}}^1(\R^d) \; : \; -\nabla \cdot \left( \ahom_p \nabla \overline{u} \right) = 0 , \; \lim_{r \to \infty} 
r^{-k-1} \left\|  \overline{u} \right\|_{\underline{L}^2 \left( B_{r} \right)} = 0 \right\}.
\end{equation*}

\begin{remark} \label{r.regularityhigher1}  
In proving Theorem~\ref{t.regularityhigher}$\mathrm{(ii)}_n$ by induction in~$n$, it will be necessary to prove a stronger statement, namely that, for every $p \in B_{\mathsf{M}}$, $(\overline{w}_1,\ldots,\overline{w}_n) \in \overline{\mathsf{W}}_n^{p}$ and $(w_1,\ldots,w_{n-1}) \in \mathsf{W}_{n-1}^p$ satisfying ~\eqref{e.liouvillec} for $m\in \{1,\ldots,n-1\}$,  there exists~$w_n$ satisfying~\eqref{e.liouvillec} for~$m=n$ such that $(w_1,\ldots,w_{n}) \in \mathsf{W}_{n}^p$.
\end{remark}

\begin{proof}[Proof of Theorem~\ref{t.regularityhigher} $\mathrm{(ii)}_n$] For fixed $n$, we will take $\X$ as the maximum of random variables $\X$ appearing in Theorem~\ref{t.linearizehigher}, Theorem~\ref{t.regularity.linerrors} corresponding $q=2+\delta$, where $\delta$ is as in Theorem~\ref{t.linearizehigher},  and a deterministic constant $\mathsf{R}(n,\mathsf{M}, d,\Lambda) < \infty$. Clearly~\eqref{e.higherX} holds then.

\smallskip

Given $p \in B_{\mathsf{M}}$ and $(\overline{w}_1,\ldots,\overline{w}_n) \in \overline{\mathsf{W}}_n^p $, our goal is to prove that there exists a tuplet $(w_1,\ldots,w_n) \in \mathsf{W}_n^p$ such that~\eqref{e.liouvillec} holds for every $R\geq \X$ and $k \in \{1,\ldots,n\}$.

\smallskip

\emph{Step 1.} Induction assumption. We proceed inductively and assume that there is a tuplet $(w_1,\ldots,w_{n-1}) \in \mathsf{W}_{n-1}^p$ such that~\eqref{e.liouvillec} is true for every $R\geq \X$ and $m \in \{1,\ldots,n-1\}$. The base case for the induction is valid by the results of~\cite{AFK}, as mentioned at the beginning of this section. Our goal is therefore to construct $w_n$ such that $(w_1,\ldots,w_n) \in \mathsf{W}_n^p$ and~\eqref{e.liouvillec} holds for every $R\geq \X$ and $k \in \{1,\ldots,n\}$.  Recall that since $(\overline{w}_1,\ldots,\overline{w}_n) \in \overline{\mathsf{W}}_n^p$, we have, for every $R\geq \X$, that  
\begin{equation}  
\label{e.overlinewgrowth}
\sum_{i=1}^{n} 
\left(
 \frac1{R} \left\| \overline{w}_i  \right\|_{\underline{L}^2(B_R)} 
\right)^{\frac{n}{i}}  
\leq 
C(d)
\left(\frac{R}{\X}\right)^n   \sum_{i=1}^{n} 
\left(
 \frac1{\X} \left\| \overline{w}_i  \right\|_{\underline{L}^2(B_\X)}
\right)^{\frac{n}{i}}  
.
\end{equation}
Moreover, by the induction assumption, if the lower bound $\mathsf{R}$ for $\X$ is large enough, we deduce by~\eqref{e.liouvillec} that, for $R \geq \X$ and $m \in \{1,\ldots,n-1\}$,  
\begin{equation}  
\label{e.overlinewpluswgrowth}
\sum_{i=1}^{n} 
\left(
 \frac1{R} \left\| w_i  \right\|_{\underline{L}^2(B_R)} 
\right)^{\frac{m}{i}}  
\leq 2
\left(\frac{R}{\X}\right)^m   \sum_{i=1}^{m} 
\left(
 \frac1{\X} \left\| \overline{w}_i  \right\|_{\underline{L}^2(B_\X)}
\right)^{\frac{n}{i}}  
.
\end{equation}

\smallskip

\emph{Step 1.} 
Let $r_j:= 2^j \X$ and let $w_{n,j}$,  $j \in \N$,  solve
\begin{equation}
\left\{ 
\begin{aligned}
& -\nabla \cdot  \left( \a_p \nabla w_{n,j} \right) =  \nabla \cdot \mathbf{F}_n(p + \nabla \phi_p,\nabla w_{1},\ldots, \nabla w_{n-1},\cdot) 
& \mbox{in} &  \ B_{r_{j}},\\
& w_{n,j} =  \overline{w}_n & \mbox{on} & \ \partial B_{r_{j}}.
\end{aligned} 
\right. 
\end{equation}     
We show that
\begin{equation}  \label{e.wmjhomog2}
 \left\| w_{n,j} - \overline{w}_{n} \right\|_{\underline{L}^2 \left( B_{r_{j-1}} \right)} 
\leq 
C r_j^{1-\alpha}  
 \left(\frac{r_j}{\X}\right)^n   \sum_{i=1}^{n} 
\left(
 \frac1{\X} \left\| \overline{w}_i  \right\|_{\underline{L}^2(B_\X)}
\right)^{\frac{n}{i}}  
.
\end{equation}
Theorem~\ref{t.linearizehigher} 
yields that, for $m \in \{1,\ldots,n\}$,  solutions of 
\begin{equation*}
\left\{ 
\begin{aligned}
& -\nabla \cdot  \left( D_p\overline{L}(\nabla \overline{u}_j)  \right) = 0 & & \mbox{in }   \ B_{r_{j+n}},\\
& -\nabla \cdot  \left( D_p^2\overline{L}(\nabla \overline{u}_j)  \nabla \overline{w}_{m,j} \right) =  \nabla \cdot \overline{\mathbf{F}}_n(\nabla \overline{u}_j,\nabla  \overline{w}_{1,j},\ldots, \nabla \overline{w}_{m-1,j}) 
& & \mbox{in }   \ B_{r_{j+n-m}},\\
& \overline{u}_{j}(x) = p\cdot x + \phi_p(x) &  &\mbox{for }  x \in \partial B_{r_{j+n}},\\
& \overline{w}_{m,j} =  w_{m}, \quad m \in \{1,\ldots,n-1\},  & & \mbox{on }  \partial B_{r_{j+n-m}},
\\
& \overline{w}_{n,j} =  \overline{w}_{n} & & \mbox{on }  \partial B_{r_{j}}
\end{aligned} 
\right. 
\end{equation*}  
satisfy the estimate 
\begin{equation*} 
\left\| w_{n,j} - \overline{w}_{n,j} \right\|_{\underline{L}^2 \left( B_{r_{j}} \right)} 
\leq 
C r_j^{1-\alpha}  \left( \left\| \nabla  \overline{w}_{n} \right\|_{\underline{L}^{2+\delta}(B_{r_{j}})} +   
\sum_{i=1}^{n-1} 
\left( 
\left\| \nabla w_i \right\|_{\underline{L}^{2+\delta}(B_{r_{j + n-m}})} \right)^{\frac{n}{i}} 
\right)  .
\end{equation*}
By Theorem~\ref{t.regularity.linerrors}, together with~\eqref{e.overlinewgrowth}  and~\eqref{e.overlinewpluswgrowth}, assumed for $m \in \{1,\ldots,n-1\}$, we obtain
\begin{equation} \label{e.wmjhomog00}
\left\| w_{n,j} - \overline{w}_{n,j} \right\|_{\underline{L}^2 \left( B_{r_{j}} \right)}  
\leq
C r_j^{1-\alpha}  \left(\frac{r_j}{\X}\right)^n   \sum_{i=1}^{n} 
\left(
 \frac1{\X} \left\| \overline{w}_i  \right\|_{\underline{L}^2(B_\X)}
\right)^{\frac{n}{i}}  
.
\end{equation}
Similarly, by Theorems~\ref{t.linearizehigher} and~\ref{t.regularity.linerrors}, together with~\eqref{e.overlinewpluswgrowth} valid for $m \in \{1,\ldots,n-1\}$, we get, for $m \in \{1,\ldots,n-1\}$,  that
\begin{equation*} %
\left\| w_{m} - \overline{w}_{m,j} \right\|_{\underline{L}^2 \left( B_{r_{j+n-m}} \right)} \leq 
C r_j^{1-\alpha} 
 \left(\frac{r_j}{\X}\right)^m  \sum_{i=1}^{m} 
\left(
 \frac1{\X} \left\| \overline{w}_i  \right\|_{\underline{L}^2(B_\X)}
\right)^{\frac{m}{i}}  .
\end{equation*}
Consequently, again by~\eqref{e.liouvillec} assumed for $m \in \{1,\ldots,n-1\}$, that 
\begin{equation}  \label{e.wmjhomog0}
\left\| \overline{w}_{m} - \overline{w}_{m,j} \right\|_{\underline{L}^2 \left( B_{r_{j}} \right)} \leq 
C r_j^{1-\delta}  
 \left(\frac{r_j}{\X}\right)^m   \sum_{i=1}^{m} 
\left(
 \frac1{\X} \left\| \overline{w}_i  \right\|_{\underline{L}^2(B_\X)}
\right)^{\frac{m}{i}}  
.
\end{equation}
We denote, in short,
\begin{equation*} 
 \overline{\f}_{m,j} := \overline{\mathbf{F}}_m(\nabla \overline{u}_j,\nabla  \overline{w}_{1,j},\ldots, \nabla \overline{w}_{m-1,j}) ,
\end{equation*}
together with
\begin{equation*} 
\ahom_p := D_p^2 \overline{L}(p) \quad \mbox{and} \quad \overline{\f}_{m} := \overline{\mathbf{F}}_m(p, \nabla  \overline{w}_{1},\ldots, \nabla \overline{w}_{m-1}) .
\end{equation*}
Now, it is easy to see (cf. the proof of~\cite[Theorem 5.2]{AFK}) that 
\begin{equation*} 
\left\| \nabla \overline{u}_{j} - p \right\|_{L^\infty(B_{r_{j+n-1}})} \leq C r_j^{-\alpha},
\end{equation*}
so that, by Theorem~\ref{t.regularity.Lbar},
\begin{equation*} 
\left\| D_p^2\overline{L}(\nabla \overline{u}_{j}) - \ahom_p  \right\|_{L^\infty(B_{r_{j+n-1}})} \leq 
\left\| D_p^3 \overline{L} \right\|_{L^\infty(B_{2\mathsf{M}})} \left\| \nabla \overline{u}_{j} - p \right\|_{L^\infty(B_{r_{j+n-1}})} 
\leq C r_j^{-\alpha}. 
\end{equation*}
Therefore, by an analogous computation  to the proof of Lemma~\ref{l.diff.linearizedsystem}, using~\eqref{e.wmjhomog0},  we get
\begin{equation*} 
\left\| \overline{\f}_{n,j} -  \overline{\f}_{n} \right\|_{\underline{L}^{2}(B_{r_{j}})}
 \leq
C r_j^{-\alpha/2}  
 \left(\frac{r_j}{\X}\right)^n   \sum_{i=1}^{n} 
\left(
 \frac1{\X} \left\| \overline{w}_i  \right\|_{\underline{L}^2(B_\X)}
\right)^{\frac{n}{i}}  
\end{equation*}
as well as
\begin{equation*} 
\left\| D_p^2 \overline{L}(\nabla \overline{u}_j) - \ahom_p )  \nabla \overline{w}_{n,j} \right\|_{\underline{L}^{2}(B_{r_{j}})} 
\leq 
C r_j^{- \alpha/2-1}  
\left\| \overline{w}_n - (\overline{w}_n)_{B_{r_n}}  \right\|_{\underline{L}^{2}(B_{r_{j}})}.
\end{equation*}
By testing the equation 
\begin{equation*}
\left\{ 
\begin{aligned}
& -\nabla \cdot  \left( \ahom_p   \nabla ( \overline{w}_{n,j}  -  \overline{w}_{n} ) \right) =  \nabla \cdot \left((D_p^2 \overline{L}(\nabla \overline{u}_j) - \ahom_p )  \nabla \overline{w}_{n,j}  +   \overline{\f}_{n,j} -  \overline{\f}_{n} \right) 
& & \mbox{in }   \ B_{r_{j}},
\\
& \overline{w}_{n,j}  -  \overline{w}_{n}  =  0 & & \mbox{on }  \partial B_{r_{j}}
\end{aligned} 
\right. 
\end{equation*} 
with $\overline{w}_{n,j}  -  \overline{w}_{n}$, we then obtain
\begin{equation*} 
\left\| \overline{w}_{n,j}  -  \overline{w}_{n} \right\|_{\underline{L}^{2}(B_{r_{j}})} \leq 
C (r_j^{1-\delta} + r_j^{1-\alpha/2})  
 \left(\frac{r_j}{\X}\right)^n   \sum_{i=1}^{n} 
\left(
 \frac1{\X} \left\| \overline{w}_i  \right\|_{\underline{L}^2(B_\X)}
\right)^{\frac{n}{i}}  
.
\end{equation*}
Therefore,~\eqref{e.wmjhomog2} follows by~\eqref{e.wmjhomog00} and the above display by taking $\delta = \alpha/2$.

\smallskip

\emph{Step 2.} We show that there is $w_n$ such that $(w_1,\ldots,w_n) \in \mathsf{W}_n^p$ and $w_n$ satisfies~\eqref{e.liouvillec}. 
Setting $z_{n,j} := w_{n,j} - w_{n,j+1}$ we have by the triangle inequality that 
\begin{equation}  \label{e.znjest}
  \left\| z_{n,j} \right\|_{\underline{L}^2 \left( B_{r_j} \right)} \leq 
C r_j^{1 - \delta}  
 \left(\frac{r_j}{\X}\right)^n   \sum_{i=1}^{n} 
\left(
 \frac1{\X} \left\| \overline{w}_i  \right\|_{\underline{L}^2(B_\X)}
\right)^{\frac{n}{i}}   .
\end{equation}
Notice that $z_{n,j}$ is $\a_p$-harmonic in $B_{r_j}$. Thus, by~\cite[Theorem 5.2]{AFK}, we find $\phi_{n,j} \in \A_{n}^p$ such that, for every $r \in [\X,r_j]$, 
\begin{equation} \label{e.znjest2}
 \left\| z_{n,j} - \phi_{n,j}  \right\|_{\underline{L}^2 \left( B_{r} \right)} 
\leq 
C \left(\frac{r}{r_j}\right)^{n+1}  \left\|  z_{n,j}   \right\|_{\underline{L}^2 \left( B_{r_j} \right)} .
\end{equation}
Consequently, for every $r \in [\X,r_j]$, 
\begin{equation*} 
\left\| z_{n,j} - \phi_{n,j}  \right\|_{\underline{L}^2 \left( B_{r} \right)} 
\leq 
C \left( \frac{r}{r_j} \right)^{\delta} 
 r^{1 - \delta}  
  \left(\frac{r}{\X}\right)^n   \sum_{i=1}^{n} 
\left(
  \frac1{\X} \left\| \overline{w}_i  \right\|_{\underline{L}^2(B_\X)}
\right)^{\frac{n}{i}}  . 
\end{equation*}
Setting then $\tilde w_{n,j} := w_{n,j} - \sum_{h=1}^{j-1} \phi_{n,h}$, we have that 
$
\tilde w_{n,k} - \tilde w_{n,j} = \sum_{i= j}^{k-1}  (z_{n,i} -  \phi_{n,i})
$ and it follows that, for all $j,k \in \N$, $j < k$, 
\begin{equation}  \label{e.wmjhomog3prepre}
\left\|  \tilde w_{n,k} - \tilde w_{n,j}  \right\|_{\underline{L}^2 \left( B_{r_j} \right)} 
\leq 
C  r_j^{1-\delta}
 \left(\frac{r_j}{\X}\right)^n   \sum_{i=1}^{n} 
\left(
 \frac1{\X} \left\| \overline{w}_i  \right\|_{\underline{L}^2(B_\X)}
\right)^{\frac{n}{i}}  
. 
\end{equation}
Therefore, $\{\tilde w_{n,k} \}_{k=j}^\infty$ is a Cauchy sequence and, by the Caccioppoli estimate and the diagonal argument, we find~$w_n$ such that 
\begin{equation} \label{e.wmjhomog3pre}
\begin{aligned}
& -\nabla \cdot  \left( \a_p \nabla w_{n} \right) 
=  
\nabla \cdot \mathbf{F}_n(p + \nabla \phi_p,\nabla w_1,\ldots, \nabla w_{n-1},\cdot) 
& \mbox{in} &  \ \R^d
\end{aligned} 
\end{equation}     
and, for all $j \in \N$, 
\begin{equation}  \label{e.wmjhomog3}
 \left\| w_{n} - \tilde w_{n,j} \right\|_{\underline{L}^2 \left( B_{r_{j}} \right)} 
\leq 
C r_j^{1-\delta}  
 \left(\frac{r_j}{\X}\right)^n   \sum_{i=1}^{n} 
\left(
  \frac1{\X} \left\| \overline{w}_i  \right\|_{\underline{L}^2(B_\X)}
\right)^{\frac{n}{i}}  . 
\end{equation}
We now use the facts that $(\overline{w}_1,\ldots,\overline{w}_n) \in \overline{\mathsf{W}}_n^{p} $ and $\phi_{n,j} \in \A_{n}^p$, together with~\eqref{e.znjest} and~\eqref{e.znjest2}, to deduce that
\begin{align} \notag 
\sum_{h=1}^{j-1} \left\| \phi_{n,h} \right\|_{\underline{L}^2 \left( B_{r_j} \right)} 
& 
\leq 
\sum_{h=1}^{j-1} \left( \frac{r_j}{r_h} \right)^n  \left\| \phi_{n,h} \right\|_{\underline{L}^2 \left( B_{r_h} \right)}
\\ \notag & 
\leq
\sum_{h=1}^{j-1} 
\left( \frac{r_j}{r_h} \right)^n 
\left( \left\|  \phi_{n,j}  \right\|_{\underline{L}^2 \left( B_{r_h} \right)} +
\left\| z_{n,j} - \phi_{n,j}  \right\|_{\underline{L}^2 \left( B_{r_h} \right)} 
\right)
\\ \notag & 
\leq 
C \sum_{h=1}^{j-1} \left( \frac{r_j}{r_h} \right)^n r_h^{1-\delta} 
 \left(\frac{r_h}{\X}\right)^n   \sum_{i=1}^{n} 
\left(
 \left\| \overline{w}_i  \right\|_{\underline{L}^2(B_\X)}
\right)^{\frac{n}{i}}  
\\ \notag & 
\leq 
C r_j^{1-\delta}  
\left(\frac{r_j}{\X}\right)^n   \sum_{i=1}^{n} 
\left(
 \left\| \overline{w}_i  \right\|_{\underline{L}^2(B_\X)}
\right)^{\frac{n}{i}}  
 . 
\end{align}
Combining this with~\eqref{e.wmjhomog3} yields that $w_n$ satisfies~\eqref{e.liouvillec}. Moreover, obviously $(w_1,\ldots,w_n) \in \mathsf{W}_n^p$. The proof is complete.
\end{proof}

\begin{proof}[Proof of Theorem~\ref{t.regularityhigher} $\mathrm{(i)}_n$] Let $\X$ be as in the beginning of the proof of $\mathrm{(ii)}_n$.  Fix $R \geq \X$. We proceed via induction. Assume that $(w_1,\ldots,w_{n-1})$ satisfy~\eqref{e.liouvillec}, i.e.,  we find $(\overline{w}_1,\ldots,\overline{w}_{n-1}) \in \overline{\mathsf{W}}_{n-1}^p$ such that, for $k \in \{1,\ldots,n-1\}$ and~$t \geq \X$, 
\begin{equation} \label{e.liouvillec2}
\left\| w_k - \overline{w}_k  \right\|_{\underline{L}^2(B_t)} 
\leq 
C t^{1-\delta} 
\left(\frac{t}{\X}\right)^n   \sum_{i=1}^{n} 
\left(
\frac{1}{\X} \left\| \overline{w}_i  \right\|_{\underline{L}^2(B_\X)}
\right)^{\frac{n}{i}}  
.
\end{equation}
Since $(\overline{w}_1,\ldots,\overline{w}_{n-1}) \in \mathcal{P}_2 \times \ldots \times \mathcal{P}_{n}$, we have by the homogeneity that 
\begin{equation*} 
\overline{\mathbf{F}}(p,\nabla\overline{w}_1,\ldots,\nabla\overline{w}_{n-1}) \in \mathcal{P}_{n}. 
\end{equation*}
We find, by Remark~\ref{r.barwregtrivial} below, a solution $\overline{w} \in \mathcal{P}_{n+1}$ such that  $(\overline{w}_1,\ldots,\overline{w}_{n-1},\overline{w}) \in \overline{\mathsf{W}}_{n}^p$ and that there exists a constant $\mathsf{C}(n,d,\Lambda)< \infty$ such that, for every $t \geq \X$, 
\begin{equation} \label{e.barwregtrivial}
 \left\| \overline{w}  \right\|_{\underline{L}^2 \left( B_{t} \right)}  
 \leq
  \mathsf{C} t 
\left(\frac{t}{\X}\right)^{n}   \sum_{i=1}^{n-1} 
\left(
\frac{1}{\X} \left\| \overline{w}_i  \right\|_{\underline{L}^2(B_\X)}
\right)^{\frac{n}{i}}  
.
\end{equation}
Consequently, Remark~\ref{r.regularityhigher1} provides us $\tilde w$ such that $(w_1,\ldots,w_{n-1},\tilde w) \in \mathsf{W}_{n}^p$ and, for~$t\geq \X$, 
\begin{equation} \label{e.liouvillec3}
 \left\| \tilde w - \overline{w}  \right\|_{\underline{L}^2(B_t)}
\leq 
C \mathsf{C} 
t^{1-\alpha} 
\left(\frac{t}{\X}\right)^n   \sum_{i=1}^{n-1} 
\left(
\frac{1}{\X} \left\| \overline{w}_i  \right\|_{\underline{L}^2(B_\X)}
\right)^{\frac{n}{i}}  
.
\end{equation}
Moreover, by the equations and the growth condition at infinity, $w_n - \tilde w \in \A_{n+1}^p$. Therefore, by~\cite[Theorem 5.2]{AFK}, there is $q \in \Ahom_{n+1}^p$ such that, for all $t \geq \X$, 
\begin{equation*} 
 \left\| w_n - \tilde w  - q  \right\|_{\underline{L}^2(B_t)} \leq C t^{-\delta}\left\| q \right\|_{\underline{L}^2 \left( B_{t} \right)} .
\end{equation*}
 We set $\overline{w}_n :=  \overline{w} + q$. Obviously, $(\overline{w}_1,\ldots,\overline{w}_{n-1},\overline{w}_n) \in \overline{\mathsf{W}}_{n}^p$ since $q$ is $\ahom_p$-harmonic. By~\eqref{e.barwregtrivial} and the triangle inequality we have, for $t\geq \X$, that 
 \begin{multline*} 
 t 
\left(\frac{t}{\X}\right)^n   \sum_{i=1}^{n-1} 
\left(
\frac{1}{\X} \left\| \overline{w}_i  \right\|_{\underline{L}^2(B_\X)}
\right)^{\frac{n}{i}}  
 \leq  
 \frac1{2\mathsf{C}} \left\| q \right\|_{\underline{L}^2 \left( B_{t} \right)} 
 \\    \; \implies \;  
 \left\| q \right\|_{\underline{L}^2 \left( B_{t} \right)} +  \left\| \overline{w}  \right\|_{\underline{L}^2 \left( B_{t} \right)}  
 \leq 
3 \left\| \overline{w}_n  \right\|_{\underline{L}^2 \left( B_{t} \right)} \leq C \left( \frac{t}{\X} \right)^{n+1}  \left\| \overline{w}_n  \right\|_{\underline{L}^2 \left( B_{\X} \right)}
\end{multline*}
and 
\begin{multline*} 
\left\| q \right\|_{\underline{L}^2 \left( B_{t} \right)}  
\leq 
2\mathsf{C}   t 
\left(\frac{t}{\X}\right)^n   \sum_{i=1}^{n-1} 
\left(
\frac{1}{\X} \left\| \overline{w}_i  \right\|_{\underline{L}^2(B_\X)}
\right)^{\frac{n}{i}}  
\\
 \; \implies \;  
 \left\| q \right\|_{\underline{L}^2 \left( B_{t} \right)} +  \left\| \overline{w}  \right\|_{\underline{L}^2 \left( B_{t} \right)} 
 \leq 
 3\mathsf{C}
 t 
\left(\frac{t}{\X}\right)^n   \sum_{i=1}^{n-1} 
\left(
\frac{1}{\X} \left\| \overline{w}_i  \right\|_{\underline{L}^2(B_\X)}
\right)^{\frac{n}{i}}  
.
\end{multline*}
We thus have by the triangle inequality that 
\begin{align} \notag 
\left\| w_n - \overline{w}_n  \right\|_{\underline{L}^2(B_t)} 
& 
\leq 
\frac1t \left\| w_n - \tilde w - q  \right\|_{\underline{L}^2(B_t)} +   \left\| \tilde w  - \overline{w} \right\|_{\underline{L}^2(B_t)} 
\\ \notag &
\leq 
C t^{-\delta \wedge \alpha} 
\left(
 \left\| q \right\|_{\underline{L}^2 \left( B_{t} \right)} + \left\|  \overline{w}  \right\|_{\underline{L}^2(B_t)} 
+ t\left(\frac{t}{\X}\right)^n   \sum_{i=1}^{n-1} \left( \frac{1}{\X} \left\| \overline{w}_i  \right\|_{\underline{L}^2(B_\X)} \right)^{\frac{n}{i}}  
\right) 
\\ \notag &
\leq 
Ct^{1-\delta \wedge \alpha} 
\left(\frac{t}{\X}\right)^n   \sum_{i=1}^{n} 
\left(
\frac{1}{\X} \left\| \overline{w}_i  \right\|_{\underline{L}^2(B_\X)}
\right)^{\frac{n}{i}}  
,
\end{align}
proving the induction step and finishing the proof of~$\mathrm{(i)}_n$.
\end{proof}

\begin{remark} \label{r.barwregtrivial}
We show that there is $\overline{w}$ such that~\eqref{e.barwregtrivial} is valid. Indeed, by letting~$\overline{v}$ solve
\begin{equation} \label{e.barwregtrivial2}
\left\{ 
\begin{aligned}
& -\nabla \cdot \ahom_p \nabla \overline{v} = \nabla \cdot   \overline{\mathbf{F}}(p,\nabla\overline{w}_1,\ldots,\nabla\overline{w}_{n-1}) 
& \mbox{in} &  \ B_{R},\\
& \overline{v} = 0 & \mbox{on} & \ \partial B_{R},
\end{aligned} 
\right. 
\end{equation}
using the fact that $\overline{\mathbf{F}}(p,\nabla\overline{w}_1,\ldots,\nabla\overline{w}_{n-1}) $ is a polynomial of degree $n$, 
it is straightforward to show by homogeneity that 
\begin{equation*} 
\overline{w}(x) = \sum_{m=2}^{n+1} \frac{1}{m!} \nabla^m \overline{v}(0) x^{\otimes m}
\end{equation*}
solves
\begin{equation*} 
-\nabla \cdot \left( \ahom \nabla \overline{w}\right) = \nabla \cdot   \overline{\mathbf{F}}(p,\nabla\overline{w}_1,\ldots,\nabla\overline{w}_{n-1})  \quad  \mbox{in }   \; \R^d .
\end{equation*}
We have the estimate
\begin{equation*} 
\left| \nabla^{m+1} \overline{v}(0) \right| \leq C \sum_{k=0}^{m+1} R^{k - m} \left\|  \nabla^k \overline{\mathbf{F}}(p,\nabla\overline{w}_1,\ldots,\nabla\overline{w}_{n-1})  \right\|_{\underline{L}^2 \left( B_{R} \right)}
\end{equation*}
for all $m \in \{1,\ldots,n\}$, and by the equations of $\overline{w}_1,\ldots,\overline{w}_{n-1}$ we see that
\begin{equation*} 
 \left\|  \nabla^k \overline{\mathbf{F}}(p,\nabla\overline{w}_1,\ldots,\nabla\overline{w}_{n-1})  \right\|_{\underline{L}^2 \left( B_{R} \right)}
 \leq C R^{-k} \left( \sum_{i=1}^{n-1}  \frac1R \left\| \overline{w}_i  \right\|_{\underline{L}^2 \left( B_{R} \right)} \right)^{\frac ni} .
\end{equation*}
Therefore, for all $R>0$ and $m \in \{1,\ldots,n\}$, 
\begin{equation*} 
\left| \nabla^{m+1} \overline{v}(0) \right| 
\leq 
C R^{-m}  \left( \sum_{i=1}^{n-1}  \frac1R \left\| \overline{w}_i  \right\|_{\underline{L}^2 \left( B_{R} \right)} \right)^{\frac ni}  .
\end{equation*}
Thus we have that, for $t \geq R$, 
\begin{equation*} 
\left\|  \nabla \overline{w} \right\|_{\underline{L}^2 \left( B_{t} \right)} \leq C  \left( \frac{t}{R} \right)^{n}  \left( \sum_{i=1}^{n-1}  \frac1R \left\| \overline{w}_i  \right\|_{\underline{L}^2 \left( B_{R} \right)} \right)^{\frac ni} ,
\end{equation*}
which yields~\eqref{e.barwregtrivial}.  
\end{remark}

We now turn to the proof of the large-scale $C^{n,1}$ estimate. 

\begin{proof}[Proof of Theorem~\ref{t.regularityhigher} $\mathrm{(iii)}_n$]
Fix $\mathsf{M} \in [1,\infty)$. By Theorem~\ref{t.regularity.linerrors} there exist constants $\sigma(n,\mathsf{M},\data) \in (0,1)$ and $C(n,\mathsf{M},\data)<\infty$ and a random variable $\X$ satisfying $\X \leq \O_\sigma(C)$ so that the statement of Theorem~\ref{t.regularity.linerrors}  is valid with $q = 2(n+2)$.   We now divide the proof in two steps. 

\smallskip

\emph{Step 1.} 
Induction assumption. Assume $\mathrm{(iii)}_{n-1}$. Consequently there is $p \in B_{C}$ and a tuplet $(w_1,\ldots,w_{n-1})$ such that, for $k \in \{0,\ldots,n-1\}$ and
\begin{equation*} 
\xi_k(x) = v(x) - p \cdot x  -  \phi_p(x) - \sum_{i=1}^{k} \frac{w_i(x)}{i!} , \qquad \xi_0(x) := v(x) - p \cdot x  -  \phi_p(x),
\end{equation*}
we have that  there exists $C(k,\mathsf{M},\data)$ such that, for every $m \in \{0,\ldots,n-1\}$ and for every $r \in [\X, \frac12 (1+ 2^{-k-2})R]$, 
\begin{equation}
\label{e.intrinsicregnminusone1}
\left\| \nabla \xi_m \right\|_{\underline{L}^{2}(B_r)}
\leq 
C \left( \frac r R \right)^{m+1} 
\left( \mathsf{H}_k^{m+1}  \wedge \inf_{\phi \in \mathcal{L}_1} \frac1R \left\|  v - \phi \right\|_{\underline{L}^2 \left( B_{R} \right)} \right) ,
\end{equation}
where we denote, in short, 
\begin{equation*} 
\mathsf{H}_m  
:= 
\sum_{i=0}^{m}   
\left(   \frac1R
 \left\|  \xi_i - (\xi_i)_{B_R}  \right\|_{\underline{L}^2(B_{R})} 
\right)^{\frac{1}{i+1}} .
\end{equation*}
Our goal is to show that~\eqref{e.intrinsicregnminusone1} continues to hold with $m=n$ and for every $r \in [\X, \frac12 (1+  2^{-n-2})R]$. The base case $n=1$ is valid since, by~\cite[Theorem 1.3]{AFK}, we have that there is $ p \in B_C$ such that, for all $r \in  [\X, \frac  34 R]$, 
\begin{equation}  \label{e.intrinsicregnminusonebbase}
\left\|  \nabla \xi_0 \right\|_{\underline{L}^2 \left( B_{r} \right)} \leq C\left( \frac rR \right) \inf_{\phi \in \mathcal{L}_1} \frac1R \left\|  v - \phi \right\|_{\underline{L}^2 \left( B_{R} \right)}. 
\end{equation}

\smallskip

\smallskip

\emph{Step 2.} Construction of a special solution. We construct a solution $\tilde w_n$ of
\begin{equation}  \label{e.tildewneq}
- \nabla \cdot \left( \a_p  \nabla \tilde w_n \right) 
=
 \nabla \cdot  \mathbf{F}_n \left( p+ \nabla \phi_p , \nabla w_1,\ldots, \nabla w_{n-1} ,\cdot \right) \quad \mbox{in } \, \R^d 
\end{equation}
satisfying, for $r \geq \X$, 
\begin{equation}  \label{e.tildewnest}
  \left\| \nabla \tilde w_n \right\|_{\underline{L}^{2(n+1)}(B_r)} 
\leq
C \left( \frac r R \right)^{n}    \left( \mathsf{H}_{n-1}^{n} \wedge \inf_{\phi \in \mathcal{L}_1} \frac1R \left\|  v - \phi \right\|_{\underline{L}^2 \left( B_{R} \right)}\right) Ã.
\end{equation}
To show this, it first follows by~\eqref{e.intrinsicregnminusone1} and the triangle inequality that, for $m \in \{1,\ldots,n-1\}$, 
\begin{equation*} 
\left\| \nabla w_m \right\|_{\underline{L}^{2}(B_r)}
 \leq
C \left( \inf_{\phi \in \mathcal{L}_1} \frac1R \left\|  v - \phi \right\|_{\underline{L}^2 \left( B_{R} \right)}\right)^{\frac{1}{m+1}}  \left\| \nabla \xi_m \right\|_{\underline{L}^{2}(B_r)}^{\frac{m}{m+1}} + \left\| \nabla \xi_{m-1} \right\|_{\underline{L}^{2}(B_r)}
.
\end{equation*}
Since we have Theorem~\ref{t.regularity.linerrors} at our disposal with $q = 2(n+1)$, we can increase the integrability and obtain by~\eqref{e.C01linsols} and~\eqref{e.intrinsicregnminusone1} that, for $r \in [\X, \frac12 (1+ \frac38 2^{-n})R]$ and $m \in \{1,\ldots,n-1\}$,
\begin{equation}
\label{e.intrinsicregnminusone3n}
  \left\| \nabla w_m \right\|_{\underline{L}^{2(n+1)}(B_r)} 
\leq
C \left( \frac r R \right)^{m}   
\left( \mathsf{H}_{m}^{m} \wedge \inf_{\phi \in \mathcal{L}_1} \frac1R \left\|  v - \phi \right\|_{\underline{L}^2 \left( B_{R} \right)}\right) 
.
\end{equation}
Consequently, by~\eqref{e.C01linerror} and~\eqref{e.intrinsicregnminusone3n}, we also get, for $m \in \{0,\ldots,n-1\}$ and $r \in [\X, \frac12 (1+\frac3{16} 2^{-m})R]$, that
\begin{equation}
\label{e.intrinsicregnminusone1n}
\left\| \nabla \xi_m \right\|_{\underline{L}^{2(n+1)}(B_r)}
\leq 
C \left( \frac r R \right)^{m+1}   \left( \mathsf{H}_{m}^{m+1} \wedge \inf_{\phi \in \mathcal{L}_1} \frac1R \left\|  v - \phi \right\|_{\underline{L}^2 \left( B_{R} \right)}\right) 
.
\end{equation}
Next, Theorem~\ref{t.regularityhigher}$\mathrm{(i)}_{n-1}$ yields that we find $(\overline{w}_1,\ldots,\overline{w}_{n-1}) \in \overline{\mathsf{W}}_{n-1}^{p}$ such that, for $m \in \{1,\ldots,n-1\}$, 
\begin{equation} \label{e.liouvillec.applied1}
\frac1\X \left\| w_m - \overline{w}_m \right\|_{\underline{L}^2(B_\X)} \leq C \X^{-\delta}  \sum_{i=1}^m \left(  \frac1{\X} \left\| \overline{w}_i \right\|_{\underline{L}^2(B_{\X})}\right)^{\frac{m}{i}} .
\end{equation}
 In particular, applying this inductively, assuming that the lower bound $\mathsf{R}$ for $\X$ is such that $C \mathsf{R}^{-\delta}  \leq \frac12$, we deduce by~\eqref{e.intrinsicregnminusone3n} that, for $k \in \{1,\ldots,n-1\}$, 
 \begin{equation} \label{e.liouvillec.applied2}
\left\| \nabla \overline{w}_k  \right\|_{\underline{L}^2(B_\X)} 
\leq
C \left( \frac {\X}{R} \right)^{k} 
\left( \mathsf{H}_{k}^{k} \wedge \inf_{\phi \in \mathcal{L}_1} \frac1R \left\|  v - \phi \right\|_{\underline{L}^2 \left( B_{R} \right)}\right) 
.
\end{equation}
By Remark~\ref{r.barwregtrivial} we find a solution $\overline{w}_n$ of 
\begin{equation*} 
-\nabla \cdot \left( \ahom_p \nabla \overline{w}_n \right)  
= 
\nabla \cdot   \overline{\mathbf{F}}(p,\nabla\overline{w}_1,\ldots,\nabla\overline{w}_{n-1}) 
\end{equation*}
satisfying, for $r \geq \X$, 
\begin{equation*} 
\left\|  \nabla \overline{w}_n \right\|_{\underline{L}^2 \left( B_{r} \right)}
 \leq 
 C \left( \frac{r}{\X} \right)^{n}  \left( \sum_{i=1}^{n-1}  \left\| \nabla \overline{w}_i  \right\|_{\underline{L}^2 \left( B_{\X} \right)} \right)^{\frac ni} 
.
\end{equation*}
In view of~\eqref{e.liouvillec.applied2} this yields, for $r \geq \X$, that
\begin{equation} \label{e.liouvillec.applied3}
\left\|  \nabla \overline{w}_n \right\|_{\underline{L}^2 \left( B_{r} \right)} 
\leq 
C \left( \frac{r}{R} \right)^{n} 
\left( \mathsf{H}_{n-1}^{n} \wedge \inf_{\phi \in \mathcal{L}_1} \frac1R \left\|  v - \phi \right\|_{\underline{L}^2 \left( B_{R} \right)}\right) 
.
\end{equation}
By Remark~\ref{r.regularityhigher1} we then find $\tilde w_n$ solving~\eqref{e.tildewneq} such that, for $r\geq \X$, 
\begin{equation} \label{e.liouvillec.applied4}
\left\| \tilde w_n - \overline{w}_n  \right\|_{\underline{L}^2(B_r)} \leq C r^{1-\delta}  \left( \frac{r}{\X} \right)^{n} \sum_{i=1}^n \left(  \frac1{\X} \left\| \overline{w}_i \right\|_{\underline{L}^2(B_{\X})}\right)^{\frac{n}{i}} .
\end{equation}
Now~\eqref{e.tildewnest} follows by~\eqref{e.liouvillec.applied2},~\eqref{e.liouvillec.applied3} and~\eqref{e.liouvillec.applied4} together with Theorem~\ref{t.regularity.linerrors}.

\smallskip

\emph{Step 3.} We show the induction step, that is, we validate~\eqref{e.intrinsicregnminusone1} for $k=n$ and $r \in [\X, \frac12 (1+ 2^{-n-2})R]$. 
Denote~$\tilde \xi_n := \xi_{n-1} - \frac1{n!} \tilde w_n$. We begin by deducing an estimate for~$\tilde \xi_n$. Appendix~\ref{app.linerrors} tells us that~$\tilde \xi_n$ solves the equation 
\begin{equation*} 
- \nabla \cdot \left( \a_p \nabla \tilde \xi_{n}  \right) 
= \nabla \cdot \mathbf{E}_{n}
\quad \mbox{in} \ B_{R}
\end{equation*}
and there exist constants $C(n,\mathsf{M},\data)<\infty$ such that 
\begin{equation} \label{e.Enapplied}
\left\| \mathbf{E}_{n} \right\|_{\underline{L}^{2} \left( B_{r} \right)} 
\leq 
 C \sum_{i=0}^{n-1} \left\| \nabla \xi_{i} \right\|_{\underline{L}^{2(n+1)} \left( B_{r} \right)}^{\frac{n+1}{i+1}}  
+
C \sum_{i=1}^{n-1}  \left\| \nabla w_i \right\|_{\underline{L}^{2(n+1)} (B_r)}^{\frac{n+1}{i}}  
.
\end{equation}
By~\eqref{e.intrinsicregnminusone1n} and~\eqref{e.intrinsicregnminusone3n} we then obtain, for $r \in [\X, \frac12 (1+\frac38 2^{-n})R]$, that
\begin{equation} \label{e.intrinsicregnminusone4}
\left\| \mathbf{E}_{n} \right\|_{\underline{L}^{2} \left( B_{r} \right)}  
\leq 
C \left( \frac r R \right)^{n+1} 
\left( \mathsf{H}_{n-1}^{n+1} \wedge \inf_{\phi \in \mathcal{L}_1} \frac1R \left\|  v - \phi \right\|_{\underline{L}^2 \left( B_{R} \right)}\right) 
 .
\end{equation}

\smallskip

Next, set, for $j\in \N_0$, $r_j := \frac12  \theta^j (1+2^{-n-2})R$, where $\theta(n,\mathsf{M},\data) \in \left(0,\frac12 \right]$ will be fixed shortly.  
Let $\phi_0  \in \A_{n+2}^p$ and, for given $\phi_{j}  \in \A_{n+2}^p$, let $h_{j}$ solve 
\begin{equation}
\left\{ 
\begin{aligned}
& -\nabla \cdot  \left( \a_p \nabla h_j \right) =  0
& \mbox{in} &  \ B_{r_{j}},\\
& h_j = \tilde \xi_n - \phi_{j} & \mbox{on} & \ \partial B_{r_{j}}.
\end{aligned} 
\right. 
\end{equation}
By testing and~\eqref{e.intrinsicregnminusone4} we get
\begin{equation}  \label{e.intrinsicregnminusone5}
\left\|   \nabla \tilde \xi_n -   \nabla \phi_{j}  -  \nabla h_j  \right\|_{\underline{L}^{2} \left( B_{r_{j}} \right)} 
\leq   
C \left( \frac r R \right)^{n+1}  
\left( \mathsf{H}_{n-1}^{n+1} \wedge \inf_{\phi \in \mathcal{L}_1} \frac1R \left\|  v - \phi \right\|_{\underline{L}^2 \left( B_{R} \right)}\right) 
.
\end{equation}
Furthermore, by~\cite[Theorem 5.2]{AFK}, there is $\tilde \phi_{j+1} \in \A_{n+2}^p$ such that  
\begin{equation*} 
 \left\| \nabla h_j -  \nabla \tilde \phi_{j+1} \right\|_{\underline{L}^{2}(B_{r_{j+1}})} 
 \leq 
 C \theta^{n+2}  \left\| \nabla h_j -  \nabla \phi_{j}   \right\|_{\underline{L}^{2}(B_{r_{j}})}
.
\end{equation*}
Combining, we have by the triangle inequality, for $\phi_{j+1} := \tilde \phi_{j+1} + \phi_{j} \in \A_{n+2}^p$,  that 
\begin{align} \notag 
\left\|  \nabla \tilde \xi_n- \nabla \phi_{j+1} \right\|_{\underline{L}^{2} (B_{r_{j+1}})}
& 
\leq 
C \theta^{n+2}  \left\| \nabla \tilde \xi_n -  \nabla \phi_{j}   \right\|_{\underline{L}^{2}(B_{r_{j}})}  
\\ \notag & \quad
+ 
C \theta^{-\frac d2} \left( \frac r R \right)^{n+1}  \left( \mathsf{H}_{n-1}^{n+1} \wedge \inf_{\phi \in \mathcal{L}_1} \frac1R \left\|  v - \phi \right\|_{\underline{L}^2 \left( B_{R} \right)}\right) 
.
\end{align}
Choosing $C\theta^{1/2} = 1$, we thus arrive at
\begin{align} \notag 
\lefteqn{ \frac{1}{r_{j+1}^{n+1}}
\left\|  \nabla \tilde \xi_n - \nabla \phi_{j+1} \right\|_{\underline{L}^{2} (B_{r_{j+1}})}} \quad &
\\ \notag &
\leq 
   \frac{ \theta^{1/2} }{r_{j}^{n+1}} \left\| \nabla \tilde \xi_n -  \nabla \phi_{j}   \right\|_{\underline{L}^{2+\theta}(B_{r_{j}})}
+
\frac{C}{R^{n+1}}  \left( \mathsf{H}_{n-1}^{n+1} \wedge \inf_{\phi \in \mathcal{L}_1} \frac1R \left\|  v - \phi \right\|_{\underline{L}^2 \left( B_{R} \right)}\right) 
.
\end{align}
It follows by iteration that, for $r \in [\X, \frac12 (1+ 2^{-n-2})R]$, 
\begin{align} \notag 
\inf_{\phi \in \A_{n+2}^p } \left\|  \nabla \tilde \xi_n - \nabla \phi \right\|_{\underline{L}^{2} (B_{r}) } 
&
\leq 
C\left( \frac rR \right)^{n+3/2}  \inf_{\phi \in \A_{n+2}^p } \left\|  \nabla \tilde \xi_n - \nabla \phi \right\|_{\underline{L}^{2} (B_{r_0}) }  
\\ \notag & \quad 
+  C\left( \frac rR \right)^{n+3/2}   \left( \mathsf{H}_{n-1}^{n+1} \wedge \inf_{\phi \in \mathcal{L}_1} \frac1R \left\|  v - \phi \right\|_{\underline{L}^2 \left( B_{R} \right)}\right) 
.
\end{align}
We can now revisit~\cite[Proof of (3.49)]{AKMbook} and obtain that there exists $\phi \in \A_{n+1}^p$ such that, taking $w_n = \tilde w_n + \phi$ and $\xi_n := \xi_{n-1} - \frac1{n!} w_n$, 
we get, for all $r \in [\X, \frac12 (1+2^{-n-2})R]$,
\begin{equation} 
\left\|  \nabla \xi_n \right\|_{\underline{L}^{2} (B_{r}) } 
\leq 
C \left( \frac rR \right)^{n+1} 
 \left( \mathsf{H}_{n}^{n+1} \wedge \inf_{\phi \in \mathcal{L}_1} \frac1R \left\|  v - \phi \right\|_{\underline{L}^2 \left( B_{R} \right)}\right) 
\end{equation}
Since $\phi \in \A_{n+1}^p$, we see that $(w_1,\ldots,w_n) \in \mathsf{W}_n^p$. Now~\eqref{e.intrinsicregnminusone1} follows for $k=n$ by the previous inequality together with the Caccioppoli estimate and~\eqref{e.intrinsicregnminusone4}. The proof is complete. 
\end{proof}


\appendix

\section{Deterministic regularity estimates}
\label{app.CZ}
In this first appendix, we record some determinstic regularity estimates of Schauder and Calder\'on-Zygmund type for linear equations with H\"older continuous coefficients. These estimates, while well-known, are not typically written with explicit dependence on the regularity of the coefficients, which is needed for our purposes in this paper. 


\begin{proposition}[{Calder\'on-Zygmund gradient $L^q$ estimates}]
\label{p.gradientLq}
Let~$\beta \in (0,1]$, $q\in [2,\infty)$ and~$\a\in \R^{d\times d}$ be a symmetric matrix with entries in~$C^{0,\beta}(B_2)$ satisfying 
\begin{equation*}
I_d \leq \a(x) \leq \Lambda I_d, \quad \forall x\in B_2. 
\end{equation*}
Suppose~$\mathbf{f} \in L^q(B_2;\Rd)$ and~$u\in H^1(B_2)$ is a solution of 
\begin{equation*}
-\nabla \cdot \left( \a\nabla u \right) = \nabla\cdot \mathbf{f} \quad \mbox{in} \ B_2. 
\end{equation*}
Then $u \in W^{1,q}_{\mathrm{loc}}(B_2)$ and there exists $C(q,d,\Lambda)<\infty$ such that
\begin{equation}
\label{e.CZ}
\left\| \nabla u \right\|_{L^q(B_1)} 
\leq 
C \exp\left( \tfrac C\beta\left(1 - \tfrac2q\right) \right)
\left( 1+ \left[ \a \right]_{C^{0,\beta}(B_2)}^{\frac{d}{2\beta}\left( 1 - \frac{2}{q}\right)}  \right)
\left\| \nabla u \right\|_{L^2(B_2)} 
+
C \left\| \mathbf{f} \right\|_{L^q(B_2)}
.
\end{equation}
\end{proposition}
\begin{proof}
We will explain how to extract the statement of the proposition from that of~\cite[Proposition 7.3]{AKMbook}. The latter asserts the existence of $\delta_0(q,d,\Lambda)>0$ such that, for every ball $x\in B_1$ and $r\in \left( 0,\tfrac 12\right]$ satisfying 
\begin{equation*}
\osc_{B_{2r}(x)} \a  \leq \delta_0,
\end{equation*}
we have, for a constant $C(q,d,\Lambda)<\infty$,   the estimate
\begin{equation*}
\left\| \nabla u \right\|_{\underline{L}^q(B_r(x))} 
\leq 
C \left( \left\| \nabla u \right\|_{\underline{L}^2(B_{2r}(x))} + \left\| \f \right\|_{\underline{L}^q(B_{2r}(x))}  \right) .
\end{equation*}
Since $\osc_{B_{2r}} \a \leq (2r)^\beta \left[ \a \right]_{C^{0,\beta}(B_2)}$, we have the above estimate for every $x\in B_1$ and
\begin{equation*}
r:= \frac12\wedge \frac12 \left( \delta_0 \left[ \a \right]_{C^{0,\beta}(B_2)}^{-1} \right)^{\frac1\beta}.
\end{equation*}
From this, Fubini's theorem and Young's inequality for convolutions, we obtain
\begin{align*}
\left\| \nabla u \right\|_{L^q(B_1)}^q 
& 
\leq 
C \int_{B_{1}} \left( \left|\nabla u \right|^q \ast \left( \frac1{|B_r|} \indc_{B_r} \right) \right) (x)\,dx
\\ & 
= 
C \int_{B_{1}} \fint_{B_r(x)} \left|\nabla u(y) \right|^q \,dy\,dx
\\ & 
\leq 
C \int_{B_{1}} 
\left(
\fint_{B_{2r}(x)} \left| \nabla u(y) \right|^2 \, dy
\right)^{\frac q2}\,dx
+ 
C \int_{B_{1}} \fint_{B_{2r}(x)} \left| \f(y) \right|^q\,dy\,dx
\\ & 
\leq 
C \int_{B_{1}} \left| \left(  \left| \nabla u \right|^2 \ast 
\left( \frac1{|B_{2r}|} \indc_{B_{2r}} \right) \right) (x) \right|^{\frac q2} \,dx
+
C \left\| \f \right\|_{L^q(B_2)}^q
\\ & 
\leq 
C \left\| \nabla u \right\|_{L^2(B_2)}^q
\left\| \frac1{|B_{2r}|} \indc_{B_{2r}} \right\|_{L^{q/2}(B_2)}^{\frac q2}
+
C \left\| \f \right\|_{L^q(B_2)}^q
\\ & 
= C r^{-d\left(\frac q2-1\right)} \left\| \nabla u \right\|_{L^2(B_2)}^q
+
C \left\| \f \right\|_{L^q(B_2)}^q.
\end{align*}
This completes the proof. 
\end{proof}

\begin{proposition}
\label{p.schauder}
Let~$\beta \in (0,1)$ and~$\a\in \R^{d\times d}$ be a symmetric matrix with entries in~$C^{0,\beta}(B_2)$ satisfying 
\begin{equation*}
I_d \leq \a(x) \leq \Lambda I_d, \quad \forall x\in B_2. 
\end{equation*}
Suppose~$\mathbf{f} \in C^{0,\beta} (B_2;\Rd)$ and~$u\in H^1(B_2)$ is a solution of 
\begin{equation*}
-\nabla \cdot \left( \a\nabla u \right) = \nabla\cdot \mathbf{f} \quad \mbox{in} \ B_2. 
\end{equation*}
Then $u \in C^{1,\beta}_{\mathrm{loc}}(B_2)$ and there exists $C(\beta,d,\Lambda)<\infty$ such that
\begin{equation}
\label{e.schauder1}
\left\| \nabla u \right\|_{L^\infty(B_1)} 
\leq 
C \left( 1 + \left[ \a \right]_{C^{0,\beta}(B_2)}^{\frac d{2\beta}} \right) 
\left\| \nabla u \right\|_{L^2(B_2)} 
+ 
C \left[ \f \right]_{C^{0,\beta}(B_{2})} 
\end{equation}
and
\begin{equation}
\label{e.schauder2}
\left[ \nabla u \right]_{C^{0,\beta}(B_1)}
\leq 
C\left( 1 + \left[ \a \right]_{C^{0,\beta}(B_2)}^{1+\frac d{2\beta}} \right)  \left\| \nabla u \right\|_{L^2(B_2)} 
+ 
C\left[ \f \right]_{C^{0,\beta}(B_{2})}.
\end{equation}
\end{proposition}
\begin{proof}
We will explain how to extract the statement of the proposition from the gradient H\"older estimate found in~\cite[Theorem 3.13]{HL}. The latter states that, under the assumption that 
\begin{equation*}
\left[ \a \right]_{C^{0,\beta}(B_2)} \leq 1, 
\end{equation*}
there exists $C(\beta,d,\Lambda)<\infty$ such that
\begin{equation*}
\left\| \nabla u \right\|_{C^{0,\beta}(B_1)} 
\leq 
C\left( \left\| \nabla u \right\|_{L^2(B_2)} + \left[ \f \right]_{C^{0,\beta}(B_2)} \right). 
\end{equation*}
After changing the scale, we obtain the corresponding statement in $B_r$, which asserts that, under the assumption that 
\begin{equation*}
r^{\beta} \left[ \a \right]_{C^{0,\beta} (B_{2r})} \leq 1, 
\end{equation*}
there exists $C(\beta,d,\Lambda)<\infty$ such that
\begin{equation*}
\left\| \nabla u \right\|_{L^\infty(B_r)} 
+ r^{\beta} \left[ \nabla u \right]_{C^{0,\beta}(B_r)} 
\leq 
C\left( \left\| \nabla u \right\|_{\underline{L}^2(B_{2r})} + r^{\beta} \left[ \f \right]_{C^{0,\beta}(B_{2r})} \right). 
\end{equation*}
Therefore we take~$r:= \frac12\wedge \left[ \a \right]_{C^{0,\beta} (B_{2})}^{-\frac1\beta}$ and apply the previous statement in every ball~$B_r(x)$ with $x\in B_1$ to obtain
\begin{align*}
\left\| \nabla u \right\|_{L^\infty(B_1)} 
+ \sup_{x\in B_1} r^{\beta} \left[ \nabla u \right]_{C^{0,\beta}(B_r(x))} 
&
\leq 
C\sup_{x\in B_1}\left(  \left\| \nabla u \right\|_{\underline{L}^2(B_{2r}(x))} + r^{\beta} \left[ \f \right]_{C^{0,\beta}(B_{2r}(x))} \right)
\\ & 
\leq 
C \left(r^{-\frac d2} \left\| \nabla u \right\|_{L^2(B_2)} 
+ 
r^{\beta} \left[ \f \right]_{C^{0,\beta}(B_{2})} \right).
\end{align*}
After a covering argument, we obtain 
\begin{equation*}
\left\| \nabla u \right\|_{L^\infty(B_1)} + r^{\beta} \left[ \nabla u \right]_{C^{0,\beta}(B_1)}
\leq 
C \left(r^{-\frac d2} \left\| \nabla u \right\|_{L^2(B_2)} 
+ 
r^{\beta} \left[ \f \right]_{C^{0,\beta}(B_{2})} \right),\end{equation*}
which yields the proposition. 
\end{proof}

\section{Differentiation of \texorpdfstring{$\mathbf{F}_m$}{Fm}}
\label{s.AppendixFm}
In this appendix, we show that \eqref{e.Fmrelation} holds.

\begin{lemma} \label{l.Fmrelation} Fix $m \in \N$ and $h,p \in \R^d$. Suppose that $z\mapsto L(z,x)$ is $C^{m+2}$ and $t \mapsto \mathbf{g}( p + th)$ is~$m$ times differentiable at~$0$.  Then 
\begin{align}  \label{e.FmrelationApp}
\lefteqn{
\mathbf{F}_{m+1} (\mathbf{g}(p) , D_p \mathbf{g}(p) h^{\otimes 1} ,\ldots,D_p^{m} \mathbf{g}(p) h^{\otimes m},x) 
} \quad &
\\ \notag &
=  D_p \left( \mathbf{F}_{m} (\mathbf{g}(p),D_p \mathbf{g}(p) h^{\otimes 1},\ldots,D_p^{m-1} \mathbf{g}(p) h^{\otimes (m-1)},x) \right) \cdot h   
\\ \notag & \quad 
+ D_p \left( D_p^{2} L(\mathbf{g}(p),x) \right) h^{\otimes 1} \left( D_p^{m} \mathbf{g}(p) h^{\otimes m}  \right)^{\otimes 1}.
\end{align}
\end{lemma}

\begin{proof}
We first observe that the terms $D_p \mathbf{g}(p) h^{\otimes 1},\ldots,D_p^{m} \mathbf{g}(p) h^{\otimes m}$ in~\eqref{e.FmrelationApp} are precisely the directional derivatives of $\mathbf{g}$ in the $h$ direction, up to $m$th degree. The terms in \eqref{e.FmrelationApp} involve derivatives of $z\mapsto L(z,x)$ up to the $(m+2)\text{th}$ degree. Hence, we can assume by approximation, without loss of generality, that $z \mapsto L(z,x)$ and $p \mapsto \mathbf{g}(p)$ are polynomials, of degrees at most $m+2$ and $m$, respectively. Fix $h \in \R^d$ and let $t \in \R$. We write
\begin{equation*} 
\mathbf{g}(p+t h) = \mathbf{g}(p) + \sum_{j=1}^m \frac{t^j }{j!} D_p^j \mathbf{g}(p) h^{\otimes j} .
\end{equation*}
Denote
\begin{equation*} 
\mathbf{z}_j(p) :=  D_p^j \mathbf{g}(p) h^{\otimes j} \quad \mbox{and} \quad \mathbf{Z}(p,t) := \sum_{j=1}^{m} \frac{t^j }{j!} \mathbf{z}_{j}(p).
\end{equation*}
Examining the relation between the $p$ and $t$ derivatives of $\mathbf{Z}(p,t)$, we find that
\begin{equation} \label{e.diffZ}
D_p \mathbf{Z}(p,t) \cdot h = \sum_{j=1}^{m} \frac{t^j }{j!} \mathbf{z}_{j+1}(p) 
= 
\partial_t \mathbf{Z}(p,t)  - \mathbf{z}_{1}(p) .
\end{equation}
Set now, for fixed $h,x\in\R^d$,
\begin{equation*}
\mathbf{G}_{h,x}(t,p) := \sum_{k=2}^{m+1} \frac1{k!} D_p^{k+1} L(\mathbf{g}(p),x) \left( \mathbf{Z}(p,t) \right)^{\otimes k} 
\end{equation*}
and, by the definition of $\mathbf{F}_m$ in~\eqref{e.defFm}, 
\begin{equation*} 
\partial_t^m \mathbf{G}_{h,x} (0,p) = \mathbf{F}_m (\mathbf{g}(p),\mathbf{z}_1(p) ,\ldots,\mathbf{z}_{m-1}(p),x) ,
\end{equation*}
and similarly for $\mathbf{F}_{m+1}$.  Computing the directional derivative, recalling that we assume that $z \mapsto L(z,x)$ is a $(m+2)$th degree polynomial,   
\begin{align} \notag 
D_p \mathbf{G}_{h,x}(t,p) \cdot h 
& 
=  
\sum_{k=2}^{m+1} \frac1{(k-1)!} D_p^{k+1} L(\mathbf{g}(p),x) \left( \mathbf{Z}(p,t) \right)^{\otimes (k-1)} \left( D_p \mathbf{Z}(p,t) \cdot h \right)^{\otimes 1}
\\ \notag & \quad 
+ 
\sum_{k=3}^{m+1} \frac1{(k-1)!} D_p^{k+1} L(\mathbf{g}(p),x)  \left( \mathbf{Z}(p,t) \right)^{\otimes (k-1)} \left( \mathbf{z}_{1}(p) \right)^{\otimes 1},
\end{align}
we get by~\eqref{e.diffZ} that
\begin{align} \notag 
\partial_t \mathbf{G}_{h,x}(t,p) 
& = D_p \mathbf{G}_{h,x}(t,p) \cdot h   +  D_p^{3} L(\mathbf{g}(p),x)  \left( \mathbf{Z}(p,t) \right)^{\otimes 1} \left( \mathbf{z}_{1}(p) \right)^{\otimes 1}.
\end{align}
Consequently, we have
\begin{equation*} 
\partial_t^{m+1} \mathbf{G}_{h,x}(0,p) =
D_p \partial_t^{m} \mathbf{G}_{h,x}(0,p) \cdot h   +  D_p^{3} L(\mathbf{g}(p),x)   \left( \mathbf{z}_{1}(p) \right)^{\otimes 1}  \left( \mathbf{z}_{m}(p) \right)^{\otimes 1},
\end{equation*}
which is~\eqref{e.FmrelationApp}, concluding the proof.
\end{proof}

\section{Linearization errors}
\label{app.linerrors}

In this appendix we compute the equation satisfied by a higher-order linearization error and thereby obtain gradient estimates.

\smallskip

\begin{lemma}[{Equation for the linearization error}] 
\label{l.lineq}
Fix $\mathsf{M},R \in (0,\infty)$ and $n\in \N$ with $n \geq 2$. Assume that $p \mapsto L(p,x)$ is $C^{n+1,1}$ for every $x\in \R^d$ and
\begin{equation*} 
\sum_{k=1}^{n} \frac{1}{k!} \|          D_p^{k+1} L \|_{L^\infty(\R^d \times \R^d)} \leq \mathsf{M}. 
\end{equation*}
Suppose that $u,v,w_1,\ldots,w_n \in H^1(B_R)$ satisfy
\begin{equation*} 
\nabla \cdot \left( D_pL(\nabla u,x) - D_pL(\nabla v,x) \right) = 0
\quad \mbox{in} \ B_R
\end{equation*}
and, for each~$m \in \{1,\ldots,n\}$, 
\begin{equation*}
-\nabla \cdot  \left( D^2_pL\left( \nabla u, x \right) \nabla w_m \right) = \nabla \cdot \mathbf{F}_m(\nabla u,\nabla w_1,\ldots,\nabla w_{m-1},x)
\quad\mbox{in} \ B_R,
\end{equation*}
where $ \mathbf{F}_m$ is defined in~\eqref{e.defFm}. Denote $\xi_0 = v -u$ and, for each~$m \in \{1,\ldots,n\}$, 
\begin{equation*} 
\xi_m := v - u - \sum_{k=1}^m \frac{w_k}{k!}.
\end{equation*}
Then there is vector field $\mathbf{E}_{n}$ such that  $\xi_n$ solves 
\begin{equation*} 
- \nabla \cdot \left( D^2 L(\nabla u,\cdot) \nabla \xi_{n}  \right) = \nabla \cdot \mathbf{E}_{n}
\quad \mbox{in} \ B_R
\end{equation*}
and there exists a constant $C(n,\mathsf{M},d)<\infty$ such that 
\begin{align}  \label{e.Enprel}
\left| \mathbf{E}_{n} \right|
&
\leq
 C \sum_{h=0}^{n-1} \left| \nabla \xi_{h} \right| \left( \left| \nabla \xi_{0} \right| + \sum_{i=1}^{n-1} \left| \nabla \frac{w_i}{i!} \right|^{\frac{1}{i}}  \right)^{n-h}.
\end{align}
Furthermore, there exist constants $C(n,\mathsf{M},d)<\infty$, $q(n,d) \in (2,\infty)$ and $\delta(d,\Lambda) \in \left( 0 ,\tfrac 12 \right]$  such that
\begin{align} 
\label{e.EnL2}
\left\| \nabla \xi_n \right\|_{\underline{L}^{2+\delta} \left( B_{\frac12 (1 + 2^{-n})R} \right)} 
& \leq 
 C  \sum_{i=1}^{n} \left( \frac1R\left\|  \xi_{i} - (\xi_{i})_{B_R}\right\|_{\underline{L}^{2} \left( B_{R} \right)} \right)^{\frac{n+1}{i+1}} 
\\  \notag & \qquad  + 
C \left\| \nabla \xi_{0} \right\|_{\underline{L}^{q}(B_R)}^{n+1}
+
C \sum_{i=1}^{n-1}  \left\| \nabla \frac{w_i}{i!} \right\|_{\underline{L}^{q}(B_R)}^{\frac{n+1}{i}}  .
\end{align}
\end{lemma}

\begin{proof}
Throughout the proof we use the notation $s_k = \sum_{j=1}^k \frac{w_j}{j!}$ and $\xi_0 = v-u$, so that $\xi_k = \xi_0 - s_k$.  

\smallskip

\emph{Step 1.}  Recalling that $\mathbf{F}_1 = 0$, we may rewrite
\begin{align} \notag 
\lefteqn{D_p L(\nabla v,x) - D_p L(\nabla u,x)  -
D^2 L(\nabla u,x) \nabla \xi_n    } \quad &
\\ \notag & 
= 
\sum_{k=1}^{n} \frac{1}{k!}   \left(D^2 L(\nabla u,x) \nabla w_k + \mathbf{F}_{k}(\nabla u,\nabla w_1,\ldots,\nabla w_{k-1},x) \right) + \mathbf{E}_{n}
\end{align} 
where we define
\begin{align} \notag 
\mathbf{E}_{n} & :=  \sum_{k=2}^{n} \frac{1}{k!}   \left( D_p^{k+1} L(\nabla u,x) (\nabla \xi_0)^{\otimes k} -  \mathbf{F}_{k}(\nabla u,\nabla w_1,\ldots,\nabla w_{k-1},x)  \right)
\\ \notag & 
\quad 
+ D_p L(\nabla v,x) - \sum_{k=0}^{n} \frac1{k!} D_p^{k+1} L(\nabla u,x) (\nabla \xi_0)^{\otimes k}   .
\end{align}
By the equations of $u$, $v$ and $w_k$, we have that 
\begin{equation*} 
- \nabla \cdot \left( D^2 L(\nabla u,x) \nabla \xi_{n}    \right) = \nabla \cdot \mathbf{E}_{n}.
\end{equation*}
It thus remains to estimate~$\mathbf{E}_{n}$.

\smallskip

\emph{Step 2.} We show that, for $k \in \{2,\ldots n\}$ and $m \in \{k,\ldots,n\}$,
\begin{equation} \label{e.recursiveformulaforL}
(\nabla \xi_0)^{\otimes k}  = \mathbf{S}^{(k)}_{m}  + \mathbf{E}^{(k)}_{m},
\end{equation}
where $\mathbf{S}^{(j)}_{m}$ and  $\mathbf{E}^{(j)}_{m}$ are defined, for $j \in \{2,\ldots,k\}$, recursively as 
\begin{equation} \label{e.Sjm}
\mathbf{S}^{(j)}_{m} :=  \sum_{i = 1}^{m+1-j} \mathbf{S}^{(j-1)}_{m - i}  \otimes \nabla \frac{w_{i}}{i!} 
\end{equation}
and
\begin{equation} \label{e.Ejm}
\mathbf{E}^{(j)}_{m} := \sum_{i = 1}^{m+1-j} \mathbf{E}^{(j-1)}_{m-i} \otimes \nabla \frac{w_{i}}{i!} 
+  (\nabla \xi_0)^{\otimes (j-1)}  \otimes  \nabla \xi_{m-(j-1)} ,
\end{equation}
with
\begin{equation}  \label{e.S1E1mdef}
\mathbf{S}^{(1)}_{i} :=  \nabla s_{i}  
\qquad \mbox{and} \qquad  
\mathbf{E}^{(1)}_{i} :=  \nabla \xi_{i} .
\end{equation}
Indeed, suppose that we have, for $j \in \{1,\ldots,k-1\}$ and $ m \in \{j,\ldots,n\}$, that 
\begin{equation*} 
(\nabla \xi_0)^{\otimes j}  = \mathbf{S}^{(j)}_{m}  + \mathbf{E}^{(j)}_{m}. 
\end{equation*}
This is obviously true for $j=1$. We compute, for $m \in \{k,\ldots,n\}$, 
\begin{align} 
\notag  
(\nabla \xi_0)^{\otimes k}  
&  
=  \sum_{i=1}^{m+1-k} (\nabla \xi_0)^{\otimes (k-1)} \otimes   \nabla \frac{w_i}{i!}  +  (\nabla \xi_0)^{\otimes (k-1)}   \otimes  \nabla \xi_{m-(k-1)} 
\\  \notag & 
= \sum_{i = 1}^{m+1- k} \mathbf{S}^{(k-1)}_{m - i}  \otimes   \nabla \frac{w_{i}}{i!}
+ \sum_{i = 1}^{m+1-k}  \mathbf{E}^{(k-1)}_{m-i} \otimes  \nabla \frac{w_{i}}{i!}  + (\nabla \xi_0)^{\otimes (k-1)}   \otimes  \nabla \xi_{m-(k-1)}  
\\  \notag & 
=  \mathbf{S}^{(k)}_{m}  + \mathbf{E}^{(k)}_{m} ,
\end{align}
which proves the recursive formula~\eqref{e.recursiveformulaforL}.

\smallskip

\emph{Step 3.}
We show that, for $k \in \{2,\ldots, n\}$ there exists a constant $C(n,k,d)<\infty$ such that
\begin{align}  \label{e.Eknest}
\left| \mathbf{E}^{(k)}_{n} \right|  
\leq   
C \sum_{h=1}^{n+1-k} \left| \nabla \xi_{h} \right| \left(\left| \nabla \xi_0 \right|  +  \sum_{i=1}^{n+1-k} \left|\nabla \frac{w_{i}}{i!}\right|^{\frac1i} \right)^{n-h}.
\end{align}
The statement is easy to verify by induction. Indeed, for $m=j=2$ we have by~\eqref{e.Ejm} that  
\begin{equation*} 
\left|\mathbf{E}^{(2)}_{2}   \right| \leq  \left| \nabla \xi_{1} \otimes  \nabla w_{1} +   \nabla \xi_{0} \otimes  \nabla \xi_{1} \right| \leq  
C \sum_{h=1}^{1} \left| \nabla \xi_{h} \right| \left(\left| \nabla \xi_0 \right|  + \left|\nabla w_{1}\right| \right)^{2-h}.
\end{equation*}
Assume then that, for $ m\in \{2,\ldots,n-1\}$ and $j \in \{2,\ldots,m\}$, we have 
\begin{align}  \label{e.Ekmjest}
\left| \mathbf{E}^{(j)}_{m} \right|  
\leq   
C \sum_{h=0}^{m+1-j} \left| \nabla \xi_{h} \right| \left(\left| \nabla \xi_0 \right|  +  \sum_{i=1}^{m+1-j} \left|\nabla \frac{w_{i}}{i!}\right|^{\frac1i} \right)^{m-h}.
\end{align}
By the definition of $\mathbf{E}^{(j)}_{n}$, we have, for $j \in \{2,\ldots,n\}$, that 
\begin{align} \label{e.Ejm.again}
\left|\mathbf{E}^{(j)}_{n}  \right| & \leq C \sum_{i = 1}^{n+1-j} \left| \mathbf{E}^{(j-1)}_{n-i} \right| \left|\nabla \frac{w_{i}}{i!}\right|
+  C \left| \nabla \xi_0 \right|^{j-1}   \left| \nabla \xi_{n-(j-1)} \right| 
\end{align}
By~\eqref{e.Ekmjest}, using Fubini for sums, we obtain, 
\begin{align*} 
\sum_{i = 1}^{n+1-j} \left| \mathbf{E}^{(j-1)}_{n-i} \right| \left|\nabla \frac{w_{i}}{i!}\right| 
&
\leq 
C \sum_{i = 1}^{n+1-j} \sum_{h=0}^{n+2-i-j} \left| \nabla \xi_{h} \right| \left(\left| \nabla \xi_0 \right|  +  \sum_{\ell=1}^{n+2-i-j} \left|\nabla \frac{w_{\ell}}{\ell !}\right|^{\frac1\ell} \right)^{n-i-h} |\nabla \frac{w_i}{i!}|
\\ \notag &
\leq 
C \sum_{h = 0}^{n+1-j} 
\left| \nabla \xi_{h} \right|  
\sum_{i=1}^{n+2-j-h} \left(\left| \nabla \xi_0 \right|  +  \sum_{\ell=1}^{n+1-j} \left|\nabla w_{\ell}\right|^{\frac1\ell}\right)^{n-i-h} |\nabla \frac{w_i}{i!}|
\\ \notag &
\leq 
C \sum_{h = 0}^{n+1-j} \left| \nabla \xi_{h} \right| \left(\left| \nabla \xi_0 \right|  +  \sum_{\ell=1}^{n+1-j} \left|\nabla \frac{w_{\ell}}{\ell!}\right|^{\frac1\ell}\right)^{n-h}.
\end{align*}
This, together with~\eqref{e.Ejm.again}, proves the induction step, and gives also~\eqref{e.Eknest}. 

\smallskip

\emph{Step 4.}
 We show that 
\begin{equation}
\label{e.L0vsSkns}
\left| \sum_{k=2}^{n} \left( \frac1{k!} D_p^{k+1} L(\nabla u,x) \left((\nabla \xi_0)^{\otimes k} -  \mathbf{S}^{(k)}_{n} \right)  \right) \right| 
\leq C \sum_{h=0}^{n-1} \left| \nabla \xi_{h} \right| \left(\left| \nabla \xi_0 \right|  +  \sum_{i=1}^{n-1} \left|\nabla \frac{w_{i}}{i!}\right|^{\frac1i} \right)^{n-h}.
\end{equation}
By the recursive formula we have, for $k \in \{2,\ldots,n\}$, that
\begin{align} 
\notag  
(\nabla \xi_0)^{\otimes k}  
&  
=  \mathbf{S}^{(k)}_{n}  + \mathbf{E}^{(k)}_{n},
\end{align}
and thus~\eqref{e.L0vsSkns} follows by~\eqref{e.Eknest}.
\smallskip

\emph{Step 5.}
 We show that 
\begin{align}
\label{e.obvioushomo}
\sum_{k=2}^{n} \frac1{k!} \left(  D_p^{k+1} L(\nabla u,x) \mathbf{S}^{(k)}_{n} - \mathbf{F}_k(\nabla u,\nabla w_1,\ldots,\nabla w_{k-1},x)  \right) = 0 .
\end{align}

For this, we first abbreviate
\[
\mathbf{F}_k=\mathbf{F}_{k}(\nabla u,\nabla w_1,\ldots,\nabla w_{k-1},x)
\]
and observe that, by definition,
\[
\frac{1}{k!}\mathbf{F}_{k}=\sum_{j \geq 2}\frac{1}{j!}D_{p}^{j+1}L(\nabla u, x)\left(\sum_{i_{1}+\cdots i_{j}= k \, : \, i_{1},\dots, i_{j}\geq 1}\nabla \frac{w_{i_1}}{i_{1}!}\otimes \cdots \otimes \nabla \frac{w_{i_j}}{i_{j}!}\right).
\]
Second, we observe that, by induction on $j\geq 2$, we have
\[
\mathbf{S}^{(j)}_{n}=\sum_{m\leq n}\left(\sum_{i_{1}+\cdots i_{j}= m \, : \, i_{1}, \dots, i_{j}\geq 1}\nabla \frac{w_{i_1}}{i_{1}!}\otimes\cdots \nabla \frac{w_{i_j}}{i_{j}!}\right)
\]
for all $n\geq j$.  Third, by commutativity of addition, we observe that
\[
\sum_{m\leq n}\left(\sum_{j\geq 2}\frac{1}{j!}D_{p}^{j+1}L(\nabla u, x)\left(\sum_{i_{1}+\cdots i_{j}=m\ :\ i_{1}, \dots, i_{j}\geq 1}\nabla \frac{w_{i_1}}{i_{1}!}\otimes\cdots\otimes\nabla \frac{w_{i_j}}{i_{j}!}\right)\right)=
\]
\[
=\sum_{j\geq 2}\frac{1}{j!}D_{p}^{j+1}L(\nabla u, x)\left(\sum_{m\leq n}\left(\sum_{i_{1}+\cdots i_{j}=m\ :\ i_{1}, \dots, i_{j}\geq 1}\nabla \frac{w_{i_1}}{i_{1}!}\otimes\cdots\otimes\nabla \frac{w_{i_j}}{i_{j}!}\right)\right).
\]
Finally, letting $\mathbf{F}_{m}=0$ for $m<2$ and $\mathbf{S}^{(j)}_{n}=0$ for $j>n$ for notational convenience, we note that the above equation may be rewritten as
\[
\sum_{m\leq n}\frac{1}{m!}\mathbf{F}_{m}=\sum_{j\geq 2}\frac{1}{j!}D_{p}^{j+1}L(\nabla u, x)\mathbf{S}^{(j)}_{n},
\]
which is \eqref{e.obvioushomo}.

\smallskip

\emph{Step 6.} Conclusion. 
We have that 
\begin{equation} 
\label{e.ws.restaylor}
\left| D_p L(\nabla v,x) - \sum_{k=0}^{n} \frac1{k!} D_p^{k+1} L(\nabla u,x) (\nabla v - \nabla u)^{\otimes k}  \right|
\leq 
C \left|  \nabla v - \nabla u  \right|^{n+1}   .
\end{equation}
Indeed, by a Taylor expansion, we see that
\begin{equation*} 
\left| D_p L(z_0 + z,x) - \sum_{k=0}^{n} \frac1{k!} D_p^{k+1} L(z_0,x) z^{\otimes k}  \right| \leq \left[ D_p^{n+1} L(\cdot,x)\right]_{C^{0,1}\left( B_{|z|}(z_0) \right)}  \frac{\left| z \right|^{n+1} }{(k+1)!}  .
\end{equation*}
Applying this with $z_0 = \nabla u$ and $z = \nabla v - \nabla u$ gives~\eqref{e.ws.restaylor}. Combining this with the previous steps yields the desired estimate~\eqref{e.Enprel} for $\mathbf{E}_{n}$. Finally, by the H\"older and Young inequalities, we get, for all $q \in [2,\infty)$ and $r \in (0,R]$,
\begin{multline}  \label{e.EnLq}
\left\| \mathbf{E}_{n} \right\|_{\underline{L}^{q} \left( B_{r} \right)}  
 \leq
\sum_{h=1}^{n-1} \left\| \nabla \xi_{h} \right\|_{\underline{L}^{q+\delta} \left( B_{r} \right)} \left\| \left| \nabla \xi_{0} \right| + \sum_{i=1}^{n-1} \left| \nabla \frac{w_i}{i!} \right|^{\frac{1}{i}}  \right\|_{\underline{L}^{\frac{nq(q +\delta)}{\delta}} \left( B_{r} \right)}^{n-h}
\\ 
\leq C
\left(\sum_{h=1}^{n-1} \left\| \nabla \xi_{h} \right\|_{\underline{L}^{q+\delta} \left( B_{R} \right)}^{\frac{n+1}{h+1}} + 
\left\| \nabla \xi_{0}  \right\|_{\underline{L}^{\frac{nq(q +\delta)}{\delta}} \left( B_{R} \right)}^{n+1} +
 \sum_{i=1}^{n-1}  \left\| \nabla \frac{w_i}{i!}  \right\|_{\underline{L}^{\frac{nq(q +\delta)}{\delta}} \left( B_{R} \right)}^{\frac{n+1}{i}}\right),
\end{multline}
Let $\delta_0$ be the Meyers exponent  corresponding $\Lambda$. Let 
\begin{equation*} 
q_h := q +  \frac12 \delta_0 + \frac12 \frac{n+1-h}{n+1} \delta_0 
\quad \mbox{and} \quad 
q :=  \frac{16}{\delta_0} n(n+1) . 
\end{equation*}
Set also $R_h := \frac12 (1+2^{-h}) R$.  With this notation the previous display yields, by H\"older's inequality, that 
\begin{equation*} 
\left\| \mathbf{E}_{m} \right\|_{\underline{L}^{q_m} \left( B_{R_m} \right)}   
\leq C
\left(\sum_{h=1}^{m-1} \left\| \nabla \xi_{h} \right\|_{\underline{L}^{q_h} \left( B_{R_h} \right)}^{\frac{n+1}{h+1}} + 
\left\| \nabla \xi_{0}  \right\|_{\underline{L}^{q} \left( B_{R} \right)}^{n+1} +
 \sum_{i=1}^{n-1}  \left\| \nabla \frac{w_i}{i!}  \right\|_{\underline{L}^{q} \left( B_{R} \right)}^{\frac{n+1}{i}}\right),
\end{equation*}
Now~\eqref{e.EnL2} follows by the Caccioppoli estimate, concluding the proof.
\end{proof}

\section{Regularity for constant coefficient linearized equations} 
\label{s.appendixconstant}

In this appendix we prove a lemma tracking down the regularity of a solution $(\overline{w}_1,\ldots,\overline{w}_n)$ of the linearized system in the case that~$\overline{L}$ is a smooth, constant-coefficient Lagrangian.

\smallskip

Throughout we fix $n \in \N_0$, $\Lambda \in [1,\infty)$, $\beta \in (0,1)$, and assume that $\overline{L}$ satisfies 
\begin{equation} \label{e.appC.barL1}
I_d \leq D_p^2 \overline{L} \leq \Lambda I_d 
\end{equation}
and for all $\mathsf{M}_0 \in [1,\infty)$ there is $C(\mathsf{M}_0,\beta,d)<\infty$ such that 
\begin{equation} \label{e.appC.barL2}
\left\| D^2 \overline{L} \right\|_{C^{n,\beta}(B_{ \mathsf{M}_0 })}  \leq  C. 
\end{equation}

\smallskip

\begin{lemma}[{Regularity of $\overline{w}_{m}$}]
\label{l.appC.C1alphabarwn}
Let $\eta \in [\tfrac12,1)$, $\mathsf{M} \in [1,\infty)$ and $R \in (0,\infty)$. Assume that $\overline{L}$ satisfies~\eqref{e.appC.barL1} and~\eqref{e.appC.barL2}. Let $\overline{u},\overline{w}_1,\ldots,\overline{w}_n$ solve the equations, for $m \in \{1,\ldots,n+1\}$,
\begin{equation}
\label{e.appC.eqs}
\left\{ 
\begin{aligned}
& -\nabla \cdot  \left( D_p\overline{L}\left( \nabla \overline{u} \right) \right) =  0 & \mbox{in} & \ B_{R}, \\
& -\nabla \cdot  \left( D^2_p\overline{L}\left( \nabla \overline{u}\right) \nabla  \overline{w}_m \right) = \nabla \cdot \left( \overline{\mathbf{F}}_m(\nabla \overline{u},\nabla \overline{w}_1,\ldots,\nabla \overline{w}_{m-1})\right) & \mbox{in} & \ B_{R},
\end{aligned} 
\right. 
\end{equation}
where $\overline{\mathbf{F}}_m$ has been defined in~\eqref{e.defbarFm} and $\overline{u}$ satisfies 
\begin{equation}  \label{e.appC.norm1pre}
\frac1R  \left\|  \overline{u} - ( \overline{u})_{B_{R}} \right\|_{\underline{L}^2\left( B_{R} \right)} 
 \leq  \mathsf{M}.
\end{equation}
Then, for $m \in \{1,\ldots,n+1\}$, there exists a constant $C(m,\eta,\mathsf{M},\beta,d,\Lambda)<\infty$ such that
 \begin{equation} \label{e.appC.wres1} 
\left\| \nabla \overline{w}_m  \right\|_{L^\infty(B_{\eta R})} 
\leq 
C  \sum_{i=1}^m \left( \frac1R \left\|  \overline{w}_{i} - \left( \overline{w}_{i} \right)_{B_R}   \right\|_{\underline{L}^2(B_R)} \right)^{\frac mi} 
\end{equation}
Moreover, letting  
\begin{equation}  \label{e.deltacondonbaru}
\delta\in (0,\infty) \cap  \left[  \left( \frac{1}{R} \inf_{\ell \in \mathcal{P}_1} \left\|  \bar{u} - \ell  \right\|_{L^2(B_{R})} \right)^\beta, \infty \right), 
\end{equation}
we have, for $m \in \{1,\ldots,n+1\}$, that 
\begin{align} \label{e.appC.C1alphabarw}
\lefteqn{  R^\beta \left[ \nabla \overline{w}_{m} \right]_{C^{0,\beta}(B_{\eta R})}} \quad & 
\\ \notag 
& \leq 
C \delta \sum_{i=1}^m \left( \frac1{\delta R} \inf_{\ell \in \mathcal{P}_1} \left\|  \overline{w}_{i}  - \ell \right\|_{\underline{L}^2(B_R)} \right)^{\frac mi}  
+  C \delta \sum_{i=1}^{m-1} \left( \frac1R \left\|  \overline{w}_{i} - \left( \overline{w}_{i} \right)_{B_R}   \right\|_{\underline{L}^2(B_R)} \right)^{\frac mi}   
\\ \notag & \quad  
+  C \left( \frac{1}{R} \inf_{\ell \in \mathcal{P}_1} \left\|  \bar{u} - \ell \right\|_{\underline{L}^2(B_R)} \right)^\beta \sum_{i=1}^{m} \left( \frac1R \left\|  \overline{w}_{i} - \left( \overline{w}_{i} \right)_{B_R}   \right\|_{\underline{L}^2(B_R)} \right)^{\frac mi}  .
\end{align}
and, for $m \in \{1,\ldots,n\}$ and $k \in \{1,\ldots,n+1-m\}$, 
\begin{align} \label{e.appC.C1alphabarw2}
\lefteqn{ R^{k}  \left\| \nabla^{k+1} \overline{w}_{m} \right\|_{L^\infty(B_{\eta R})} +   R^{k +\beta} \left[ \nabla^{k+1} \overline{w}_{m} \right]_{C^{0,\beta}(B_{\eta R})}} \quad & 
\\ \notag 
& \leq 
C \delta \sum_{i=1}^m \left( \frac1{\delta R} \inf_{\ell \in \mathcal{P}_1} \left\|  \overline{w}_{i}  - \ell \right\|_{\underline{L}^2(B_R)} \right)^{\frac mi}  
+  C \delta \sum_{i=1}^{m-1} \left( \frac1R \left\|  \overline{w}_{i} - \left( \overline{w}_{i} \right)_{B_R}   \right\|_{\underline{L}^2(B_R)} \right)^{\frac mi}   
\\ \notag & \quad  
+  C \left( \frac{1}{R} \inf_{\ell \in \mathcal{P}_1} \left\|  \bar{u} - \ell \right\|_{\underline{L}^2(B_R)} \right)^\beta \sum_{i=1}^{m} \left( \frac1R \left\|  \overline{w}_{i} - \left( \overline{w}_{i} \right)_{B_R}   \right\|_{\underline{L}^2(B_R)} \right)^{\frac mi}  .
\end{align}
\end{lemma}

Notice that by~\eqref{e.appC.norm1pre} we may always take $\delta =  \mathsf{M}^{\beta}$ in~\eqref{e.deltacondonbaru}. When applying the result in practice, we typically take $\delta$ to be very small.

\begin{proof}
Fix $m \in \{1,\ldots,n+1\}$, $\eta \in [\tfrac12,1)$, $\mathsf{M} \in [1,\infty)$. Let $\overline{u},\overline{w}_1,\ldots,\overline{w}_n$ solve~\eqref{e.appC.eqs} and assume~\eqref{e.appC.norm1pre}. 
Fix also $\delta$ as in~\eqref{e.deltacondonbaru}. 

\smallskip

Throughout the proof we denote, for~$\theta \in (0,\infty)$,
\begin{align*} 
\mathsf{E}_m ^{(\theta)} & :=  
 \theta \sum_{i=1}^m \left( \frac1{\theta R} \inf_{\ell \in \mathcal{P}_1} \left\|  \overline{w}_{i}  - \ell \right\|_{\underline{L}^2(B_R)} \right)^{\frac mi}  
+  \theta \sum_{i=1}^{m-1} \left( \frac1R \left\|  \overline{w}_{i} - \left( \overline{w}_{i} \right)_{B_R}   \right\|_{\underline{L}^2(B_R)} \right)^{\frac mi}   
\\ \notag & \quad  
+  \left( \frac{1}{R} \inf_{\ell \in \mathcal{P}_1} \left\|  \bar{u} - \ell \right\|_{\underline{L}^2(B_R)} \right)^\beta \sum_{i=1}^{m} \left( \frac1R \left\|  \overline{w}_{i} - \left( \overline{w}_{i} \right)_{B_R}   \right\|_{\underline{L}^2(B_R)} \right)^{\frac mi} .
\end{align*}
We also denote
\begin{equation*} 
\overline{\f}_{m} :=  \overline{\mathbf{F}}_{m}( \nabla \bar{u},\nabla \overline{w}_1,\ldots,\nabla \overline{w}_{m-1})
\end{equation*}
and, for $j \in \N$, 
\begin{equation*} 
R_j := (\eta + 2^{-j}(1-\eta))R  \quad \mbox{and} \quad  r_j :=  2^{-j-8}(1-\eta) R   . 
\end{equation*}
Below we denote by $C$ a constant depending only on parameters $(m,\eta,\mathsf{M},\beta,d,\Lambda)$. It may change from line to line.

\smallskip

\emph{Step 1}.
Basic properties of $\overline{u}$.  In view of~\cite[Proposition A.1]{AFK}, assumption~\eqref{e.appC.barL1} and normalization~\eqref{e.appC.norm1pre} imply that there exists a constant $\mathsf{M_0}(\eta,\mathsf{M},d,\Lambda)<\infty$ such that
\begin{equation}  \label{e.appC.norm1}
 \left\|  \nabla \overline{u}  \right\|_{L^\infty\left( B_{\frac13 (2+\eta)R} \right)}  
  \leq \mathsf{M_0} .
\end{equation}
Therefore,~\eqref{e.appC.barL2} is applicable in $B_{\frac13 (2+\eta)R}$, and we obtain by~\cite[Proposition A.1]{AFK} that
\begin{equation}  \label{e.appC.norm3}
 R^2 \left\|  \nabla^2 \overline{u}  \right\|_{L^\infty\left( B_{\frac12(1+\eta)R} \right)} 
 \leq 
 C \inf_{\ell \in \mathcal{P}_1} \left\|  \bar{u} - \ell  \right\|_{L^2(B_{R})}.
\end{equation}
We also define
\begin{equation} \label{e.appC.b}
\b(x) := D_p^2 \overline{L} \left( \nabla \bar u \right).
\end{equation}
We have by~\eqref{e.appC.norm3} and~\eqref{e.appC.norm1pre} that
\begin{equation} \label{e.appC.bbound}
I_d \leq \b(x)  \leq \Lambda I_d   
\quad \mbox{and} \quad 
R \left\| \nabla \b \right\|_{L^{\infty}(B_{\frac12 (1+\eta)R})} \leq  \frac{C}{R} \inf_{\ell \in \mathcal{P}_1} \left\|  \bar{u} - \ell  \right\|_{L^2(B_{R})} \leq C \delta^{1/\beta}.
\end{equation}
Notice also that, by~\eqref{e.appC.bbound}, we have
\begin{equation}  \label{e.appCHm1bnd}
\mathsf{E}_m ^{(1)} \leq C \sum_{i=1}^m \left( \frac1R \left\|  \overline{w}_{i} - \left( \overline{w}_{i} \right)_{B_R}   \right\|_{\underline{L}^2(B_R)} \right)^{\frac mi} .
\end{equation}

\smallskip

\emph{Step 2}. Induction assumption on degree $m$.  We assume inductively that, for $j \in \{1,\ldots,m-1\}$ there exists a constant $\mathsf{K}_{j}(\eta,\mathsf{M},\mathsf{K}_{0},d,\Lambda)<\infty$ such that 
 \begin{equation}  \label{e.appC.inductionnabla2w}
R_j^\beta \left[ \nabla \overline{w}_{j} \right]_{C^{0,\beta}(B_{R_j})}  \leq  \mathsf{K}_{j} \mathsf{E}_{j}^{(\delta)} 
\end{equation}
 and 
 \begin{equation} \label{e.appC.inductionnablaw} 
\left\| \nabla \overline{w}_{j} \right\|_{\underline{L}^\infty(B_{R_{j}})} 
\leq 
\mathsf{K}_{j} \sum_{i=1}^j \left( \frac1R \left\|  \overline{w}_{i} - \left( \overline{w}_{i} \right)_{B_R}   \right\|_{\underline{L}^2(B_R)} \right)^{\frac ji} .
\end{equation}
Notice that the case $m=2$ has been established in~\cite[Proposition A.1]{AFK}.  

\smallskip

Throughout the next steps of the proof, we let constants $C$ depend on parameters $(\{\mathsf{K}_{i}\}_{i=1}^{m-1}, m,\eta,\mathsf{M},\beta,\mathsf{K}_{0},d,\Lambda)$, and they may change from line to line. 

\smallskip

\emph{Step 3}. Bounds on $\f$. We show that under induction assumptions~\eqref{e.appC.inductionnabla2w} and~\eqref{e.appC.inductionnablaw}, we have that 
\begin{align}  \label{e.appC.fm}
 \left\| \overline{\f}_{m} \right\|_{L^\infty(B_{R_{m-1}})} 
 & 
\leq C \sum_{i=1}^{m-1}  \left( \frac1R \left\|  \overline{w}_{i} - \left( \overline{w}_{i} \right)_{B_R}   \right\|_{\underline{L}^2(B_R)} \right)^{\frac mi}   
\end{align}
and, for $r \in (0,r_m]$ and $y \in B_{R_m}$, defining 
\begin{equation}  \label{e.appC.fmyr}
\overline{\f}_{m,y,r} := \overline{\mathbf{F}}_m \left(  (\nabla \overline{u} )_{B_r(y)} ,  (\nabla \overline{w}_1 )_{B_r(y)},\ldots, (\nabla \overline{w}_{m-1} )_{B_r(y)}\right)    ,
\end{equation}
we have that
\begin{equation}  \label{e.appC.fm2}
\left\|  \overline{\f}_{m} - \overline{\f}_{m,y,r} \right\|_{L^\infty(B_r(y))} \leq C \left(\frac{r}{R}\right)^\beta \mathsf{E}_m ^{(\delta)} .
\end{equation}

\smallskip

To show~\eqref{e.appC.fm}, we have by~\eqref{e.appC.inductionnablaw} that 
\begin{align}  \label{e.appC.fmpre}
  \sum_{i=1}^{m-1} \left\| \nabla \overline{w}_{i} \right\|_{L^\infty(B_{R_{m-1}})}^{\frac mi}   
 & 
\leq  
 C \sum_{i=1}^{m-1} \mathsf{K}_{i}^m  \sum_{j=1}^{i} \left( \frac1R \left\|  \overline{w}_{j} - \left( \overline{w}_{j} \right)_{B_R}   \right\|_{\underline{L}^2(B_R)} \right)^{\frac{m}{j}}  
\\ \notag &
\leq  
C 
\sum_{i=1}^{m-1}  \left( \frac1R \left\|  \overline{w}_{i} - \left( \overline{w}_{i} \right)_{B_R}   \right\|_{\underline{L}^2(B_R)} \right)^{\frac mi}  ,
\end{align}
which yields~\eqref{e.appC.fm} by~\eqref{e.Fmbasic}. 

\smallskip

To show~\eqref{e.appC.fm2}, using H\"older regularity of $\overline{\f}_{m}$ with respect to $\nabla \overline{u}$ variable, similarly to~\eqref{e.Fmbasic3}, gives us
\begin{align*} 
\left| \overline{\f}_{m} -  \overline{\f}_{m,r,y} \right| 
& 
\leq 
C  \left( r^\beta   \left\| \nabla^2 \overline{u} \right\|_{L^\infty (B_{r}(y)) }^\beta  +\delta \left(\frac rR\right)^\beta  \right)\sum_{i=1}^{m-1} \left( \left| \nabla \overline{w}_{j} \right| + \left| (\nabla \overline{w}_{j})_{B_r(y)} \right| \right) ^{\frac{m}{i}}  
\\ \notag & \quad
+  C \delta \left(\frac rR\right)^\beta \sum_{i=1}^{m-1} \left( \delta^{-1} \left(\frac rR\right)^{-\beta}  \left| \nabla \overline{w}_{i} - (\nabla \overline{w}_{i})_{B_r(y)}\right| \right)^{\frac{m}{i}}.     
\end{align*}
Applying~\eqref{e.appC.inductionnabla2w} and~\eqref{e.deltacondonbaru} yields that 
\begin{equation*} 
\delta \left(  \delta^{-1} \left(\frac rR\right)^{-\beta}  \left| \nabla \overline{w}_{i} - (\nabla \overline{w}_{i})_{B_r(y)}\right| \right)^{\frac mi}   \leq 
 C  \delta \left( \delta^{-1} \mathsf{E}_i ^{(\delta)} \right)^{\frac mi}   \leq 
  C  \mathsf{E}_m ^{(\delta)} .
\end{equation*}
Thus~\eqref{e.appC.fm2} follows by~\eqref{e.appC.bbound} and~\eqref{e.appC.fmpre}.

\smallskip

\emph{Step 4}. Caccioppoli estimate. We show that under induction assumptions~\eqref{e.appC.inductionnabla2w} and~\eqref{e.appC.inductionnablaw}, we have that, for all $y \in B_{R_m + r_m}$ and $r \in (0,r_m]$,
\begin{multline}  \label{e.appC.cacc}
\left\|  \nabla \overline{w}_{m} - \nabla \ell  \right\|_{\underline{L}^2 \left( B_{r/2}(y) \right)} \leq \frac{C}{r} \left\|  \overline{w}_{m} - \ell  \right\|_{\underline{L}^2 \left( B_{r}(y) \right)}   
\\
+  C r \left\| \nabla^2 \overline{u} \right\|_{L^\infty (B_{r}(y))  } ( \left| \nabla \ell \right|  \wedge  \left\|  \nabla \overline{w}_{m}   \right\|_{\underline{L}^2 \left( B_{r}(y) \right)}  )
+ 
C \left(\frac{r}{R}\right)^\beta \mathsf{E}_m ^{(\delta)}  .
\end{multline}
Since $\overline{w}_{m}$ solves the equations, for $\f_{m,y,r}$ defined in~\eqref{e.appC.fmyr} and any affine function $\ell$,
\begin{equation*} 
-\nabla \cdot  \left( \b  \nabla ( \overline{w}_{m} -\ell) \right) = \nabla \cdot \left(( \b - \b(y))  \nabla  \ell + \mathbf{f}_{m}  -  \mathbf{f}_{m,y,r} \right),
\end{equation*}
and
\begin{equation*} 
-\nabla \cdot  \left( \b(y)  \nabla  ( \overline{w}_{m} - \ell) \right) = \nabla \cdot \left( (\b - \b(y)) \nabla  \overline{w}_{m} + \mathbf{f}_{m} - \mathbf{f}_{m,y,r} \right),
\end{equation*}
we obtain~\eqref{e.appC.cacc} simply by testing and~\eqref{e.appC.fm2}.

\smallskip

\emph{Step 5.}
Induction assumption on the scale. We now assume that we find $\ep \in (0,1]$, a constant $\mathsf{C}_\ep$, and $r^* \in (0, \ep r_m]$ such that
\begin{equation}  \label{e.appC.indscale}
\sup_{y \in B_{R_m}}\sup_{t \in [r^*, \ep r_m]}\left\|  \nabla \overline{w}_{m}   \right\|_{\underline{L}^2 \left( B_{t} (y) \right)} \leq \mathsf{C}_\ep \sum_{i=1}^m \left( \frac1R \left\|  \overline{w}_{i} - \left( \overline{w}_{i} \right)_{B_R}   \right\|_{\underline{L}^2(B_R)} \right)^{\frac mi} .
\end{equation}
Notice that, by the Caccioppoli estimate~\eqref{e.appC.cacc} and~\eqref{e.appCHm1bnd}, we have, for any $\ep \in (0,1]$, that
\begin{equation}  \label{e.appC.indscaleinit}
\left\|  \nabla \overline{w}_{m} \right\|_{\underline{L}^2 \left( B_{\ep r_m}(y) \right)} \leq C_\ep  \sum_{i=1}^j \left( \frac1R \left\|  \overline{w}_{i} - \left( \overline{w}_{i} \right)_{B_R}   \right\|_{\underline{L}^2(B_R)} \right)^{\frac ji} ,
\end{equation}
implying that~\eqref{e.appC.indscale} is valid for $r^* = \ep r_m$ provided that $\mathsf{C}_\ep \geq C_\ep$.  
 
\smallskip

\emph{Step 6.} We verify that~\eqref{e.appC.inductionnablaw} is true for $j=m$. This gives us also~\eqref{e.appC.wres1}. Actually, we prove that if~\eqref{e.appC.indscale} is valid for some $r^* \in (0, \ep r_m]$, then it remains valid for $\tfrac12 r^*$ instead. This proves, by induction, that we may take any $r^* \in (0, \ep r_m]$ in~\eqref{e.appC.indscale}. In particular, we obtain~\eqref{e.appC.inductionnablaw} for $j=m$.

\smallskip

Fix $y \in B_{R_m}$.  Rewriting the equation of $ \overline{w}_{m}$ as before, for $r \in \left(0 ,r_m \right]$, 
\begin{equation*} 
-\nabla \cdot  \left( \b(y)  \nabla  \overline{w}_{m}  \right) = \nabla \cdot \left( (\b - \b(y)) \nabla  \overline{w}_{m} + \left(\overline{\f}_m- \overline{\f}_{m,y,r}   \right)  \right),
\end{equation*}
where $\f_{m,y,r}$ defined in~\eqref{e.appC.fmyr}, and consequently solving 
\begin{equation*}
\left\{ 
\begin{aligned}
& -\nabla \cdot \left( \b(y) \nabla \overline{w}_{m,y,r}\right) = 0  & \mbox{in} & \ B_{r}(y), \\
& \overline{w}_{m,y,r} = \overline{w}_{m}  & \mbox{on} & \ \partial B_{r}(y ),
\end{aligned} 
\right. 
\end{equation*}
we obtain by testing and~\eqref{e.appC.fm2} that 
\begin{equation*} 
\left\| \nabla \overline{w}_{m,y,r} - \nabla  \overline{w}_{m}  \right\|_{\underline{L}^2 \left( B_r(y)\right)} 
\\ 
\leq C  r   \left\| \nabla^2 \overline{u} \right\|_{L^\infty (B_{r}(y)) } \left\| \nabla   \overline{w}_{m}  \right\|_{\underline{L}^2 \left( B_r(y)\right)} 
+C \left(\frac{r}{R}\right)^\beta \mathsf{E}_m ^{(\delta)} .
\end{equation*}
In particular, we get by~\eqref{e.appC.indscale} for $r \in [r^*, \ep r_m]$ that        
\begin{equation*} 
\left\| \nabla \overline{w}_{m,y,r} - \nabla  \overline{w}_{m}  \right\|_{\underline{L}^2 \left( B_r(y)\right)} \leq C \left(\frac rR\right)^{\beta} \left( \mathsf{C}_\ep   \ep^{1-\beta}  + 1 \right) \mathsf{E}_m ^{(\delta)}. 
\end{equation*}
By decay estimate for harmonic functions we have for small enough $\theta(\beta,d,\Lambda) \in \left(0,\frac12\right]$ that 
\begin{equation*} 
\left\| \nabla \overline{w}_{m,y,r} - (\nabla \overline{w}_{m,y,r})_{B_{\theta r} (y) } \right\|_{\underline{L}^2 \left( B_{\theta r} (y) \right)} 
\leq 
\frac12 \theta^\beta \left\| \nabla \overline{w}_{m,y,r} - (\nabla \overline{w}_{m,y,r})_{B_{r} (y) }  \right\|_{\underline{L}^2 \left( B_{r} (y) \right)}.
\end{equation*}
 Therefore, by the triangle inequality, we get 
 \begin{multline} \notag 
\left\| \nabla \overline{w}_{m} - (\nabla \overline{w}_{m})_{B_{\theta r}(y)} \right\|_{\underline{L}^2 \left( B_{\theta r}(y) \right)} 
\\
 \leq 
\frac12 \theta^\beta \left\| \nabla \overline{w}_{m} - (\nabla \overline{w}_{m})_{B_{r}(y)}  \right\|_{\underline{L}^2 \left( B_{r}(y) \right)}
+ C \left(\frac rR\right)^{\beta} \left( \mathsf{C}_\ep   \ep^{1-\beta}  + 1 \right) \mathsf{E}_m ^{(\delta)}. 
\end{multline}
By an iteration argument we thus obtain that, for $r \in [\theta r^*, \ep r_m]$ 
\begin{multline}  \label{e.appC.almostthere}
\left(\frac {r}{\ep r_m}\right)^{-\beta}  \left\| \nabla \overline{w}_{m} - (\nabla \overline{w}_{m})_{B_{r}(y)} \right\|_{\underline{L}^2 \left( B_{r}(y)\right)}  
\\ \leq 
C \left\| \nabla \overline{w}_{m} - (\nabla \overline{w}_{m})_{B_{\ep r_m} (y)}  \right\|_{\underline{L}^2 \left( B_{\ep r_m}(y) \right)} 
+ C \left(  \ep \mathsf{C}_\ep    + 1 \right) \mathsf{E}_m ^{(\delta)}. 
\end{multline}
Letting $r \in [\theta r^*, \ep r_m]$ and $n\in \N_0$ be such that $r \in (\theta^{n+1} \ep r_m, \theta^n \ep r_m]$, we obtain by the triangle inequality that 
\begin{align} \notag 
\left\| \nabla \overline{w}_{m} \right\|_{\underline{L}^2 \left( B_{r}(y) \right)} & \leq C 
\left\| \nabla \overline{w}_{m} \right\|_{\underline{L}^2 \left( B_{\theta^n \ep r_m}(y) \right)} 
\\       \notag &
\leq 
C \left\| \nabla \overline{w}_{m} - (\nabla \overline{w}_{m})_{B_{\theta^{n} \ep r_m}(y)} \right\|_{\underline{L}^2 \left( B_{\theta^n \ep r_m}(y) \right)} 
\\      \notag & \quad
+C \left|  (\nabla \overline{w}_{m})_{B_{\ep r_m}(y)}  \right| 
+ C \sum_{i=1}^n \left| (\nabla \overline{w}_{m})_{B_{\theta^{i} \ep r_m}(y) }  - (\nabla \overline{w}_{m})_{B_{\theta^{i-1} \ep r_m}(y) }\right|.
\end{align}
Thus, by the previous two displays and~\eqref{e.appC.indscaleinit}, we obtain, for $r \in [\theta r^*, \ep r_m]$, that 
\begin{equation*} 
\left\| \nabla \overline{w}_{m} \right\|_{\underline{L}^2 \left( B_{r}(y) \right)} \leq \left(C_\ep + C \ep \mathsf{C}_\ep  \right)
\sum_{i=1}^m \left( \frac1R \left\|  \overline{w}_{i} - \left( \overline{w}_{i} \right)_{B_R}   \right\|_{\underline{L}^2(B_R)} \right)^{\frac mi} .
\end{equation*}
We first take $\ep$ so small that $C\ep = \frac12$ and then choose  $\mathsf{C}_\ep \geq 2C_\ep$. All in all, we have proved that 
\begin{equation*} 
\sup_{t \in [\theta r^*, \ep r_m]}\left\|  \nabla \overline{w}_{m}   \right\|_{\underline{L}^2 \left( B_{t} (y) \right)} \leq \mathsf{C}_\ep \sum_{i=1}^m \left( \frac1R \left\|  \overline{w}_{i} - \left( \overline{w}_{i} \right)_{B_R}   \right\|_{\underline{L}^2(B_R)} \right)^{\frac mi} ,
\end{equation*}
which implies that~\eqref{e.appC.indscale} is valid for $\frac12 r^*$ instead of $r^*$, which was to be shown. 

\smallskip

\emph{Step 7.} 
We now prove that~\eqref{e.appC.inductionnabla2w} is valid for $j=m$, giving also~\eqref{e.appC.C1alphabarw}. An application of the Caccioppoli estimate~\eqref{e.appC.cacc}, together with~\eqref{e.appC.wres1}, which was proved in Step 6 above, we have that, by giving up volume factors,
\begin{equation*} 
\left\| \nabla \overline{w}_{m} - (\nabla \overline{w}_{m})_{B_{\ep r_m} (y)}  \right\|_{\underline{L}^2 \left( B_{\ep r_m}(y) \right)}  
\leq C \mathsf{E}_m ^{(\delta)} .
\end{equation*}
Therefore,~\eqref{e.appC.almostthere} yields, for all $y \in B_{\eta R}$, that
\begin{equation*} 
\sup_{r \in (0,r_m)} \left(\frac rR\right)^{-\beta}  \left\| \nabla \overline{w}_{m} - (\nabla \overline{w}_{m})_{B_r(y)} \right\|_{\underline{L}^2 \left( B_{r}(y) \right) } \leq C \mathsf{E}_m ^{(\delta)}.
\end{equation*}
This yields, via telescoping summation as in Step 6, that
\begin{equation*} 
\left| \nabla \overline{w}_{m}(y)  - (\nabla \overline{w}_{m})_{B_r(y)} \right| \leq C \left(\frac rR\right)^{\beta}  \mathsf{E}_m ^{(\delta)}.
\end{equation*}
Thus, if, on the one hand, $r = |x-y| \in (0,r_m]$, we get by the above two displays that 
\begin{align} \notag 
\left| \nabla \overline{w}_{m}(y) -  \nabla \overline{w}_{m}(x)   \right|
& \leq 
\left| \nabla \overline{w}_{m}(y)  - (\nabla \overline{w}_{m})_{B_r(y)} \right| 
+\left| \nabla \overline{w}_{m}(x)  - (\nabla \overline{w}_{m})_{B_r(x)} \right|
\\ \notag &
\quad + C \left\| \nabla \overline{w}_{m} - (\nabla \overline{w}_{m})_{B_{2r}(y)} \right\|_{\underline{L}^2 \left( B_{4r}(y) \right) }
\\ \notag & 
\leq C \left(\frac rR\right)^{\beta}  \mathsf{E}_m ^{(\delta)}.
\end{align}
If, on the other hand $ |x-y|> r_m$, we get
\begin{equation*} 
\left| \nabla \overline{w}_{m}(y) -  \nabla \overline{w}_{m}(x)   \right| \leq C  \mathsf{E}_m ^{(\delta)}
\end{equation*}
by noticing that
\begin{equation*} 
\left|  (\nabla \overline{w}_{m})_{B_{r_m}(y)}   - (\nabla \overline{w}_{m})_{B_{r_m}(x)} \right| \leq C
\left\| \nabla \overline{w}_{m} - (\nabla \overline{w}_{m})_{B_{R_m + r_m }(y)} \right\|_{\underline{L}^2 \left( B_{R_m + r_m}(y) \right) } ,
\end{equation*}
and applying once more~\eqref{e.appC.cacc}. Thus we have proved~\eqref{e.appC.C1alphabarw}. 

\smallskip

\emph{Step 8.} 
We finally sketch the proof of~\eqref{e.appC.C1alphabarw2}. Since it is very similar to the above reasoning, we will omit most of the details. We prove the statement by using induction in $m$ and in $k$. First, we observe that by differentiation we see that $\partial_{x_j}^k \overline{u}$ satisfies the equation 
\begin{equation*} 
-\nabla \cdot (\b \nabla \partial_{x_j}^k u) = \nabla \cdot \mathbf{F}_k(\nabla \overline{u}, \nabla \partial_{x_j} \overline{u},\ldots, \nabla \partial_{x_j}^{k-1} \overline{u})  .
\end{equation*}
Thus we can apply~\eqref{e.appC.wres1} and~\eqref{e.appC.C1alphabarw2} for $\overline{w}_k = \partial_{x_j}^k u$ recursively and obtain, by polarization as in Lemma~\ref{l.polarization}, that, for every $k \in \{1,\ldots,n+1\}$ and $\eta \in \left[\tfrac12,1\right)$, there is a constant $C(k,\eta,\mathsf{M},\beta,d,\Lambda)$ such that 
\begin{equation}  \label{e.baruhigherreg}
R^{k}\left\| \nabla^{k+1} \overline{u} \right\|_{L^\infty(B_{(2+\eta)R/3})}
+
R^{k+\beta} \left[ \nabla^{k+1} \overline{u} \right]_{C^{0,\beta}(B_{(2+\eta)R/3})} \leq C \frac{1}{R} \inf_{\ell \in \mathcal{P}_1} \left\|  \bar{u} - \ell  \right\|_{L^2(B_{R})} .
\end{equation}

\smallskip

Next, $\overline{w}_1$ solves $-\nabla \cdot  \left( \b \nabla  \overline{w}_1 \right) = 0$ and $z \mapsto \b(z) = D^2_p\overline{L}(\nabla u(z))$ is in $C^{n,\beta}$ by~\eqref{e.baruhigherreg}, we may differentiate the equation at most $n$ times and obtain that $\overline{w}_1 \in C^{n+1,\beta}$ and it is also straightforward to show that $\overline{w}_1$ satisfies~\eqref{e.appC.C1alphabarw2} for $m=1$; this is just classical Schauder theory. 

\smallskip

We than assume that~\eqref{e.appC.C1alphabarw2} is valid for every $m \in \{1,\ldots,M\}$ and $k \in \{1,\ldots,n+1-m\}$  with some $M \in \{1,\ldots,n-1\}$. We then show that it continues to hold for $m = M+1$ and $k \in \{1,\ldots,n+M\}$  as well. Now 
\begin{equation*} 
-\nabla \cdot (\b \nabla \overline{w}_{M+1} ) = -\nabla \cdot \overline{F}_m(\nabla \overline{u}, \overline{w}_{1},\ldots, \overline{w}_{M}).
\end{equation*}
Recalling that $(h_1,\ldots,h_{m-1}) \mapsto \overline{F}_m(\nabla \overline{u},h_1,\ldots,h_{m-1})$ is a polynomial, using~\eqref{e.baruhigherreg} and~\eqref{e.appC.C1alphabarw2}  for $m \in \{1,\ldots,M\}$ and $k \in \{1,\ldots,n+1-m\}$, we can actually deduce that 
\begin{equation*} 
R^{k+\beta} \left[ \nabla^k \overline{F}_m(\nabla \overline{u}, \overline{w}_{1},\ldots, \overline{w}_{M}) \right]_{C^{0,\beta}(B_{(1+\eta)R/2})} \leq C \mathsf{E}_m ^{(\delta)} 
\end{equation*}
Therefore, using~\eqref{e.baruhigherreg} once more, we can differentiate the equation of $ \overline{w}_{M+1}$ $k$ times and then show that $ \overline{w}_{M+1}$ satisfies~\eqref{e.appC.C1alphabarw2}. The proof is complete. 
\end{proof}

\section{\texorpdfstring{$C^\infty$}{{C-infty}} regularity for smooth constant-coefficient Lagrangians} 
\label{a.constantcoeff}

In this section we give an alternative proof of the statement that $C^{1,1}$ regularity implies $C^\infty$ regularity for smooth, constant-coefficient Lagrangians. Our argument is similar to the classical argument by Schauder theory, but we keep track of the linearized equations to obtain a Taylor series with an explicit representation of the Taylor polynomials in terms of the linearized equations. We note that it is relatively simple to obtain real analyticity for solutions using this argument.

\begin{proposition} \label{p.cc}
Fix $\ep \in (0,\frac12 ]$, $\mathsf{M} \in [0,\infty)$, $\mathsf{N} \in \N$, $\sigma \in [2,\infty)$, and $R \in (0,\infty]$. Suppose that $\overline{L} \in C^{\mathsf{N}+2,1}(\R^d)$ is uniformly convex, that is, for all $\zeta,\xi \in \R^d$, 
\begin{equation*} 
|\zeta|^2  \leq  D^2 \overline{L}(\xi)\zeta \cdot \zeta \leq \Lambda |\zeta|^2 
\end{equation*}
Let $u \in H^1(B_R)$ solve
$\nabla \cdot D\overline{L}(\nabla u)  = 0$ such that $\left\| \nabla u \right\|_{\underline{L}^2(B_R)} \leq \mathsf{M}$.  Then there exist constants $C(\overline{L},\mathsf{M},\mathsf{N},\ep,\data)$ and polynomials $q_1, \ldots, q_{\mathsf{N+2}}$ such that $q_{m+1}$ is homogeneous polynomial of degree $m+1$ solving 
\begin{equation*} 
 -\nabla \cdot  \left( D^2_p\overline{L}\left( \nabla q_1 \right)   \nabla \frac{q_{m+1}}{m+1}  \right) = \nabla \cdot  \overline{\mathbf{F}}_{m} \left(\nabla q_1 , \nabla 
 \frac{q_{2}}{2} ,\ldots,  \nabla  \frac{q_{m}}{m} \right) \quad  \mbox{in } \R^d ,
\end{equation*}
and, for all $r \in \left(0,\frac R2\right]$,    
\begin{equation*} 
\left\| \nabla u - \nabla \sum_{j=1}^{\mathsf{N}+2} \frac{q_j}{j!} \  \right\|_{\underline{L}^p(B_{r})} \leq C \left( \frac{r}{R} \right)^{\mathsf{N} + 2-\ep}.
\end{equation*}
\end{proposition}

\begin{proof}
Without loss of generality, we may take $R=1$. 
By $C^{1,1}$-estimates, see for example~\cite[Proposition A.1] {AFK}, we have that  
\begin{equation} \label{e.C11app}
 \left| \nabla u(0) \right| \leq C \mathsf{M} \quad \mbox{and} \quad \sup_{r \in (0,3/4)}r^{-1} \left\|  \nabla u - \nabla u(0) \right\|_{\underline{L}^{\infty} \left( B_{r} \right)} \leq C . 
\end{equation}
We set
\begin{equation*} 
q_0 = u(0) \quad \mbox{and} \quad q_1(x) = \nabla u(0) \cdot x. 
\end{equation*}
Assume then inductively that, for $m \in \{1,\ldots,n\}$, there exists homogeneous polynomials $q_m$ of degree $m$ such that, for every $\sigma \in [2,\infty)$ and $\ep \in (0,\tfrac12]$, there exists a constant $\mathsf{N}_{m,\sigma}(\ep,d,\Lambda)$ such that 
\begin{equation*} 
\left| \nabla^{m} q_{m}   \right| \leq \mathsf{N}_{m,2}, \qquad \sup_{r \in \left( 0, \frac {2m+1}{4 m} \right)} r^{-m+\ep} \left\| \nabla u - \nabla \sum_{j=1}^m \frac{q_j}{j!} \right\|_{\underline{L}^{\sigma}(B_r)} \leq \mathsf{N}_{m,\sigma}. 
\end{equation*}
and that, for $m \in \{1,\ldots,n-1\}$, $q_{m+1}$ satisfies the equation 
\begin{equation*} 
 -\nabla \cdot  \left( D^2_p\overline{L}\left( \nabla q_1 \right)   \nabla \frac{q_{m+1}}{m+1}  \right) = \nabla \cdot  \overline{\mathbf{F}}_{m} \left(\nabla q_1 , \nabla 
 \frac{q_{2}}{2} ,\ldots,  \nabla  \frac{q_{m}}{m} \right) \quad  \mbox{in } \R^d ,
\end{equation*}
Let us denote $\ahom:=  D^2_p\overline{L}\left( \nabla q_1 \right)$ and, for $m \in \{1,\ldots,n-1\}$, 
\begin{equation*} 
\overline{w}_{m} := \frac{q_{m+1}}{m+1} .
\end{equation*}

\smallskip

By homogeneity, we find a homogeneous polynomial $q_{n+1}$ of degree $n+1$ solving the equation  
\begin{equation*} 
 -\nabla \cdot  \left( D^2_p\overline{L}\left( \nabla q_1 \right) \nabla  \frac{q_{n+1}}{n+1} \right) 
 = 
 \nabla \cdot  \overline{\mathbf{F}}_{n} \left(\nabla q_1 , \nabla \frac{q_{2}}{2} ,\ldots,  \nabla  \frac{q_{n}}{n} \right) \quad  \mbox{in } \R^d  ,
\end{equation*}
Notice that there is a degree of freedom in the choice of $q_{n+1}$. Namely, the solution is unique up to an $\ahom$-harmonic polynomial of degree $n+1$. We will fix this shortly. To draw parallels between this appendix and Appendix~\ref{app.linerrors}, we set
\begin{equation*} 
\overline{w}_{n} := \frac{q_{n+1}}{n+1} \quad \mbox{and} \quad \xi_m := u - q_1 -  \sum_{k=1}^{m} \frac{\overline{w}_k}{k!} = u - \sum_{k=1}^{m+1} \frac{q_k}{k!}.
\end{equation*}
 Rewrite
\begin{align} \notag 
\lefteqn{
D_p \overline{L}\left( \nabla u \right) - D_p \overline{L}\left( \nabla q_1 \right)  
-
D^2 \overline{L}\left( \nabla q_1 \right) \nabla \xi_n
} \quad &
\\ \notag & 
=  
\sum_{m=1}^{n} \frac{1}{m!}   \left(D^2 \overline{L}\left( \nabla q_1 \right) \nabla \overline{w}_{m} 
+ \overline{\mathbf{F}}_m \left(\nabla q_1,\nabla \overline{w}_1,\ldots,\nabla \overline{w}_{m-1} \right) \right) 
+ \overline{\mathbf{E}}_{n}
\end{align} 
where
\begin{align} \notag 
\overline{\mathbf{E}}_{n} & :=  \sum_{m=2}^{n} \frac{1}{m!}   \left( D_p^{m+1} \overline{L}(\nabla q_1) (\nabla u - \nabla q_1)^{\otimes m} -   \overline{\mathbf{F}}_{m} \left(\nabla q_1 , \nabla \overline{w}_{1},\ldots,  \nabla  \overline{w}_{m-1} \right)  \right)
\\ \notag & 
\quad 
+ D_p \overline{L}(\nabla u) - \sum_{k=0}^{n+1} \frac1{k!} D_p^{k+1} \overline{L}(\nabla q_1) (\nabla u - \nabla q_1)^{\otimes k}   .
\end{align}
By the estimate in Appendix~\ref{app.linerrors}, we have that
\begin{align*} 
\left| \overline{\mathbf{E}}_{n} \right|
&
\leq
 C \sum_{h=0}^{n-1} \left| \nabla \xi_h  \right| \left( \left| \nabla \xi_0 \right| + \sum_{i=1}^{n-1} \left| \nabla \frac{\overline{w}_{i} }{i!} \right|^{\frac{1}{i}}  \right)^{n-h} 
 .
\end{align*}
Taking divergence gives us, by the equations of $u$, $ \overline{w}_{1},\ldots,  \overline{w}_{n} $, that 
\begin{equation*} 
-\nabla \cdot \ahom \nabla \xi_n  = \nabla \cdot \overline{\mathbf{E}}_{n}. 
\end{equation*}
Using the induction assumption we get that 
\begin{equation*} 
 \left\| \overline{\mathbf{E}}_{n}  \right\|_{\underline{L}^{\sigma}(B_r)} \leq C r^{n+1 - \ep}.
\end{equation*}
Now Lemma~\ref{l.Aharmdecay} below allows us to identify the homogeneous $\ahom$-harmonic polynomial part of $q_{n+1}$ of degree $n+1$ such that 
\begin{equation*} 
\sup_{r \in \left( 0, \frac {2(n+1)+1}{4(n+1)} \right)} r^{-(n+1)+\ep} \left\| \nabla u - \nabla \sum_{j=1}^{n+1} \frac{q_j}{j!} \right\|_{\underline{L}^{\sigma}(B_r)} \leq C 
\end{equation*}
and
\begin{equation*} 
\left| \nabla^{n+1} q_{n+1}   \right| \leq C.
\end{equation*}
This proves the induction step, and finishes the proof. 
\end{proof}

\begin{lemma} \label{l.Aharmdecay}
Let $n \in \N$ and let $\alpha \in (n,n+1)$. Let $\mathsf{M} \in [0,\infty)$ and $\ep  \in (0,1)$.   
Suppose that $\mathbf{A} $ is a constant symmetric matrix having eigenvalues on the interval $[1,\Lambda]$. There is a constant $C(n,\alpha,\ep,d,\Lambda)$ such that if  $\mathbf{F} \in L^p(B_1)$ and $u \in H^1(B_1)$ solve 
\begin{equation*} 
\nabla \cdot \mathbf{A} \nabla u=  \nabla \cdot \mathbf{F},
\end{equation*}
and that $\mathbf{F} \in L^p(B_1)$ satisfies, for $r \in (0,1)$,
\begin{equation}  \label{e.Aharmdecay.cond1}
\left\| \mathbf{F} \right\|_{\underline{L}^p(B_r)} \leq  \mathsf{M}  r^{\alpha},
\end{equation}
then there is $\mathbf{A}$-harmonic $q \in \mathcal{P}_{n+1}$ such that 
\begin{equation*} 
\sup_{r \in (0,1-\ep)} r^{-\alpha} \left\| \nabla u - \nabla q \right\|_{L^p(B_r)}  \leq C \left( \left\| \nabla u  \right\|_{L^2(B_1)}  +  \mathsf{M}\right).
\end{equation*}

\end{lemma}

\begin{proof}
We proceed via harmonic approximation. Let $v_r \in u + H_0^1(B_r)$  be $A$-harmonic. Denote
\begin{equation*} 
 \Ahom_{n+1} := \{ p \in \mathcal{P}_{n+1} \, : \, p \mbox{ is }  A-\mbox{harmonic}\}
\end{equation*}
By Calder\'on-Zygmund estimates and~\eqref{e.Aharmdecay.cond1}, 
\begin{equation*} 
\left\| \nabla v_r -  \nabla u  \right\|_{\underline{L}^2(B_r)} \leq C \mathsf{M} r^{\alpha}. 
\end{equation*}
Using the oscillation decay estimate
\begin{equation*} 
\inf_{ \tilde q \in \Ahom_{n+1}} \left\| \nabla v_r-  \nabla \tilde q \right\|_{L^\infty(B_{r})} \leq C \theta^{n+1} \inf_{ \tilde q \in \Ahom_{n+1}} \left\| \nabla v_r-  \nabla \tilde q \right\|_{\underline{L}^2(B_r)}
\end{equation*}
and defining
\begin{equation*} 
D(r) := r^{-\alpha} \inf_{ \tilde q \in \Ahom_{n+1}} \left\| \nabla u -  \nabla \tilde q \right\|_{\underline{L}^p(B_{r})},
\end{equation*}
we obtain by the triangle inequality, for $\theta>0$ such that $ C \theta^{n+1-\alpha} = \frac12$, that
\begin{equation*} 
D(\theta r)  \leq \frac12 D(r) + C \mathsf{M} .
\end{equation*}
It follows by reabsorption that, for $r \in \left( 0,1\right)$, 
\begin{equation*} 
\sup_{t  \in (0,r)} D(t)  \leq C \left(D(r)  + \mathsf{M} \right) .
\end{equation*}
In particular, letting $\tilde q_t $ be the minimizing element of $\Ahom_{n+1}$ in the definition of $D(t)$, we get by the triangle inequality that
\begin{equation*} 
t^{-n}\left\| \nabla \tilde q_{t/2} -   \nabla \tilde q_t \right\|_{\underline{L}^p(B_{t})} \leq C t^{\alpha-n} \left(D(t/2) + D(t)   \right) \leq C t^{\alpha-n}  \left(D(r)  + \mathsf{M} \right) .
\end{equation*}
This allows us to identify $q \in \Ahom_{n+1}$ such that, for $t  \in (0,r)$, 
\begin{equation*} 
\sum_{j=1}^{n+1} t^{n+1-j} \left| \nabla^{j} \tilde q_{t}(0) -  \nabla^{j} q(0) \right|  \leq  C t^{\alpha-n} \left(D(1)  +  \mathsf{M}  \right) ,
\end{equation*}
and it follows that
\begin{equation*} 
\left\| \nabla u -  \nabla  q \right\|_{\underline{L}^p(B_{r})} \leq C r^{\alpha}    \left(D(1-\ep)  +  \mathsf{M} \right) .
\end{equation*}
The proof is complete by an easy estimate $ D(1-\ep) \leq C_\ep(\left\| \nabla u  \right\|_{\underline{L}^2 \left( B_{1} \right)} +  \mathsf{M} )$. 
\end{proof}

\subsection*{Acknowledgments}
SA was partially supported by the NSF grant DMS-1700329. SF was partially supported by NSF grants DMS-1700329 and DMS-1311833. T.K. was supported by the Academy of Finland and the European Research Council (ERC) under the European Union's Horizon 2020 research and innovation programme (grant agreement No 818437).

\small
\bibliographystyle{abbrv}
\bibliography{linearization}

\end{document}